\documentclass[12pt, a4paper, twoside]{report}

\usepackage[ngerman,english]{babel}	

\usepackage[utf8]{inputenc}
\usepackage[T1]{fontenc}

\usepackage{titling}
\newcommand{\subtitle}[1]{%
  \posttitle{%
    \par\end{center}
    \begin{center}\large#1\end{center}
    \vskip0.5em}%
}

\usepackage[headsepline,nouppercase]{scrpage2} 
\ohead[]{}
\ihead[]{}
\usepackage{graphicx}
\usepackage{amsmath} 
\usepackage{amssymb} 
\usepackage{amsfonts}
\usepackage{amsthm}
\usepackage{enumerate}
\usepackage{makeidx} 
\usepackage{verbatim}
\usepackage{epsfig}
\usepackage{hyperref} 
\usepackage{dsfont}
\usepackage{wasysym} 
\usepackage{caption}	
\usepackage{float}
\usepackage{stackrel}	
\usepackage{mathtools} 
\usepackage{ltablex}	
\usepackage{nicefrac}
\usepackage{stmaryrd}	

\usepackage{hyperref}
\hypersetup{colorlinks=true,
						linkcolor=black,
						citecolor=black,
						urlcolor=blue}

\makeindex

\makeindex      
     
\newlength{\fixboxwidth}     
\newcommand{\R}{{\mathbb R}}
     
\newcommand{\N}{{\mathbb N}}    
\newcommand{\Z}{{\mathbb Z}}     

\renewcommand{\P}{\mathbb P}

\newcommand\expect{\operatorname{\mathbb{E}}}
\newcommand{\Torus}{{\mathbb T}}
\newcommand\Borel{{\mathcal B}}
\newcommand{\Linop}{{\mathcal L}}
\newcommand{\Hilbert}{{\mathcal H}}

\newcommand{\ind}{\mathds{1}}
\newcommand\euler{{\mathrm e}}
\newcommand\eps{\varepsilon}

\newcommand{\cost}{\operatorname{cost}}      

\newcommand\Boole{\circ\shortmid}
\newcommand\mon{\mathrm{mon}}

\newcommand\ran{\mathrm{ran}}     
\newcommand\deter{\mathrm{det}}
\newcommand\avg{\mathrm{avg}}
\newcommand\nonada{\mathrm{nonada}}
\newcommand\non{\mathrm{non}}
\newcommand\ada{\mathrm{ada}}
\newcommand\lin{\mathrm{lin}}
\newcommand\homog{\mathrm{hom}}
\newcommand\normal{\mathrm{normal}}

\DeclareMathOperator\diam{diam}
\DeclareMathOperator\App{APP}
\DeclareMathOperator\Int{INT}
\DeclareMathOperator*\argmin{argmin} 
\newcommand\Lall{\Lambda^{\mathrm{all}}}
\newcommand\Lstd{\Lambda^{\mathrm{std}}}

\DeclareMathOperator\Vol{Vol}
\DeclareMathOperator\dist{dist}
\DeclareMathOperator\supp{supp}
\DeclareMathOperator\trace{tr}
\DeclareMathOperator\rank{rk}
\DeclareMathOperator\image{img}
\DeclareMathOperator*\esssup{ess\,sup}
\DeclareMathOperator\id{id}

\DeclareMathOperator\sgn{sgn}
\DeclareMathOperator\abs{abs}
\DeclareMathOperator\Uniform{unif}

\DeclareMathOperator\Cov{Cov}
\DeclareMathOperator\modulo{mod}
\newcommand\LHS{\mathrm{LHS}}
\newcommand\RHS{\mathrm{RHS}}
\newcommand\Korobov{\mathrm{Kor}}

\DeclareMathOperator\card{card}
\DeclareMathOperator\linspan{span}     

\newcommand{\diff}{\mathrm{d}} 

\newcommand{\zeros}{\boldsymbol{0}}
\newcommand{\ones}{\boldsymbol{1}}
\newcommand\vecc{\mathbf{c}}
\newcommand\vece{\mathbf{e}}
\newcommand\veci{\mathbf{i}}
\newcommand\vecj{\mathbf{j}}
\newcommand\veck{\mathbf{k}}
\newcommand\vecm{\mathbf{m}}
\newcommand\vecn{\mathbf{n}}
\newcommand\vecr{\mathbf{r}}
\newcommand\vecu{\mathbf{u}}
\newcommand\vecv{\mathbf{v}}
\newcommand\vecw{\mathbf{w}}
\newcommand\vecx{\mathbf{x}}
\newcommand\vecy{\mathbf{y}}
\newcommand\vecz{\mathbf{z}}
\newcommand\vecX{\mathbf{X}}
\newcommand\vecY{\mathbf{Y}}
\newcommand\vecZ{\mathbf{Z}}

\newcommand\vecalpha{\boldsymbol{\alpha}}
\newcommand\vecbeta{\boldsymbol{\beta}}
\newcommand\veckappa{\boldsymbol{\kappa}}
\newcommand\veclambda{\boldsymbol{\lambda}}
\newcommand\vecsigma{\boldsymbol{\sigma}}
\newcommand\vecxi{\boldsymbol{\xi}}

\newcommand{\rd}{\,\mathrm{d}} 

\usepackage{microtype} 

\usepackage{bookmark}
\makeatletter 
  \providecommand*{\toclevel@author}{999}
  \providecommand*{\toclevel@title}{0}
\makeatother

\theoremstyle{plain}       
     
\newtheorem{theorem}{Theorem}[chapter]      
\newtheorem{corollary}[theorem]{Corollary}      
\newtheorem{lemma}[theorem]{Lemma}     
\newtheorem{proposition}[theorem]{Proposition}      
     
\theoremstyle{definition}     
     
\newtheorem{definition}[theorem]{Definition}

\newtheorem{example}[theorem]{Example}     
\newtheorem{remark}[theorem]{Remark}

\numberwithin{equation}{section}   

\numberwithin{theorem}{chapter}

\newcounter{proofstep}[theorem]
\newcommand{\proofstep}[2]{\refstepcounter{proofstep} \label{#1}
												\textbf{Step~\arabic{proofstep}: #2 \\}}
\newcounter{proofsubstep}[proofstep]
\newcommand{\proofsubstep}[2]{\refstepcounter{proofsubstep} \label{#1}
												\textit{Step~\arabic{proofstep}.\arabic{proofsubstep}: #2 \\}}

\newcommand{\proofstepref}[1]{\ref{#1}}

\newcommand{\thmref}[1]{\hyperref[#1]{Theorem~\ref*{#1}}}
\newcommand{\lemref}[1]{\hyperref[#1]{Lemma~\ref*{#1}}}
\newcommand{\propref}[1]{\hyperref[#1]{Proposition~\ref*{#1}}}
\newcommand{\corref}[1]{\hyperref[#1]{Corollary~\ref*{#1}}}
\newcommand{\remref}[1]{\hyperref[#1]{Remark~\ref*{#1}}}
\newcommand{\chapref}[1]{\hyperref[#1]{Chapter~\ref*{#1}}}
\newcommand{\secref}[1]{\hyperref[#1]{Section~\ref*{#1}}}
\newcommand{\defref}[1]{\hyperref[#1]{Definition~\ref*{#1}}}
\newcommand{\exref}[1]{\hyperref[#1]{Example~\ref*{#1}}}

\newcommand{\Kapref}[1]{\hyperref[#1]{Kapitel~\ref*{#1}}}
\newcommand{\Abschref}[1]{\hyperref[#1]{Abschnitt~\ref*{#1}}}

\begin{document}

\pagenumbering{alph}
\pagestyle{empty}



\begin{center}

\includegraphics[width = 4.72cm, height = 5.55cm]{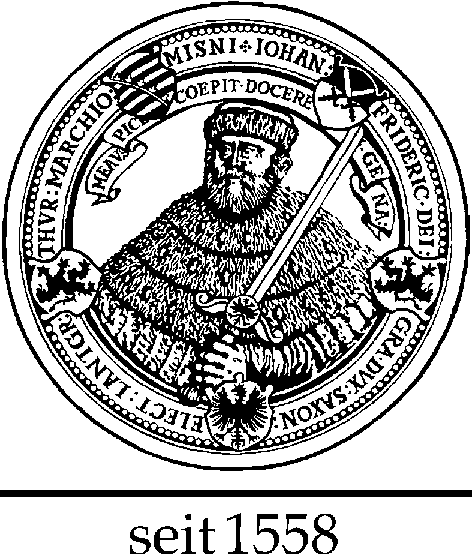}

\vspace{2cm}

{\textbf{\LARGE{%
High-Dimensional\\
Function Approximation:\\
Breaking the Curse\\
\vspace{0.2cm}
with Monte Carlo Methods
}}}

\vspace{3cm}
\textbf{D I S S E R T A T I O N}\\
\textit{zur Erlangung des akademischen Grades\\
doctor rerum naturalium (Dr.\ rer.\ nat.)}\\
\vspace*{\fill}
vorgelegt dem Rat der\\
Fakult\"at f\"ur Mathematik und Informatik\\
der Friedrich-Schiller-Universit\"at Jena\\
\vspace*{\fill}
von M.Sc.\ Robert Kunsch\\
geboren am 19.\ Februar 1990 in Gotha
\end{center}

\newpage

\vspace*{\fill}
\textbf{Gutachter}
\begin{enumerate}
	\item ...
	\item ...
	\item ...
\end{enumerate}

\vspace{1cm}
\textbf{Tag der \"offentlichen Verteidigung}: ...
\newpage

\chapter*{Acknowledgements}

First of all, I wish to express my deepest gratitude
to my supervisor Prof.~Dr.~Erich Novak.
His sincere advice, not only on scientific matters,
carried me forward in so many ways during my time as his student.
He helped me to realize when attempts to solve a problem might be unsuccesful.
Even more often his encouragement led me to pursue new ideas.
Not to forget, the numerous hints on the written text
allowed the dissertation to evolve.

I am also grateful to my many colleagues for all their support.
Especially David Krieg,
who meticulously read large portions of the manuskript.
Hopefully, his remarks led to a more comprehensive text.
I would also like to thank
Prof.~Dr.~Stefan Heinrich, Prof.~Dr.~Aicke Hinrichs, Prof.~Dr.~Winfried Sickel,
Prof.~Dr.~Henryk Wo\'zniakowski,
Dr.~Daniel Rudolf, Dr.~Mario Ullrich, Dr.~Tino Ullrich, Dr.~Jan Vyb{\'\i}ral,
and my fellow students Glenn Byrenheid and Dr.~Van Kien Nguyen,
for interesting and encouraging scientific discussions during conferences and seminars.
If I were to include all the friendly associates at mathematical meetings,
this list would extend considerably.

Finally, I wish to thank my friends Nico Kennedy and Rafael Silveira
for reading the english introduction for linguistic aspects,
never deterred from the formulas, even as non-mathematicians.
I almost forgot to mention Kendra Horner who helped polish this very text.

\thispagestyle{empty}
\newpage
\phantom{empty page}
\thispagestyle{empty}
\newpage


\selectlanguage{ngerman}
\pagestyle{scrheadings}
\pagenumbering{roman}
\manualmark
\cleardoublepage

\chapter*{Zusammenfassung}
\addcontentsline{toc}{chapter}{Zusammenfassung}
\markboth{Zusammenfassung}{Zusammenfassung}

F\"ur viele Probleme, die in wissenschaftlich-technischen Anwendungen auftreten,
ist es praktisch unm\"oglich, exakte L\"osungen zu finden.
Stattdessen sucht man N\"aherungsl\"osungen
mittels Verfahren, die in endlich vielen Schritten umsetzbar sind.
Insbesondere muss man mit unvollst\"andiger Information
\"uber die jeweilige Probleminstanz arbeiten;
abgesehen von Strukturannahmen (dem sogenannten \mbox{\emph{a priori}}-Wissen)
k\"onnen wir nur endlich viel Information sammeln,
typischerweise in Form von $n$~reellen Zahlen
aus Messungen oder vom Nutzer bereitzustellenden Unterprogrammen.
Es besteht wachsendes Interesse an der L\"osung hochdimensionaler Probleme,
das sind Probleme mit Funktionen,
die auf $d$-dimensionalen Gebieten definiert sind.
Wir untersuchen die \emph{Informationskomplexit\"at}~\mbox{$n(\eps,d)$},
die minimal ben\"otigte Anzahl von Informationen,
um ein Problem bis auf einen Fehler~\mbox{$\eps > 0$} zu l\"osen.
\emph{Tractability}
ist die Frage nach dem Verhalten dieser Funktion~\mbox{$n(\eps,d)$},
also nach der \emph{Durchf\"uhrbarkeit} einer Aufgabe.
In vielen F\"allen w\"achst die Komplexit\"at exponentiell in~$d$
f\"ur festgehaltenes~$\eps$,
wir sprechen dann vom \emph{Fluch der Dimension}.
Es gibt im Grunde zwei Wege, dem Fluch der Dimension zu begegnen --
so er denn auftritt.
Die eine Variante ist, mehr \emph{a priori}-Wissen einzubeziehen
und so die Menge der denkbaren Eingabegr\"o{\ss}en einzugrenzen.
Die andere M\"oglichkeit besteht in der Erweiterung der Klasse
zul\"assiger Algorithmen.
In dieser Dissertation liegt der Fokus auf dem zweiten Ansatz,
und zwar untersuchen wir das Potential von Randomisierung
f\"ur die Approximation von Funktionen.

Ein $d$-dimensionales Approximationsproblem ist eine Identit\"atsabbildung
\begin{equation*}
	\App: F^d \hookrightarrow G^d, \quad f \mapsto f \,,
\end{equation*}
mit einer \emph{Inputmenge}~$F^d$, die $d$-variate reellwertige Funktionen enth\"alt,
und einem normierten Raum~$G^d$.
Funktionen aus~$F^d$ sind in dieser Arbeit f\"ur gew\"ohnlich
auf dem $d$-dimensionalen Einheitsw\"urfel~\mbox{$[0,1]^d$} definiert,
der Zielraum~$G^d$ ist dann beispielsweise
\mbox{$L_1([0,1]^d)$} oder~\mbox{$L_{\infty}([0,1]^d)$}.
Deterministische Algorithmen sind Abbildungen
\mbox{$A_n = \phi \circ N : F^d \rightarrow G^d$},
wobei die \emph{Informationsabbildung}
\begin{equation*}
	N : F^d \rightarrow \R^n , \quad f \mapsto (L_1(f),\ldots,L_n(f)) \,,
\end{equation*}
mittels $n$~linearer Funktionale~$L_i$
Information \"uber die Probleminstanz~$f$ sammelt.
Der Fehler des Verfahrens ist durch den ung\"unstigsten Fall bestimmt,
\begin{equation*}
	e(A_n,F^d) := \sup_{f \in F^d} \|f - A_n(f)\|_{G^d} \,.
\end{equation*}
Randomisierte Verfahren sind Familien~%
\mbox{$(A_n^{\omega})_{\omega \in \Omega}$}
von Abbildungen~\mbox{$A_n^{\omega}:
												F^d \rightarrow G^d$}
obiger Struktur,
indiziert durch ein Zufallselement~\mbox{$\omega \in \Omega$}
aus einem Wahrscheinlichkeitsraum~\mbox{$(\Omega,\Sigma,\P)$}.
Der Fehler eines solchen \emph{Monte-Carlo}-Algorithmus
ist der erwartete Fehler f\"ur den schlechtesten Input,
\begin{equation*}
	e((A_n^{\omega})_{\omega},F^d)
		:= \sup_{f \in F^d} \expect \|f - A_n^{\omega}(f)\|_{G^d} \,.
\end{equation*}
In beiden F\"allen wird eine signifikante Verkleinerung
des \emph{Anfangsfehlers} angestrebt,
\begin{equation*}
	e(0,F^d) := \inf_{g \in G^d} \sup_{f \in F^d} \|f - g\|_{G^d} \,,
\end{equation*}
welcher bereits ohne Information erreichbar ist.
Insbesondere interessieren wir uns f\"ur den Vergleich
der Komplexit\"at im deterministischen und randomisierten Fall,
\begin{align*}
	n^{\deter}(\eps,d) &:= \inf\{n \in \N_0 \mid
														\exists A_n: e(A_n,F^d) \leq \eps\} \,,\\
	n^{\ran}(\eps,d) &:=	\inf\{n \in \N_0 \mid
														\exists (A_n^{\omega})_{\omega}:
															e((A_n^{\omega})_{\omega},F^d) \leq \eps\} \,,
\end{align*}
wobei~\mbox{$0 < \eps < e(0,F^d)$}.
S\"amtliche in dieser Arbeit angegebene Algorithmen
sind \emph{nichtadaptiv} mit der einfachen Struktur von~$N$ siehe oben.
Untere Fehlerschranken werden f\"ur wesentlich allgemeinere Verfahren gezeigt,
welche die Information auch \emph{adaptiv} sammeln
oder eine \emph{ver\"anderliche Kardinalit\"at}~\mbox{$n(\omega,f)$} aufweisen.
Zu diesen Begriffen und einer ausf\"uhrlichen Einf\"uhrung
in das Themengebiet der \emph{Informationskomplexit\"at},
siehe \Kapref{chap:basics}.

Neue Resultate sind in den Kapiteln~2--4 enthalten,
welche mehr oder weniger f\"ur sich stehende Themen behandeln.
\Kapref{chap:Bernstein} befasst sich mit unteren Schranken
f\"ur randomisierte Verfahren,
mittels derer sich f\"ur verschiedene Beispiele zeigen l\"asst,
dass Monte-Carlo-Methoden nicht viel besser
als deterministische Algorithmen sein k\"onnen.
Im Gegensatz dazu ist \Kapref{chap:Hilbert} der Suche nach Problemen gewidmet,
wo deterministische Algorithmen unter dem Fluch der Dimension leiden,
Randomisierung diesen jedoch auf recht eindrucksvolle Weise zu brechen vermag.
\Kapref{chap:monotone} bespricht ein konkretes Problem
f\"ur welches Zufallsalgorithmen zwar den Fluch aufheben,
das Problem aber trotzdem noch sehr schwer ist.

\subsubsection{Zu \Kapref{chap:Bernstein}:
	Untere Schranken f\"ur lineare Probleme mittels Bernstein-Zahlen}

Das Hauptergebnis dieses Kapitels stellen untere Schranken
f\"ur den Fehler von Monte-Carlo-Algorithmen
f\"ur allgemeine lineare Probleme
\begin{equation*}
	S: F \rightarrow G
\end{equation*}
dar, d.h.\ $S$ ist ein linearer Operator
zwischen normierten R\"aumen~$\widetilde{F}$ und~$G$,
zudem ist die Inputmenge~$F$ die Einheitskugel in~$\widetilde{F}$.
Es wird gezeigt,
dass f\"ur jede (adaptive) Monte-Carlo-Methode~\mbox{$(A_n^{\omega})_{\omega}$},
welche $n$~beliebige stetige lineare Funktionale~$L_i$ zur Informationsgewinnung einsetzt,
die Absch\"atzung
\begin{equation*}
	e((A_n^{\omega})_{\omega},F) \geq \textstyle{\frac{1}{30}} \, b_{2n}(S)
\end{equation*}
gilt, wobei $b_m(S)$ die $m$-te Bernstein-Zahl des Operators~$S$ ist,
siehe \thmref{thm:BernsteinMCada}.
Der Beweis basiert auf einem Ergebnis von Heinrich~\cite{He92},
welches Normerwartungswerte von Gau{\ss}-Ma{\ss}en
in Beziehung zum Monte-Carlo-Fehler setzt.
Die Neuerung besteht in der Anwendung des Theorems von Lewis
f\"ur die Wahl eines optimalen Gau{\ss}-Ma{\ss}es.
Dieses Ergebnis wurde in~\cite{Ku16} angek\"undigt
und ein kurzer Beweis ohne explizite Konstanten aufgef\"uhrt.

In \Abschref{sec:Cinf->Linf}
wird dieses allgemeine Werkzeug f\"ur die $L_{\infty}$-Approximation
bestimmter Klassen von $C^{\infty}$-Funktionen angewandt.
Wir betrachten das Problem
\begin{equation*}
	\App : F_p^d \hookrightarrow L_{\infty}([0,1]^d)
\end{equation*}
mit der Inputmenge
\begin{equation*}
	\begin{split}
		F_p^d := \{f \in C^{\infty}([0,1]^d)
								\, \mid \,&
								\|\nabla_{\vecv_k} \cdots \nabla_{\vecv_1} f\|_{\infty}
									\leq |\vecv_1|_p \cdots |\vecv_k|_p\\
								&\text{f\"ur } k \in \N_0, \,
								\vecv_1,\ldots,\vecv_k \in \R^d\} \,,
	\end{split}
\end{equation*}
wobei~\mbox{$\nabla_{\vecv} f$} die Richtungsableitung
entlang eines Vektors~\mbox{$\vecv \in \R^d$} bezeichnet
und \mbox{$|\vecv|_p$} die $p$-Norm von~\mbox{$\vecv \in \ell_p^d$},
\mbox{$1 \leq p \leq \infty$}.
\"Uber die dazugeh\"origen Bernstein-Zahlen erhalten wir
die untere Schranke
\begin{equation*}
	n^{\ran}(\eps,d,p)
		> 2^{\lfloor \frac{d^{1/p}}{3} \rfloor - 1}
	\quad \text{f\"ur\, $0 < \eps \leq {\textstyle \frac{1}{30}}$,}
\end{equation*}
siehe \corref{cor:LinfAppLB}.
F\"ur~\mbox{$p=1$} ergibt sich daraus der Fluch der Dimension,
auch bei Randomisierung.
Die Beweistechnik zur Bestimmung der Bernstein-Zahlen
ist von Novak und Wo\'zniakoswki~\cite{NW09b} bekannt,
welche den Fluch der Dimension f\"ur den Fall~\mbox{$p=1$}
im deterministischen Szenario gezeigt haben.

Eine einfache Taylor-Approximation liefert obere Schranken
f\"ur die Komplexit\"at mittels deterministischer Verfahren,
\begin{equation*}
	n^{\ran}(\eps,d,p)
		\,\leq\, n^{\deter}(\eps,d,p)
		\,\leq\, \exp\left(\euler \, \log(d+1)
										\, 
										\max\left\{\log {\textstyle \frac{1}{\eps}},\,
																			d^{1/p}
												\right\}
								\right) \,,
\end{equation*}
siehe \thmref{thm:Cinf->LinfUB}.
Die algorithmische Idee stammt von Vyb{\'\i}ral~\cite{Vyb14},
welcher ein Problem untersucht hat,
das dem Fall \mbox{$p=\infty$} nahekommt.

F\"ur dieses Beispiel beobachten wir grob gesprochen
eine exponentielle Abh\"angigkeit
der Komplexit\"at von~$d^{1/p}$.
Dies kann durch Randomisierung nicht verbessert werden.
Das betrachtete Problem ist zudem ein Beispiel daf\"ur,
wie eine Einschr\"ankung der Input-Menge die Durchf\"uhrbarkeit
der Approximation beeinflusst.

\subsubsection{Zu \Kapref{chap:Hilbert}:
	Gleichm\"a{\ss}ige Approximation von Funktionen aus einem Hilbert-Raum}

Dieses Kapitel enth\"alt einen neuen Monte-Carlo-Ansatz
f\"ur die $L_{\infty}$-Approximation
von Funktionen aus einem Hilbert-Raum~$\Hilbert$ mit reproduzierendem Kern.
Die Menge~$F$ der Eingabegr\"o{\ss}en ist die Einheitskugel in~$\Hilbert$,
d.h.\ f\"ur eine Orthonormalbasis~\mbox{$(\psi_k)_{k \in \N}$} von~$\Hilbert$
haben wir
\begin{equation*}
	F := \Bigl\{\sum_{k=1}^{\infty} a_k \, \psi_k\
								\,\Big|\,
									a_k \in \R, \sum_{k=1}^{\infty} a_k^2 \leq 1
				\Bigr\} \,.
\end{equation*}
Die Idee f\"ur den neuen Algorithmus
basiert auf einer fundamentalen Monte-Carlo-Approximationsmethode
nach Math\'e~\cite{Ma91},
siehe auch \Abschref{sec:HilbertFundamental}.
Jene wurde in der Originalarbeit zur Rekonstruktion
in endlichdimensionalen Folgenr\"aumen
\mbox{$\ell_2^m \hookrightarrow \ell_q^m$}, \mbox{$q > 2$},
angewandt und diente in Verbindung mit Diskretisierungstechniken
f\"ur Funktionenr\"aume
der Bestimmung der Konvergenzordnung f\"ur Einbettungsoperatoren.
In \Abschref{sec:HilbertPlainMCUB} verfolgen wir einen direkteren Ansatz
\"uber die lineare Monte-Carlo-Methode
\begin{equation*}
	A_n^{\omega}(f)
		:= \frac{1}{n}
					\sum_{i=1}^n
						L_i^{\omega}(f) \, g_i^{\omega} \,,
\end{equation*}
wobei
\begin{equation*}
	L_i^{\omega}(f)
		:= \sum_{k=1}^{\infty} X_{ik} \, \langle \psi_k,f \rangle_{\Hilbert},
		\quad\text{und}\quad
	g_i^{\omega}
		:= \sum_{k=1}^{\infty} X_{ik} \, \psi_k \,,
\end{equation*}
mit unabh\"angigen standardnormalverteilten Zufallsvariablen $X_{ik}$.
Die Funktionen~$g_i^{\omega}$ sind unabh\"angige Realisierungen
des mit~$\Hilbert$ assoziierten \emph{Gau{\ss}-Feldes}~$\Psi$,
die Kovarianzfunktion von~$\Psi$
ist der reproduzierende Kern von~$\Hilbert$.
F\"ur dieses Verfahren gilt die Fehlerabsch\"atzung
\begin{equation*}
	e((A_n^{\omega})_{\omega},F)
		\leq \frac{2 \, \expect \|\Psi\|_{\infty}}{\sqrt{n}} \,.
\end{equation*}
Zugegebenerma{\ss}en sind die zuf\"alligen Funktionale~$L_i^{\omega}$
unstetig mit Wahrscheinlichkeit~$1$,
f\"ur festes~\mbox{$f \in \Hilbert$} jedoch sind
die Werte~\mbox{$L_i^{\omega}(f)$} zentrierte Gau{\ss}-Variablen
mit Varianz~\mbox{$\|f\|_{\Hilbert}^2$}
und somit fast sicher endlich.
Das Verfahren~$(A_n^{\omega})_{\omega}$ kann allerdings auch als Grenzwert von Methoden
gesehen werden, die fast sicher stetige Funktionale verwenden,
siehe \lemref{lem:stdMCapp}.

Mit Werkzeugen aus der Stochastik,
siehe \Abschref{sec:E|Psi|_sup} f\"ur eine Zusammenstellung,
k\"onnen wir den Wert~\mbox{$\expect \|\Psi\|_{\infty}$} absch\"atzen,
sofern die zuf\"allige Funktion~$\Psi$ beschr\"ankt ist.
Insbesondere mit der Technik \emph{majorisierender Ma{\ss}e} nach Fernique
l\"asst sich der Fall periodischer Funktionen
auf dem~$d$-dimensionalen Torus~$\Torus^d$ angehen,
siehe \Abschref{sec:HilbertPeriodic}.
Hierbei ist $\Torus$ das Intervall~$[0,1]$ mit identifizierten Randpunkten.
Im eindimensionalen Fall bezeichnen wir mit \mbox{$\Hilbert_{\veclambda}(\Torus)$}
den Raum mit Orthonormalbasis
\begin{equation*}
	\left\{\lambda_0,\,
				\lambda_k \, \sin(2 \, \pi \, k \,\cdot),\,
				\lambda_k \, \cos(2 \, \pi \, k \,\cdot)
	\right\}_{k \in \N} \,,
\end{equation*}
wobei~\mbox{$\lambda_k > 0$}.
Der $d$-variate Fall ist \"uber das Tensorprodukt definiert,
\begin{equation*}
	\Hilbert_{\veclambda}(\Torus^d)
		:= \bigotimes_{j=1}^{d} \Hilbert_{\veclambda}(\Torus) \,.
\end{equation*}
Wir nehmen \mbox{$\sum_{k=0}^{\infty} \lambda_k^2 \stackrel{\text{!}}{=} 1$} an,
sodass der Anfangsfehler konstant~$1$ ist.
Unter diesen Voraussetzungen gilt der Fluch der Dimension f\"ur
deterministische Verfahren, siehe~\thmref{thm:curseperiodic}.
Die deterministische untere Schranke basiert auf einer Beweistechnik
von Kuo, Wasilkowski und Wo\'zniakowski~\cite{KWW08},
ebenso Cobos, K\"uhn und Sickel~\cite{CKS16},
siehe \Abschref{sec:HilbertWorLB}.
Im randomisierten Fall leiten wir Bedingungen
an den reproduzierenden Kern periodischer Hilbert-R\"aume ab,
f\"ur die der assoziierte Gau{\ss}-Prozess beschr\"ankt ist.
Im Speziellen betrachten wir Korobov-R\"aume
\mbox{$H_r^{\Korobov}(\Torus^d) = \Hilbert_{\veclambda}(\Torus^d)$}
mit~\mbox{$\lambda_k := \sqrt{\beta_1} \, k^{-r}$} f\"ur~\mbox{$k \in \N$}.
Hierbei ist \mbox{$\beta_1 > 0$} so gew\"ahlt,
dass der Anfangsfehler immer noch
\"uber die Wahl von \mbox{$0 < \lambda_0 < 1$} angepasst werden kann.
F\"ur Glattheit~\mbox{$r > 1$} l\"asst sich zeigen,
dass das Approximationsproblem
\begin{equation*}
	\App: H_r^{\Korobov}(\Torus^d) \hookrightarrow L_{\infty}(\Torus^d)
\end{equation*}
eine polynomiell beschr\"ankte Monte-Carlo-Komplexit\"at besitzt,
\begin{equation*}
	n^{\ran}(\eps,d,r) \leq C_r \, d \, (1 + \log d) \, \eps^{-2} \,,
\end{equation*}
wobei \mbox{$C_r > 0$}.
F\"ur weniger Glattheit, konkret \mbox{$\frac{1}{2} < r \leq 1$},
k\"onnen wir immer noch Durchf\"uhrbarkeit der Approximation
mit polynomiell beschr\"anktem Aufwand (\emph{polynomial tractability}) zeigen,
wobei die Schranken f\"ur die Komplexit\"at schlechter werden.
Hierbei wird die fundamentale Monte-Carlo-Methode
nur noch auf endlichdimensionale Teilr\"aume
von~\mbox{$H_r^{\Korobov}(\Torus^d)$} angewandt,
siehe~\thmref{thm:Korobov}.
Auf diese Weise bricht Monte Carlo den Fluch.

\subsubsection{Zu \Kapref{chap:monotone}:
	Approximation monotoner Funktionen}

Wir untersuchen das Problem der $L_1$-Approximation f\"ur die Klasse
beschr\"ankter, monotoner Funktionen,
\begin{equation*}
	F_{\mon}^d
		:= \{f : [0,1]^d \rightarrow [0,1]
					\, \mid \, \vecx \leq \vecz \Rightarrow f(\vecx) \leq f(\vecz)
				\} \,,
\end{equation*}
unter Nutzung von Funktionswerten als Information.
Dies ist kein lineares Problem, weil die Inputmenge asymmetrisch ist.
Hinrichs, Novak und Wo\'zniakowski~\cite{HNW11}
zeigten, dass das Problem im deterministischen Fall
dem Fluch der Dimension unterliegt.
Dies ist bei Randomisierung nicht mehr der Fall,
dennoch bleibt das Problem sehr schwer zu l\"osen.

Aus einem Ergebnis von Blum, Burch und Langford~\cite{BBL98}
f\"ur Boole'sche monotone Funktionen~\mbox{$f: \{0,1\}^d \rightarrow \{0,1\}$}
kann man ableiten, dass f\"ur festes~\mbox{$\eps > 0$}
die Komplexit\"at~\mbox{$n^{\ran}(\eps,d)$}
mindestens exponentiell von~\mbox{$\sqrt{d}$} abh\"angt.
\Abschref{sec:monoMCLBs} enth\"alt einen modifizierten Beweis,
dank dem wir eine untere Schranke mit aussagekr\"aftiger $\eps$-Abh\"angigkeit
bekommen,
\begin{equation*}
	n^{\ran}(\eps,d) > \nu \exp(c \, \sqrt{d} \, \eps^{-1})
		\quad \text{f\"ur\, $\eps_0 \, \sqrt{d_0/d} \leq \eps \leq \eps_0$} \,,
\end{equation*}
wobei~\mbox{$\eps_0,\nu,c > 0$} und \mbox{$d \geq d_0 \in \N$},
siehe \thmref{thm:monotonLB}.
Insbesondere wenn wir eine gem\"a{\ss}igt abfallende Folge
von Fehlertoleranzen \mbox{$\eps_d := \eps_0 \, \sqrt{d_0/d}$} w\"ahlen,
l\"asst sich beobachten,
dass die Komplexit\"at~\mbox{$n^{\ran}(\eps_d,d)$}
exponentiell in~$d$ w\"achst.
Man sagt, das Problem sei \emph{nicht "`weakly tractable"'},
siehe \remref{rem:monMCLBintractable}.

In \Abschref{sec:monoUBs} werden obere Schranken bewiesen,
die zeigen, dass die Komplexit\"at~\mbox{$n^{\ran}(\eps,d)$}
f\"ur festes~\mbox{$\eps > 0$}
tats\"achlich "`nur"' exponentiell von~$\sqrt{d}$
modulo logarithmischer Terme abh\"angt.
Die algorithmische Idee wurde bereits von Bshouty und Tamon~\cite{BT96}
f\"ur \emph{Boole'sche} monotone Funktionen umgesetzt,
siehe \Abschref{sec:BooleanUBs}.
Ein vergleichbarer Ansatz f\"ur \emph{reellwertige},
auf~\mbox{$[0,1]^d$} definierte, monotone Funktionen
wird nun in \Abschref{sec:monoRealUBs} verfolgt.
Darin beschreiben und analysieren wir
einen neuen Monte-Carlo-Algorithmus~$(A_{r,k,n}^{\omega})_{\omega}$
mit w\"unschenswerten Fehlerschranken,
hierbei~\mbox{$r,k,n \in \N$}.
Im Wesentlichen basiert dieser auf einer Standard Monte-Carlo-N\"aherung
f\"ur die wichtigsten Wavelet-Koeffizienten der Haar-Basis in~\mbox{$L_2([0,1]^d)$},
wobei die zu approximierende Funktion an~$n$ zuf\"allig gew\"ahlten Stellen
ausgewertet wird.
Die ausgegebene Funktion ist konstant auf Teilw\"urfeln der Seitenl\"ange~$2^{-r}$,
d.h.\ nur Wavelet-Koeffizienten bis zu einer bestimmten Aufl\"osung
kommen in Betracht.
Au{\ss}erdem sind nur solche Wavelet-Koeffizienten von Interesse,
die -- f\"ur eine Input-Funktion~$f$ --
die gleichzeitige Abh\"angigkeit von bis zu $k$~Variablen messen.
F\"ur festes~$\eps$ hat dieser Parameter das asymptotische Verhalten
\mbox{$k \asymp \sqrt{d\,(1+\log d)}$}.
Es gibt eine lineare Version des Algorithmus, siehe~\thmref{thm:monoUBsreal},
sowie eine nichtlineare mit verbesserter $\eps$-Abh\"angigkeit der Komplexit\"at,
siehe \remref{rem:monoMCUBeps}.

\selectlanguage{english}
\cleardoublepage

\chapter*{Introduction and Results}
\addcontentsline{toc}{chapter}{Introduction and Results}
\markboth{Introduction and Results}{Introduction and Results}

For many problems arising in technical and scientific applications
it is practically impossible to give exact solutions.
Instead, one is interested in approximate solutions
that are to be found with methods that perform a finite number of steps.
In particular, we need to cope with incomplete information
about a problem instance;
apart from structural assumptions
(the so-called \emph{a priori} knowledge),
we may collect only a finite amount of information,
let us say $n$~real numbers originating from measurements
or from subprograms provided by the user.
There is a growing interest in solving high-dimensional problems
that involve functions defined on a $d$-dimensional domain.
We study the so-called
\emph{information-based complexity}~\mbox{$n(\eps,d)$},
that is the minimal number of information needed in order to solve
the problem within a given error tolerance~\mbox{$\eps > 0$}.
\emph{Tractability} studies in general are concerned with
the behaviour of this function~\mbox{$n(\eps,d)$}.
In many cases the complexity increases exponentially in~$d$
for some fixed~$\eps$,
this phenomenon is called the \emph{curse of dimensionality}.
If a problem suffers from the curse of dimensionality,
there are basically two ways to deal with it.
One way is to include more a priori knowledge,
thus narrowing the set of possible inputs.
The other way is to widen the class of admissible algorithms.
In this dissertation we focus on the second approach,
namely, we study the potential of randomization
for function approximation problems.

A $d$-dimensional function approximation problem is an identity mapping
\begin{equation*}
	\App: F^d \hookrightarrow G^d, \quad f \mapsto f \,,
\end{equation*}
with an \emph{input set}~$F^d$ which contains $d$-variate real-valued functions,
and a normed space~$G^d$.
In this study, functions from~$F^d$ are usually defined
on the $d$-dimensional unit cube~\mbox{$[0,1]^d$},
the output space~$G^d$ could be~\mbox{$L_1([0,1]^d)$}
or~\mbox{$L_{\infty}([0,1]^d)$}.
Deterministic algorithms are mappings~%
\mbox{$A_n = \phi \circ N : F^d \rightarrow G^d$},
where the \emph{information mapping}
\begin{equation*}
	N : F^d \rightarrow \R^n , \quad f \mapsto (L_1(f),\ldots,L_n(f)) \,,
\end{equation*}
uses $n$~linear functionals~$L_i$ as information about the problem instance~$f$.
The error is then defined by the worst case,
\begin{equation*}
	e(A_n,F^d) := \sup_{f \in F^d} \|f - A_n(f)\|_{G^d} \,.
\end{equation*}
Randomized methods are modelled as a family~%
\mbox{$(A_n^{\omega})_{\omega \in \Omega}$}
of mappings~\mbox{$A_n^{\omega}:
											F^d \rightarrow G^d$}
as before,
where~\mbox{$\omega \in \Omega$} is a random element
from a probability space~\mbox{$(\Omega,\Sigma,\P)$}.
The error of such a \emph{Monte Carlo} algorithm
is defined as the expected error for the worst input,
\begin{equation*}
	e((A_n^{\omega})_{\omega},F^d)
		:= \sup_{f \in F^d} \expect \|f - A_n^{\omega}(f)\|_{G^d} \,.
\end{equation*}
In both cases, the aim is to significantly reduce the \emph{initial error}
\begin{equation*}
	e(0,F^d) := \inf_{g \in G^d} \sup_{f \in F^d} \|f - g\|_{G^d} \,,
\end{equation*}
which is achievable without any information.
We are interested in the comparison of the complexity
in the deterministic and the randomized setting,
\begin{align*}
	n^{\deter}(\eps,d) &:= \inf\{n \in \N_0 \mid
														\exists A_n: e(A_n,F^d) \leq \eps\} \,,\\
	n^{\ran}(\eps,d) &:=	\inf\{n \in \N_0 \mid
														\exists (A_n^{\omega})_{\omega}:
															e((A_n^{\omega})_{\omega},F^d) \leq \eps\} \,,
\end{align*}
where~\mbox{$0 < \eps < e(0,F^d)$}.
All algorithms presented in this thesis
for upper bounds on these quantities
are \emph{non-adaptive} algorithms with the simple structure of~$N$
as indicated above.
The lower bounds are proven for more general \emph{adaptive} algorithms,
even \emph{varying cardinality}~\mbox{$n(\omega,f)$} is considered.
For these notions and a detailed introduction to \emph{information-based complexity}
see \chapref{chap:basics}.

New results are contained in Chapters~2--4,
which treat more or less stand-alone topics.
\chapref{chap:Bernstein} is concerned with lower bounds for randomized methods,
by means of which in some cases one can show
that Monte Carlo methods are not much better than optimal deterministic algorithms.
In contrast to this,
in \chapref{chap:Hilbert} we find settings where deterministic algorithms
suffer from the curse of dimensionality but randomization can break
the curse in a very impressive way.
\chapref{chap:monotone} deals with a problem
for which randomization breaks the curse of dimensionality,
yet the problem is quite difficult.

\subsubsection{On \chapref{chap:Bernstein}:
	Lower Bounds for Linear Problems via Bernstein Numbers}

The main result of this chapter 
is a lower bound for Monte Carlo algorithms
for general linear problems
\begin{equation*}
	S: F \rightarrow G \,,
\end{equation*}
that is, $S$ is a linear operator between normed spaces~$\widetilde{F}$ and~$G$,
and the input set~$F$ is the unit ball in~$\widetilde{F}$.
We show that for any
adaptive Monte Carlo method~\mbox{$(A_n^{\omega})_{\omega}$}
using $n$~arbitrary continuous linear functionals~$L_i$ as information,
we have
\begin{equation*}
	e((A_n^{\omega})_{\omega},F) \geq \textstyle{\frac{1}{30}} \, b_{2n}(S) \,,
\end{equation*}
where~$b_m(S)$ is the $m$-th Bernstein number of the operator~$S$,
see \thmref{thm:BernsteinMCada}.
The proof is based on a result due to Heinrich~\cite{He92}
which connects norm expectations for Gaussian measures with
the Monte Carlo error.
The innovation is that we use Lewis' theorem
for choosing optimal Gaussian measures.
This result has been announced in~\cite{Ku16},
a short proof without the explicit constant has been included there.

In \secref{sec:Cinf->Linf}
we apply this general tool to the $L_{\infty}$-approximation
of certain classes of $C^{\infty}$-functions,
\begin{equation*}
	\App : F_p^d \hookrightarrow L_{\infty}([0,1]^d) \,.
\end{equation*}
Here, the input set is defined as
\begin{equation*}
	\begin{split}
		F_p^d := \{f \in C^{\infty}([0,1]^d)
								\, \mid \,&
								\|\nabla_{\vecv_k} \cdots \nabla_{\vecv_1} f\|_{\infty}
									\leq |\vecv_1|_p \cdots |\vecv_k|_p\\
								&\text{for } k \in \N_0, \,
								\vecv_1,\ldots,\vecv_k \in \R^d\} \,,
	\end{split}
\end{equation*}
where~\mbox{$\nabla_{\vecv} f$} denotes the directional derivative along
a vector~\mbox{$\vecv \in \R^d$},
and we write~$|\vecv|_p$ for the $p$-norm of~\mbox{$\vecv \in \ell_p^d$},
\mbox{$1 \leq p \leq \infty$}.
Via the corresponding Bernstein numbers we obtain
the lower bound
\begin{equation*}
	n^{\ran}(\eps,d,p)
		> 2^{\lfloor \frac{d^{1/p}}{3} \rfloor - 1}
	\quad \text{for\, $0 < \eps \leq {\textstyle \frac{1}{30}}$,}
\end{equation*}
see \corref{cor:LinfAppLB}.
For~\mbox{$p=1$} this implies the curse of dimensionality
even in the randomized setting.
The technique for determining the Bernstein numbers
is known from Novak and Wo\'zniakoswki~\cite{NW09b},
where the curse of dimensionality for the case~\mbox{$p=1$}
was shown in the deterministic setting.

A simple Taylor approximation provides upper bounds
for the complexity with deterministic methods,
\begin{equation*}
	n^{\ran}(\eps,d,p)
		\,\leq\, n^{\deter}(\eps,d,p)
		\,\leq\, \exp\left(\euler \, \log(d+1)
										\, 
										\max\left\{\log {\textstyle \frac{1}{\eps}},\,
																			d^{1/p}
												\right\}
								\right) \,,
\end{equation*}
see \thmref{thm:Cinf->LinfUB}.
The algorithmic idea goes back to Vyb{\'\i}ral~\cite{Vyb14} who considered a setting
similar to the case~\mbox{$p=\infty$}.

For this example we observe an exponential dependency of the complexity
on~$d^{1/p}$, roughly,
which cannot be removed with randomization.
It is also an example which shows how narrowing the input set may affect
tractability.

\subsubsection{On \chapref{chap:Hilbert}:
	Uniform Approximation of Functions from a Hilbert Space}

In this chapter 
we study a new Monte Carlo approach to the $L_{\infty}$-approximation
of functions from a reproducing kernel Hilbert space~$\Hilbert$.
The input set~$F$ is the unit ball of~$\Hilbert$, that is,
for an orthonormal basis~\mbox{$(\psi_k)_{k \in \N}$} of~$\Hilbert$
we have
\begin{equation*}
	F := \Bigl\{\sum_{k=1}^{\infty} a_k \, \psi_k\
								\,\Big|\,
									a_k \in \R, \sum_{k=1}^{\infty} a_k^2 \leq 1
				\Bigr\} \,.
\end{equation*}
The idea for the new algorithm
is based on a fundamental Monte Carlo approximation method
which is due to Math\'e~\cite{Ma91},
see also \secref{sec:HilbertFundamental}.
In the original paper
it has been applied to finite dimensional sequence recovery~%
\mbox{$\ell_2^m \hookrightarrow \ell_q^m$}, \mbox{$q > 2$},
it was then used in combination with discretization techniques
for function space embeddings in order to determine the order of convergence.
In \secref{sec:HilbertPlainMCUB}
we take a more direct approach,
proposing the linear Monte Carlo method
\begin{equation*}
	A_n^{\omega}(f)
		:= \frac{1}{n}
					\sum_{i=1}^n
						L_i^{\omega}(f) \, g_i^{\omega}
\end{equation*}
where
\begin{equation*}
	L_i^{\omega}(f)
		:= \sum_{k=1}^{\infty} X_{ik} \, \langle \psi_k,f \rangle_{\Hilbert},
		\quad\text{and}\quad
	g_i^{\omega}
		:= \sum_{k=1}^{\infty} X_{ik} \, \psi_k \,,
\end{equation*}
with the $X_{ik}$~being independent standard Gaussian random variables.
The functions~$g_i^{\omega}$ are independent copies
of the \emph{Gaussian field}~$\Psi$ associated to~$\Hilbert$,
the covariance function of~$\Psi$ is the reproducing kernel of~$\Hilbert$.
We have the error estimate
\begin{equation*}
	e((A_n^{\omega})_{\omega},F)
		\leq \frac{2 \, \expect \|\Psi\|_{\infty}}{\sqrt{n}} \,.
\end{equation*}
Admittedly,
the random functionals $L_i^{\omega}$ are discontinuous with probability~$1$,
but for any fixed~\mbox{$f \in \Hilbert$}
the value~\mbox{$L_i^{\omega}(f)$} is a zero-mean Gaussian random variable
with variance~\mbox{$\|f\|_{\Hilbert}^2$},
hence it is almost surely finite.
This method, however, can be seen as the limiting case of methods
that use continuous random functionals, see \lemref{lem:stdMCapp}.

Using tools from stochastics,
see \secref{sec:E|Psi|_sup} for a summary,
we can estimate the value~\mbox{$\expect \|\Psi\|_{\infty}$},
provided that the random function~$\Psi$ is bounded.
Namely, via the technique of \emph{majorizing measures} due to Fernique,
we tackle the case of periodic functions on the~$d$-dimensional
torus~$\Torus^d$, see \secref{sec:HilbertPeriodic}.
Here, $\Torus$~is the interval~$[0,1]$ where the endpoints are identified.
In the univariate case,
we denote by \mbox{$\Hilbert_{\veclambda}(\Torus)$}
the space with orthonormal basis
\begin{equation*}
	\left\{\lambda_0,\,
				\lambda_k \, \sin(2 \, \pi \, k \,\cdot),\,
				\lambda_k \, \cos(2 \, \pi \, k \,\cdot)
	\right\}_{k \in \N} \,,
\end{equation*}
where~\mbox{$\lambda_k > 0$}.
The $d$-variate case is defined by the tensor product,
\begin{equation*}
	\Hilbert_{\veclambda}(\Torus^d)
		:= \bigotimes_{j=1}^{d} \Hilbert_{\veclambda}(\Torus) \,.
\end{equation*}
We assume \mbox{$\sum_{k=0}^{\infty} \lambda_k^2 \stackrel{\text{!}}{=} 1$},
and then the initial error is constant~$1$.
For this situation
we obtain the curse of dimensionality in the deterministic setting,
see~\thmref{thm:curseperiodic}.
The deterministic lower bound is based on a technique due to
Kuo, Wasilkowski, and Wo\'zniakowski~\cite{KWW08},
also Cobos, K\"uhn, and Sickel~\cite{CKS16},
see \secref{sec:HilbertWorLB}.
For the randomized setting,
we derive conditions on the reproducing kernel of periodic Hilbert spaces
such that the associated Gaussian process is bounded.
Specifically, we consider Korobov spaces~%
\mbox{$H_r^{\Korobov}(\Torus^d) = \Hilbert_{\veclambda}(\Torus^d)$}
with~\mbox{$\lambda_k := \sqrt{\beta_1} \, k^{-r}$} for~\mbox{$k \in \N$},
here \mbox{$\beta_1 > 0$} such that the initial error may still be adjusted
with \mbox{$0 < \lambda_0 < 1$}.
For smoothness~\mbox{$r > 1$} we can show that
the approximation problem
\begin{equation*}
	\App: H_r^{\Korobov}(\Torus^d) \hookrightarrow L_{\infty}(\Torus^d)
\end{equation*}
possesses the Monte Carlo complexity
\begin{equation*}
	n^{\ran}(\eps,d,r) \leq C_r \, d \, (1 + \log d) \, \eps^{-2} \,,
\end{equation*}
where \mbox{$C_r > 0$}.
Hence this problem is \emph{polynomially tractable}.
For smaller smoothness~\mbox{$\frac{1}{2} < r \leq 1$},
we can still prove polynomial tractability with a worse complexity bound,
in that case the fundamental Monte Carlo method is only applied
to a finite dimensional subspace of~\mbox{$H_r^{\Korobov}(\Torus^d)$},
see~\thmref{thm:Korobov}. By this, Monte Carlo breaks the curse.

\subsubsection{On \chapref{chap:monotone}:
	Approximation of Monotone Functions}

We study the $L_1$-approximation for the class of bounded monotone functions,
\begin{equation*}
	F_{\mon}^d
		:= \{f : [0,1]^d \rightarrow [0,1]
					\, \mid \, \vecx \leq \vecz \Rightarrow f(\vecx) \leq f(\vecz)
				\} \,,
\end{equation*}
based on \emph{function values} as information.
This problem is not linear since the input set is unbalanced.
Hinrichs, Novak, and Wo\'zniakowski~\cite{HNW11}
showed that the problem suffers from the curse of dimensionality
in the deterministic setting.
This is not the case in the randomized setting anymore,
still the problem is very difficult.

From a result by Blum, Burch, and Langford~\cite{BBL98}
for monotone Boolean functions~\mbox{$f: \{0,1\}^d \rightarrow \{0,1\}$},
we can conclude that for fixed~\mbox{$\eps > 0$}
the complexity~\mbox{$n^{\ran}(\eps,d)$}
depends exponentially on~\mbox{$\sqrt{d}$} at least.
In \secref{sec:monoMCLBs} a modified proof is given
by what we obtain a lower bound that includes a meaningful $\eps$-dependency,
\begin{equation*}
	n^{\ran}(\eps,d) > \nu \exp(c \, \sqrt{d} \, \eps^{-1})
		\quad \text{for\, $\eps_0 \, \sqrt{d_0/d} \leq \eps \leq \eps_0$} \,,
\end{equation*}
where~\mbox{$\eps_0,\nu,c > 0$} and \mbox{$d \geq d_0 \in \N$},
see \thmref{thm:monotonLB}.
In particular, choosing a moderately decaying sequence of error tolerances~%
\mbox{$\eps_d := \eps_0 \, \sqrt{d_0/d}$},
we observe that the complexity~\mbox{$n^{\ran}(\eps_d,d)$}
grows exponentially in~$d$.
This implies that the problem is \emph{not weakly tractable},
see \remref{rem:monMCLBintractable}.

In \secref{sec:monoUBs} we prove upper bounds which show
that, for fixed~\mbox{$\eps > 0$},
the complexity~\mbox{$n^{\ran}(\eps,d)$}
indeed depends exponentially on~$\sqrt{d}$ times some logarithmic terms only.
The algorithmic idea has been performed for monotone Boolean functions
in Bshouty and Tamon~\cite{BT96}, see \secref{sec:BooleanUBs}.
Inspired by this,
in~\secref{sec:monoRealUBs}
a new Monte Carlo algorithm~$(A_{r,k,n}^{\omega})_{\omega}$ with desirable error bounds
for real-valued monotone functions defined on~\mbox{$[0,1]^d$}
is proposed and studied, here~\mbox{$r,k,n \in \N$}.
Essentially, we use standard Monte Carlo approximation
for the most important wavelet coefficients of the Haar basis in~\mbox{$L_2([0,1]^d)$},
using $n$~random samples.
The output will be constant on subcubes of sidelength~$2^{-r}$,
so only wavelet coefficients up to a certain resolution come into consideration.
Further, only those wavelet coefficients are of interest
that -- for an input function~$f$ --
measure the simultanious dependency on at most $k$~variables.
For fixed~$\eps$,
this parameter has the asymptotic behaviour
\mbox{$k \asymp \sqrt{d\,(1+\log d)}$}.
There is a linear version of the algorithm, see~\thmref{thm:monoUBsreal},
and a non-linear version with improved $\eps$-dependency of the complexity,
see \remref{rem:monoMCUBeps}.


\newpage

\pagestyle{scrplain}

\cleardoublepage
\tableofcontents
\newpage

\pagenumbering{arabic}
\setcounter{page}{1}
\pagestyle{scrheadings}
\automark[section]{chapter}

\cleardoublepage

\chapter{Basic Notions in Information-Based Complexity}
\label{chap:basics}

%
%
%
%

In \secref{sec:IBC e(n),N} the basic notions for the model of
computation and approximation in \emph{information-based complexity} (IBC)
are introduced.
In \secref{sec:tractability} on \emph{tractability} we provide the notions
for a classification of multi-dimensional problems
by the difficulty of solving them.
After these two sections the reader may immediately go forward
to one of the three main chapters (Chapters~2--4) that cover different topics.
\secref{sec:VaryCard} on algorithms with varying cardinality
is an extension of the computational model,
we collect tools that help to extend lower bounds
to this broader class of algorithms.
\secref{sec:measurable} is a comment on the computational model,
especially on measurability assumptions,
it has no further connection to the rest of the thesis.

\section{Types of Errors and Information}
\label{sec:IBC e(n),N}

We collect all notions we need for a basic understanding of
\emph{information-based complexity} (IBC).
For an elaborate introduction to this field, refer to the book
of Traub, Wasilkowski, and Wo\'zniakowski~\cite{TWW88}.

Let $S: \widetilde{F} \rightarrow G$ be the so-called \emph{solution mapping}
between the~\emph{input space}~$\widetilde{F}$, and the \emph{target space}~$G$
which is a metric space.
We aim to approximate~$S$
for inputs from an \emph{input set}~\mbox{$F \subseteq \widetilde{F}$}
with respect to the metric~$\dist_G$ of the \emph{target space}~$G$,
using algorithms
that collect only a limited amount of information on the input~\mbox{$f \in F$}
by evaluating finitely many functionals from a given class~$\Lambda$.

A very common example are \emph{linear problems} where
\begin{itemize}
	\item $S$ is a linear operator between Banach spaces,
	\item the input set~$F$ is the unit ball in~$\widetilde{F}$,
		or -- more generally -- a centrally symmetric convex set, and
	\item the class~$\Lambda$ of all admissible functionals
		is a subclass of the class~$\Lall$ of \emph{all}
		continuous linear functionals.
\end{itemize}
Chapters~\ref{chap:Bernstein} and~\ref{chap:Hilbert} deal with linear problems.
In \chapref{chap:monotone}, however, we will consider an input set~$F$
consisting of monotone functions which is not centrally symmetric.
Within this research we mainly examine
approximation problems~\mbox{$S = \App: \widetilde{F} \hookrightarrow G$},
\mbox{$f \mapsto f$},
that is, $\widetilde{F}$~is identified with a subset of~$G$.
Another typical example for problems is the computation of the
definite integral~\mbox{$S = \Int : \widetilde{F} \rightarrow \R$},
\mbox{$f \mapsto \int_{[0,1]^d} f(\vecx) \rd \vecx$},
with~$\widetilde{F}$ being a class of
integrable functions~\mbox{$f:[0,1]^d \rightarrow \R$};
here, algorithms may use function values,
also called \emph{standard information}~$\Lstd$.

Let~$(\Omega,\Sigma,\mathbb{P})$ be a suitable probability space.
Further, let $\Borel(G)$ denote the Borel
\mbox{$\sigma$-algebra} of~$G$,
and~$\mathcal{F}$ be a suitable \mbox{$\sigma$-algebra} on~$\widetilde{F}$,
e.g.\ the Borel \mbox{$\sigma$-algebra} if~$\widetilde{F}$ is a metric space.
By \emph{randomized algorithms}, also called \emph{Monte Carlo algorithms},
we understand
\mbox{$(\Sigma \otimes \mathcal{F})-\Borel(G)$}-measurable mappings
\mbox{$A_n = (A_n^{\omega}(\cdot))_{\omega \in \Omega}:
	\Omega \times \widetilde{F} \rightarrow G$}.
This means that the output~\mbox{$A_n(f)$} for an input~$f$ is random, depending
on~\mbox{$\omega \in \Omega$}.
We consider algorithms of \emph{cardinality}~$n$ that use at most
$n$~pieces of information,\footnote{%
	See \secref{sec:VaryCard} for the extention of the computational model
	to algorithms with varying cardinality.}
i.e.~\mbox{$A_n^{\omega} = \phi^{\omega} \circ N^{\omega}$} where
\mbox{$N^{\omega} : \widetilde{F} \rightarrow \R^n$} is the so-called
\textit{information mapping}.
The mapping \mbox{$\phi^{\omega}: \R^n \rightarrow G$} generates
an output~\mbox{$g = \phi^{\omega}(\vecy) \in G$}
as a compromise for all possible inputs that lead to
the same information~\mbox{$\vecy = N^{\omega}(f) \in \R^n$}.\footnote{%
	Some authors call~$\phi^{\omega}$ an \emph{algorithm}
	and \mbox{$\phi^{\omega} \circ N^{\omega}$}	a~\emph{method}.
	In this dissertation, ``method'' and ``algorithm'' are used synonymously,
	both referring to \mbox{$A_n^{\omega} = \phi^{\omega} \circ N^{\omega}$}.
}
If, for any information vector~\mbox{$\vecy \in N^{\omega}(F)$},
we take the output~\mbox{$\phi^{\omega}(\vecy) = S(\tilde{f})$}
as the solution for an element \mbox{$\tilde{f} \in F$} from the input set
which \emph{interpolates} the data,
that means \mbox{$N^{\omega}(\tilde{f}) = \vecy$},
then the algorithm is called \emph{interpolatory}.
The combinatory cost for the computation of~$\phi$
(arithmetic operations, comparison of real numbers,
operations in~$G$)
is usually neglected.\footnote{%
	We make one exception in \remref{rem:monoMCUBphicost},
	where we compare two different outputs~$\phi$.}

There are different types of information mappings.
In this research the information is obtained by computing
$n$~\emph{functionals} from the class~$\Lambda$ for the particular input.
This could be function values~$\Lstd$,
or arbitrary continuous linear functionals~$\Lall$.
We do not care about how these functionals are evaluated
-- they could be obtained by some measuring device
or by a subroutine provided by the user --
to us, evaluating an information functional is an \emph{oracle call}.
An information mapping is called \emph{non-adaptive}, if
\begin{equation}
	N^{\omega}(f)
	= [L_1^{\omega}(f),\ldots,L_n^{\omega}(f)]
	= (y_1,\ldots,y_n)
	= \vecy\,,
\end{equation}
where all functionals~\mbox{$L_k^{\omega} \in \Lambda$}
are chosen independently from~$f$.
In that case, $N^{\omega}$~is a linear mapping for
fixed~\mbox{$\omega \in \Omega$}.
For \emph{adaptive} information $N^{\omega}$ the choice of the
functionals may depend on previously obtained information, we assume that
the choice of the $k$-th functional is a measurable mapping
\mbox{$(\omega;\vecy_{[k-1]})
	\mapsto L_{k;\vecy_{[k-1]}}^{\omega}(\cdot)$}
into the space of functionals,
here, \mbox{$\vecy_{[k]} = (y_1,\ldots,y_k)$}
for~\mbox{$k=1,\ldots,n$}.
Further, the mapping~\mbox{$N = (N^{\omega})_{\omega}:
														\Omega \times \widetilde{F} \rightarrow \R^n$}
as a whole shall be
\mbox{$(\Sigma \otimes \mathcal{F})
				-\Borel(\R^n)$}-measurable.
By~\mbox{$\mathcal{A}_n^{\ran,\ada}(\Lambda)$} we denote the class of all
Monte Carlo algorithms that use~$n$ pieces of adaptively obtained information,
for the subclass of non-adaptive algorithms we
write~\mbox{$\mathcal{A}_n^{\ran,\nonada}(\Lambda)$}.

If the solution operator~$S$ is a linear operator that maps between
Banach spaces, we consider two more special types of algorithms.\\
\emph{Linear} algorithms~\mbox{$\mathcal{A}_n^{\ran,\lin}(\Lambda)$}
comprise non-adaptive algorithms where not only~$N^{\omega}$,
but also~$\phi^{\omega}$, and therefore
\mbox{$A_n^{\omega} = \phi^{\omega} \circ N^{\omega}$},
is linear for every~\mbox{$\omega \in \Omega$}.
For linear algorithms we usually say \emph{rank} instead of cardinality.\\
As another special class we consider \emph{homogeneous}
algorithms~\mbox{$\mathcal{A}_n^{\ran,\homog}(\Lambda)$}.
The information mapping may still be adaptive,
however with the special constraint
\begin{equation*}
	L_{k,\vecy_{[k-1]}}^{\omega} = L_{k,\lambda \, \vecy_{[k-1]}}^{\omega}
\end{equation*}
for information vectors~$\vecy = N^{\omega}(f)$
and~\mbox{$\lambda \in \R \setminus \{0\}$}.
In particular, this implies homogeneity
for the info mapping, \mbox{$N^{\omega}(\lambda \, f)
																		= \lambda \, N^{\omega}(f)$}
for all~\mbox{$\lambda \in \R$} and~\mbox{$f \in \widetilde{F}$}.
For the mapping~$\phi$ we assume the same,
\mbox{$\phi^{\omega}(\lambda \, \vecy) = \lambda \, \phi^{\omega}(\vecy)$},
thus inducing \mbox{$A_n^{\omega}(\lambda \, f) = \lambda \, A_n^{\omega}(f)$}.

We regard the class of \emph{deterministic algorithms} as a
subclass~\mbox{$\mathcal{A}_n^{\deter,\star} \subset \mathcal{A}_n^{\ran,\star}$}
\mbox{($\star \in \{\ada,\nonada,\lin,\homog\}$)}
of algorithms that are independent from~\mbox{$\omega \in \Omega$},\footnote{%
	This means in particular that we assume
	deterministic algorithms to be measurable. For a deeper discussion on
	measurability see \secref{sec:measurable}.}
for a particular algorithm we write \mbox{$A_n = \phi \circ N$},
omitting the random element~\mbox{$\omega$}.
For a deterministic algorithm~$A_n$,
the (\emph{absolute}) \emph{error~at~$f$} is
defined as the distance between output and exact solution,
\begin{equation}
	e(A_n,S,f) := \dist_G(A_n(f),S(f)) \,.
\end{equation}
For randomized algorithms~$A_n = (A_n^{\omega}(\cdot))_{\omega \in \Omega}$,
this can be generalized as the \emph{expected error at~$f$},
\begin{equation}
	e(A_n,S,f) := \mathbb{E} \dist_G(A_n^{\omega}(f),S(f)) \,,
\end{equation}
however, some authors prefer the \emph{root mean square error}
\begin{equation}
	e_2(A_n,S,f) := \sqrt{\mathbb{E} \dist_G(A_n^{\omega}(f),S(f))^2} \,.
\end{equation}
(The expectation~$\expect$ is written for the integration over
all~$\omega \in \Omega$ with respect to~$\P$.)
Note that~\mbox{$e(A_n,S,f) \leq e_2(A_n,S,f)$}.
Another criterion for rating Monte Carlo methods is
the \emph{margin of error}\footnote{%
	This is a common notion in statistics.}
for some preferably small \emph{uncertainty level}~\mbox{$\delta \in (0,1)$},
\begin{equation*}
	e_{\delta}(A_n,S,f)
		:= \inf \{ \eps > 0
							\mid \P(\dist_G(A_n^{\omega}(f),S(f)) > \eps) \leq \delta\} \,,
\end{equation*}
in other words,
we have a confidence level~\mbox{$(1-\delta)$} for the error~$\eps$.
This criterion is more difficult to analyse than the other two
definitions of a Monte Carlo error,
however, a basic understanding of the power of randomization can already
be gained with a simple mean error criterion.\footnote{%
	In \secref{sec:BooleanUBs} we cite an algorithm
	proposed by Bshouty and Tamon~\cite{BT96}.
	They studied the margin of error, but we only reproduce the analysis
	for the expected error.}\\
If the input space~$\widetilde{F}$ is a normed space,
one can also consider the \emph{normalized} error criterion
where for deterministic algorithms the \mbox{error~at~$f \not= 0$} is defined as
\begin{equation}
	e_{\normal}(A_n,S,f) := \frac{\|S(f) - A_n(f)\|_G}{\|f\|_F} \,.
\end{equation}
The normalized error for randomized algorithms is defined analogously.

The \emph{global error} of an algorithm~$A_n$ is defined as the error for the
worst input from the input set~$F \subset \widetilde{F}$, we write
\begin{equation}
	e(A_n,S,F) := \sup_{f \in F} e(A_n,S,f) \,.
\end{equation}
For technical purposes, we also need the \emph{$\mu$-average error},
which is defined for any (sub-)probability measure~$\mu$ (the so-called
\emph{input distribution})
on the input space~$\widetilde{F}$,
\begin{equation}
	e(A_n,S,\mu) := \int e(A_n,S,f) \rd \mu(f) \,.
\end{equation}
(A sub-probability measure~$\mu$ on~$\widetilde{F}$ is a positive measure with
\mbox{$0 < \mu(\widetilde{F}) \leq 1$}.)

The difficulty of a problem within a particular setting refers to the error of
optimal algorithms, we define the \emph{$n$-th minimal error}
\begin{align*}
	e^{\diamond,\star}(n,S,F,\Lambda)
		&:= \inf_{A_n \in \mathcal{A}_n^{\diamond,\star}(\Lambda)} e(A_n,S,F)
	\quad \text{and} \quad \\
	e^{\diamond,\star}(n,S,\mu,\Lambda)
		&:= \inf_{A_n \in \mathcal{A}_n^{\diamond,\star}(\Lambda)} e(A_n,S,\mu) \,,
\end{align*}
where~\mbox{$\diamond \in \{\ran,\det\}$}
and \mbox{$\star \in \{\ada,\nonada,\lin,\homog\}$}.
These quantities are inherent properties of the problem~$S$ with proper names.
So, given an input set~$F$,
the worst input error for optimal randomized algorithms
\mbox{$e^{\ran,\star}(n,S,F,\Lambda)$}
is called the \emph{Monte Carlo error},
the worst input error for deterministic algorithms
\mbox{$e^{\deter,\star}(n,S,F,\Lambda)$}
is called the \emph{worst case error} of the problem~$S$.
Given an input distribution~$\mu$, we only consider deterministic algorithms,
and -- for better distinction from the other two settings --
we introduce a new labelling
\mbox{$e^{\avg,\star}(n,S,\mu,\Lambda)
				:= e^{\deter,\star}(n,S,\mu,\Lambda)$},
calling it the \emph{$\mu$-average (case) error} of the problem~$S$.
For~$n=0$ we obtain the \emph{initial error}, that is the minimal error
that we achieve if we have to generate an output without collecting any
information about the actual input.\\
The inverse notion is the \emph{$\eps$-complexity}\footnote{%
	More precisely, we should call this quantity
	\emph{information-based $\eps$-complexity}.
	In the book on IBC by Traub et al.~\cite{TWW88} it is called
	\emph{$\eps$-cardinality},
	whereas the notion \emph{complexity} is associated to the total
	computational cost taking combinatory operations such as addition,
	multiplication, comparisons and evaluation of certain elementary functions
	into account.}
for a given error tolerance~\mbox{$\eps > 0$},
\begin{equation*}
	n^{\diamond,\star}(\eps,S,\bullet,\Lambda)
		:= \inf \{n \in \N_0 \mid
							\exists_{A_n \in \mathcal{A}_n^{\diamond,\star}(\Lambda)} ,\,
								e(A_n,S,\bullet) \leq \eps
						\} \,,
\end{equation*}
where $\bullet$ either stands for an input set~$F\subset\widetilde{F}$,
or for an input distribution~$\mu$.

\begin{remark}[Monotonic properties of error quantities]
	Obviously, in any setting,
	the error~\mbox{$e^{\diamond,\star}(n,S,\bullet,\Lambda)$}
	is monotonously decreasing for growing~$n$.
	Similarly, the inverse notion of
	complexity~\mbox{$n^{\diamond,\star}(\eps,S,\bullet,\Lambda)$}
	is growing for~\mbox{$\eps \rightarrow 0$}.
	
	By definition, the error (or the complexity, respectively) is smaller or equal
	for smaller input sets~\mbox{$F' \subseteq F$},
	\begin{equation*}
		e^{\diamond,\star}(n,S,F',\Lambda) \leq e^{\diamond,\star}(n,S,F,\Lambda) \,.
	\end{equation*}
	
	In general,
	a broader class of algorithms~\mbox{$\mathcal{A}'(\Lambda)
																					\supseteq \mathcal{A}(\Lambda)$}
	can only lead to a smaller error (and complexity),
	so, since adaption and randomization are additional features for algorithms,
	we have
	\begin{equation}
		e^{\ran,\star}(n,S,\bullet,\Lambda)
				\leq e^{\deter,\star}(n,S,\bullet,\Lambda)
		\quad \text{and} \quad
		e^{\diamond,\ada}(n,S,\bullet,\Lambda)
				\leq e^{\diamond,\nonada}(n,S,\bullet,\Lambda) \,.
	\end{equation}
	For the same reason,
	more general classes of information functionals~\mbox{$\Lambda' \supseteq \Lambda$}
	will diminish the error (and the complexity),
	\begin{equation*}
		e^{\diamond,\star}(n,S,\bullet,\Lambda')
			\leq e^{\diamond,\star}(n,S,\bullet,\Lambda) \,.
	\end{equation*}
	If for a particular problem function evaluations are continuous,
	then arbitrary continuous functionals are a generalization,
	so in that case we have~\mbox{$\Lall \supseteq \Lstd$}.
\end{remark}

Another important relationship connects average errors and the Monte Carlo
error. It has already been used by Bakhvalov~\cite[Sec~1]{Bakh59}.
\begin{proposition}[Bakhvalov's technique] \label{prop:Bakh}
	Let $\mu$ be an arbitrary (sub-)probability measure
	supported on the input set~\mbox{$F \subseteq \widetilde{F}$}.
	Then
	\begin{displaymath}
		e^{\ran,\star}(n,S,F,\Lambda) \geq e^{\avg,\star}(n,S,\mu,\Lambda) \,.
	\end{displaymath}
\end{proposition}
\begin{proof}
	Let \mbox{$A_n = \left(A_n^{\omega}\right)_{\omega \in \Omega}
								\in \mathcal{A}_n^{\ran,\star}(\Lambda)$}
	be a Monte Carlo algorithm.
	We find
	\begin{displaymath}
		\begin{split}
			e(A_n,S,F)
				&= \sup_{f \in F} \expect e(A_n^{\omega},S,f) \\
				&\geq \int \expect e(A_n^{\omega},S,f) \, \mu(\diff f) \\
			\text{[Fubini]}\quad
				&= \expect \int e(A_n^{\omega},S,f) \, \mu(\diff f) \\
				&= \mathbb{E} \, e(A_n^{\omega},S,\mu)\\
				&\geq \inf_{\omega} e(A_n^{\omega},S,\mu)\\
				&\geq \inf_{A_n^{\prime} \in \mathcal{A}_n^{\deter,\star}(\Lambda)}
								e(A_n^{\prime},S,\mu) \,.
		\end{split}
	\end{displaymath}
	Here, we used that for any fixed elementary event~\mbox{$\omega \in \Omega$}
	the realization~$A_n^{\omega}$ can be seen as a deterministic algorithm.
\end{proof}

The proof of the above relation
also shows
\begin{equation*}
	e^{\ran,\star}(n,S,\mu,\Lambda) = e^{\deter,\star}(n,S,\mu,\Lambda)
	 \equiv e^{\avg,\star}(n,S,\mu,\Lambda) \,,
\end{equation*}
so there is no need for randomized algorithms
in an average case setting.

\begin{remark}[On upper and lower bounds]
	Bakhvalov's technique provides the standard tool for proving lower bounds
	for the Monte Carlo error by considering particular average case situations.
	This has the advantage that we have to deal only with deterministic algorithms.
	We have some freedom to choose a suitable distribution~$\mu$.\footnote{%
		There are only few situations where lower bounds for the Monte Carlo
		error have been proven directly without switching to the average case
		setting, see for example the non-adaptive Monte Carlo setting
		for the integration of univariate monotone functions in
		Novak~\cite{No92mon},
		or an estimate for small errors for the
		approximation of monotone Boolean functions
		in Bshouty and Tamon~\cite[Thm~5.3.1]{BT96},
		see also \remref{rem:monMCLBintractable}.}\\
	The proof of upper bounds basically relies on the analysis of proposed
	algorithms.\\
	Math\'e~\cite{Ma93} showed that in several cases
	one can theoretically find input distributions~$\mu$ supported on~$F$
	such that the $\mu$-average error matches the Monte Carlo error.
	
	Lower (upper) bounds for the $n$-th minimal error correspond to
	lower (upper) bounds for the $\eps$-complexity, in detail,
	\begin{equation*}
		e^{\diamond,\star}(n_0,S,\bullet,\Lambda) > \eps_0
			\quad \Rightarrow \quad
		n^{\diamond,\star}(\eps_0,S,\bullet,\Lambda) > n_0 \,.
	\end{equation*}
	Consequently, Bakhvalov's technique can also be written down in the notion
	of $\eps$-complexity,
	\begin{equation*}
		n^{\ran,\star}(\eps,S,F,\Lambda) \geq n^{\avg,\star}(\eps,S,\mu,\Lambda)\,.
	\end{equation*}
\end{remark}

If no confusion is possible, in the future we will use a reduced notation, e.g.\ writing
\mbox{$e^{\ran}(n,S)$} instead of~\mbox{$e^{\ran,\ada}(n,S,F,\Lambda)$}
if the class~$\Lambda$ of information functionals is known from the context,
the input set~$F$ is the unit ball of the input space~$\widetilde{F}$
in the setting of a linear operator~$S$ between Banach spaces,
and taking into account that adaptive algorithms are the most general
type of algorithms we consider.
The same applies for the complexity \mbox{$n^{\ran}(\eps,S)$}.
In any case, the notation should be compact, yet include
all aspects needed to distinguish different settings within the context.


\section{Tractability}
\label{sec:tractability}

We give a short overview over different notions used in tractability theory.
For a more detailed introduction
refer to the book by Novak and Wo\'zniakowski~\cite[Chap~2]{NW08}.

In tractability analysis we do not just consider a single solution operator
but an entire family of solution operators
\begin{equation*}
	(S^d : F^d \rightarrow G^d)_{d \in \N},
\end{equation*}
with~$d$ being a dimensional parameter.
This could mean, for example,
that~$F^d$ and $G^d$ are classes of $d$-variate functions
defined on the unit cube~\mbox{$[0,1]^d$}.

In classical numerical analysis, however,
the dimension~$d$ is typically considered a fixed parameter
-- along with smoothness parameters etc.\ --
so within a complexity setting\footnote{%
	The notion \emph{complexity setting} comprises all features of algorithms
	like adaptivity or non-adaptivity,
	randomization, the class of information functionals~$\Lambda$,
	as well as the error criterion,
	be it the absolute or the normalized error.}
the error is perceived as a function in~$n$,
\begin{equation*}
	e^{\diamond,\star}(n,S^d,F^d,\Lambda) = e(n) \,.
\end{equation*}
For solvable problems this function is monotonously decreasing
and converging to~$0$ for~\mbox{$n\rightarrow \infty$}.
Problems are then classified by their \emph{speed of convergence}:
\begin{itemize}
	\item a function~$e(n)$ converges faster than a function~$e'(n)$
		iff~\mbox{$e(n)/e'(n) \xrightarrow[n \rightarrow \infty]{} 0$},
		we write~\mbox{$e(n) \prec e'(n)$},
	\item a function~$e(n)$ converges at least as fast as a function~$e'(n)$
		iff there exists a constant~\mbox{$c > 0$} and~\mbox{$n_0 \in \N$}
		such that~\mbox{$e(n) \leq c \, e'(n)$} for~\mbox{$n \geq n_0$},
		we write~\mbox{$e(n) \preceq e'(n)$},
	\item two functions~$e(n)$ and~$e'(n)$ have the same speed of convergence
		iff~\mbox{$e(n) \preceq e'(n)$} and~\mbox{$e'(n) \preceq e(n)$},
		we write~\mbox{$e(n) \asymp e'(n)$}.\footnote{%
			In \chapref{chap:Bernstein} we will encounter relations
			like~\mbox{$e(n) \geq \frac{1}{2} \, b_{2n}$}.
			It is worth thinking about an alternative definition
			of \emph{equal speed}, which holds if there exist
			constants~\mbox{$k_1,k_2,k_3 \in \N$} and \mbox{$c,C > 0$} such that
			\begin{equation*}
				c\, e(k_1 n) \leq e'(k_2 n) \leq C \, e(k_3 n)
			\end{equation*}
			for sufficiently large~$n$.
			For polynomial rates this will not make any difference,
			but if exponential functions are involved,
			two functions~$\exp(-n^p)$ and $\exp(-2 n^p)$
			would be classified \emph{the same speed of decay}
			only for the new notion.}
\end{itemize}
A widely used classification is done by the comparison to polynomial decay,
a problem has the \emph{order of convergence} at least~$p$
iff~\mbox{$e(n) \preceq n^{-p}$}, where~\mbox{$p > 0$}.
Determining the optimal order of convergence means finding constants
$c,C > 0$ such that for large~$n$ we have
\begin{equation*}
	c \, n^{-p} \leq e(n) \leq C \, n^{-p}.
\end{equation*}
A common phenomenon when determining the optimal order~$p_d$ for $d$-variate
problems is that the corresponding constants $c_d$ and~$C_d$ deviate widely,
and even worse, ``large~$n$'' means \mbox{$n \geq n_0(d) \in \N$}
and~$n_0(d)$ can be huge for growing dimension~$d$.
In \secref{sec:MonoOrder} we find an example
where difficulties become apparent as soon as we consider the inverse notion
of $\eps$-complexity.
Last but not least, for discrete problems such as the approximation of Boolean
functions, see \chapref{chap:monotone}, the concept of order of convergence
is meaningless since discrete problems may be solved with a finite amount of
information.

For tractability analysis now, we regard the complexity as a function
depending on~$\eps > 0$ and $d \in \N$,
\begin{equation*}
	n^{\diamond,\star}(\eps,S^d,F^d,\Lambda)
		= n(\eps,d) \,.
\end{equation*}

A first approach to this complexity function is to fix~$\eps > 0$
and to consider the growth in~$d$,
see for example the results on lower bounds
in \corref{cor:LinfAppLB}, \thmref{thm:curseperiodic},
or \thmref{thm:monotonLB}.
It is unpleasant if the complexity depends exponentially on~$d$,
we say that a problem suffers from the \emph{curse of dimensionality}\footnote{%
	This notion goes back to Bellman~1957~\cite{Bel57}.}
iff there exist~\mbox{$\eps,\gamma,c> 0$} and~\mbox{$d_0 \in \N$}
such that
\begin{equation*}
	n(\eps,d)
		\geq c \, (1+\gamma)^d
	\quad \text{for\, $d \geq d_0$.}
\end{equation*}
There are problems that have arbitrarily high order of convergence
but suffer from the curse of dimensionality,
see for example the case~\mbox{$p=1$} in \secref{sec:Cinf->Linf}.

For positive results, we do not only want the dependency on~$d$ to be
moderate, but also the dependency on~$\eps$.
A problem is \emph{polynomially tractable} iff there exist constants
\mbox{$C,p,q>0$} such that
\begin{equation*}
	n(\eps,d) \leq C \, \eps^{-p} \, d^q \,.
\end{equation*}
If we can even choose~$q = 0$,
that is if the complexity is essentially independent from the dimension~$d$,
we have \emph{strong polynomial tractability}.

In contrast to the curse of dimensionality,
problems for that the complexity
does \emph{not} depend exponentially on~$d$ or~$\eps^{-1}$,
in detail, where
\begin{equation*}
	\lim_{\eps^{-1} + d \rightarrow \infty}
			\frac{\log n(\eps,d)}{\eps^{-1} + d}
		= 0 \,,
\end{equation*}
are called \emph{weakly tractable}.
This notion is fairly new and has been studied first around the time
where the book on tractability, Novak and Wo\'zniakowski 2008~\cite{NW08},
has been written.
A problem which is \emph{not weakly tractable} is called
\emph{intractable}.\footnote{%
	The notion of ``intractability'' as it is used within the IBC community since
	the book on the tractability of multivariate problems,
	Novak and Wo\'zniakowski~\cite[p.~14]{NW08},
	is different from definitions	of ``intractability''
	in other scientific communities.
	In computer science,
	see for example the book on \emph{NP-completeness}
	by Garey and Johnson~\cite{GJ79NP},
	all problems that, for solving a problem \emph{exactly},
	need a running time
	which is \emph{superpolynomial} in the size~$m$ of the input,
	are called ``intractable''.
	Thus even~$m^{\log m}$ would fall into that category.
	In tractability studies for IBC, instead of the input size we consider
	the dimension~$d$ and the error tolerance~$\eps$,
	so automatically new notions arose.
	But also the observation that many problems have a
	\emph{sub-exponential} yet \emph{superpolynomial} running time
	motivated the introduction of new notions like \emph{weak tractability}.}
Note that there are intractable problems
that do not suffer from the curse of dimensionality,
for example the randomized approximation of monotone functions,
see \chapref{chap:monotone}.

More recently, the refined notion of \emph{\mbox{$(s,t)$}-weak tractability}
has been promoted in Siedlecki and Weimar~\cite{SW15}.
It is fulfilled iff
\begin{equation*}
	\lim_{\eps^{-1} + d \rightarrow \infty}
			\frac{\log n(\eps,d)}{\eps^{-s} + d^t}
		= 0
\end{equation*}
with~\mbox{$s,t > 0$}.
This notion coincides with weak tractability for~\mbox{$s=t=1$}.

Last but not least, Gnewuch and Wo\'zniakowski~\cite{GW11} promoted
the notion \emph{quasi-polynomial tractability}.
It holds iff there exist constants~\mbox{$C,p > 0$} such that
\begin{equation*}
	n(\eps,d) \leq C \, \exp[p \, (1+\log \eps^{-1})(1+\log d)] \,.
\end{equation*}
In this case the complexity behaves almost polynomially in~$d$
with an exponent that grows very slowly in~$\eps^{-1}$, and vice versa.

For an example of quasi-polynomial and \mbox{$(s,t)$}-weak tractability,
see \thmref{thm:Cinf->LinfUB}.

Whether or not a problem falls into one of the tractability classes above,
highly depends on the particular choice of the
$d$-dependent setting~\mbox{$(S^d)_{d \in \N}$}.
One criterion of a \emph{natural} $d$-dependent problem could be
that the input set $F^d$~can be identified with a subset of~$F^{d+1}$, and
therefore we can consider~$S^d$ to be a restriction of~$S^{d+1}$.
Another possible criterion is whether the initial error is properly \emph{normalized},
that is, the initial error should be a constant,
\begin{equation*}
	e(0,S^d,F^d) = c > 0 \quad \text{for all $d \in \N$.}
\end{equation*}
Typically~$c = 1$, see for example the problem in \secref{sec:Cinf->Linf};
however, in \chapref{chap:monotone} we have~\mbox{$c = \frac{1}{2}$},
see \remref{rem:monoInit}.

\section{Algorithms with Varying Cardinality}
\label{sec:VaryCard}

For some problems it might be convenient to allow algorithms
that collect a varying amount of information,
but in average they do not use more than~$n$ pieces of information.
In Ritter~\cite[Chap~VII and VIII]{Rit00} one can find examples
of average case settings where varying cardinality \emph{does} help.
Anyways,
for upper bounds we try to find algorithms
that are as simple as possible,
whereas for lower bounds it is desirable that they hold
for as general classes of algorithms as possible,
that is, we allow for randomization, adaption,
or even varying cardinality.\footnote{%
	Similarly, it is good to find lower bounds for very small input sets,
	but upper bounds that hold for very general and large input sets.}
In the end one might see what features are really making a big difference.

We need to adjust our model of
algorithms~\mbox{$A^{\omega} = \phi^{\omega} \circ N^{\omega}$}
where the number of information we collect may depend
on the random element~$\omega$ and (adaptively) on the input.
Now, the information mapping shall
be a mapping~\mbox{$N^{\omega}: F \rightarrow \R^\N$}
yielding an information sequence~\mbox{$\vecy = (y_k)_{k \in \N}$},
and for possible information sequences we need to define an output
via a mapping~\mbox{$\phi^{\omega}: N^{\omega}(F) \rightarrow G$}.
As before, the $k$-th piece of information is obtained by evaluating
an adaptively chosen functional from a given class~$\Lambda$,
\emph{or the zero functional},\footnote{%
	Considering for example $\Lstd$,
	in general the zero functional is not a function evaluation.}
\begin{equation*}
	y_k := L_{k,\vecy_{[k-1]}}^{\omega}(f) \,.
\end{equation*}

At some point we need to stop collecting further information.
Within the model, this means that for some index~\mbox{$n \in \N_0$} we choose
$L_{k,\vecy_{[k-1]}}^{\omega}$ to be the zero functional for all~\mbox{$k > n$},
so the actual amount of information for a particular algorithm is a function
\begin{equation*}
	n(\omega,\vecy) := \inf\{n \in \N_0 \mid
														L_{k,\vecy_{[k-1]}}^{\omega} = 0
														\text{ for all }
														k > n
											\} \,,
\end{equation*}
with~\mbox{$\vecy := N^{\omega}(f)$} being
a proper information sequence.\footnote{%
	For fixed~$\omega \in \Omega$, 
	not all sequences~$\vecy \in \R^\N$ can be
	the outcome of the information mapping~$N^{\omega}$.}
(For non-adaptive algorithms this function is independent from the input,
\mbox{$n(\omega,\vecy) = n(\omega)$}, for deterministic algorithms
it is a function~\mbox{$n(\omega,\vecy) = n(\vecy)$}.)
For convenience, we will also write~\mbox{$n(\omega,f)$}
instead of~\mbox{$n(\omega,N^{\omega}(f))$}.
Then the \emph{worst input cardinality}
of the algorithm~\mbox{$A = (A^{\omega})_{\omega \in \Omega}$}
is defined as
\begin{equation}
	\card(A) = \card(A,F) := \sup_{f \in F} \expect n(\omega,f) \,.
\end{equation}
For any (sub)-probability measure~$\mu$
the \emph{$\mu$-average cardinality} is
\begin{equation}
	\card(A,\mu) := \int \expect n(\omega,f) \, \mu(\diff f) \,.
\end{equation}
The $\mu$-average cardinality is usually defined for deterministic algorithms.

As before, we define different classes of algorithms
\mbox{$\mathcal{A}^{\diamond,\star}(\Lambda)$}
where~\mbox{$\diamond \in \{\ran,\det\}$}
and~\mbox{$\star \in \{\ada,\nonada,\lin,\homog\}$}.
The definition of the error for a particular algorithm does not change,
however, the new concept of cardinality
brings about new error and complexity
notions associated to a problem~\mbox{$S: F \rightarrow G$}.
For~\mbox{$\bar{n} \geq 0$} we have
\begin{equation*}
	\bar{e}^{\diamond,\star}(\bar{n},S,\bullet,\Lambda)
		:= \inf_{\substack{A \in \mathcal{A}^{\diamond,\star}(\Lambda) \\
												\card(A,\bullet) \leq \bar{n}
											}
						}
					e(A,S,\bullet) \,,
\end{equation*}
and for a given error tolerance~\mbox{$\eps > 0$} we define
\begin{equation*}
	\bar{n}^{\diamond,\star}(\eps,S,\bullet,\Lambda)
		:= \inf \{\bar{n} \geq 0 \mid
							\exists \,
								A \in \mathcal{A}^{\diamond,\star}(\Lambda) :\,
								\card(A,\bullet)\leq \bar{n} ,\,
								e(A,S,\bullet) \leq \eps
						\} \,,
\end{equation*}
where for~$\bullet$ we may insert an input set~\mbox{$F\subset\widetilde{F}$},
or an input distribution~$\mu$.
Be aware that the cardinality may be a real number now.

Note that algorithms from classes of fixed cardinality
\mbox{$\mathcal{A}_n^{\diamond,\star}(\Lambda)$}
can be identified with methods
from~\mbox{$\mathcal{A}^{\diamond,\star}(\Lambda)$},
so for~$\bar{n} \geq 0$ we have the general estimate
\begin{equation*}
	\bar{e}^{\diamond,\star}(\bar{n},S,\bullet,\Lambda)
		\leq e^{\diamond,\star}(\lfloor \bar{n} \rfloor,S,\bullet,\Lambda) \,.
\end{equation*}
For the worst case setting it is easy to see
that the new notion even coincides with the old notion of fixed cardinality,
that is, for~\mbox{$n \in \N_0$} we have
\begin{equation*}
	\bar{e}^{\deter,\star}(n,S,F,\Lambda)
		= e^{\deter,\star}(n,S,F,\Lambda) \,.
\end{equation*}

For Monte Carlo methods with non-adaptively varying cardinality~$n(\omega)$,
there is a direct relation to the fixed cardinality setting.
This relation is well known, see Heinrich~\cite[p.~289/290]{He92}.
\begin{lemma} \label{lem:n(om)MC}
	For $n \in \N$ we have
	\begin{equation*}
		\bar{e}^{\ran,\nonada}(n,S,F,\Lambda)
			\geq {\textstyle \frac{1}{2}}
							\, e^{\ran,\nonada}(2n, S, F, \Lambda) \,.
	\end{equation*}
\end{lemma}
\begin{proof}
	The proof also works for classes of adaptive algorithms
	as long as the actual cardinality does not depend on the input.
	In this sense, let \mbox{$A = (A^{\omega})_{\omega \in \Omega}
															\in \mathcal{A}^{\ran,\star}(\Lambda)$}
	be a Monte Carlo algorithm
	with non-adaptively varying cardinality~\mbox{$n(\omega)$}
	such that \mbox{$\expect n(\omega) \leq n \in \N$}.
	Then we have
	\begin{align*}
		e(A,F)
			&= \sup_{f \in F} \expect e(A^{\omega},S,f) \\
			&\geq \sup_{f \in F}
							\expect e(A^{\omega},S,f)
								\, \ind_{\{n(\omega) \leq 2n\}} \\
			&= \P\{n(\omega) \leq 2n\}
						\, \sup_{f \in F} \expect' e(A^{\omega},S,f) \\
		\intertext{Here, $\expect'$ denotes the expectation for the
				conditional probability space~\mbox{$(\Omega',\Sigma \cap \Omega',\P')$}
				where we integrate over~\mbox{$\omega \in \Omega'
													:= \{\omega \mid n(\omega) \leq 2n\}
													\subseteq \Omega$}
				with respect to the conditional measure
				\mbox{$\P'(\cdot) := \P(\cdot \mid \Omega')
									= \P(\cdot \, \cap \Omega')/\P(\Omega')$}.
				We can regard~\mbox{$A' := (A^{\omega})_{\omega \in \Omega'}$}
				as a Monte Carlo algorithm from the
				class~\mbox{$\mathcal{A}_{2n}^{\ran,\star}(\Lambda)$}
				with another underlying probability space than for~$A$.
				Together with
				\mbox{$\P\{n(\omega) \leq 2n\} \geq \frac{1}{2}$}
				(by Markov's inequality),
				this gives the lower bound}
		e(A,F)
			&\geq {\textstyle \frac{1}{2}}
							\, e^{\ran,\star}(2n,S,F,\Lambda) \,.
	\end{align*}
\end{proof}

For Monte Carlo methods with adaptively varying cardinality~\mbox{$n(\omega,f)$}
we need special versions of Bakhvalov's technique.
\begin{lemma}[Bakhvalov's technique for varying cardinality]
	\label{lem:n(om,f)Bakhvalov}
	For any (sub-)proba\-bility measure~$\mu$ on~$F$,
	and \mbox{$\bar{n} \geq 0$},
	the Monte Carlo error and the average error in the setting
	of (adaptively) varying cardinality are related by
	\begin{equation*}
		\bar{e}^{\ran,\ada}(\bar{n},S,F,\Lambda)
			\geq {\textstyle \frac{1}{2}}
						\, \bar{e}^{\avg,\ada}(2\bar{n}, S, \mu, \Lambda) \,.
	\end{equation*}
	If we have an estimate~%
	\mbox{$\bar{e}^{\avg,\ada}(\bar{n}, S, \mu, \Lambda)
						\geq \hat{\eps}(\bar{n})$}
	with a convex and decaying function~$\hat{\eps}(\bar{n})$
	for~\mbox{$\bar{n} \geq 0$},
	the lower bound can be improved to
	\begin{equation*}
		\bar{e}^{\ran,\ada}(\bar{n},S,F,\Lambda)
			\geq \hat{\eps}(\bar{n}) \,.
	\end{equation*}
\end{lemma}
\begin{proof}
	Let \mbox{$A = \left(A^{\omega}\right)_{\omega \in \Omega}
								\in \mathcal{A}^{\ran,\star}(\Lambda)$}
	be a Monte Carlo algorithm
	with adaptively varying cardinality~\mbox{$n(\omega,f)$}
	such that
	\begin{align*}
		\bar{n} &\geq \card(A,F)
			:= \sup_{f \in F} \expect n(\omega,f) \,. \\
		\intertext{We can relate this to the average cardinality
				with respect to~$\mu$
				regarding~$A^{\omega}$ as a deterministic algorithm
				for fixed~$\omega$,}
			&\geq \int \expect n(\omega,f)
					\, \mu(\diff f) \\
		\text{[Fubini]}\quad
			&= \expect \underbrace{\int n(\omega,f) \, \mu(\diff f)
														}_{= \card(A^{\omega},\,\mu)}\\
		\text{[Markov's ineq.]}\quad
			&\geq 2\bar{n} \, \P\{\card(A^{\omega},\mu) > 2\bar{n}\} \,.
	\end{align*}
	This gives us the estimate
	\begin{equation} \label{eq:Bakh,P(2n)}
		\P\{\card(A^{\omega},\mu) \leq 2\bar{n}\}
			\geq {\textstyle \frac{1}{2}} \,.
	\end{equation}
	Now, considering the error, we find
	\begin{align*}
		e(A,F)
			&= \sup_{f \in F} \expect e(A^{\omega},f) \\
			&\geq
				\int \expect e(A^{\omega},f) \, \mu(\diff f) \\
		\text{[Fubini]}\quad
			&= \expect \int e(A^{\omega},f) \, \mu(\diff f) \\
			&= \expect e(A^{\omega},\mu) \\
			&\geq
				\expect \bar{e}^{\avg,\star}(\card(A^{\omega},\mu),
																			S, \mu, \Lambda) \,. \\
		\intertext{%
	A rough estimate via~\eqref{eq:Bakh,P(2n)} will give}
		e(A,F)
			&\geq
				\underbrace{\P\{\card(A^{\omega},\mu) \leq 2\bar{n}\}
									}_{= \frac{1}{2}}
					\, \bar{e}^{\avg,\star}(2\bar{n}, S, \mu, \Lambda) \,.
		\intertext{%
	If we have a specially structured estimate~$\hat{\eps}$
	for the average error, we can proceed in a better way,}
		e(A,F)
			&\geq \expect \hat{\eps}(\card(A^{\omega},\mu)) \\
		\text{[convexity]}\quad
			&\geq \hat{\eps}(\expect \card(A^{\omega},\mu)) \\
		\text{[monotonicity]}\quad
			&\geq \hat{\eps}(\bar{n}) \,.
	\end{align*}
	This finishes the proof.
\end{proof}

A similar convexity argument will help to find good bounds
for average case settings with varying cardinality.
\begin{lemma}[Special average settings with varying cardinality]
	\label{lem:n(om,f)avgspecial}
	Let~$\mu$ be a probability measure on~$\widetilde{F}$.
	Assume that for any deterministic algorithm~\mbox{$\phi \circ N$}
	with varying cardinality
	there exists a version of the conditional measure~$\mu_{\vecy}$
	such that
	\begin{equation*}
		\int \dist_G(\phi(\vecy),S(f)) \, \mu_{\vecy}(\diff f)
			\geq \hat{\eps}(n(\vecy)) \,,
	\end{equation*}
	where~\mbox{$\hat{\eps}(\bar{n})$}
	is convex and decaying for~\mbox{$\bar{n} \geq 0$}.
	Then the average error for algorithms with varying cardinality
	is bounded by this function,
	\begin{equation*}
		\bar{e}^{\avg,\ada}(\bar{n},S,\mu,\Lambda)
			\geq \hat{\eps}(\bar{n}) \,.
	\end{equation*}
	If $\mu$ is supported on~$F$,
	by \lemref{lem:n(om,f)Bakhvalov} the very lower bound holds
	for the Monte Carlo error as well.
\end{lemma}
\begin{proof}
	Let~\mbox{$A = \phi \circ N$} be a deterministic algorithm
	with adaptively varying cardinality~$n(f)$ such that
	\mbox{$\bar{n} \geq \int n(f) \, \mu(\diff f)$}.
	By definition we have
	\begin{align*}
		e(A,\mu)
			&= \int e(A,f) \, \mu(\diff f) \,. \\
	\intertext{%
	We split the integral
	into the integration over~$\vecy \in N^{\omega}(F)$,
	with an appropriate conditional distributions~$\mu_{\vecy}$
	on~\mbox{$N^{-1}(\vecy) \cap F$},
	fulfilling the assumptions of the lemma,
	and obtain}
			&= \int\left[\int \dist_G(\phi(\vecy),S(f))
											\,\mu_{\vecy}(\diff f)
							\right]
						\, \mu \circ N^{-1}(\diff\vecy) \\
			&\geq
				\int \hat{\eps}(n(\vecy))
						\, \mu\circ N^{-1}(\diff\vecy) \\
		\text{[convexity]}\quad
			&\geq
				\hat{\eps}\left(\int n(f) \, \mu(\diff f)\right)\\
		\text{[monotonicity]}\quad
			&\geq
				\hat{\eps}(\bar{n})\,.
	\end{align*}
\end{proof}

By \lemref{lem:n(om)MC} we see
that non-adaptively varying cardinality does not help a lot
when trying to find better Monte Carlo algorithms.
For adaptively varying cardinality the situation is slightly more
complicated; however, \lemref{lem:n(om,f)Bakhvalov}
gives us a tool to prove lower bounds that are similar to those
that we can obtain for the fixed cardinality setting,
see \secref{sec:n(om,f)Bernstein} for an application
of this lemma.
In many more cases even better,
\lemref{lem:n(om,f)avgspecial} applies to the average setting
so that we obtain lower bounds which coincide
with the computed estimates for the fixed cardinality setting.
In this dissertation, we have this nice situation 
for the lower bounds in
\thmref{thm:BernsteinMChom} (homogeneous algorithms and Bernstein numbers),
and in \thmref{thm:MonAppOrderConv} and \thmref{thm:monotonLB}
(approximation of monotone functions).
This justifies that in the main parts of this thesis we focus on algorithms
with fixed cardinality.

\section{On the Measurability of Algorithms}
\label{sec:measurable}

It seems \emph{natural} to assume measurability for algorithms since
\emph{real computers} can only deal with a finite amount of states.
In the IBC setting, however, it is convenient to assume that we can operate
with \emph{real numbers},
otherwise the concept of linear algorithms for real-valued functions
would not make sense.
Further justification for why we work with the real number model
is gathered in Novak and Wo\'zniakowski~\cite[Sec~4.1.3]{NW08}.
Unfortunately, the \emph{real number model} tails the problem of
measurability.
Heinrich and Milla~\cite{HeM11} presented
a simple Monte Carlo sampling algorithm for indefinite integration
that at first view appears natural but, in fact, is not measurable.\footnote{%
	I would like to thank Mario Hefter for interesting discussions
	on measurability of algorithms
	during our stay at Brown University's ICERM in fall 2014.
	I would also like to thank Prof.~Dr.~Klaus Ritter for pointing me
	to the paper of Heinrich and Milla~\cite{HeM11}.}
We will comment on that.

Consider the indefinite integration
\begin{equation*}
	S^d : L_p([0,1]^d) \rightarrow L_{\infty}([0,1]^d), \quad
		 [S^d(f)](\vecx)
				:= \int_{\llbracket \zeros,\vecx \rrbracket}
							f \rd \lambda^d \,,
\end{equation*}
where~\mbox{$1 < p \leq \infty$}.
The simple Monte Carlo sampling algorithm~$A_n$ is given by
\begin{equation*}
	[A_n(f)](\vecx)
		:= \frac{1}{n} \sum_{i=1}^n
					\ind[\vecX_i \leq \vecx] \, f(\vecX_i) \,,
\end{equation*}
with iid random variables~\mbox{$\vecX_i \sim \Uniform([0,1]^d)$}.
As discussed in~\cite[Sec~6.3]{HeM11}, this algorithm is not measurable
since the method is not separably valued.
Indeed, considering the constant function~$f_1 = 1$,
for two realizations~$A_n^{\omega}$ and~$A_n^{\omega'}$
with distinct sample points~\mbox{$\vecX_i(\omega)$} and~\mbox{$\vecX_i(\omega')$}
modulo ordering,
we have~\mbox{$\| A_n^{\omega} f_1 - A_n^{\omega'} f_1 \|_{\infty}
								\geq 1/n$}.
Still,
the error mapping \mbox{$\omega \mapsto \|S^d(f) - A_n^{\omega}(f)\|_{\infty}$}
is measurable and an error analysis makes sense.\footnote{%
	In detail, Heinrich and Milla~\cite[Thm~3.4]{HeM11}
	showed polynomial tractability in the randomized setting.
	Note that in the deterministic setting the problem is unsolvable
	because we may only use function values of~$L_p$-functions as information.
	By this, indefinite integration is an example of a problem
	where the output space consists of functions and where randomization
	does help.
	Heinrich and Milla also note
	that only few polynomially tractable problems with unweighted dimensions
	have been known so far.
	Their example of indefinite integration is such an unweighted problem
	with polynomial tractability.
	In \secref{sec:HilbertPeriodic} of this dissertation we add another example:
	The $L_{\infty}$-approximation of Hilbert space functions from
	unweighted periodic Korobov spaces with standard information~$\Lall$
	is polynomially tractable in the Monte Carlo setting.}
In detail, Heinrich and Milla show that it suffices to consider
the pointwise difference~\mbox{$|[S^d(f)](\vecx) - [A_n^{\omega}(f)](\vecx)|$}
for points~$\vecx$ from a regular grid~\mbox{$\Gamma_m \subset [0,1]^d$}
with mesh size~$1/m$, see~\cite[Sec~3]{HeM11}.

This motivates a very natural measurable modification of the sample algorithm.
A computer can only store finitely many digits of
the coordinates of~\mbox{$\vecX_i$},
in general, for fixed~$\omega$,
the function~\mbox{$g_i^{\omega}(\vecx) := \ind[\vecX_i \leq \vecx]$}
is not exactly implementable.
Therefore,
let~\mbox{$\vecX_i^{(r)}(j) \leq \vecX_i(j)$}
be the largest rational number
representable with~$r$ binary digits after the radix point,
\mbox{$j=1,\ldots,d$}.
The algorithm 
\begin{equation*}
	[A_{n,r}(f)](\vecx)
		:= \frac{1}{n} \sum_{i=1}^n
					\ind[\vecX_i^{(r)} \leq \vecx] \, f(\vecX_i)
\end{equation*}
is composed of measurable mappings, and thereby measurable itself.
Indeed, the mapping
\begin{equation*}
	\Omega \rightarrow L_{\infty}([0,1]^d),\quad
	\omega \mapsto \ind[(\vecX_i^{(r)})^{\omega} \leq \cdot\,]
\end{equation*}
is measurable since it only has discrete values.
Furthermore, the pointwise error of~$A_n$ and~$A_{n,r}$
coincides on the grid~$\Gamma_{2^r}$.
With increasing~$r$ we can get arbitrarily close to the error of~$A_n$,
compare Heinrich and Milla~\cite[Thm~3.4]{HeM11}.
The original publication
contains another modification with continuous outputs.
The modification given here, however, nourishes the belief
in measurability of implementable algorithms.

This was an example of non-measurability of Monte Carlo algorithms.
Typically, measurability is an assumption in the randomized and
in the average setting, but we do not need it for the worst case setting.
If we assume measurability only for randomized algorithms,
but allow non-measurability for deterministic algorithms
in the worst case setting,
for linear problems~$S$
with the input set~$F$ being the unit ball in~$\widetilde{F}$
and with general linear information~$\Lall$,
one can still state
\begin{equation*}
	\bar{e}^{\ran}(n,S,F,\Lall) \leq 4 \, e^{\deter}(n,S,F,\Lall) \,,
\end{equation*}
see Heinrich~\cite[p.~282, (4)]{He92}.

As we have seen in the example above, we do not really need
measurability for the algorithm as long as the error mapping is measurable.
Measurability, however, is a convenient assumption,
especially for the average case analysis in \chapref{chap:Bernstein},
where we need to establish
the conditional measure for given information~$\vecy$,
see \secref{sec:GaussCond}.
As long as we do not find meaningful non-measurable algorithms
that could not be replaced by equally successful measurable algorithms,
measurability is a justifiable assumption for lower bound studies.
For the lower Monte Carlo bounds in \chapref{chap:monotone}, however,
measurability is unproblematic since we consider average settings
with discrete measures,
see \thmref{thm:MonAppOrderConv} and \thmref{thm:monotonLB}.
In that case, relaxed measurability assumptions would suffice,
but we do not go into details.

We finish with a final remark on an alternative approximation concept,
aside from the IBC setting with the real number model.
Given a numerical problem~\mbox{$S: F \rightarrow G$},
we define \emph{entropy numbers} for~$n \in \N_0$,
\begin{equation*}
	e_n(S,F) := \inf \bigl\{\eps > 0 \,\big|\,
										\exists g_1,\ldots,g_{2^n} \in G :
											S(F) \subseteq \bigcup_{i=1}^{2^n} B_G(g_i,\eps)
										\bigr\} \,,
\end{equation*}
where~\mbox{$B_G(g,\eps)$} denotes
the closed $\eps$-ball around~\mbox{$g \in G$},
see Carl~\cite{Carl81}, alternatively Pisier~\cite[Chap~5]{Pis89}.
One interpretation of this concept is the question on
how well we can approximate the problem~$S$
if we are only allowed to use $n$~bits
to represent $2^n$~different outputs.\footnote{%
	The given definition contains an index shift
	compared to the definition to be found in Carl~\cite{Carl81}.
	This is a matter of taste.
	Here, \mbox{$e_0(S,F)$} coincides with the initial error from the IBC setting.
	According to Carl's notation, we would start with~$n=1$,
	and the initial error would match~\mbox{$e_1(S,F)$}.
	Similar index shifts compared to related notions from IBC
	are commonly found for \emph{$s$-numbers},
	following the axiomatic scheme of Pietsch~\cite{Pie74}.
	Contrarily, Hutton, Morrell, and Retherford~\cite{HMR76}
	use a definition of approximation numbers
	which happens to fit the IBC notion.
	Heinrich~\cite{He92}, in turn, in his paper on lower bounds,
	on which \chapref{chap:Bernstein} is based on,
	and Math\'e~\cite{Ma91} in his fundamental research on random approximation
	by~$\Lall$, which inspired \chapref{chap:Hilbert},
	both kept consitency with the $s$-number conventions,
	even for the definition of the Monte Carlo error.
	In this thesis, however, we strictly follow IBC conventions
	for error quantities.
	In contrast,
	for the definition of Bernstein numbers in \secref{sec:BernsteinSetting},
	we use a definition which fits to the $s$-number scheme,
	see the footnote given there for additional justification.
	See also \lemref{lem:H->,lin=opt,singular}~(b)
	for the link between singular values and the worst case error
	in the Hilbert space setting.}
Carl studies linear problems
and establishes a lower bound for certain \emph{$s$-numbers}
based on entropy numbers.
Some of the $s$-numbers are closely related to the
error quantities for deterministic algorithms with~$\Lall$,
\emph{approximation numbers} correspond to the error of linear methods,
\emph{Gelfand numbers} are linked to the error
	of general deterministic methods.\footnote{%
	See the book on IBC by Traub et al.~\cite[pp.~70--73]{TWW88}.}\\
Within this dissertation,
in \remref{rem:monMCLBintractable}
we cite a lower bound for the Monte Carlo error
for the approximation of monotone functions
which is due to Bshouty and Tamon~\cite{BT96}.
Their proof uses an entropy argument.\\
In \chapref{chap:Hilbert},
in the context of estimates on the expected maximum
of zero-mean Gaussian fields,
we will step across the inverse concept
\emph{metric entropy}~\mbox{$H(\eps)$},
that is the logarithm of the minimal number of $\eps$-balls
needed to cover a set,
see \propref{prop:Dudley} (Dudley).

\chapter{Lower Bounds for Linear Problems via Bernstein Numbers}
\label{chap:Bernstein}

We consider adaptive Monte Carlo methods for linear problems
and establish a lower bound via Bernstein numbers,
see \secref{sec:BernsteinSetting} for the definition
and an overview of already known relations.
The abstract main result and the proof is contained
in \secref{sec:BernsteinAda}.
It is based on a technique due to Heinrich~\cite{He92}
that relates the Monte Carlo error to norm expectations
of Gaussian measures.
The innovation is the application of Lewis' theorem
in order to find optimal Gaussian measures,
see \secref{sec:OptGauss}.
Within the supplementary \secref{sec:BernsteinSpecial}
we present versions of the main result for two interesting special settings:
varying cardinality, and homogeneous algorithms.
A major application is
the~$L_{\infty}$-approximation of certain classes of $C^{\infty}$-functions,
see \secref{sec:Cinf->Linf}.
With the new technique we obtain lower bounds via Bernstein numbers,
which show that in these cases randomization cannot give us
better tractability than that what we already have with deterministic methods.

\section{The Setting and Bernstein Numbers}
\label{sec:BernsteinSetting}

Let~\mbox{$S: \widetilde{F} \rightarrow G$}
be a compact linear operator between Banach spaces over the reals.
Throughout this chapter the \emph{input set}~\mbox{$F \subset \widetilde{F}$}
is the unit ball of~$\widetilde{F}$.
We consider algorithms that may use arbitrary continuous linear
functionals~$\Lall$ as information.

The operator~$S$ can be analysed in terms of \emph{Bernstein numbers}
\begin{equation} \label{eq:BernsteinN}
	b_m(S) := \sup_{X_m \subseteq \widetilde{F}}
							\inf_{\substack{f \in X_m \\
															\|f\| = 1}}
								\|S(f)\|_G \,,
\end{equation}
where the supremum is taken
over $m$-dimensional\footnote{%
	Some authors take the supremum over \mbox{$(m+1)$}-dimensional
	spaces~\cite{NguyenVK15,OP95},
	which might be motivated by relations like~\eqref{eq:deterBernstein}.
	The present version, however, is also in common use~\cite{Kudr99,NguyenVK16},
	besides it looks quite natural,
	and in view of the sharp estimate \eqref{eq:BernsteinLBhom2n},
	any index shift would appear like a disimprovement.
	Although Bernstein numbers are not $s$-numbers,
	according to the definition in Pietsch~\cite{Pie74},
	in some cases they coincide with certain $s$-numbers,
	in particular for operators between Hilbert spaces
	where Bernstein numbers match the singular values.
	}
linear subspaces~\mbox{$X_m \subseteq \widetilde{F}$}.
These quantities are closely related to the
\emph{Bernstein widths}\footnote{%
	I wish to thank Prof.~Dr.~Stefan~Heinrich for making me aware
	of the non-equivalence of both notions.}
of the image~\mbox{$S(F)$} within~$G$,
\begin{equation} \label{eq:BernsteinW}
	b_m(S(F),G) := \sup_{Y_m \subseteq G}
									\sup \{r \geq 0 \, \mid \,
													B_r(0) \cap Y_m \subseteq S(F) \} \,,
\end{equation}
where the first supremum is taken
over $m$-dimensional linear subspaces~\mbox{$Y_m \subseteq G$}.
By~\mbox{$B_r(g)$} we denote the (closed) ball around~\mbox{$g \in G$} with radius~$r$.
In general, Bernstein widths are greater than Bernstein numbers,
however, for injective operators (like embeddings) both notions coincide
(consider $Y_m = S(X_m)$).
In the case of Hilbert spaces $\widetilde{F}$ and~$G$,
Bernstein numbers and widths match the singular values~$\sigma_m(S)$.

For deterministic algorithms it can be easily seen that
\begin{equation} \label{eq:deterBernstein}
	e^{\deter}(n,S) \geq b_{n+1}(S(F),G) \geq b_{n+1}(S) \,,
\end{equation}
since for any information mapping~\mbox{$N: \widetilde{F} \rightarrow \R^n$}
and any~\mbox{$\eps>0$}, there always exists an~\mbox{$f \in N^{-1}(\zeros)$}
with~\mbox{$\|S(f)\|_G \geq b_{n+1}(S(F),G) - \eps$}
and \mbox{$\pm f \in F$},
i.e.\ $f$~cannot be distinguished from~\mbox{$-f$}.

If both $\widetilde{F}$ and~$G$ are Hilbert spaces, lower bounds for the
(root mean square) Monte Carlo error have been found by Novak~\cite{No92},
\begin{equation}
	e_2^{\ran}(n,S)
		\geq {\textstyle \frac{\sqrt{2}}{2}} \, \sigma_{2n}(S) \,.	
\end{equation}
For operators between arbitrary Banach spaces
the estimate reads quite similar, see \thmref{thm:BernsteinMCada},
\begin{equation} \label{eq:BernsteinLBada2n}
	e^{\ran,\ada}(n,S) > {\textstyle \frac{1}{30}} \, b_{2n}(S) \,.
\end{equation}
The constant can be improved for extremely large~$n$,
see \remref{rem:1/6*b_2m}, or when imposing further assumptions.
The following lower bound for non-adaptive algorithms has been proven
first within the author's master thesis and published later in~\cite{Ku16},
\begin{equation}
	e^{\ran,\nonada}(n,S) \geq {\textstyle \frac{1}{2}} \, b_{2n+1}(S) \,.
\end{equation}
For homogeneous algorithms, possibly adaptive, and even with varying cardinality,
one can prove an estimate with optimal constant,
see \thmref{thm:BernsteinMChom},
\begin{equation} \label{eq:BernsteinLBhom2n}
	e^{\ran,\homog}(n,S) \geq {\textstyle \frac{1}{2}} \, b_{2n}(S) \,.
\end{equation}
Within~\cite{Ku16} the results of~\eqref{eq:BernsteinLBada2n}
and~\eqref{eq:BernsteinLBhom2n} have been mentioned,
for the adaptive setting a proof for a result with slightly worse constants
based on results from Heinrich~\cite{He92} has been given.
In this chapter now one can find
a self-contained proof following the lines of Heinrich~\cite{He92}
but with optimized constants and slight simplifications that are
possible when relying on Bernstein numbers.

\section{Adaptive Monte Carlo Methods}
\label{sec:BernsteinAda}

\thmref{thm:BernsteinMCada} below is the main result of this chapter.
The proof needs several results that are provided in the subsequent subsections.

The proof is based on the idea of Heinrich~\cite{He92}
to use truncated Gaussian measures in order to obtain lower
bounds for the Monte Carlo error.
Considering Gaussian measures is quite convenient
as there is an easy representation for the conditional distribution,
even when collecting adaptive information, see \secref{sec:GaussCond}.
The key tool for Heinrich's technique
is a deviation result for zero-mean Gaussian measures~$\tilde{\mu}$
on a normed space~$\widetilde{F}$,
\begin{equation*}
	\tilde{\mu}\{f \,:\, \|f\|_F > \lambda \, \expect^{\tilde{\mu}} \|f\|_F\}
			\leq \exp\left(- \frac{(\lambda-1)^2}{\pi} \right) \,,
\end{equation*}
see \corref{cor:deviationGauss}.
This shows how far the norm may deviate from its expected value,
that way enabling us to estimate how much we lose when truncating a Gaussian measure.
The expected norm for a truncated Gaussian measure is estimated
in \secref{sec:E|JX|trunc}, in the case of Bernstein numbers a simplified
result with slightly better constants is feasible.

The new idea
now is to apply Lewis' theorem
in order to find optimal Gaussian measures
that are ``well spread'' into all directions
within the input space~$\widetilde{F}$, see \secref{sec:OptGauss}.\footnote{%
	This idea has already been published in~\cite{Ku16}.
	I wish to thank Prof.~Dr.~Aicke~Hinrichs and my doctoral advisor Prof.~Dr.~Erich~Novak
	for pointing me to the book of Pisier~\cite{Pis89}
	in search of optimal Gaussian measures.}

This dissertation includes a self-contained proof of the theorem.
That way we are able to adapt for simplifications that are possible
in the particular situation.
Furthermore, we work on the improvement of constants.

\begin{theorem} \label{thm:BernsteinMCada}
	For~\mbox{$S: \widetilde{F} \rightarrow G$} being a compact linear
	operator between Banach spaces, and the input set~$F$ being the
	unit ball in~$\widetilde{F}$, we have
	\begin{equation*}
		e^{\ran,\ada}(n,S,F,\Lall)
			> \frac{1}{15} \, \frac{m-n}{m} \, b_m(S)
		\quad \text{for\, $m > n$.}
	\end{equation*}
\end{theorem}
\begin{proof}
	For all $\eps > 0$ there exists an $m$-dimensional
	subspace~$X_m \subseteq \widetilde{F}$ such that
	\begin{equation*}
		\|S(f)\|_G \geq \|f\|_F \, (b_m(S) - \eps)
			\quad \text{for\, $f \in X_m$.}
	\end{equation*}
	Note that for the restricted operator we have
	\mbox{$b_m(S|_{X_m}) \geq b_m(S) - \eps$},
	and in general \mbox{$e^{\ran}(n,S,F) \geq e^{\ran}(n,S|_{X_m},F \cap X_m)$}.
	Hence it suffices to show the theorem for~$S|_{X_m}$,
	so without loss of generality
	we assume \mbox{$X_m = \widetilde{F} = \R^m$},
	and therefore
	\mbox{$\|S(f)\|_G \geq \|f\|_F \, b_m(S)$} holds
	for all~\mbox{$f \in \widetilde{F}$}.
	
	Below, $\P$ and $\expect$ are used to describe probabilities and
	expectations for an average case setting whenever it seems convenient.
	This is not to be confused
	with the probability space~\mbox{$(\Omega,\Sigma,\P)$}
	used to define Monte Carlo algorithms within \chapref{chap:basics}.
	Let~$\vecX$ be a standard Gaussian vector within~$\R^m = \ell_2^m$.
	We choose a
	matrix~\mbox{$J:\R^m = \ell_2^m \rightarrow \R^m = \widetilde{F}$}
	in order to define a Gaussian measure~$\tilde{\mu}$ on~$\widetilde{F}$
	as the distribution of~\mbox{$J \vecX$}. The restricted measure
	\begin{equation*}
		\mu(E) := \tilde{\mu}|_F (E) = \P\{J \vecX \in E \cap F\} \,,
		\quad \text{for measurable $E \subseteq \widetilde{F}$,}
	\end{equation*}
	is a sub-probability measure
	supported on the unit ball~\mbox{$F \subset \widetilde{F}$}.
	By Bakhvalov's technique, see \propref{prop:Bakh}, we know
	\begin{equation*}
		e^{\ran,\ada}(n,S,F,\Lall) \geq e^{\avg,\ada}(n,S,\mu,\Lall) \,.
	\end{equation*}
	
	Let~$\phi \circ N : \widetilde{F} \rightarrow G$
	be an adaptive deterministic algorithm using $n$~pieces of information.
	Let~\mbox{$\tilde{\nu} = \tilde{\mu} \circ N^{-1}$} denote the distribution
	of the information~\mbox{$N(J \vecX)$}.
	Without loss of generality,
	\mbox{$N : \widetilde{F} \rightarrow \R^n$}
	is surjective, so by \lemref{lem:condGauss},
	for all~\mbox{$\vecy \in \R^n$} we have an orthogonal projection~$P_{\vecy}$
	and an element~$m_{\vecy} \in \widetilde{F}$ to describe the
	conditional distribution~$\tilde{\mu}_{\vecy}$
	as the distribution of~\mbox{$J P_{\vecy} \vecX + m_{\vecy}$}.
	We write the error
	\begin{align*}
		e(\phi \circ N,\mu)
			\,&= \int_{\R^n}
						\int_{F \cap N^{-1}(\vecy)}
								\|S(f) - \phi(\vecy)\|_G
							\, \tilde{\mu}_{\vecy}(\diff f)
					\, \tilde{\nu}(\diff \vecy) \,. \\
		\intertext{%
	Defining \mbox{$g_{\vecy}:=\phi(\vecy) - S(m_{\vecy})$},
	we can continue
	using the representation of the conditional measure~$\tilde{\mu}_{\vecy}$,
	further cutting off parts of the integral,}
			&\geq
				\int_{\R^n}
						\expect
							\Bigr[ \|S J P_{\vecy} \vecX - g_{\vecy}\|_G
										\, \ind_{\{\|J P_{\vecy} \vecX\|_F \leq 1 - \|m_{\vecy}\|_F
														\}}
							\Bigl]
					\, \tilde{\nu}(\diff \vecy) \,.\\
		\intertext{%
	Due to symmetry,
	the two versions
	\mbox{$\|S J P_{\vecy} \vecX \pm g_{\vecy}\|_G
					\, \ind_{\{\|J P_{\vecy} \vecX\|_F \leq r\}}$}
	are identically distributed (\mbox{$r>0$}),
	so we can rewrite}
			&=
				\int_{\R^n}
						\expect
							\Bigr[\underbrace{
												{ \frac{1}{2}} \sum_{\sigma = \pm 1}
													\|S J P_{\vecy} \vecX + \sigma g_{\vecy}\|_G
											}_{[\text{$\Delta$-ineq.}] \quad \geq \|S J P_{\vecy} \vecX\|_G}
										\, \ind_{\{\|J P_{\vecy} \vecX\|_F \leq 1 - \|m_{\vecy}\|_F
														\}}
							\Bigl]
					\, \tilde{\nu}(\diff \vecy) \,.\\
		\intertext{%
	Applying the triangle inequality,
	and further truncating the integral,
	we get}
			&\geq
				\int_{\{\|m_{\vecy}\|_F \leq 1 - r\}}
						\expect
							\Bigr[\|S J P_{\vecy} \vecX\|_G
										\, \ind_{\{\|J P_{\vecy} \vecX\|_F
																			\leq r
														\}}
							\Bigl]
					\, \tilde{\nu}(\diff \vecy) \,.		
		\intertext{%
	Using the definition of the Bernstein numbers,
	and replacing the projections~$P_{\vecy}$
	by a general estimate for orthogonal rank-$(m-n)$ projections,
	we end up with}
		e(\phi \circ N,\mu)
			\,&\geq\,
				\tilde{\nu}\{\vecy \,:\, \|m_{\vecy}\|_F
												\leq 1-r\} \,
					\inf_{\substack{\text{$P$ orth.~Proj.}\\
													\rank P = m-n}}
						\expect
							\Bigr[\|J P \vecX\|_F
										\, \ind_{\{\|J P_{\vecy} \vecX\|_F
																			\leq {\textstyle r}
														\}}
							\Bigl]
					\, b_m(S) \,.
	\end{align*}
	
	From now on we write~\mbox{$\alpha := \expect \|J \vecX\|_F$}.
	
	First,
	we need an estimate for the probability of a small~$m_{\vecy}$,
	this is done in \lemref{lem:|m_y|<r}
	with~\mbox{$1-r = 2 \, \kappa \, \alpha$},
	\begin{equation} \label{eq:nu(m_y<1-r)}
		\tilde{\nu}\{\vecy \,:\, \|m_{\vecy}\|_F \leq 1-r\}
			\geq 1 - 2 \, \exp\left(- \left(\frac{1-r}{2 \, \alpha}-1\right)^2 \bigg/\pi
												\right)
			=: \nu(r,\alpha) \,.
	\end{equation}
	This estimate is meaningful
	for~\mbox{$r < 1- 2 \, \alpha \, (1+\sqrt{\pi \log 2})$}.
	
	Second,
	for orthogonal projections~$P$ on~$\ell_2^m$ and~$t>0$,
	the truncated expectation can be estimated by
	\begin{equation} \label{eq:E|JPX|trunc}
	\begin{split}
		\expect\Bigl[\|J P \vecX\|_F
									\, \ind_{\{\|J P \vecX\|_F
																\leq \lambda\,\expect \|J \vecX\|_F
													\}}
						\Bigr]
			&\geq \expect\Bigl[\|J P \vecX\|_F
													\, \ind_{\{\|J P \vecX\|_F
																\leq \lambda\,\expect \|J P \vecX\|_F
													\}}
									\Bigr] \\
			&\geq \beta(\lambda) \, \expect \|J P \vecX\|_F \,.
	\end{split}
	\end{equation}
	Here,
	within the first step we used
	\mbox{$\expect \|J P \vecX\|_F \leq \expect \|J \vecX\|_F$},
	see \lemref{lem:E|JPX|<=E|JX|}.
	The second inequality is the application of \lemref{lem:E|JX|trunc}
	with the operator~\mbox{$J' := J P$}
	and the constant~\mbox{$\beta(\lambda)$} defined there.
	In our situation~\mbox{$\lambda = \frac{r}{\alpha}$}.
	
	By \corref{cor:optGauss} we know that~$J$ can be chosen in a way
	such that for any rank~$(m-n)$ projection~$P$ on~$\R^m$ we have\footnote{%
		This idea is new and special for the situation of Bernstein numbers.
		However, similar properties have been known to Heinrich~\cite[Lem~3]{He92}
		for the special case of the standard Gaussian distribution
		in sequence spaces~$\ell_p^m$,
		compare also \remref{rem:UniqueGauss}.
		Heinrich used a symmetry argument
		that can be found in Math\'e~\cite[Lem~4]{Ma91}.}
	\begin{equation*}
		\expect \|J P \vecX\|_F
			\geq \frac{m-n}{m} \, \expect \|J \vecX\|_F \,.
	\end{equation*}
	In detail, we choose~\mbox{$J := \alpha \widetilde{J}$},
	with~$\widetilde{J}$ being the optimal operator from \corref{cor:optGauss},
	and \mbox{$\alpha > 0$}.
	
	Putting all this together, we obtain the estimate
	\begin{equation*}
		e(\phi \circ N,\mu)
			\geq c \, \frac{m-n}{m} \, b_m(S)
	\end{equation*}
	with \mbox{$c := \nu(r,\alpha) \, \beta\left(\frac{r}{\alpha}\right) \, \alpha$}.
	Note that~\mbox{$\nu(r,\alpha) > 0$} iff
	\mbox{$r < 1- 2 \,\alpha \, (1+\sqrt{\pi \log 2})$}.
	On the other hand, \mbox{$\beta(\lambda)$} gives meaningful results
	for~\mbox{$\lambda > 3.0513$} only,
	so we need~\mbox{$r > 3.0513 \, \alpha$}
	in order to obtain positive estimates.
	Combining this, we have the constraint
	\begin{equation*}
		0 < \alpha
			< \frac{1}{5.0513 + 2 \, \sqrt{\pi \log 2}}
			< 0.125 = \frac{1}{8} \,.
	\end{equation*}
	With~\mbox{$r = 0.37$} and \mbox{$\alpha = 0.0735$},
	we find a constant~\mbox{$c = 0.06667... > \frac{1}{15}$}.
\end{proof}

\subsection{The Conditional Measure}
\label{sec:GaussCond}

The following lemma gives the conditional measure for adaptive information
mappings applied in an Gaussian average case setting.
The conditional measure is well known since the study of average errors,
see the book on IBC, Traub et al.~\cite[pp.~471]{TWW88}.
This reference has also been given in
Heinrich~\cite[p.~287]{He92}.
The proof given here is intended to be self-contained
and uses a slightly different notation.

\begin{lemma} \label{lem:condGauss}
	Let~$\vecX$ be a standard Gaussian vector in~$\R^m$,
	and~\mbox{$J: \R^m \rightarrow \widetilde{F}$}
	an injective linear operator defining a measure~$\tilde{\mu}$
	on~\mbox{$\widetilde{F}$} as the distribution of~\mbox{$f := J \vecX$}.
	Furthermore, let~\mbox{$N:\widetilde{F} \rightarrow \R^n$} be
	a \emph{non-wasteful}\footnote{%
		That means,
		\mbox{$N(\supp(\tilde{\mu})) = \R^n$}, so if
		$\widetilde{F}$ is $m$-dimensional,
		$N$ shall be surjective.}
	adaptive deterministic information mapping.
	Then the conditional measure~$\tilde{\mu}_{\vecy}$,
	given the information~\mbox{$\vecy = N(f)$},
	can be described as the
	distribution of~\mbox{$J P_{\vecy} \vecX + m_{\vecy}$},
	with~$P_{\vecy}$ being a suitable rank-\mbox{$(m-n)$} orthogonal projection
	within~\mbox{$\R^m = \ell_2^m$},
	and a suitable vector~\mbox{$m_{\vecy} \in \widetilde{F}$}.
	That is, for all measurable~\mbox{$E \subseteq \widetilde{F}$} we have
	\begin{equation*}
		\tilde{\mu}(E)
			= \P\{J \vecX \in E\}
			= \int_{\R^n}
					\underbrace{\P\{J P_{\vecy} \vecX + m_{\vecy} \in E\}
										}_{= \tilde{\mu}_{\vecy}(E)}
					\, \tilde{\mu} \circ N^{-1} (\diff \vecy) \,.
	\end{equation*}
\end{lemma}
\begin{proof}
	Before going into the details of the proof,
	we want to clarify that
	expressions containing~$\vecX$ are random variables,
	whereas the information vector~$\vecy$,
	and everything depending on it,
	is fixed.
	In particular,
	\mbox{$N(J \vecX)$} is the information as a random vector,
	\mbox{$\{N(J \vecX) = \vecy\}$} is an event fixing the information.
	
	We denote the partial information~\mbox{$N_k(f) := \vecy_{[k]}$}
	for~\mbox{$k=0,1,\ldots,n$}.
	We will show by induction that
	the conditional measure~\mbox{$\tilde{\mu}_{\vecy_{[k]}}$},
	knowing the first~$k$ information values~$\vecy_{[k]}$,
	can be represented
	as the distribution of~\mbox{$J P_{k,\vecy_{[k-1]}} \vecX + m_{\vecy_{[k]}}$},
	with $P_{k,\vecy_{[k-1]}}$ being a suitable rank-\mbox{$(m-k)$} orthogonal
	projection within~\mbox{$\R^m = \ell_2^m$},
	and a suitable vector~\mbox{$m_{\vecy_{[k]}} \in \widetilde{F}$}.
	Moreover,
	there exists a vector~\mbox{$\vecn_{\vecy_{[k]}} \in \ell_2^m$}
	such that~\mbox{$m_{\vecy_{[k]}} = J \vecn_{\vecy_{[k]}}$}
	and~\mbox{$P_{k,\vecy_{[k-1]}} \vecn_{\vecy_{[k]}} = \zeros$}.
	For convenience,
	we also show that the information mapping can be
	chosen in a way such that the distribution~\mbox{$\tilde{\mu} \circ N_k^{-1}$}
	of the partial information~\mbox{$N_k(J \vecX) = \vecy_{[k]}$}
	is the \mbox{$k$-dimensional} standard Gaussian distribution
	which we denote by~$\gamma_k$.
		
	Starting with~\mbox{$k = 0$} means that we have no
	information~\mbox{$\vecy_{\emptyset} = 0 \in \R^0$},\footnote{%
		$\R^0 = \{0\}$ is the zero vector space.}
	the conditional distribution~\mbox{$\tilde{\mu} = \tilde{\mu}_{\vecy_{\emptyset}}$}
	is described by~\mbox{$J \vecX = J P_{0,\vecy_{\emptyset}} \vecX + m_0$}
	where~\mbox{$P_{0,\vecy_{\emptyset}} = \id_{\ell_2^m}$}
	and~\mbox{$m_0 = 0 \in \widetilde{F}$},
	or~\mbox{$\vecn_0 = \zeros \in \ell_2^m$}.
	The partial information is distributed according to the
	``zero-dimensional standard Gaussian distribution'',
	that is~\mbox{$\P\{N_0(J \vecX) = 0\} = 1$}.
	
	Now for~\mbox{$k=1,\ldots,n$}.
	Given the partial information~$\vecy_{[k-1]}$,
	the (adaptively chosen)
	\mbox{$k$-th}~information functional~$L_{k,\vecy_{[k-1]}}$
	actually gives us information about the random vector~\mbox{$\vecX \in \R^m$}.
	In detail,
	\mbox{$L_{k,\vecy_{[k-1]}} (J \, \cdot)$} is a functional in~$\ell_2^m$,
	so there exists a representing
	vector~\mbox{$\vecxi_{k,\vecy_{[k-1]}} \in \ell_2^m$}
	such that the $k$-th information value
	(as a random variable for fixed~$\vecy_{[k-1]}$) is~\mbox{
		$L_{k,\vecy_{[k-1]}} (J \vecX)
				= \langle \vecxi_{k,\vecy_{[k-1]}},
									\vecX
					\rangle$}.
	Not waisting any information actually means that
	\mbox{$L_{k,\vecy_{[k-1]}} (J P_{k-1,\vecy_{[k-2]}} \, \cdot)$}
	is not the zero functional, therefore we may assume
	\mbox{$L_{k,\vecy_{[k-1]}} (J P_{k-1,\vecy_{[k-2]}} \, \cdot)
					= L_{k,\vecy_{[k-1]}} (J \, \cdot)$}.
	This is equivalent to~\mbox{$P_{k-1,\vecy_{[k-2]}} \vecxi_{k,\vecy_{[k-1]}}
																	= \vecxi_{k,\vecy_{[k-1]}}$},
	and by induction it further implies the orthogonality
	\mbox{$L_{k,\vecy_{[k-1]}} (m_{\vecy_{[k-1]}})
					= \langle \vecxi_{k,\vecy_{[k-1]}},
										\vecn_{\vecy_{[k-1]}}
						\rangle
					= 0$}.
	In addition,
	we may assume that~\mbox{$\|\vecxi_{k,\vecy_{[k-1]}}\|_2 = 1$}
	such that \mbox{$\langle \vecxi_{k,\vecy_{[k-1]}},\vecX \rangle$}
	is a standard Gaussian random variable.
	We now set
	\begin{equation*}
		P_{k,\vecy_{[k-1]}} \vecx
			:= P_{k-1,\vecy_{[k-2]}} \vecx
					- \langle \vecxi_{k,\vecy_{[k-1]}}, \vecx \rangle
							\, \vecxi_{k,\vecy_{[k-1]}} \,,
		\quad \text{and} \quad
		\vecn_{\vecy_{[k]}}
			:= \vecn_{\vecy_{[k-1]}} + y_k \, \vecxi_{k,\vecy_{[k-1]}} \,,
	\end{equation*}
	thus defining~$\tilde{\mu}_{\vecy_{[k]}}$.
	Note that by construction~\mbox{$P_{k,\vecy_{[k-1]}} \vecn_{\vecy_{[k]}} = \zeros$}.
	Then for any measurable set~\mbox{$E \subseteq \widetilde{F}$} we have
	\begin{align*}
		\tilde{\mu}_{\vecy_{[k-1]}}(E)
			&= \P\{J P_{k-1,\vecy_{[k-2]}} \vecX + m_{\vecy_{[k-1]}} \in E\} \\
			&= \P\{J (P_{k,\vecy_{[k-1]}} \vecX
									+ \langle \vecxi_{k,\vecy_{[k-1]}}, \vecX \rangle
										\, \vecxi_{k,\vecy_{[k-1]}})
							+ J \vecn_{\vecy_{[k-1]}}
						\in E\}\\
			&\stackrel{\text{($\ast$)}}{=}
				\int_{\R}
						\P\{J P_{k,\vecy_{[k-1]}} \vecX
								+ \underbrace{J (\vecn_{\vecy_{[k-1]}}
																+ y_k \, \vecxi_{k,\vecy_{[k-1]}})
															}_{= J \vecn_{\vecy_{[k]}} = m_{\vecy_{[k]}}}
								\in E \}
					\, \gamma_1(\diff y_k) \\
			&\stackrel{\text{def.}}{=}
				\int_{\R} \tilde{\mu}_{\vecy_{[k]}}(E) \, \gamma_1(\diff y_k) \,.
	\end{align*}
	The step~\mbox{($\ast$)} of splitting the integral into two integrations
	was possible because
	\begin{itemize}
		\item the Gaussian random vector~\mbox{$P_{k,\vecy_{[k-1]}} \vecX$}
			is stochastically independent from the Gaussian random variable~
			\mbox{$\langle \vecxi_{k,\vecy_{[k-1]}}, \vecX \rangle$}
			due to orthogonality, and
		\item the span of~\mbox{$J \vecxi_{k,\vecy_{[k-1]}}$}
			is not inside the image of~\mbox{$J P_{k,\vecy_{[k-1]}}$},
			which provides that for any
			\mbox{$f \in E \cap \supp(\tilde{\mu}_{\vecy_{[k-1]}})
									= E \cap (\image(J P_{k-1,\vecy_{[k-2]}})
														+ m_{\vecy_{[k-1]}})$}
			there is a unique representation
			\mbox{$f = f_k + y_k J \vecxi_{k,\vecy_{[k-1]}} + m_{\vecy_{[k]}}$}
			with~\mbox{$y_k \in \R$} and a vector
			\mbox{$f_k \in \image(J P_{k,\vecy_{[k-1]}})$}.
	\end{itemize}
	Now, by induction, we have
	\begin{align*}
		\tilde{\mu}(E)
			&= \int_{\R^{k-1}}
						\tilde{\mu}_{\vecy_{[k-1]}}(E)
					\, \gamma_{k-1}(\diff \vecy_{[k-1]}) \,, \\
		\intertext{which by the above results and the product structure of
		the standard Gaussian measure may be continued as}
			&= \int_{\R^{k-1}}
						\int_{\R}
								\tilde{\mu}_{\vecy_{[k]}}(E)
							\, \gamma_1(\diff y_k)
					\, \gamma_{k-1}(\diff \vecy_{[k-1]}) \\
			&= \int_{\R^k} \tilde{\mu}_{\vecy_{[k]}}(E) \, \gamma_k(\diff \vecy_{[k]}) \,.
	\end{align*}
	The lemma is obtained for~\mbox{$k=n$}
	with~\mbox{$P_{\vecy} = P_{n,\vecy_{[n-1]}}$}
	and \mbox{$m_{\vecy} = J \vecn_{\vecy_{[n]}}$}.
\end{proof}

Having the representation of the conditional measure for the
untruncated Gaussian measure,
it is of interest to know the probability of obtaining information
such that the mass of the conditional measure is concentrated inside the
unit ball~\mbox{$F \subseteq \widetilde{F}$} and therefore
truncation is not a great loss.
\begin{lemma} \label{lem:|m_y|<r}
	In the situation of \lemref{lem:condGauss},
	with~\mbox{$\rho := \expect \|J \vecX\|_F / \|J\|_{2 \rightarrow F}$},
	and the image measure~\mbox{$\tilde{\nu} := \tilde{\mu} \circ N^{-1}$},
	for~\mbox{$\kappa > 1$}	we have the estimate
	\begin{align*}
		\tilde{\nu}\{\vecy \,:\,
								\|m_{\vecy}\|_F \leq 2 \, \kappa \, \expect \|J \vecX\|_F\}
			&\geq 1 - 2 \, \exp\left(- \frac{(\kappa - 1)^2 \, \rho^2}{2}
												\right) \\
			&\geq 1 - 2 \, \exp\left(- \frac{(\kappa-1)^2}{\pi}\right) \,.
	\end{align*}
\end{lemma}
\begin{proof}
	Writing~\mbox{$t :=\kappa \, \expect \|J \vecX\|_F$},
	basic estimates give
	\begin{align*}
		\tilde{\nu}\{\vecy \,:\,
								\|m_{\vecy}\|_F \leq 2 \, t\}
			&\geq \tilde{\mu}\{f \,:\,
												\|f\|_F \leq t\} \\
				&\quad	- \tilde{\mu}\{f \,:\,\|m_{\vecy}\|_F > 2 \, t
										\text{ with $\vecy = N(f)$ and }
										\|f\|_F \leq t\} \\
			&\geq 1 - \tilde{\mu}\{f \,:\,\|f\|_F > t\}
							- \sup_{\vecy} \tilde{\mu}_{\vecy}
																\{f \,:\, \|f - m_{\vecy}\|_F > t\} \\
			&\geq 1 - \P\{\|J\vecX\|_F > t\}
							- \sup_{\substack{\text{$P$ orth.\ proj.}\\
																\rank P = m-n}}
									\P\{\|J P \vecX\|_F \geq t\} \,. \\
		\intertext{%
	Applying \corref{cor:deviationGauss} 
	with~\mbox{$t > \expect \|J \vecX\|_F$} gives us}
			&\geq 1 - \exp\left(- \frac{(t - \expect \|J \vecX\|_F)^2
																	}{2 \, \|J\|_{2 \rightarrow F}^2}
										\right)
							- \exp\left(- \frac{(t - \expect \|J \vecX\|_F)^2
																	}{2 \, \|J P\|_{2 \rightarrow F}^2}
										\right) \,,\\
		\intertext{%
	which by \mbox{$\expect \|J P \vecX\|_F \leq \expect \|J \vecX\|_F$},
	see \lemref{lem:E|JPX|<=E|JX|},
	and \mbox{$\|J P\|_{2 \rightarrow F} \leq \|J \|_{2 \rightarrow F}$},
	reduces to}
			&\geq 1 - 2 \,\exp\left(- \frac{(t - \expect \|J \vecX\|_F)^2
																	}{2 \, \|J\|_{2 \rightarrow F}^2}
											\right) \\
			&= 1 - 2 \, \exp\left(- \frac{(\kappa-1)^2 \, \rho^2}{2}\right)\,.
	\end{align*}
	The first lower bound is meaningful
	for~\mbox{$\kappa > 1 + \sqrt{2 \log 2} / \rho$},
	otherwise it is not positive and should be replaced
	by the trivial lower bound~$0$.
	
	\lemref{lem:E|JX|>c|J|} gives us the general estimate
	\mbox{$\expect \|J \vecX\|_F / \|J\|_{2 \rightarrow F} \geq \sqrt{2/\pi}$},
	which leads to the second lower bound.
	This now is meaningful for~\mbox{$\kappa > 1 + \sqrt{\pi \log 2}$}.
\end{proof}

\begin{example}[Why $\|m_{\vecy}\|_F \leq r$ is a complicated constraint]
	We consider \mbox{$\widetilde{F} = \ell_{\infty}^m$} for $m \geq 2$ and
	\mbox{$J = \id: \ell_2^m \rightarrow \ell_{\infty}^m$}.
	The center~$m_{\vecy} = \vecn_{\vecy}$ of the conditional distribution
	on the affine subspace~\mbox{$N^{-1}(\vecy) \subset \widetilde{F}$}
	of inputs~\mbox{$f = \vecx$} with the same information~$\vecy$
	is orthogonal to that subspace,
	i.e.\ \mbox{$\vecn_{\vecy} \bot (\vecx-\vecn_{\vecy})$}
	for all~\mbox{$\vecx \in N^{-1}(\vecy)$}.
	
	Consider the situation
	\begin{equation*}\textstyle
		m_{\vecy}
			= \vecn_{\vecy}
			= \left(1,\frac{1}{\sqrt{m}+1},\ldots,\frac{1}{\sqrt{m}+1}\right)
						\in \R^m \,.
	\end{equation*}
	If, for example, \mbox{$N(f) := \langle \vecn_{\vecy}, f\rangle \in \R^1$}
	then
	\begin{equation*}\textstyle
		f = \vecx
			= \left(\frac{2}{\sqrt{m}+1},\ldots,\frac{2}{\sqrt{m}+1}\right)
			\in \R^m
	\end{equation*}
	lies within~\mbox{$N^{-1}(\vecy)$} because
	\begin{equation*}\textstyle
		(\vecx - \vecn_{\vecy}) 
						\bot \vecn_{\vecy} \,.
	\end{equation*}
	In this situation we have
	\begin{equation*}\textstyle
		\|m_{\vecy}\|_{\infty} = 1 \quad \text{and} \quad
		\|f\|_{\infty} = \frac{2}{\sqrt{m}+1}
									\xrightarrow[m \rightarrow \infty]{} 0
		\,.
	\end{equation*}
\end{example}

\subsection{Norm Expectation of Truncated Gaussian Measures}
\label{sec:E|JX|trunc}

The following lemma is a simplification of a result
in Heinrich~\cite[Lem~1]{He92}.
This simplification is only feasible in the situation of Bernstein numbers.
The more complicated version is given in \lemref{lem:E|SJX|trunc}.

\begin{lemma} \label{lem:E|JX|trunc}
	For~\mbox{$J:\ell_2^m \rightarrow \widetilde{F}$} and $\vecX$~being
	the standard Gaussian vector in~\mbox{$\R^m = \ell_2^m$},
	set~\mbox{$\rho := \expect \|J \vecX\|_F / \|J\|_{2 \rightarrow F}$}.
	Then for~$\lambda>1$ we have
	\begin{multline*}
		\expect\Bigl[\|J \vecX\|_F
									\, \ind_{\{\|J \vecX\|_F
																\leq \lambda\,\expect \|J \vecX\|_F
													\}}
						\Bigr]\\
			\geq 
				\underbrace{\left[1 - \left(\lambda+\frac{1}{(\lambda-1) \, \rho^2}
															\right)
																\, \exp\left(- \frac{(\lambda-1)^2 \, \rho^2}{2}
																			\right)
										\right]_{+}
									}_{=: \beta(\lambda,\rho)}
					\, \expect \|J \vecX\|_F \\
			\geq
				\underbrace{\left[1 - \left(\lambda+\frac{\pi}{2 \, (\lambda-1)}\right)
																\, \exp\left(- \frac{(\lambda-1)^2}{\pi}\right)
										\right]_{+}
									}_{=: \beta(\lambda)}
					\, \expect \|J \vecX\|_F \,.
	\end{multline*}
\end{lemma}
\noindent
	(Note that~\mbox{$\beta(\lambda)$} vanishes for \mbox{$\lambda \leq 3.0513$},
	but it is positive for~\mbox{$\lambda \geq 3.0514$}
	and monotonically increasing
	with limit~\mbox{$\beta(\lambda) \xrightarrow[\lambda \rightarrow \infty]{} 1$}.)
\begin{proof}
	With~\mbox{$t := \lambda \, \expect \|J \vecX\|_F$} we have
	\begin{equation*}
		\expect\Bigl[\|J \vecX\|_F
									\, \ind_{\{\|J \vecX\|_F \leq t\}}
						\Bigr]
			= \expect \|J \vecX\|_F
				- t \, \P\{\|J \vecX\|_F \geq t \}
							- \int_{t}^{\infty}
										\P\{\|J \vecX\|_F \geq s\}
									\rd s \,.
	\end{equation*}
	By the deviation result \corref{cor:deviationGauss},
	and substituting~\mbox{$s = \kappa \, \expect \|J \vecX\|_F$},
	we can bound this by
	\begin{equation*}
			\qquad\geq
				\left[1 - \lambda \, \exp\left(- \frac{(1-\lambda)^2 \, \rho^2}{2}
																\right)
							- \int_{\lambda}^{\infty}
										\exp\left(- \frac{(1-\kappa)^2 \, \rho^2}{2}\right)
									\rd \kappa
				\right]
					\,\expect \|J \vecX\|_F \,.
	\end{equation*}
	
	Using the estimate
	\begin{align*}
		\int_{\lambda}^{\infty}
				\exp\left(- \frac{(1-\kappa)^2 \, \rho^2}{2}\right)
			\rd \kappa
			&\leq \int_{\lambda}^{\infty} \frac{\kappa-1}{\lambda-1}
							\, \exp\left(- \frac{(\kappa-1)^2 \, \rho^2}{2}\right) \rd s \\
			&= \frac{1}{(\lambda-1) \, \rho^2}
					\, \exp\left(- \frac{(\lambda-1)^2 \, \rho^2}{2}\right) \,,
	\end{align*}
	we obtain the final lower bound.
	
	The factor~\mbox{$\beta(\lambda,\rho)$} is monotonically increasing in~$\rho$,
	so taking the general bound \mbox{$\rho \geq \sqrt{2/\pi}$},
	see \lemref{lem:E|JX|>c|J|}, we obtain the second estimate.
	
	Of course, the truncated expectation is non-negative, so we take the
	positive part~\mbox{$[\ldots]_{+}$} of the prefactor. 
\end{proof}

For comparison,
we cite the original lemma from Heinrich~\cite[Lem~1]{He92}
concerning the truncated norm expectation when dealing with
two different norms at once.
\begin{lemma}	\label{lem:E|SJX|trunc}
	Consider a similar situation to \lemref{lem:E|JX|trunc} above
	with the ratios
	\mbox{$\rho := \expect \|J \vecX\|_F / \|J\|_{2 \rightarrow F}$}
	and~\mbox{$\sigma := \expect \|S J \vecX\|_G / \|S J\|_{2 \rightarrow G}$}.
	Then for~\mbox{$\kappa,\lambda > 1$} we have
	\begin{multline*}
		\expect\Bigl[\|S J \vecX\|_G
									\, \ind_{\{\|J \vecX\|_F
																\leq \kappa\,\expect \|J \vecX\|_F
													\}}
						\Bigr]\\
			\geq 
				\underbrace{\left[\beta(\lambda,\sigma)
														- \lambda
																\exp\left(-\frac{(\kappa-1)^2 \, \rho^2}{2}
																		\right)
										\right]_{+}
									}_{=: \tilde{\beta}(\kappa,\lambda,\rho,\sigma)}
					\, \expect \|S J \vecX\|_F \\
			\geq
				\underbrace{\left[\beta(\lambda)
														- \lambda
																\exp\left(-\frac{(\kappa-1)^2}{\pi}
																		\right)
										\right]_{+}
									}_{=: \tilde{\beta}(\kappa,\lambda)}
					\, \expect \|S J \vecX\|_F \,.
	\end{multline*}
\end{lemma}
\begin{proof}[Idea of the proof.]
	The trick is that we replace the truncation with respect to the $F$-norm
	by a truncation with respect to the $G$-norm,
	for that purpose introducing an auxiliary parameter~$\lambda$.
	We estimate the difference between both truncation variants,
	\begin{multline*}
		\expect\Bigl[\|S J \vecX\|_G
									\, \ind_{\{\|J \vecX\|_F
																\leq \kappa\,\expect \|J \vecX\|_F
													\}}
						\Bigr] \\
			\geq \expect\Bigl[\|S J \vecX\|_G
													\, \ind_{\{\|S J \vecX\|_G
																				\leq \lambda\,\expect \|S J \vecX\|_G
																	\}}
									\Bigr]
					- (\lambda\,\expect \|S J \vecX\|_G)
							\, \P\{\|J \vecX \|_F > \kappa \,\expect \|J \vecX\|_F\} \,.
	\end{multline*}
	The first term may be estimated by applying \lemref{lem:E|JX|trunc}
	to the operator~\mbox{$SJ$},
	for the second term we can directly use the deviation result
	\corref{cor:deviationGauss}.
\end{proof}

\begin{remark}[Heinrich's original lower bound] \label{rem:GaussMCada}
	Heinrich's result~\cite[Prop~2]{He92} originally provides
	lower bounds for the Monte Carlo error
	via norm expectations of Gaussian measures.
	In detail, there exists a constant~\mbox{$c'>0$} such that for~\mbox{$m > n$}
	and any injective linear operator~{$J: \ell_2^m \rightarrow \widetilde{F}$}
	we have
	\begin{equation} \label{eq:GaussMCada-inf_P}
		e^{\ran,\ada}(n,S,F,\Lall)
			\geq c' \, \inf_{\substack{\text{$P$ orth.~Proj.}\\
																\rank P = m-n}}
						\frac{\expect \|S J P \vecX\|_G}{\expect \|J \vecX\|_F} \,,
	\end{equation}
	where~$\vecX$ is a standard Gaussian random vector in~$\R^m = \ell_2^m$.
	
	The proof works similarly to the proof of \thmref{thm:BernsteinMCada}.
	In detail, for any~\mbox{$\alpha > 0$} one may rescale the operator~$J$
	such that \mbox{$\|J \vecX\|_F \stackrel{!}{=} \alpha$},
	we truncate the rescaled measure.
	The constant then is determined as
	\begin{equation*}
		\mbox{$c' = \nu(r,\alpha)
								\, \tilde{\beta}\left(\frac{r}{\alpha},\lambda\right)
								\, \alpha$} \,,
	\end{equation*}
	now applying \lemref{lem:E|SJX|trunc} instead of \lemref{lem:E|JX|trunc}.
	With \mbox{$r = 0.375$}, \mbox{$\alpha = 0.073$}, and \mbox{$\lambda = 6.15$},
	we obtain~\mbox{$c' = 0.06635...$} which is not much worse
	than the constant~$c$ in \thmref{thm:BernsteinMCada}.
	
	How should $J$ be chosen?
	When applying \corref{cor:optGauss} from the next section
	to find an optimal~\mbox{$J' := SJ$}, that is, the image measure
	\mbox{$\tilde{\mu} \circ S^{-1}$} shall be ``well spread'' within~$G$,
	we may get rid of the infimum within~\eqref{eq:GaussMCada-inf_P}
	and write
	\begin{equation} \label{eq:GaussMCada-optSJ}
		e^{\ran,\ada}(n,S,F,\Lall)
			\geq c' \, \frac{m-n}{m} \,
						\frac{\expect \|S J \vecX\|_G}{\expect \|J \vecX\|_F} \,.
	\end{equation}
	Especially for the identity mapping between sequence spaces
	\mbox{$\ell_p^m \hookrightarrow \ell_q^m$},
	the optimal Gaussian measure will be the standard Gaussian measure,
	see \remref{rem:UniqueGauss}.
\end{remark}

\subsection{Optimal Gaussian Measures}
\label{sec:OptGauss}

\subsubsection{Lewis' Theorem and Application to Gaussian Measures}

We want to find optimal Gaussian measures with respect to the $F$-norm
in~$\R^m$. We therefore apply Lewis' Theorem, originally~\cite{Lewis79}.
The proof given here is taken from Pisier~\cite[Thm~3.1]{Pis89}.
It is included for completeness.

\begin{proposition}[Lewis' Theorem] \label{prop:Lewis}
	Let $\alpha$ be an arbitrary norm on the space of
	automorphisms~\mbox{$\Linop(\R^m)$}.
	Let \mbox{$\widetilde{J} \in \Linop(\R^m)$} maximize the
	determinant~\mbox{$\det(J)$} subject to~\mbox{$\alpha(J) = 1$}. Then
	for any operator~\mbox{$T \in \Linop(\R^m)$} we have
	\begin{equation*}
		\trace(\widetilde{J}^{-1} T) \leq m \,\alpha(T) \,.
	\end{equation*}
\end{proposition}
\begin{proof}
	Since any invertible operator~$J$ can be rescaled such
	that~\mbox{$\alpha(J) = 1$}, there are admissible operators that
	fulfil~\mbox{$\det(J) > 0$}. The constraint~\mbox{$\alpha(J) = 1$} defines
	a compact subset within the finite dimensional space~\mbox{$\Linop(\R^m)$},
	that is, the supremum
	\begin{equation*}
		\sup_{\alpha(J)=1} \det(J) > 0
	\end{equation*}
	is attained. Let~$\widetilde{J}$ be a maximizer. Then for
	any~\mbox{$T \in \Linop(\R^m)$} and $\eps>0$ we have
	\begin{equation*}
		\det\left(\frac{\widetilde{J} + \eps\, T
									}{\alpha(\widetilde{J} + \eps\, T)}
				\right)
			\leq \det(\widetilde{J}) \,.
	\end{equation*}
	By homogeneity, and after dividing by~\mbox{$\det(\widetilde{J})$},
	\begin{equation*}
		\det(1 + \eps \, \widetilde{J}^{-1} T)
			\leq \left(\alpha(\widetilde{J} + \eps \, T)\right)^m
			\stackrel{\text{$\Delta$-ineq.}}{\leq}
				(1 + \eps \, \alpha(T))^m \,.
	\end{equation*}
	Finally,
	\begin{equation*}
		\trace (\widetilde{J}^{-1} T)
				= \lim_{\eps \rightarrow 0}
							\frac{\det(1 + \eps \, \widetilde{J}^{-1} T) - 1
									}{\eps}
				\leq \lim_{\eps \rightarrow 0}
							\frac{(1 + \eps \, \alpha(T))^m - 1
									}{\eps}
				= m \, \alpha(T) \,.
	\end{equation*}
\end{proof}

\begin{corollary}[Optimal Gaussian measures] \label{cor:optGauss}
	For any norm~\mbox{$\|\cdot\|_F$} on~$\R^m$
	there is an operator~\mbox{$J \in \Linop(\R^m)$} with
	\mbox{$\expect \|J \vecX\|_F = 1$} and
	\begin{equation*}
		\expect \|J P \vecX\|_F \geq \frac{\rank P}{m}
	\end{equation*}
	for any projection~$P \in \Linop(\R^m)$.
	Here, $\vecX$~is a standard Gaussian vector in~$\R^m$.
\end{corollary}
\begin{proof}
	First note that \mbox{$\alpha(J):= \expect \|J \vecX\|_F$} defines a norm
	on the space~\mbox{$\Linop(\R^m)$}
	of linear operators~{$J:\R^m \rightarrow \R^m$}.
	Indeed, because the expectation operator~$\expect$ is linear and
	\mbox{$\|J \cdot \|_F$} is a semi-norm for any linear operator~$J$, and
	\mbox{$\alpha(J)>0$} if~\mbox{$J \not= 0$}.
	
	We then may apply \propref{prop:Lewis} with \mbox{$T = JP$} and
	\mbox{$\trace P = \rank P$} for projections~$P$.
\end{proof}

\subsubsection{Properties and Examples of Optimal Gaussian Measures}

The remaining part of this subsection collects some results that
expand our knowledge on optimal Gaussian measures but are not
necessary for the basic version of the chapter's main result,
\thmref{thm:BernsteinMCada}.

\begin{remark}[Uniqueness of the Gaussian measure] \label{rem:UniqueGauss}
	For any orthogonal matrix~\mbox{$A \in \Linop(\R^m)$},
	the distribution of~\mbox{$A \vecX$} is identical to that of~$\vecX$,
	of course, \mbox{$|\det A| = 1$}.
	Consequently, $\widetilde{J}$ and $\widetilde{J} A$ define the same
	distribution and are equivalent maximizers of the absolute value of the
	determinant~\mbox{$|\det J|$},
	subject to \mbox{$\expect \|J \vecX\|_F \stackrel{\text{!}}{=} 1$}.
	
	Moreover, $\widetilde{J}$ is unique modulo orthogonal transformations,
	i.e.\ for all similarly optimal operators~\mbox{$\widetilde{J}_1$}
	there exists an orthogonal matrix~$A$
	with~\mbox{$\widetilde{J}_1 = \widetilde{J} A$},
	see Pisier~\cite[Prop~3.6]{Pis89} for a proof.
	In particular, all similarly optimal operators have the same operator norm
	\mbox{$\|\widetilde{J}_1\|_{2 \rightarrow F}
				= \|\widetilde{J}\|_{2 \rightarrow F}$},
	and there is one unique optimal Gaussian measure~$\widetilde{\mu}$ associated
	with~$F$.
	
	Let us now consider sequence spaces~{$\ell_p^m$} with
	\mbox{$1 \leq p \leq \infty$},
	and let \mbox{$J \vecX \in \ell_p^m$} be a random vector
	distributed according to the corresponding optimal Gaussian measure.
	Due to the symmetry of sequence spaces,
	for any operator~$Q$ permuting
	the coordinates of a vector in~$\R^m$
	and possibly changing some of their signs,
	clearly, the distribution of~\mbox{$Q J \vecX$} will be optimal as well.
	By the uniqueness we conclude that the covariance matrix
	must be a multiple of the identity,
	so the optimal Gaussian measure for sequence spaces
	is a scaled standard Gaussian vector.
	In other words, the optimal operator~$\widetilde{J}$
	may be chosen as a multiple of the identity,
	\mbox{$\widetilde{J} = c \, \id_{\R^m}$}.
	
	Compare Math\'e~\cite[Lem~4]{Ma91}
	for similar symmetry arguments in a more general setting.
	This has been used by Heinrich~\cite{He92} to prove properties
	for standard Gaussian measures on sequences spaces~$\ell_p^m$
	which we obtain by optimality according to Lewis' theorem.
\end{remark}

Next, we find bounds for the operator norm of an optimal operator
corresponding to the optimal Gaussian measure on~\mbox{$\widetilde{F} = \R^m$}.
It is not known to the author
whether this particular upper bound has already been proven before,
however, its implication together with the deviation result
\corref{cor:deviationGauss}
in high dimensions is not surprising, and similar results are known under names
such as \emph{concentration phenomenon} and \emph{thin shell estimates}.

\begin{proposition} \label{prop:|J|<bound}
	There exists a constant~\mbox{$c>0$} such that for any
	norm~\mbox{$\|\cdot\|_F$} on~$\R^m$, \mbox{$m \geq 2$},
	the operator~$\widetilde{J}$ defined as in \corref{cor:optGauss}
	is bounded by
	\begin{equation*}
		\|\widetilde{J}\|_{2 \rightarrow F}
			\leq c \, (\log m)^{-1/2} \,.
	\end{equation*}
	On the other hand, we have the lower bound
	\begin{equation*}
		\sqrt{\frac{\pi}{2}} \, m^{-1}
			\leq \|\widetilde{J}\|_{2 \rightarrow F} \,.
	\end{equation*}
\end{proposition}
\begin{proof}
	The lower bound is rather simple. Let~\mbox{$P_i$} denote the
	projections onto the $i$-th coordinate.
	Then, using~\mbox{$\|\widetilde{J} P_i\|_{2 \rightarrow F}
											\leq \|\widetilde{J}\|_{2 \rightarrow F}$},
	we have
	\begin{align*}
		1 &= \expect \|\widetilde{J} \vecX\|_F
			\stackrel{\text{$\Delta$-ineq.}}{\leq}
				\sum_{i=1}^m \expect\| \widetilde{J} P_i \vecX \|_F
			\leq m \, \sqrt{\frac{2}{\pi}} \, \|\widetilde{J}\|_{2 \rightarrow F} \,.
	\end{align*}
	
	Now for the upper bounds.
	There exists an orthogonal \mbox{rank-$1$} projection~$P_1$
	on~\mbox{$\R^m = \ell_2^m$} such that
	\begin{equation*}
		L := \|\widetilde{J}\|_{2 \rightarrow F}
			= \|\widetilde{J} P_1\|_{2 \rightarrow F} \,.
	\end{equation*}
	For the complementary \mbox{rank-$(m-1)$} projection~\mbox{$P_2 := \id - P_1$}
	we have
	\begin{equation*}
		\|\widetilde{J} P_2\|_{2 \rightarrow F}
			\leq \|\widetilde{J}\|_{2 \rightarrow F}
			= L \,.
	\end{equation*}
	Due to orthogonality, we can split
	\mbox{$\widetilde{J} \vecX
					= \widetilde{J} P_1 \vecX + \widetilde{J} P_2 \vecX$}
	into two independent zero-mean random vectors.
	Let \mbox{$\beta:= \expect \|\widetilde{J} P_2 \vecX\|_F$} denote
	the expectation of the norm of the second part.
	Clearly, by \corref{cor:optGauss} we have
	\begin{equation}	\label{eq:propbound-c}
		 1-\frac{1}{m} \leq \beta \leq 1 \,.
	\end{equation}
	Note that \mbox{$\|\widetilde{J} P_1 \vecX\|_F$}
	is a real random variable with probability density
	\begin{equation*}
		p_1(t) := \sqrt{\frac{2}{\pi L^2}} \,
								\exp\left(- \frac{t^2}{2 L^2}\right)
								\, \ind[t \geq 0] \,.
	\end{equation*}
	For the cumulative distribution function of the real random
	variable~\mbox{$\|\widetilde{J} P_2 \vecX\|_F$} we write
	\mbox{$F_2(s):= \P\{\|\widetilde{J} P_2 \vecX\|_F \leq s\}$},
	and by~\mbox{$p_2(s):= \frac{\diff}{\diff s} F_2(s)$}
	we denote the corresponding density function for~\mbox{$s>0$}.
	\propref{prop:dev2} directly implies that for~\mbox{$s>\beta$} we have
	\begin{equation} \label{eq:propbound-F2est}
		F_2(s) \geq 1 - \exp\left(-\frac{(s-\beta)^2}{2 L^2}\right) \,.
	\end{equation}
	
	Now, by symmetry and independence we have
	\begin{equation}	\label{eq:propbound-norm<max}
		\begin{split}
			1 = \expect \|\widetilde{J} \vecX\|_F
				&= \frac{1}{2} \expect
						\left[
							\|\widetilde{J} P_1 \vecX + \widetilde{J} P_2 \vecX\|_F
							+ \|\widetilde{J} P_1 \vecX - \widetilde{J} P_2 \vecX\|_F
						\right] \\
				&\stackrel{\text{$\Delta$-ineq.}}{\geq}
					\expect \max \left(\|\widetilde{J} P_1 \vecX\|_F,
															\|\widetilde{J} P_2 \vecX\|_F\right) \,.
		\end{split}
	\end{equation}
	It follows
	\allowdisplaybreaks
	\begin{align*}
			\frac{1}{m}
				\stackrel{\eqref{eq:propbound-c}}{\geq} 1 - \beta
				&= \expect \|\widetilde{J} \vecX\|_F
					- \expect \|\widetilde{J} P_2 \vecX\|_F \\
				&\stackrel{\eqref{eq:propbound-norm<max}}{\geq}
					\expect \left[ \max \left(\|\widetilde{J} P_1 \vecX\|_F,
																		\|\widetilde{J} P_2 \vecX\|_F
															\right)
												- \|\widetilde{J} P_2 \vecX\|_F
									\right] \\
				&= \expect [(\underbrace{
											\|\widetilde{J} P_1 \vecX\|_F - \|\widetilde{J} P_2 \vecX\|_F
												}_{=: r})
										\, \ind_{\{\|\widetilde{J} P_1 \vecX\|_F
																\geq \|\widetilde{J} P_2 \vecX\|_F\}}
										] \\
				&= \int_{0}^{\infty}
									r \, \int_{0}^{\infty} p_1(r+s) p_2(s) \rd s
								\rd r \\
				&\stackrel{\text{part.\ int.}}{=}
					\int_{0}^{\infty}
							r \, \bigg[
										\underbrace{p_1(r+s) F_2(s)
																\Big|_0^{\infty}
																}_{ = 0}
										- \int_0^{\infty}
													\underbrace{p_1'(r+s)}_{<0}
														F_2(s) \rd s
									\bigg] \rd r \\
				&\stackrel{\eqref{eq:propbound-F2est}}{\geq}
					\int_{0}^{\infty}
							r \, \bigg[
										- \int_0^{\infty} p_1'(r+t+\beta)
												\left(1-\exp\left(- \frac{t^2}{2 L^2}\right)\right)
											\rd t
									\bigg] \rd r \\
				&\stackrel{\text{part.\ int.}}{=}
					\int_{0}^{\infty}
							r \, \bigg[
										- \underbrace{p_1(r+t+\beta)
													\left(1-\exp\left(- \frac{t^2}{2 L^2}\right)\right)
																	\Big|_0^{\infty}
												}_{ = 0}\\
				&\qquad\qquad\qquad\;\;	+ L^{-2}
												\int_0^{\infty} t \, p_1(r+t+\beta)
													\, \exp\left(- \frac{t^2}{2 L^2}\right) \rd t
									\bigg] \rd r \\
				&= \sqrt{\frac{2}{\pi}} \, L^{-3} \int_{0}^{\infty}
							r \, \exp\left(- \frac{(r+\beta)^2}{4 L^2}\right)
							\int_0^{\infty} t \,
								\exp\left(- \frac{\left(t+\frac{r+\beta}{2}\right)^2}{L^2}\right)
											\rd t \rd r \,.\\
		\intertext{%
	Note that \mbox{$t \geq \frac{1}{2} \left(t + \frac{r+\beta}{2}\right)$}
	for \mbox{$t \geq \frac{r+\beta}{2}$}.
	After the substitution~\mbox{$\tau = t + \frac{r+\beta}{2}$},
	the inequality can be continued as
		}
				&\geq \sqrt{\frac{2}{\pi}} \, \frac{1}{2}
								\, L^{-3} \int_{0}^{\infty}
							r \, \exp\left(- \frac{(r+\beta)^2}{4 L^2}\right)
							\int_{r+\beta}^{\infty} \tau \,
								\exp\left(- \frac{\tau^2}{L^2}\right)
											\rd \tau \rd r \\
				&= \sqrt{\frac{2}{\pi}}\,  \frac{1}{4} \, L^{-1}
							\int_{0}^{\infty}
									r \, \exp\left(-\frac{5(r+\beta)^2}{4 L^2}\right)
								\rd r \,.\\
		\intertext{%
	Using \mbox{$r \geq \frac{1}{2} (r+\beta)$} for \mbox{$r \geq \beta$},
	and with the substitution~\mbox{$\rho = r+\beta$},
	we go on to
		}
				&\geq \sqrt{\frac{2}{\pi}} \, \frac{1}{8} \, L^{-1}
								\int_{2 \beta}^{\infty}
									\rho \, \exp\left(-\frac{5\rho^2}{4 L^2}\right)
								\rd \rho \\
				&\stackrel{\beta \leq 1}{\geq}
					\sqrt{\frac{2}{\pi}} \, \frac{1}{20} \, L \,
						\exp\left(- \frac{5}{L^2}\right) \,.
	\end{align*}
	For~\mbox{$0 < L \leq \sqrt{2 / \pi}$} we have
	\begin{equation*}
		L \geq \sqrt{\frac{\pi}{2}} \, L^2 
			> \sqrt{\frac{\pi}{2}} \, 20
						\, \frac{1}{1 + \frac{20}{L^2}}
			> \sqrt{\frac{\pi}{2}} \, 20 \,
					\exp\left(- \frac{20}{L^2}\right) \,.
	\end{equation*}
	By this we finally obtain
	\begin{equation*}
		\frac{1}{m} \geq \exp\left(- \frac{25}{L^2}\right) \,.
	\end{equation*}
	Inverting the inequality, we get
	\begin{equation*}
		L \leq 5 \left(\log m\right)^{-1/2} \,.
	\end{equation*}
	This is the proposition with~\mbox{$c = 5$}.
\end{proof}

\begin{example}
	As shown before in \remref{rem:UniqueGauss},
	for sequence spaces~$\ell_p^m$, due to symmetry,
	the operator~$\widetilde{J}$ for the optimal Gaussian measure is a multiple
	of the identity.
	
	Considering~$\ell_1^m$, we observe
	\begin{equation*}
		\expect \|\vecX\|_1
			= \expect \sum_{i=1}^m |\vecX_i|
			= \sqrt{\frac{2}{\pi}} \, m
		\,.
	\end{equation*}
	Therefore, the choice
	\mbox{$\widetilde{J}:= \sqrt{\pi/2} \, m^{-1} \, \id_{\R^m}$}
	is optimal.
	The norm of~\mbox{$\widetilde{J}$} matches the lower bound
	in \propref{prop:|J|<bound}.
	Furthermore, for projections~$P$ onto coordinates, that is,
	\mbox{$P \vecx = \sum_{i \in I} x_i \vece_i$}
	with an index set~\mbox{$I \subseteq \{1,\ldots,m\}$}, we have
	\mbox{$\rank P = \# I$} and
	\begin{equation*}
		\expect \|\widetilde{J} P \vecX\|_1 = \frac{\rank P}{m} \,.
	\end{equation*}
	Therefore, the lower bound in \corref{cor:optGauss} is sharp.
	
	Now consider~$\ell_{\infty}^m$. Clearly,
	\begin{equation*}
		\|\id_{\R^m}\|_{2 \rightarrow \infty} = 1 \,.
	\end{equation*}
	Furthermore, we can show
	\begin{equation} \label{eq:ellinfexample}
		\expect \|\vecX\|_{\infty} \leq C \sqrt{1 + \log m} \,,
	\end{equation}
	see \lemref{lem:gaussqnormvector}.
	We rescale with~
	\mbox{$\alpha = \alpha(m) := (\expect \|\vecX\|_{\infty})^{-1}
								\geq \frac{1}{C} (1 + \log m)^{-1/2}$},
	i.e.\ \mbox{$\widetilde{J}:= \alpha \, \id_{\R^m}$}
	generates the optimal Gaussian
	measure,
	and the order
	of~\mbox{$\|\widetilde{J}\|_{2 \rightarrow \infty} = \alpha(m)$}
	is determined precisely
	thanks to the upper bound from \propref{prop:|J|<bound}.
\end{example}

\begin{remark}[Improved constant for large~$n$ and $m$]
	\label{rem:1/6*b_2m}
	Using \propref{prop:|J|<bound},
	the constant~$\frac{1}{15}$ in \thmref{thm:BernsteinMCada}
	can be improved for (extremely) large~$n$ and $m$.
	In detail, we have
	constants~\mbox{$0 < c_m \xrightarrow[m \rightarrow \infty]{} \frac{1}{3}$}
	such that for~$m \geq 2n$ we can state
	\begin{equation*}
		e^{\ran,\ada}(n,S,F,\Lall) \geq c_m \, \frac{m-n}{m} \, b_m(S) \,.
	\end{equation*}
	
	Following the proof of \thmref{thm:BernsteinMCada},
	thereby setting~\mbox{$r = \frac{1}{3}$}
	and~\mbox{$\alpha := \expect \|J \vecX\|_F = \frac{r}{1+\delta}$}
	with~\mbox{$\delta = \delta(\rho)$},
	we attain an estimate with a factor
	\begin{equation*}
		\tilde{\nu}\{\|m_{\vecy} \leq \textstyle{\frac{2}{3}}\}
						\, \underbrace{
									\inf_{\substack{\text{$P$ orth.~Proj.}\\
																		\rank P = m-n}}
									\expect[\|J P \vecX\|_F
										\, \ind_{\{\|J P \vecX\|_F \leq \frac{1}{3}\}}]
									}_{\text{$\geq \beta(\lambda, \rho/2)$ for $m \geq 2n$}}
						\, \alpha \,.
	\end{equation*}
	Using \lemref{lem:|m_y|<r} and
	\lemref{lem:E|JX|trunc} with
	\mbox{$\|J P \vecX\|_F/\|J P\|_{2 \rightarrow F} \geq \rho/2$},
	we can choose a $\rho$-dependent constant
	\begin{equation*}
		c = c(\rho)
			= \left[1 - 2 \exp\left(- \frac{\delta^2 \, \rho^2}{2}\right)\right]
				\, \left[1 - \left(1+ \delta + \frac{4}{\delta\,\rho^2}\right)
												\, \exp\left(- \frac{\delta^2 \, \rho^2}{8}\right)
					\right]
				\, \frac{1}{3 \, (1+\delta)} \,.
	\end{equation*}
	This expression is monotonically growing in~$\rho$,
	and converging to~\mbox{$\frac{1}{3}$} for~\mbox{$\rho \rightarrow \infty$},
	if we choose \mbox{$\frac{1}{\rho} \prec \delta(\rho) \prec 1$}.
	For example, with~\mbox{$\delta(\rho) = \delta_0/\sqrt{\rho}$}
	we have
	\begin{equation*}
		c(\rho)
			= \left[1 - 2 \exp\left(- \frac{\delta_0^2 \, \rho}{2}\right)\right]
				\, \left[1 - \left(1 + \frac{\delta_0}{\sqrt{\rho}}
														 + \frac{4}{\delta_0 \,\rho^{3/2}}\right)
												\, \exp\left(- \frac{\delta_0^2 \, \rho}{8}\right)
					\right]
				\, \frac{1}{3 \, (1+\frac{\delta_0}{\sqrt{\rho}})} \,.
	\end{equation*}
	Since by \propref{prop:|J|<bound} we know
	\mbox{$\rho = \rho_m \geq \sqrt{\log m}/5
		\xrightarrow[m \rightarrow \infty]{} \infty$},
	this shows that we can choose
	\mbox{$c_m := c(\rho_m) \xrightarrow[m \rightarrow \infty]{} \frac{1}{3}$}.
	However, this convergence is extremely slow.
	If one is really interested in better constants,
	it is recommendable to include best knowledge about~$\rho$
	directly into the estimate.
\end{remark}

\section{Special Settings}
\label{sec:BernsteinSpecial}

We study two interesting modifications
of the main result~\thmref{thm:BernsteinMCada}.
They are non-essential for the applications in \secref{sec:BernsteinAppl}.
First, in \secref{sec:n(om,f)Bernstein} we widen the class of admissible
algorithms, now allowing varying cardinality~$n(\omega,f)$.
Still, we can show a similar inequality, however with worse constants.
Second, in \secref{sec:BernsteinHomo}
we restrict to homogeneous algorithms and
obtain an estimate with \emph{sharp constants}, even for
varying cardinality.

\subsection{Varying Cardinality}
\label{sec:n(om,f)Bernstein}

Up to this point we ignored the additional feature of varying cardinality
because it does not change the big point
but gives us unpleasant constants that detract from the main relation.
However, for lower bounds it is of interest to assume the most general
shape for algorithms.

In Heinrich~\cite{He92} results were actually given for
algorithms with non-adaptively varying cardinality~$n(\omega)$.
In this setting, by \lemref{lem:n(om)MC}
and \thmref{thm:BernsteinMCada},
for~\mbox{$n \in \N$} we directly obtain an estimate like
\begin{equation*}
	\bar{e}^{\ran}(n,S) \geq {\textstyle \frac{1}{2}} \, e^{\ran}(2n,S)
		> {\textstyle \frac{1}{60}} \, b_{4n}(S)
\end{equation*}

We can even consider adaptively varying cardinality~\mbox{$n(\omega,f)$},
and still get similar bounds.

\begin{theorem} \label{thm:n(om,f)BernsteinMC}
	Let~\mbox{$S: \widetilde{F} \rightarrow G$} be
	a compact linear operator between Banach spaces,
	and the input set~$F$ be the unit ball in~$\widetilde{F}$.
	Considering algorithms with adaptively varying cardinality~\mbox{$n(\omega,f)$},
	for \mbox{$n \in \N$} we have
	\begin{equation*}
		\bar{e}^{\ran,\ada}(n,S,F,\Lall)
			> {\textstyle \frac{1}{63}} \, b_{4n}(S) \,.
	\end{equation*}
	More generally,
	for any injective linear operator~{$J: \ell_2^{8n} \rightarrow \widetilde{F}$},
	we can estimate
	\begin{equation}
		e^{\ran,\hom}(n,S,F,\Lall)
			> \frac{1}{64}
				\,\inf_{\substack{\text{$P$ orth.~Proj.}\\
																\rank P = 4n}}
						\frac{\expect \|S J P \vecX\|_G}{\expect \|J \vecX\|_F} \,,
	\end{equation}
	where~$\vecX$ is a standard Gaussian random vector in~$\R^{8n} = \ell_2^{8n}$.
\end{theorem}
\begin{proof}
	The proof works similar to that of \thmref{thm:BernsteinMCada}.
	Again, we assume~\mbox{$\widetilde{F} = \R^m$}.
	As before, we define a measure~$\tilde{\mu}$
	as the distribution of~\mbox{$J \vecX$}
	with~$\vecX$ being a standard Gaussian random vector
	in~\mbox{$\R^m = \ell_2^m$},
	and set $\mu$ to be the restriction of~$\tilde{\mu}$
	to the unit ball~\mbox{$F \subset \widetilde{F}$}.
	We write~\mbox{$\alpha := \expect \|J \vecX\|_F$}.
	
	In view of \lemref{lem:n(om,f)Bakhvalov} (Bakhvalov's technique),
	we aim to bound the $\mu$-average error
	for a deterministic algorithm~\mbox{$\phi \circ N :
																					\widetilde{F} \rightarrow G$}
	with varying cardinality~\mbox{$n(\vecy) = n(f)$} such that
	\begin{align*}
		\bar{n}
			&\geq \card(\phi \circ N, \mu) \\
			&= \int n(f) \, \mu(\diff f) \\
		\text{[Markov's ineq.]}\quad
			&\geq 2\bar{n} \, \mu\{n(f) > 2\bar{n}\} \,.
	\end{align*}
	Thus, using the definition of the truncation, we can estimate
	\begin{equation} \label{eq:n(y)<=4n}
		\begin{split}
			\tilde{\mu}\{n(f) \leq 2\bar{n}\}
				&\geq \tilde{\mu}\{n(f) \leq 2\bar{n} \text{ and } f \in F\} \\
				&= \mu\{n(f) \leq 2\bar{n} \} \\
				&\geq \mu(F) - \frac{1}{2} \\
			\text{[\corref{cor:deviationGauss}]}\quad
				&\geq \frac{1}{2} - \exp\left(- \left(\frac{1}{\alpha} - 1\right)^2 \bigg/\pi
																\right) \,.
		\end{split}
	\end{equation}

	The conditional measure~$\tilde{\mu}_{\vecy}$ can be represented
	as the distribution of~\mbox{$J P_{\vecy} \vecX + m_{\vecy}$},
	with a suitable orthogonal projection~$P_{\vecy}$ of rank~\mbox{$m-n(\vecy)$},
	and a vector~\mbox{$m_{\vecy} \in F$},
	not very different from the case of fixed cardinality,
	compare \lemref{lem:condGauss}.
	Again, we write \mbox{$\tilde{\nu} := \tilde{\mu} \circ N^{-1}$}
	for the distribution of the information.
	Following the same arguments as in the proof of \thmref{thm:BernsteinMCada},
	and in addition restricting the integral
	to the case~\mbox{$n(\vecy) \leq 2\bar{n}$},
	we obtain the estimate for~\mbox{$r \in (0,1)$}
	\begin{align*}
		e(\phi \circ N, \mu)
			&\,\geq\, \tilde{\nu}\{\vecy \,:\,
													n(\vecy) \leq 2\bar{n}
													\text{ and }
													\|m_{\vecy}\|_F \leq 1 - r \} \\
			&\qquad\qquad
					\inf_{\substack{\text{$P$ orth.~Proj.}\\
																			\rank P \geq m-2\bar{n}}}
									\expect
										\Bigr[\|J P \vecX\|_F
														\, \ind_{\{\|J P \vecX\|_F
																			\leq {\textstyle r}
																		\}}
										\Bigl]
							\, b_m(S) \,.
	\end{align*}
	
	The first factor can be estimated using inequality~\eqref{eq:n(y)<=4n}
	and a slight adjustment of \lemref{lem:|m_y|<r},
	\begin{multline*}
		\tilde{\nu}\{\vecy \,:\, \|m_{\vecy}\|_F \leq 1-r\} \\
			\geq \frac{1}{2}
							- \exp\left(- \left(\frac{1}{\alpha} - 1\right)^2 \bigg/\pi
										\right)
							- 2 \, \exp\left(- \left(\frac{1-r}{2 \, \alpha} - 1\right)^2
																	\bigg/\pi
												\right) \\
			=: \bar{\nu}(r,\alpha) \,.
	\end{multline*}
	The second factor can be bounded from below
	by~\mbox{$\frac{1}{2} \, \beta(\frac{r}{\alpha}) \, \alpha$},
	if the operator~$J$ is chosen optimally.
	This estimate is exactly the same
	as in the proof of \thmref{thm:BernsteinMCada},
	here with~\mbox{$\frac{m-2\bar{n}}{m}$}
	instead of~$\frac{m-n}{m}$, so
	\begin{equation*}
		\bar{e}^{\avg,\ada}(\bar{n},S,\mu)
			\geq \left(\frac{m-2\bar{n}}{m}\right)_+
												\, \bar{\nu}(r,\alpha)
												\, \beta\left(\frac{r}{\alpha}\right)
												\, \alpha
							\, b_{m}(S)
			=: \hat{\eps}(\bar{n}) \,.
	\end{equation*}
	This lower bound exhibits convexity,
	so by the subtle version
	of \lemref{lem:n(om,f)Bakhvalov} (Bakhvalov's technique),
	taking~\mbox{$m = 4n$}, we obtain
	\begin{equation*}
		\bar{e}^{\ran,\ada}(n,S,F)
			\geq \underbrace{\textstyle \frac{1}{2}
												\, \bar{\nu}(r,\alpha)
												\, \beta(\frac{r}{\alpha})
												\, \alpha
											}_{=: \bar{c}}
							\, b_{4n}(S) \,.
	\end{equation*}
	With~\mbox{$r = 0.35$} and \mbox{$\alpha = 0.07$},
	this gives us a constant~\mbox{$\bar{c} = 0.0159... > \frac{1}{63}$}.
	
	For the more general estimate with Gaussian measures,
	we take~\mbox{$m = 8n$} and the rough version of
	\lemref{lem:n(om,f)Bakhvalov} (Bakhvalov's technique),
	inserting~\mbox{$\bar{n} = 2n$} in the adjusted version of the above estimates,
	and obtain a constant
	\mbox{$\bar{c}' := \textstyle{\frac{1}{2}}
												\, \bar{\nu}(r,\alpha)
												\, \tilde{\beta}(\frac{r}{\alpha},\lambda)
												\, \alpha$}.
	Choosing \mbox{$r = 0.36$}, \mbox{$\alpha = 0.07$}, and \mbox{$\lambda = 6$},
	we get the numerical value \mbox{$\bar{c}' = 0.0158... > \frac{1}{64}$}.
\end{proof}

In regard of the estimate for fixed cardinality,
\begin{equation*}
	e^{\ran}(2n,S) > {\textstyle \frac{1}{30}} \, b_{4n}(S) \,,
\end{equation*}
we see that
taking twice as much information as in the varying cardinality setting
gives us bigger lower bounds by roughly a factor two only.\footnote{%
	For the general estimate for Gaussian measures we lose roughly a factor~$4$,
	for both the error and the cardinality.}
However, several estimates involved in this proof
seem to be far from optimal.
For homogeneous algorithms, see \secref{sec:BernsteinHomo},
using \lemref{lem:n(om,f)avgspecial} for the analysis of the average setting,
we will obtain sharp estimates that equally hold for algorithms
with fixed and varying cardinality.
Here, we could not apply~\lemref{lem:n(om,f)avgspecial}
because of the much more complicated situation arising from truncation.
Anyways, even without the homogeneity assumption, we can state:
\begin{quote}
	If upper bounds achieved by Monte Carlo algorithms with fixed cardinality
	are close to the lower bounds obtained using Bernstein widths
	(or directly, Gaussian measures),
	\emph{varying cardinality does not help a lot}.
\end{quote}

\subsection{Homogeneous Monte Carlo Methods}
\label{sec:BernsteinHomo}

For linear problems (as considered within this chapter),
common algorithms are homogeneous (and also non-adaptive).\footnote{%
	For example for the identity on sequence spaces~%
	\mbox{$\id: \ell_p^m \rightarrow \ell_q^m$},
	the basic structure of asymptotically best known algorithms
	is described in \secref{sec:An:lp->lq}.
	}
There is a close and very basic connection to the normalized error.
\begin{lemma} \label{lem:hom=normal}
	Concerning the approximation of a compact linear
	operator~\mbox{$S: \widetilde{F} \rightarrow G$}
	between Banach spaces over the reals
	using homogeneous algorithms,
	the absolute error criterion
	with the \emph{input set}~\mbox{$F \subset \widetilde{F}$}
	being the unit ball of~$\widetilde{F}$
	coincides with the normalized error criterion,
	\begin{equation*}
		e^{\star,\hom}(n,S,F,\Lambda)
			= e_{\normal}^{\star,\hom}(n,S,\widetilde{F} \setminus \{0\},\Lambda)\,,
	\end{equation*}
	where~\mbox{$\star \in \{\deter,\ran\}$}.
\end{lemma}
\begin{proof}
	If $A_n$ is a homogeneous algorithm
	that is defined for~\mbox{$f \in F$},
	it is naturally extended to~\mbox{$f \in \widetilde{F} \setminus \{0\}$}
	such that~\mbox{$A_n^{\omega}(f) = A_n^{\omega}(\frac{f}{\|f\|_F})$}.
	Indeed, this is unproblematic since the information mapping as well is homogeneous.
	Then for~\mbox{$f \not= 0$} we have
	\begin{align*}
		e_{\normal}(A_n,f)
			&= \expect \frac{\|S(f) - A_n^{\omega}(f)\|_G}{\|f\|_F} \\
			&= \expect \left\|S({\textstyle \frac{f}{\|f\|_F}})
												- A_n^{\omega}({\textstyle \frac{f}{\|f\|_F}})
									\right\|_G \\
			&= e\left(A_n,{\textstyle \frac{f}{\|f\|_F}}\right)  \\
			&\leq e(A_n,F) \,.
	\end{align*}
	This proves ``$\geq$''.
	
	Now, for any algorithm~$A_n$,
	and any non-zero input~\mbox{$f \in F \setminus \{0\}$} from the unit ball,
	we have
	\begin{align*}
		e(A_n,f)
			&= \expect \|S(f) - A_n^{\omega}(f)\|_G \\
			&\leq \expect \left[\frac{\|S(f) - A_n^{\omega}(f)\|_G}{\|f\|_F} \right] \\
			&= e_{\normal}(A_n,f) \\
			&\leq e_{\normal}(A_n,\widetilde{F}\setminus\{0\}) \,.
	\end{align*}
	Trivially, \mbox{$e(A_n,0) = 0$}, so this proves ``$\leq$''.
\end{proof}

\begin{theorem} \label{thm:BernsteinMChom}
	For~\mbox{$S: \widetilde{F} \rightarrow G$} being a compact linear
	operator between Banach spaces, and the input set~$F$ being the
	unit ball within~$\widetilde{F}$, we have
	\begin{equation*}
		e^{\ran,\hom}(n,S,F,\Lall)
			\geq \frac{m-n}{m} \, b_m(S)
		\quad \text{for\, $m > n$.}
	\end{equation*}
	In general,
	for any injective linear operator~{$J: \ell_2^m \rightarrow \widetilde{F}$}
	we have
	\begin{equation}
		e^{\ran,\hom}(n,S,F,\Lall)
			\geq \inf_{\substack{\text{$P$ orth.~Proj.}\\
																\rank P = m-n}}
						\frac{\expect \|S J P \vecX\|_G}{\expect \|J \vecX\|_F} \,,
	\end{equation}
	where~$\vecX$ is a standard Gaussian random vector in~$\R^m = \ell_2^m$.
	
	Actually, the same Bernstein estimate holds
	for homogeneous algorithms with varying cardinality as well,
	for \mbox{$\bar{n} \geq 0$} we can state
	\begin{equation*}
		\bar{e}^{\ran,\hom}(\bar{n},S,F,\Lall)
			\geq \left(\frac{m-\bar{n}}{m}\right)_+ \, b_m(S) \,.
	\end{equation*}
\end{theorem}
\begin{proof}
	Similarly to the proof of \thmref{thm:BernsteinMCada},
	we choose a Gaussian measure~$\tilde{\mu}$
	described as the distribution of~\mbox{$J \vecX$}.
	Here, however, we take the scaling~\mbox{$\expect \|J \vecX\|_F = 1$},
	so for measurable sets~\mbox{$E \subseteq \widetilde{F}$},
	by
	\begin{equation*}
		\mu(E) := \int_E \|f\|_F \, \tilde{\mu}(\diff f)
						= \expect \|J \vecX\|_F \, \ind_{\{J \vecX \in E\}} \,,
	\end{equation*}
	a probability measure on~$\widetilde{F}$ is defined.
	By \lemref{lem:hom=normal},
	and Bakhvalov's technique in the normalized error criterion setting
	(see \propref{prop:Bakh}
	for a proof in the absolute error criterion setting),
	we have
	\begin{equation*}
		e^{\ran,\hom}(n,F)
			= e_{\normal}^{\ran,\hom}(n,\widetilde{F}\setminus\{0\})
			\geq e_{\normal}^{\avg,\hom}(n,\mu) \,.
	\end{equation*}
	
	Now, for any homogeneous deterministic algorithm~\mbox{$A_n = \phi \circ N$},
	first with fixed cardinality,
	we have
	\begin{align*}
		e_{\normal}(A_n,\mu)
			&= \int \frac{\|S(f) - A_n(f)\|_G}{\|f\|_F} \, \mu(\diff f) \\
			&= \int \|S(f) - A_n(f)\|_G \, \tilde{\mu}(\diff f) \\
			&= \int_{\R^n}
						\left[\int \|S(f) - \phi(\vecy)\|_G
										\, \tilde{\mu}_{\vecy}(\diff f)
						\right]
						\, \tilde{\mu}\circ N^{-1} (\diff \vecy) \,.
	\intertext{%
	Using the representation of the conditional measure,
	see \lemref{lem:condGauss}, and with the same symmetrization
	argument as in the proof of \thmref{thm:BernsteinMCada},
	we continue}
			&\geq \inf_{\substack{\text{$P$ orth.~Proj.}\\
																\rank P = m-n}}
					\expect \|S J P \vecX\|_G \,.
	\intertext{%
	By the definition of Bernstein numbers
	for~$\widetilde{F} = \R^m$, choosing an optimal~$J$
	as it is found in \corref{cor:optGauss},
	we end up with}
		e_{\normal}(A_n,\mu) &\geq \frac{m-n}{m} \, b_m(S) \,.
	\end{align*}
	
	Observe that the lower bound for the conditional error,
	given~$\vecy$, exhibits the special structure
	of \lemref{lem:n(om,f)avgspecial} with a convex function
	\mbox{$\hat{\eps}(\bar{n})
					:= \left(\frac{m-\bar{n}}{m}\right)_+ \, b_m(S)$}.
	However,
	switching airily between different error criterions and measures,
	it is not immediate that this already implies lower bounds
	for homogeneous algorithms with varying cardinality,
	so we need to think about a modification
	of the proof of \lemref{lem:n(om,f)Bakhvalov}
	that fits to the present situation.
	Let~\mbox{$A = (A^{\omega})_{\omega \in \Omega}$} be a homogeneous
	Monte Carlo method with varying cardinality.
	Due to homogeneity, \mbox{$n(\omega,\lambda f) = n(\omega, f)$}
	for~\mbox{$\lambda \in \R \setminus\{0\}$} and~\mbox{$f \in \widetilde{F}$},
	hence
	\begin{align*}
		\bar{n} &:= \sup_{f \in F} \expect n(\omega,f)
			= \sup_{f \in \widetilde{F}} \expect n(\omega,f)\\
		\text{[Fubini]}\quad
			&\geq \expect \int n(\omega,f) \, \tilde{\mu}(\diff f) \,.
	\end{align*}
	The key insight is that the error~\mbox{$e(A,F)$} can be related to
	the $\tilde{\mu}$-average setting,
	here we use that \lemref{lem:hom=normal} holds for algorithms with varying
	cardinality as well,
	\begin{align*}
		e(A,F)
			&= e_{\normal}(A,\widetilde{F}\setminus\{0\})\\
		\text{[Fubini]}\quad
			&\geq \expect \int e_{\normal}(A^{\omega},f) \, \mu(\diff f) \\
			&= \expect \int e(A^{\omega},f) \, \tilde{\mu}(\diff f) \,.\\
		\intertext{%
	For the inner integral
	we apply \lemref{lem:n(om,f)avgspecial} to the $\tilde{\mu}$-average setting,
	for which we have the lower bound~\mbox{$\hat{\eps}(n(\omega,\vecy))$}
	for the conditional average error, 
	that is, with respect to~$\tilde{\mu}_{\vecy}$. We obtain}
			&\geq \expect \hat{\eps}(\card(A^{\omega},\tilde{\mu})) \,.\\
		\intertext{%
	From here we can proceed as in the last
	inequality chain within the proof of \lemref{lem:n(om,f)Bakhvalov},}
		e(A,F)
			&\geq \hat{\eps}(\bar{n}) \,.
	\end{align*}
	Hence the lower bound based on Bernstein numbers
	does even hold for algorithms with varying cardinality.
	For the direct estimate via general Gaussian measures,
	it depends on the particular situation
	how we can generalize the lower bound.
\end{proof}

\begin{remark}
	The above theorem is optimal. Indeed, consider for example
	the identity~\mbox{$\id_{\ell_1^m} : \ell_1^m \hookrightarrow \ell_1^m$}
	with Bernstein number~\mbox{$b_m(\id_{\ell_1^m}) = 1$}.
	Let \mbox{$I = I(\omega) \subseteq \{1,\ldots,m\}$}
	be a randomly chosen index set
	such that~\mbox{$\P\{i \in I\} = \frac{\bar{n}}{m}$},
	where~\mbox{$0 \leq \bar{n} \leq m$},
	and define the linear (homogeneous) Monte Carlo method
	\begin{equation*}
		A_{\bar{n}}^{\omega} (\vecx)
		:= \sum_{i \in I(\omega)} x_i \, \vece_i \,,
		\quad \vecx \in \ell_1^m\,,
	\end{equation*}
	where $\vece_i$ are the vectors of the standard basis.
	The cardinality is
	\begin{equation*}
		\card(A)
			\,=\, \expect \# I(\omega)
			\,=\, \sum_{i=1}^m \P\{i \in I\}
			\,=\, \bar{n} \,,
	\end{equation*}
	the error
	\begin{equation*}
		e(A_{\bar{n}},\id_{\ell_1^m},\vecx)
			\,=\, \textstyle {\mathbb{E}} \| \vecx - A_{\bar{n}}^{\omega}(\vecx)\|_1
			\,=\, {\displaystyle \sum_{i=1}^m}
					{\mathbb{P}}\{i \notin I(\omega)\} \, |x_i|
			\,=\, \frac{m-\bar{n}}{m} \, \|\vecx\|_1 \,,
	\end{equation*}
	so~\mbox{$e^{\ran,\hom}(A_{\bar{n}},\id_{\ell_1^m})
							= \frac{m-\bar{n}}{m}$}.
	On the other hand, by \thmref{thm:BernsteinMChom} we have
	the lower bound \mbox{$\bar{e}^{\ran,\hom}(\bar{n},\id_{\ell_1^m})
														\geq \frac{m-\bar{n}}{m}$}.
	This shows that~\mbox{$A_{\bar{n}}$} is optimal.
	If~\mbox{$\bar{n} \in \N_0$}, we can find
	a fixed-cardinality version for~\mbox{$A_{\bar{n}}$}.
\end{remark}

\section{Applications}
\label{sec:BernsteinAppl}

The first application on the recovery of sequences, \secref{sec:Bernstein:lp->lq},
is meant to give a feeling for the potentials and limitations
of Bernstein numbers when it comes to lower bounds for quantities from IBC.
We also discuss other techniques for lower bounds,
as well as general problems arising.
The main application is the~$L_{\infty}$-approximation of~$C^{\infty}$-functions,
see \secref{sec:Cinf->Linf}.
There we show that in the particular situation
randomization does not help in terms of tractability classifications.

\subsection{Recovery of Sequences}
\label{sec:Bernstein:lp->lq}

We consider the approximation of the identity
between finite dimensional sequence spaces,
\begin{equation*}
	\App: \ell_p^M \hookrightarrow \ell_q^M,
\end{equation*}
with~\mbox{$1 \leq p,q \leq \infty$} and~\mbox{$M \in \N$}.
This example problem is also of interest
when dealing with embeddings between function spaces,
compare \secref{sec:ranApp:e(n)}.

\subsubsection{The Case $M = 4n$}

The following result on Bernstein numbers is well known
and has been used e.g.\ in~\cite{Kudr99,NguyenVK15,NguyenVK16}
in order to determine the order of decay of Bernstein numbers
in different function space settings,
\begin{equation*}
	b_m(\ell_p^{2m} \hookrightarrow \ell_q^{2m})
		\asymp
			\begin{cases}
			m^{1/q-1/p}		&\text{if } 1 \leq p \leq q \leq \infty
										\text{ or } 1 \leq q \leq p \leq 2 \,, \\
			m^{1/q-1/2} 	&\text{if } 1 \leq q \leq 2 \leq p \leq \infty \,, \\
			1							&\text{if } 2 \leq q \leq p \leq \infty \,.
		\end{cases}
\end{equation*}
Here, the hidden constants may depend on~$p$ and~$q$.
Applying \thmref{thm:BernsteinMCada} with~\mbox{$m = 2n$},
this implies the estimate
\begin{equation*}
	e^{\ran}(n,\ell_p^{4n} \hookrightarrow \ell_q^{4n})
		\succeq
			\begin{cases}
			n^{1/q-1/p}		&\text{if } 1 \leq p \leq q \leq \infty
										\text{ or } 1 \leq q \leq p \leq 2 \,, \\
			n^{1/q-1/2} 	&\text{if } 1 \leq q \leq 2 \leq p \leq \infty \,, \\
			1							&\text{if } 2 \leq q \leq p \leq \infty \,.
		\end{cases}
\end{equation*}

Since the lower bounds with Bernstein numbers
were obtained using Gaussian measures,
it is not surprising that in some parameter settings we will get
significantly better estimates when directly working with Gaussian measures,
as it has been done in Heinrich~\cite{He92},
see \remref{rem:GaussMCada}.
In detail, with~$\vecX$ being a standard Gaussian vector on~$\R^{4n}$,
by~\eqref{eq:GaussMCada-optSJ} we have\footnote{%
	We chose~$m = 4n$ for better comparison with the results
	that were obtained via Bernstein numbers.
	In this case however, \mbox{$m = 2n$} would give the same asymptotics.}
\begin{equation} \label{eq:MC(lp->lq)>=Gauss}
	e^{\ran}(n, \ell_p^{4n} \hookrightarrow \ell_q^{4n})
		\geq c \, \frac{\expect \|\vecX\|_q}{\expect \|\vecX\|_p}
		\succeq \begin{cases}
			n^{1/q-1/p} 						&\text{if } 1 \leq p,q < \infty \,, \\
			n^{-1/p}(\log n)^{1/2}	&\text{if } 1 \leq p < q = \infty \,, \\
			n^{1/q}(\log n)^{-1/2}	&\text{if } 1 \leq q < p = \infty \,, \\
			1												&\text{if } p = q = \infty \,,
		\end{cases}
\end{equation}
\sloppy
where~$c>0$ is a universal constant,
and the hidden constant for the second relation may depend on~$p$ and~$q$.
For this result we only need to know
\mbox{$\expect \|\vecX\|_p \asymp m^{1/p}$} for~\mbox{$1 \leq p < \infty$},
and \mbox{$\expect \|\vecX\|_{\infty} \asymp \sqrt{1 + \log(m)}$},
for a standard Gaussian vector~$\vecX$ in~\mbox{$\R^m = \R^{4n}$},
see \lemref{lem:gaussqnormvector}.
In Heinrich~\cite{He92} we also find upper bounds,
which are achieved by non-adaptive and homogeneous methods,
see also \secref{sec:An:lp->lq},
\begin{equation} \label{eq:MC(lp->lq)<=He92}
	e^{\ran}(n, \ell_p^{4n} \hookrightarrow \ell_q^{4n})
		\preceq \begin{cases}
			n^{1/q-1/p} 						&\text{if } 1 \leq p,q < \infty \,, \\
			n^{-1/p}(\log n)^{1/2}	&\text{if } 1 \leq p < q = \infty \,, \\
			n^{1/q}									&\text{if } 1 \leq q < p = \infty \,, \\
			1												&\text{if } p = q = \infty \,.
		\end{cases}
\end{equation}
We see that in most cases
the lower bounds~\eqref{eq:MC(lp->lq)>=Gauss} obtained by Gaussian measures
match the upper bounds~\eqref{eq:MC(lp->lq)<=He92},
but for~\mbox{$q < p = \infty$} there is a logarithmic gap.
\fussy

For non-adaptive algorithms this gap can be closed.
Within the authors master's thesis, see also~\cite{Ku16},
Bernstein numbers have been related
to the error of non-adaptive Monte Carlo methods
by means of volume ratios.\footnote{%
	Instead of truncated Gaussian measures,
	in~\cite{Ku16} the average case
	for the uniform distribution
	on finite-dimensional sub-balls of the input set~$F$
	was considered.}
In the general linear setting of this chapter,
for~\mbox{$n<m$} we have
\begin{equation} \label{eq:Vol-ratioMCnonada}
	e^{\ran,\nonada}(n,S)
		\geq \frac{m-n}{m+1} \,
			 \sup_{X_m}
					\inf_{Y_{m-n}}
						\left(\frac{\Vol_{m-n}(S(F) \cap Y_{m-n})}{
												\Vol_{m-n}(B_G \cap Y_{m-n})}\right)^{1/(m-n)} \,,
\end{equation}
where \mbox{$X_m \subseteq \widetilde{F}$}
and \mbox{$Y_{m-n} \subseteq S(X_m)$}
are subspaces with dimension \mbox{$\dim(X_m) = m$} and
\mbox{$\dim(Y_{m-n}) = m-n$},
furthermore, $B_G$~denotes the unit ball in~$G$,
and for each choice of~$Y_{m-n}$
the volume measure~$\Vol_{m-n}$
may be any $(m-n)$-dimensional Lebesgue-like measure
since we are only interested in the ratio of volumes.
In the case of sequence spaces~\mbox{$\ell_p^m \hookrightarrow \ell_q^m$},
the volume ratios could actually be computed in~\cite{Ku16}
(based on results from Meyer and Pajor~\cite{MP88}),
and by that we obtained
\begin{equation} \label{eq:MC(lp->lq)>=VolRat}
	e^{\ran,\nonada}(n,S) \succeq n^{1/q-1/p}
\end{equation}
for the whole parameter range~\mbox{$1 \leq p,q \leq \infty$},
thus closing the gap for~\mbox{$q<p=\infty$}.
For~\mbox{$p<q=\infty$}, in turn, Gaussian measures do a better job.

\subsubsection{The Case of Little Information $n \ll M$}

Up to this point the dimension of the vector spaces in consideration
was a constant times the number of information.
That setting is good enough when aiming for the order of convergence
for function space embeddings.
But what if the information is much smaller
than the size of the sequence spaces?

For an example of rather disappointing lower bounds
we restrict to the case
\begin{equation*}
	\ell_1^M \hookrightarrow \ell_2^M \,.
\end{equation*}
There is a well known result on deterministic algorithms,
\begin{equation*}
	e^{\deter}(n,\ell_1^M \hookrightarrow \ell_2^M)
		\asymp \min\left\{1, \sqrt{\frac{1 + \log \frac{M}{n}}{n}} \right\}
		\quad \text{for\, $n<M$.}
\end{equation*}
A proof based on compressive sensing
can be found e.g.\ in Foucart and Rauhut~\cite[Chap~10]{FR13CS},
the idea goes back to Kashin~1977~\cite{Ka77} and
Garnaev and Gluskin 1984~\cite{GG84}.
These errors are obtained with homogeneous and non-adaptive algorithms,
see \secref{sec:An:lp->lq}.

What lower bounds for the Monte Carlo error do we know?
We could use the Bernstein numbers for~\mbox{$n < m \leq M$},
\begin{equation*}
	b_m(\ell_1^M \hookrightarrow \ell_2^M) = \frac{1}{\sqrt{m}} \,,
\end{equation*}
see Pinkus~\cite[pp.~202--205]{Pin85}.
Then by \thmref{thm:BernsteinMCada}, for~\mbox{$M \geq 3n$} we obtain
\begin{equation*}
	e^{\ran}(n,\ell_1^M \hookrightarrow \ell_2^M)
		\geq \frac{1}{15}
						\, \max_{n<m \leq M}
									\frac{m-n}{m} \, \frac{1}{\sqrt{m}}
		\stackrel{m = 3n}{=}
			\frac{1}{15} \, \frac{2}{3\sqrt{3}} \, \frac{1}{\sqrt{n}} \,.
\end{equation*}
It is not reasonable that the size~$M$ of the problem did not contribute
at all to the error quantity, so this lower bound for the Monte Carlo error
is seemingly not quite optimal.
Directly considering Gaussian measures does not change the big point,
as the next lemma shows.
\begin{lemma}
	Let~\mbox{$J: \ell_2^m \rightarrow \R^M$} be an injective linear operator
	and~$\vecX$ be a standard Gaussian random vector in~$\ell_2^m$.
	Then for~\mbox{$n < m$} we have
	\begin{equation*}
		\inf_{\substack{\text{$P$ orth.~Proj.}\\
										\rank P = m - n}
					}
				\frac{\expect \|J P \vecX\|_2}{\expect \|J \vecX\|_1}
			\leq \sqrt{\frac{\pi}{2}} \, \frac{\sqrt{m-n}}{m}
			\leq \sqrt{\frac{\pi}{2}} \, \frac{1}{2} \, \frac{1}{\sqrt{n}} 
			\,.
	\end{equation*}
\end{lemma}
\begin{proof}
	Let~\mbox{$\vecr_1,\ldots,\vecr_M \in \ell_2^m$} denote the rows,
	and \mbox{$\vecc_1,\ldots,\vecc_m \in \R^M$} the columns
	of~$J$.
	Without loss of generality, the columns of~$J$ are orthogonal,
	compare \remref{rem:UniqueGauss}.
	
	Obviously, the term
	\begin{equation*}
		\inf_{\substack{\text{$P$ orth.~Proj.}\\
										\rank P = m - n}
					}
			\expect \|J P \vecX\|_2
	\end{equation*}
	does not change when transforming~$J$ into~\mbox{$\widetilde{J} := Q_M J Q_m$},
	with~$Q_M$ being an orthogonal matrix operating on~$\ell_2^M$,
	and $Q_m$~respectively on~$\ell_2^m$.
	For the denominator we have
	\begin{equation} \label{eq:E|JX|_1}
		\expect \|J \vecX\|_1 = \sqrt{\frac{2}{\pi}} \, \sum_{i=1}^M \|\vecr_i\|_2.
	\end{equation}
	What happens to this, when we rotate rows
	(thus performing a transformation that contributes to~$Q_M$)?
	Consider the rotation of the~$i_1$-th and the $i_2$-th rows,
	with~\mbox{$\|r_{i_1}\|_2 > \|r_{i_2}\|_2$}
	and \mbox{$\langle \vecr_{i_1}, \vecr_{i_2}\rangle \not= 0$},
	defined by
	\begin{align*}
		\vecr_{i_1}
			&\mapsto \vecr_{i_1}'
									:= \sqrt{\xi} \, \vecr_{i_1}
											+ \sqrt{1-\xi} \, \vecr_{i_2} \,, \\
		\vecr_{i_2}
			&\mapsto \vecr_{i_2}'
									:= -\sqrt{1-\xi} \, \vecr_{i_1}
											+ \sqrt{\xi} \, \vecr_{i_2}
	\end{align*}
	with~$\xi \in [0,1]$. By construction,
	\begin{equation} \label{eq:|r1|^2+|r2|^2}
		\|\vecr_{i_1}\|_2^2 + \|\vecr_{i_2}\|_2^2
			= \|\vecr_{i_1}'\|_2^2 + \|\vecr_{i_2}'\|_2^2 \,.
	\end{equation}
	Now, for the special choice
	\begin{equation*}
		\xi := \frac{1}{2}
							\, \left(1
											+ \frac{\sgn \langle \vecr_{i_1}, \vecr_{i_2}\rangle
														}{\sqrt{1+ \frac{4 \langle \vecr_{i_1},
																											\vecr_{i_2}
																								\rangle^2
																						}{\|\vecr_{i_1}\|_2^2
																							- \|\vecr_{i_2}\|_2^2}
																		}
															}
								\right) \,,
	\end{equation*}
	one can check that
	\begin{equation*}
		\langle \vecr_{i_1}, \vecr_{i_2} \rangle = 0 \,,
		\quad\text{and}\quad
			\|\vecr_{i_1}'\|_2 > \|\vecr_{i_1}\|_2
				\geq \|\vecr_{i_2}\|_2 > \|\vecr_{i_2}'\|_2 \,.
	\end{equation*}
	Together with~\eqref{eq:|r1|^2+|r2|^2}, one can easily prove that
	\begin{equation*}
		\|\vecr_{i_1}'\|_2 + \|\vecr_{i_2}'\|_2
			< \|\vecr_{i_1}\|_2 + \|\vecr_{i_2}\|_2 \,.
	\end{equation*}
	This means that by such transformations performed on~$J$,
	the expression~\eqref{eq:E|JX|_1} will be reduced.
	Now, repeatedly performing such transformations,
	and permuting rows of~$J$,
	one can find a matrix~\mbox{$J' = Q_M J$}
	with orthogonal rows~\mbox{$\vecr_1',\ldots,\vecr_M' \in \ell_2^m$}
	of descending norm,
	in particular~\mbox{$\vecr_{m+1}' = \ldots = \vecr_M' = \zeros$},
	moreover,~\mbox{$\expect \|J' \vecX\|_1 \leq \expect \|J \vecX\|_1$}.
	Since the columns of~$J$ are orthogonal, so are the columns of~$J'$.
	Hence we can find an orthogonal matrix~$Q_m$
	such that the only non-zero entries
	of the matrix~\mbox{$\widetilde{J} = J' Q_m$}
	lie on the diagonal, \mbox{$j_{kk} = \lambda_k$},
	and~\mbox{$\lambda_1 \geq \ldots \geq \lambda_m > 0$}.
	Without loss of generality,
	\begin{equation} \label{eq:sum(lambda)=1}
		\sum_{k=1}^m \lambda_k = 1 \,,
	\end{equation}
	so
	\begin{equation*}
		\expect \|J' \vecX\|_1
							= \expect \|\widetilde{J} \vecX\|_1
							= \sqrt{\frac{2}{\pi}} \,.
	\end{equation*}
		
	Now, consider the projection~$P_n$ onto the last~\mbox{$m-n$}~coordinates,
	i.e.\ the mapping \mbox{$P_n \vecx := \sum_{k = n+1}^m x_k \, \vece_k$}
	with~$\vece_k$ being the standard basis in~$\ell_2^m$.
	Then
	\begin{equation*}
		\expect \|\widetilde{J} P_n \vecX\|_2
			\leq \sqrt{\expect \|\widetilde{J} P_n \vecX\|_2^2}
			= \sqrt{\sum_{k = n+1}^m \lambda_k^2}
			\leq \frac{1}{\sqrt{m-n}} \, \sum_{k = n+1}^m \lambda_k
			\leq \frac{\sqrt{m-n}}{m} \,.
	\end{equation*}
	Here, we used the general
	inequality~\mbox{$\|\vecx\|_1 \leq \sqrt{d} \, \|\vecx\|_2$}
	for~\mbox{$\vecx \in \R^d$} with~\mbox{$d = m-n$},
	and \eqref{eq:sum(lambda)=1} together with the monotonicity
	of the coefficients~$\lambda_k$.
	
	Altogether we have
	\begin{equation*}
		\inf_{\substack{\text{$P$ orth.~Proj.}\\
										\rank P = m - n}
					}
				\frac{\expect \|J P \vecX\|_2}{\expect \|J \vecX\|_1}
			\leq
				\frac{\expect \|\widetilde{J} P_n \vecX\|_2
						}{\expect \|\widetilde{J} \vecX\|_1}
			\leq \sqrt{\frac{\pi}{2}} \, \frac{\sqrt{m-n}}{m} \,.
	\end{equation*}
	This is maximized for~\mbox{$m=2n$}.
\end{proof}

The situation for volume ratios~\eqref{eq:Vol-ratioMCnonada}
is the same, see~\cite[Sec~3.1]{Ku16} for further hints.
This disappointing phenomenon is not new,
Vyb{\'\i}ral~\cite{Vyb10} showed that Gaussian measures,
as well as uniform distributions, are not suitable in this and many other cases
to obtain lower bounds that -- up to a constant -- match the upper bounds.
In his paper on \emph{best $m$-term approximation}\footnote{%
	The concept of best-$m$-term approximation is not covered
	by our framework of information-based complexity,
	but results on that topic often have implications on the
	performance of some types of algorithms.
	For instance, in the case
	\mbox{$\ell_1^M \hookrightarrow \ell_2^M$},
	typical algorithms usually return vectors with
	only~$n$ non-zero entries, compare \secref{sec:An:lp->lq}.}
he basically shows
that in those particular average case settings the initial error is already
too small.
He proposes other average case settings which work perfectly
for best $m$-term approximation but are hard to use in the information
based complexity framework.

\label{para:lp->lq,q<p,LB}
There are also situations where the lower bounds perfectly reflect
the situation for~\mbox{$n \ll M$}, consider
\begin{equation*}
	\ell_p^M \hookrightarrow \ell_q^M
\end{equation*}
for the parameter range~\mbox{$1 \leq q \leq p \leq \infty$}.
The initial error is
\begin{equation*}
	e(0,\ell_p^M \hookrightarrow \ell_q^M) = M^{1/q-1/p} \geq 1 \,,
\end{equation*}
whereas the lower bounds by Gaussian measures~\eqref{eq:MC(lp->lq)>=Gauss}
give us
\begin{equation*}
	e^{\ran,\ada}(n,\ell_p^M \hookrightarrow \ell_q^M)
		\succeq M^{1/q-1/p}
	\quad\text{for\, \mbox{$n \leq M/4$}}
\end{equation*}
in the parameter range \mbox{$1 \leq q \leq p < \infty$}
or~\mbox{$q=p=\infty$},\footnote{%
	Bernstein numbers only give comparably satisfying lower bounds
	for~\mbox{$1 \leq q \leq p \leq 2$} or \mbox{$q = p$}.}
where the hidden constant may depend on~$p$ and~$q$.
In the case~\mbox{$1 \leq q < p = \infty$},
the lower bounds by Gaussian measures
are worse by a logarithmic factor~\mbox{$(\log M)^{-1/2}$},
but in the non-adaptive setting~\eqref{eq:MC(lp->lq)>=VolRat}
we get rid of that term.
What does this mean for our strategies to approximate this embedding
for~\mbox{$q \leq p$}?
Basically, we have the two alternatives of either
\begin{itemize}
	\item taking no information~\mbox{$n=0$} and accepting the initial error, or
	\item taking full information~\mbox{$n = M$}, thus having no error at all.
\end{itemize}
This is a reasonable approach to the problem,
because any choice~\mbox{$n \leq M/4$} will be insufficient
if we want to reduce the initial error by a significant factor,
and taking at most four times as much information than really necessary
is no big deal.

\subsection[Approximation of Ultra Smooth Functions]{
						$L_{\infty}$-Approximation for $C^{\infty}$-Functions}
\label{sec:Cinf->Linf}

We consider the $L_{\infty}$-approximation
for subclasses of $C^{\infty}$-functions
defined on the $d$-dimensional unit cube~\mbox{$[0,1]^d$},
\begin{equation*}
	\App: C^{\infty}([0,1]^d) \hookrightarrow L_{\infty}[0,1]^d \,.
\end{equation*}
The space~\mbox{$C^{\infty}([0,1]^d)$} has no natural norm and it
will be crucial for tractability what input sets we choose.
Novak and Wo\'zniakoswki~\cite{NW09b}
considered the input set
\begin{equation*}
	F^d := \{f \in C^{\infty}([0,1]^d)
					\, \mid \,
					\|D^{\vecalpha} f\|_{\infty} \leq 1
					\text{ for }\vecalpha \in \N_0^d\}\,.
\end{equation*}
Here,
\mbox{$D^{\vecalpha} f = \partial_1^{\alpha_1} \cdots \partial_d^{\alpha_d} f$}
denotes the partial derivative of~$f$ belonging
to a multi-index~\mbox{$\vecalpha \in \N_0^d$}.
In their study, Novak and Wo\'zniakowski showed that with this input set
the problem suffers from the curse of dimensionality
for deterministic algorithms.
Since the proof was based on the Bernstein numbers,
thanks to \thmref{thm:BernsteinMCada},
the curse of dimensionality extends to randomized algorithms.

We will cover this case within a slightly more general setting,
considering the input sets
\begin{equation*}
	\begin{split}
		F_p^d := \{f \in C^{\infty}([0,1]^d)
								\, \mid \,&
								\|\nabla_{\vecv_k} \cdots \nabla_{\vecv_1} f\|_{\infty}
									\leq |\vecv_1|_p \cdots |\vecv_k|_p\\
								&\text{for all } k \in \N_0, \,
								\vecv_1,\ldots,\vecv_k \in \R^d\}\,,
	\end{split}
\end{equation*}
where~\mbox{$\nabla_{\vecv} f$} denotes the directional derivative along
a vector~\mbox{$\vecv \in \R^d$},
and we write~$|\vecv|_p$ for the $p$-norm of~\mbox{$\vecv \in \ell_p^d$},
\mbox{$1 \leq p \leq \infty$}.
Note that, indeed, this is a generalization of the original problem
since~\mbox{$F^d = F_1^d$}.
The set~$F_p^d$ can be seen as the unit ball of the space
\begin{equation}
	\widetilde{F}_p^d := \{f \in C^{\infty}([0,1]^d)
											\, \mid \, \|f\|_{F_p} < \infty \}\,,
\end{equation}
equipped with the norm
\begin{equation} \label{kunsch-eq:Coonorm}
	\|f\|_{F_p} := \sup_{\substack{k \in \N_0 \\
																\vecv_1,\ldots,\vecv_k \in \R^d}
											}
										|\vecv_1|_p^{-1} \cdots |\vecv_k|_p^{-1}
										\, \|\nabla_{\vecv_k} \cdots \nabla_{\vecv_1} f\|_{\infty}
										\,.
\end{equation}

First, we aim for lower bounds, to this end starting with the Bernstein numbers
of the restricted operator
\mbox{$\App: \widetilde{F}_p^d \hookrightarrow L_{\infty}[0,1]^d$}.
The proof follows the lines of Novak and Wo\'zniakowski~\cite{NW09b}.
\begin{proposition}
	For $1 \leq p < \infty$ we have
	\begin{equation*}
		b_m(\widetilde{F}_p^d \hookrightarrow L_{\infty}) = 1
		\quad \text{for\, $n \leq 2^{\lfloor \frac{d^{1/p}}{3} \rfloor}$.}
	\end{equation*}
	In the case~\mbox{$p = 1$}, this even holds
	for~\mbox{$n \leq 2^{\lfloor d/2 \rfloor}$}.
\end{proposition}
\begin{proof}
	Note that~\mbox{$\|\cdot\|_{F_p} \geq \|\cdot\|_{\infty}$},
	and therefore
	\mbox{$b_m(\widetilde{F}_p^d \hookrightarrow L_{\infty}) \leq 1$}
	for all~\mbox{$m \in \N$}.
	
	We set~\mbox{$r := \lceil 2 \, d^{1-1/p} \rceil$}
	and \mbox{$s:= \lfloor d/r \rfloor
							\geq \lfloor \frac{d^{1/p}}{2 + d^{-1}} \rfloor
							\geq \lfloor \frac{d^{1/p}}{3} \rfloor$}.
	Consider the following linear subspace of~$\widetilde{F}_p^d$,
	\begin{equation} \label{eq:Vpd}
		\begin{split}
			V_p^d:= \bigl\{ f \mid\,&
								f(\vecx)
									= \sum_{\veci \in \{0,1\}^s}
											a_{\veci} \,
												(x_1+ \ldots + x_r)^{i_1}
												\cdots
												(x_{r(s-1)+1} + \ldots + x_{rs})^{i_s}, \\
								&a_{\veci} \in \R
							\bigr\}\,
		\end{split}
	\end{equation}
	with \mbox{$\dim V_p^d = 2^s$}.
	For~\mbox{$f \in V_p^d$} and~\mbox{$\vecv \in \R^d$},
	we will show
	\mbox{$\|\nabla_{\vecv} f\|_{\infty} \leq |\vecv|_p \, \|f\|_{\infty}$}.
	Besides, \mbox{$\nabla_{\vecv} f \in V_p^d$},
	so it then easily follows that \mbox{$\|f\|_{F} = \|f\|_{\infty}$}
	for~\mbox{$f \in V_p^d$}.
	Therefore, with~\mbox{$m = 2^s$} and
	the subspace~\mbox{$X_m = V_p^d \subset \widetilde{F}_p^d$},
	we obtain~\mbox{$b_m(\widetilde{F}_p^d \hookrightarrow L_{\infty}) = 1$}.
	Since the sequence of Bernstein numbers is decreasing,
	we know the~first~$2^s$~Bernstein numbers.
	
	In order to estimate~\mbox{$\|\nabla_{\vecv} f\|_{\infty}$},
	we first consider partial derivatives.
	For an index \mbox{$(k-1)r < i \leq kr$},
	where~\mbox{$k \in \{1,\ldots,s\}$},
	for~\mbox{$f \in V_p^d$} and~\mbox{$\vecx \in [0,1]^d$} we have
	\begin{equation*}
		\begin{split}
			|\partial_i f(\vecx)|
				&= {\textstyle \frac{1}{r}} \,
						|f(x_1,\ldots,x_{(k-1)r},1,\ldots,1,x_{kr+1},\ldots,x_d)\\
				&\qquad-f(x_1,\ldots,x_{(k-1)r},0,\ldots,0,x_{kr+1},\ldots,x_d)| \\
				&\leq {\textstyle \frac{1}{r}} \,
						|f(x_1,\ldots,x_{(k-1)r},1,\ldots,1,x_{kr+1},\ldots,x_d)|\\
				&\qquad+ |f(x_1,\ldots,x_{(k-1)r},0,\ldots,0,x_{kr+1},\ldots,x_d)| \\
				&\leq {\textstyle \frac{2}{r}} \, \|f\|_{\infty} \\
				&\leq d^{-1 + 1/p} \, \|f\|_{\infty} \,.
		\end{split}
	\end{equation*}
	For the directional derivative~\mbox{$\nabla_{\vecv} f$},
	this gives us
	\begin{equation*}
		\begin{split}
			\|\nabla_{\vecv} f\|_{\infty}
				&\leq \sum_{i=1}^d |v_i| \, \|\partial_i f\|_{\infty} \\
				&\leq |\vecv|_1 \, d^{-1 + 1/p} \, \|f\|_{\infty} \\
				&\leq |\vecv|_p \, \|f\|_{\infty} \,.
		\end{split}
	\end{equation*}
\end{proof}

By \thmref{thm:BernsteinMCada}
(or \thmref{thm:BernsteinMChom}, respectively)
we directly obtain the following result on the Monte Carlo complexity.
\begin{corollary} \label{cor:LinfAppLB}
	Consider the approximation problem
	\mbox{$\App:F_p^d \hookrightarrow L_{\infty}[0,1]^d$}
	with parameter~\mbox{$1 \leq p < \infty$}.
	Then the Monte Carlo complexity for achieving an error smaller than
	\mbox{$\eps \leq \frac{1}{30}$} is bounded from below by
	\begin{equation*}
		n^{\ran,\ada}(\eps,\App,F_p^d)
			> 2^{\lfloor \frac{d^{1/p}}{3} \rfloor - 1} \,.
	\end{equation*}
	For homogeneous algorithms we have the same complexity
	for~\mbox{$\eps = \frac{1}{2}$} already,
	\begin{equation*}
		n^{\ran,\homog}({\textstyle \frac{1}{2}},\App,F_p^d)
			\geq 2^{\lfloor \frac{d^{1/p}}{3} \rfloor - 1} \,.
	\end{equation*}
	In the case~\mbox{$p=1$}, the problem suffers from the curse of dimensionality.
	In general, for small~\mbox{$\eps>0$}, the $\eps$-complexity
	depends exponentially on~$d^{1/p}$.
\end{corollary}

Note that the initial error is~\mbox{$e(0,\App,F_p^d) = 1$},
hence properly normalized.
Furthermore, functions from~$F_p^d$ can be identified with functions
in~$F_p^{d+1}$ that are independent from~$x_{d+1}$.

The upper bounds
actually get close to the lower bound in terms of $d$-dependency.
The idea originates from Vyb{\'\i}ral~\cite{Vyb14},
where it has been used for slightly different settings,
but included a case similar to the case \mbox{$p = \infty$}
here.
\begin{theorem} \label{thm:Cinf->LinfUB}
	For the $L_{\infty}$-approximation of smooth functions from
	the classes~$F_p^d$, \mbox{$1 \leq p \leq \infty$},
	for~\mbox{$\eps > 0$},
	we obtain the following upper bounds on	the $\eps$-complexity
	achieved by linear deterministic algorithms,
	\begin{equation*}
		n^{\deter,\lin}(\eps,\App,F_p^d,\Lall)
			\leq \exp\left( \log(d+1)
									\, 
									\max\left\{\frac{\log {\textstyle
																								\frac{1}{\eps}}
																			}{\log 2},\,
																		\euler \, d^{1/p}
											\right\}
							\right) \,.
	\end{equation*}
	In particular, in the case~\mbox{$p > 1$},
	the problem is \mbox{$(s,t)$}-weakly tractable
	for all~\mbox{$t>\frac{1}{p}$} and~\mbox{$s > 0$},
	but not for~\mbox{$t<\frac{1}{p}$}.
	In the case of~\mbox{$p = \infty$},
	the problem is quasi-polynomially tractable.
\end{theorem}
\begin{proof}
	As an algorithm we consider the $k$-th Taylor polynomial, \mbox{$k \in \N_0$},
	at the point~\mbox{$\vecx_0 := (\frac{1}{2},\ldots,\frac{1}{2})$},
	then the output function
	is defined for~\mbox{$\vecx = \vecx_0 + \vecv \in [0,1]^d$} as
	\begin{equation*}
		[A_k(f)](\vecx)
			:= \sum_{j = 0}^k
					\frac{[\nabla_{\vecv}^j f](\vecx_0)}{j!}
			= \sum_{\substack{\vecalpha \in \N_0^d\\
												|\vecalpha|_1 \leq k}
							}
					\frac{[D^{\vecalpha} f](\vecx_0)}{\vecalpha !}
						\, \vecv^{\vecalpha}
			\,.
	\end{equation*}
	For this deterministic algorithm we need
	\begin{equation} \label{eq:n(k)}
		n(k) := \card(A_k) = \binom{d+k}{d} \leq (d+1)^k
	\end{equation}
	partial derivatives of the input~$f$ at~$\vecx_0$ as information.\footnote{%
		Vyb{\'\i}ral~\cite{Vyb14} even shows
		that the same amount of function values is actually sufficient
		to approximate the partial derivatives at the point~$\vecx_0$
		with arbitrarily high accuracy.\\
		The problem of counting the number of
		partial derivatives~\mbox{$D^{\vecalpha} f(\vecx_0)$}
		up to the order~\mbox{$|\vecalpha|_1 \leq k$}, \mbox{$\vecalpha \in \N_0^d$},
		is equivalent to choosing $d$~numbers~\mbox{$t_1 < \ldots < t_d$}
		from the set~\mbox{$\{1,\ldots,d+k\}$}
		by the transformation \mbox{$\alpha_j := t_j - t_{j-1} - 1$},
		where~\mbox{$t_0 = 0$}.
		For our purpose,
		it is sufficient to know
		that we do not need more than~\mbox{$(d+1)^k$} partial derivatives
		(like deciding $k$~times
		in which coordinate direction to derive -- or not to derive -- the function)
		since~$k$ is very small.}
	The error estimate then is
	\begin{equation*}
		|f(\vecx) - [A_k(f)](\vecx)|
			\leq \frac{\|\nabla_{\vecv}^{k+1} f\|_{\infty}}{(k+1)!} \,,
	\end{equation*}
	and with~\mbox{$|\vecv|_p \leq \frac{1}{2} \, d^{1/p}$},
	for the input set~$F_p^d$ we obtain the error bound
	\begin{equation} \label{eq:e(k)}
		\begin{split}
			e(A_k,F_p^d)
				&\leq \frac{1}{(k+1)!} \, \left(\frac{d^{1/p}}{2}\right)^{k+1} \\
			\text{[Stirling's formula]} \quad
				&\leq \frac{1}{\sqrt{2 \, \pi \, (k+1)}}
								\, \left(\frac{\euler \, d^{1/p}
															}{2 \, (k+1)}
									\right)^{k+1} \,.
		\end{split}
	\end{equation}
	What~\mbox{$k \in \N_0$} should we choose
	in order to guarantee an error smaller or equal
	a given tolerance~\mbox{$\eps > 0$}?
	This is ensured for
	\begin{equation*}
		\left(\frac{\euler \, d^{1/p}
							}{2 \, (k+1)}
									\right)^{k+1}
			\leq \eps
		\quad \Leftrightarrow \quad
		k+1 \geq \frac{\euler \, d^{1/p}}{2}
							\, \eps^{-\frac{1}{k+1}} \,.
	\end{equation*}
	Note that~\mbox{$(1/\eps)^{1/(k+1)} \leq 2$}
	for~\mbox{$k+1 \geq \log \frac{1}{\eps} / \log 2$},
	so choosing
	\begin{equation*}
		k = k(\eps)
			= \left\lfloor
					\max \left\{\frac{\log {\textstyle \frac{1}{\eps}}
														}{\log 2},\,
											\euler \, d^{1/p}
							\right\}
				\right\rfloor
	\end{equation*}
	will give us the guarantee we aim for.
	By this and~\eqref{eq:n(k)},
	we obtain the theorem on the $\eps$-complexity.
\end{proof}

This upper bound is not optimal in terms of the speed of convergence,
which is superpolynomial
(as it has already been mentioned in Novak and Wo\'zniakowski~\cite{NW09b}).
However, together with \corref{cor:LinfAppLB}, it shows
that in these cases randomization does not help to improve
the tractability classification of the problems.
Here, only narrowing the input set affects tractability.

There are several other publications worth mentioning
that study the tractability of the approximation of smooth functions
in the worst case setting.
Weimar~\cite{Wei12} discusses several settings with weighted Banach spaces.
Xu~\cite{Xu15} considers the $L_p$-approximation for~\mbox{$1 \leq p < \infty$}
and the same input set~$F^d$ as in Novak and Wo\'zniakowski~\cite{NW09b}.

\chapter[Uniform Approximation of Functions from a Hilbert Space]{
Uniform Approximation of Functions from a Hilbert Space}
\label{chap:Hilbert}

We study the $L_{\infty}$-approximation of functions from Hilbert spaces
with linear functionals~$\Lall$ as information.
Based on a fundamental Monte Carlo approximation method
(originally for finite dimensional input spaces)
which goes back to Math\'e~\cite{Ma91},
see \secref{sec:HilbertFundamental},
we propose a function approximation analogue to
standard Monte Carlo integration,
now using ``Gaussian linear functionals'' as random information,
see \secref{sec:HilbertPlainMCUB}.
This method is intended to break the curse of dimensionality
where it holds in the deterministic setting.
The analysis relies on
the theory of Gaussian fields, see \secref{sec:E|Psi|_sup},
some theory of reproducing kernel Hilbert spaces is needed as well,
see \secref{sec:RKHS}.
Using a known proof technique
for lower bounds in the worst case setting,
see \secref{sec:HilbertWorLB},
we can prove the curse of dimensionality
for the deterministic approximation of functions
from unweighted periodic tensor product Hilbert spaces,
whereas for the randomized approximation
we can show polynomial tractability
under certain assumptions,
see \secref{sec:HilbertPeriodic} for this particular application.
A specific example are Korobov spaces,
see \thmref{thm:Korobov}.

\section{Motivation and the General Setting}

For the integration problem with standard information~$\Lstd$,
it is known since more than half a century
that randomization can speed-up the order of convergence.
For example, for $r$-times continuously differentiable functions
\begin{equation*}
	F_r^d := \{f \in C^r([0,1]^d)
					\, \mid \,
					\|D^{\vecalpha} f\|_{\infty} \leq 1
					\text{ for $\vecalpha \in \N_0^d$
					with $|\vecalpha|_1 \leq r$}\}\,,
\end{equation*}
one can show
\begin{equation*}
	e^{\ran}(n,\Int,F_r^d) \asymp n^{-r/d-1/2}
		\prec e^{\deter}(n,\Int,F_r^d) \asymp n^{-r/d} \,,
\end{equation*}
where the hidden constants depend on~$r$ and~$d$,
see for instance the lecture notes of Novak~\cite[Secs~1.3.8/9 and 2.2.9]{No88},
the original result is due to Bakhvalov 1959~\cite{Bakh59}.
We see that for fixed smoothness and high dimensions the deterministic rate
gets arbitrarily bad,
whereas for Monte Carlo methods we have a guaranteed rate
of convergence of~$n^{-1/2}$.
Even worse, in the deterministic setting
upper bounds are achieved by product rules that use
\mbox{$n = m^d$}~function values on a regular grid as information,
which for high dimensions is of no practical use.
The classical approach to lower bounds --
when proving the rate of convergence, for both settings --
involved constants that are exponentially small in~$d$.
It is only recent that Hinrichs, Novak, Ullrich, and Wo\'zniakowski~\cite{HNUW17}
proved the curse of dimensionality for these classes of $C^r$-functions
in the deterministic setting,
\begin{equation*}
	n^{\deter}(\eps,\Int,F_r^d) \geq c_r^d \, (d/\eps)^{d/r}
	\quad\text{for all~\mbox{$d \in \N$} and \mbox{$0 < \eps < 1/2$},}
\end{equation*}
moreover, they proved that product rules are really the best what we can do.
In contrast to this, the problem is strongly polynomially tractable
in the randomized setting by simple means of the classical Monte Carlo method
\begin{equation} \label{eq:stdMCint}
	M_n(f) := \frac{1}{n} \sum_{i = 1}^n f(\vecX_i) \,,
\end{equation}
where the~$\vecX_i$ are iid uniformly distributed on the domain~\mbox{$[0,1]^d$}.
Here, we have the bound~\mbox{$e(M_n,F_r^d) \leq n^{-1/2}$}, and therefore
\begin{equation*}
	n^{\ran}(\eps,F_r^d) \leq \lceil \eps^{-2} \rceil
	\quad \text{for\, $\eps > 0$.}
\end{equation*}
When aiming for the optimal rate~$n^{-r/d-1/2}$,
a proof would usually consider algorithms
that use exponentially in~$d$ many function values,
so in many high-dimensional cases
the standard Monte Carlo method
might be the best approach to a practical solution.

This observation raises the question
whether there are approximation problems
where randomization significantly reduces the complexity.
Even more,
can we find a comparably simple Monte Carlo method
that breaks the curse of dimensionality?\\
The short answer is: Yes, we can -- at least for some problems.

Throughout this chapter we consider linear problems
\begin{equation*}
	S: \widetilde{F} \hookrightarrow G
\end{equation*}
with the input set~$F$ being the unit ball of~$\widetilde{F}$,
allowing algorithms to use arbitrary continuous linear functionals~$\Lall$
for information. The latter is a crucial assumption for the new upper bounds
based on a fundamental Monte Carlo approximation method,
see \propref{prop:Ma91_l2G} in \secref{sec:HilbertFundamental} below.
Whilst for the introductory part~\secref{sec:HilbertIntro}
the input space~$\widetilde{F}$ is not necessarily a Hilbert space,
this will be the case
for Sections~\ref{sec:HilbertTools} and~\ref{sec:HilbertExamples}.
The examples we present towards the end of the chapter
are all with the output space~\mbox{$G = L_{\infty}$},
but it should also be possible to consider embeddings into
classical smoothness spaces~\mbox{$C^r$}.

\section{Introduction to Randomized Approximation}
\label{sec:HilbertIntro}

The most important part of this introduction
is \secref{sec:HilbertFundamental}
with the fundamental Monte Carlo approximation method \propref{prop:Ma91_l2G},
an idea which is due to Math\'e~\cite{Ma91}.
Sections~\ref{sec:An:lp->lq}, \ref{sec:ranApp:e(n)},
and \ref{sec:l2->linf,curse},
give an overview of different settings
where this method can be applied.
The reader can decide to skip these sections
and go directly to~\secref{sec:HilbertTools}
where we collect tools for the tractability analysis of function approximation
in both, the deterministic and the randomized setting.
These methods are applied to exemplary problems in \secref{sec:HilbertExamples}.

Still, those last three sections
of the present introductory part (Sections~\ref{sec:An:lp->lq}--\ref{sec:l2->linf,curse})
may be helpful to gain some deeper insight
into the potential of randomized approximation
and the historical background.
Within \secref{sec:An:lp->lq} we consider finite dimensional
sequence spaces and summarize what is known about this topic.
Recovery of sequences is a keystone for the understanding of
how the speed of convergence can be enhanced by randomization
for function space embeddings,
a short overview on that issue is to be found in \secref{sec:ranApp:e(n)}.
Finally, in \secref{sec:l2->linf,curse}
we discuss a sequence space model for $d$-dependent problems
where randomization breaks the curse of dimensionality.
This example will give us strong indication that we should restrict
to Hilbert spaces~$\Hilbert^d$ of $d$-variate functions as input spaces,
and~$L_{\infty}$ as the target space,
in search for examples where randomization helps to break the curse.

\subsection{A Fundamental Monte Carlo Approximation Method}
\label{sec:HilbertFundamental}

The following result originates from Math{\'e}~\cite{Ma91}
and is a key component for the Monte Carlo approximation
of Hilbert space functions.
Here we keep it a little more general than in the original paper,
where the output space was a sequence space~$\ell_q^m$ with~\mbox{$q>2$}.

\begin{proposition}
\label{prop:Ma91_l2G}
	Let \mbox{$S:\ell_2^m \rightarrow G$} be a linear operator between normed spaces
	and consider the unit ball~\mbox{$B_2^m \subset \ell_2^m$} as the input set.
	Let the information mapping~$N^{\omega} = N$
	be a random \mbox{$(n \times m$)}-Matrix with entries
	\mbox{$N_{ij} = \frac{1}{\sqrt{n}} \, X_{ij}$},
	where the $X_{ij}$ are
	independent standard Gaussian random variables.
	Then \mbox{$A_n^{\omega} := S \, N^{\top} N$}
	defines a linear rank-$n$ Monte Carlo method
	(\mbox{$\phi^{\omega}(\vecy) = S \, N^{\top} \vecy$})
	and its error is bounded from above by
	\begin{equation*}
		e(A_n,S:\ell_2^m \rightarrow G)
			\leq \frac{2}{\sqrt{n}} \, \expect \|S \vecX\|_G
	\end{equation*}
	where $\vecX$ is a standard Gaussian vector in $\R^m$.
\end{proposition}

\begin{proof}
	Note that~$A_n$ is an unbiased linear Monte Carlo algorithm.
	To see this, let $\vecx \in \R^m$, then
	\begin{equation*}
		\expect (N^{\top} N \vecx)(i)
			= \frac{1}{n} \sum_{j=1}^n \sum_{k=1}^m
					\underbrace{\expect X_{ji} X_{jk}}_{= \delta_{ik}} \, x_k
			= x_i \,,
	\end{equation*}
	i.e.\ \mbox{$\expect N^{\top} N \vecx = \vecx$},
	and by linearity of~$S$ we have \mbox{$\expect A_n^{\omega} \vecx = S \vecx$}.
	
	We start from the definition of the error for an input~$\vecx \in \ell_2^m$,
	\begin{align*}
		e(A_n,\vecx)
			&= \expect\|S\vecx-S N^{\top} N \vecx\|_G \,.
		\intertext{%
	Now, let~$M$~be an independent copy of~$N$.
	We write~$\expect'$ for expectations with respect to~$M$,
	and $\expect$ with respect to~$N$.
	Using
	\mbox{$\expect^{\prime} M^{\top}M = \id_{\R^m}$}
	and \mbox{$\expect^{\prime} M = 0$},
	we can write}
			&=
				\expect \| \expect^{\prime}
											S(M^{\top}M - M^{\top}N + N^{\top} M - N^{\top} N) \vecx
								\|_G \\
			&\stackrel{\text{$\Delta$-ineq.}}{\leq}
				2 \expect \expect^{\prime}
										\left\| S \left(\frac{M+N}{\sqrt{2}}\right)^{\top}
														\left(\frac{M-N}{\sqrt{2}}\right) \vecx
										\right\|_G \,.
		\intertext{%
	The distribution of $(M,N)$ is identical to that
	of~\mbox{$\left(\frac{M+N}{\sqrt{2}},\frac{M-N}{\sqrt{2}}\right)$},
	therefore}
			&= 2 \expect \expect^{\prime} \|S N^{\top} M \vecx \|_G.
		\intertext{%
	Here, $M \vecx$ is a Gaussian vector distributed
	like~$\frac{\|\vecx\|_2}{\sqrt{n}} \, \vecY$ with $\vecY$~being a
	standard Gaussian vector on $\R^n$.
	So we continue,
	$\expect^{\prime}$ now denoting the expectation with respect to~$\vecY$,}
			&= \frac{2 \|\vecx\|_2}{\sqrt{n}}
						\expect \expect^{\prime} \|S N^{\top} \vecY \|_G. \\
		\intertext{%
	For fixed~$\vecY$,
	the distribution of~\mbox{$N^{\top} \vecY$} is identical to
	that of~\mbox{$\frac{\|\vecY\|_2}{\sqrt{n}} \, \vecX$}
	with $\vecX$~being a standard Gaussian vector on~$\R^m$.
	Let $\expect$ denote the expectation with respect to~$\vecX$.
	By Fubini's theorem we get}
			&= \frac{2 \|\vecx\|_2}{n}
						\, \expect^{\prime}
								[\|\vecY\|_2 \, \expect \|S \vecX\|_G] \,.
		\intertext{%
	Using \mbox{$\expect^{\prime} \|\vecY\|_2
								\leq \sqrt{\expect^{\prime} \|\vecY\|_2^2} = \sqrt{n}$},
	we finally obtain}
			e(A_n,\vecx) &\leq \frac{2 \|\vecx\|_2}{\sqrt{n}}
											\, \expect \|S \vecX\|_G \,.
	\end{align*}
\end{proof}

\begin{remark}[Properties of the fundamental function approximation method]
	\label{rem:FundMC}
	As mentioned within the proof of the error bound,
	the algorithm is \emph{unbiased},
	that is, \mbox{$\expect A_n^{\omega} \vecx = S \vecx$}
	for~\mbox{$\vecx \in \ell_2^m$}.
	
	However, in general the method is \emph{non-interpolatory}
	since for non-trivial problems~$S$ with positive probability
	the output will be outside the image~\mbox{$S(B_2^m)$}
	of the input set~\mbox{$B_2^m \subset \ell_2^m$},
	which is the unit ball.
	If the solution operator~$S$ is injective,
	then the output is the solution for~\mbox{$N^{\top} N \vecx \in \ell_2^m$},
	and one can show
	\begin{equation*}
		\sqrt{\expect \|N^{\top} N \vecx\|_2^2}
			= \sqrt{1 + \frac{m+1}{n}} \, \|\vecx\|_2 \,.
	\end{equation*}
	We will put this to an extreme in \secref{sec:HilbertPlainMCUB}.
	Applied to function approximation problems,
	the method will produce an output function for which the Hilbert norm
	is almost surely infinite.
	This means that the output does not only lie outside of the input set,
	but it actually drops out of the input space~$\Hilbert$,
	see \remref{rem:Lstoch} for the general phenomenon,
	and \remref{rem:SmoothnessLost} on the loss of smoothness
	in the particular context of Korobov spaces.
\end{remark}

\subsection{Methods for the Recovery of Sequences}
\label{sec:An:lp->lq}

We consider again the identity operator between sequence spaces,
\begin{equation*}
	\App: \ell_p^M \hookrightarrow \ell_q^M \,,
\end{equation*}
where \mbox{$1 \leq p,q \leq \infty$} and~\mbox{$M \in \N$},
compare \secref{sec:Bernstein:lp->lq} where lower bounds have been
discussed.
Now, we summarize what is known on upper bounds.
In addition to well-known deterministic bounds,
we owe linear Monte Carlo results to Math\'e 1991~\cite{Ma91},
and non-linear Monte Carlo estimates to Heinrich 1992~\cite{He92}.

What is the basic structure of deterministic and randomized algorithms,
depending on the parameters~$p$ and~$q$?
In what cases does randomization help?

\subsubsection{The Case $1 \leq q \leq p \leq \infty$
							-- Practically Complete Information Needed}

The simplest case is~\mbox{$1 \leq q \leq p \leq \infty$}
where the initial error by simple norm estimates is
\begin{equation*}
	e(0,\ell_p^M \hookrightarrow \ell_q^M) = M^{1/q-1/p} \,.
\end{equation*}
If we allow to use $n$~information functionals,
it is optimal to simply compute the first~$n$ entries of the input vector
and set the other entries to~$0$ for the approximant,
which gives us
\begin{equation*}
	e(n,\ell_p^M \hookrightarrow \ell_q^M) = (M-n)^{1/q-1/p} \,,
\end{equation*}
see Pietsch~\cite[Thm~7.2]{Pie74} for a proof.
This is a fairly small reduction of the initial error.
At best, we gain a factor at most~$2$ for~$n \leq \frac{M}{2}$,
in case~\mbox{$p = q$} any~\mbox{$n < M$} will be useless.
Randomization and adaption does not help a lot,
see page~\pageref{para:lp->lq,q<p,LB} for a deeper discussion.
So basically, in this case we can rely on deterministic linear methods. \\
\phantom{Ghost line.}

We now discuss several cases for \mbox{$1 \leq p < q \leq \infty$},
where the initial error is
\begin{equation*}
	e(0,\ell_p^M \hookrightarrow \ell_q^M) = 1 \,.
\end{equation*}

\subsubsection{The Case $1 \leq p < q \leq 2$
							-- Non-Linear Deterministic Methods}

In the case~\mbox{$p=1$} and~\mbox{$q=2$},
optimal deterministic error bounds
are obtained by non-linear methods
with a subtly chosen non-adaptive (that is linear) information mapping~$N$.
For an information~\mbox{$\vecy = N \vecx \in \R^n$}
obtained for an input~\mbox{$\vecx \in \ell_1^M$},
one then finds an output~\mbox{$\vecz \in \R^M$} by $\ell_1$-minimization,
\begin{equation*}
	\phi(\vecy)
		:= \argmin_{\substack{\vecz \in \R^M \\
													N\vecz = \vecy}
								}
					\| \vecz \|_1 \,.
\end{equation*}
This definition of the output 
simply guarantees that the algorithm is interpolatory.
Indeed, by construction, \mbox{$\| \vecz \|_1 \leq \| \vecx\|_1$},
and for inputs~$\vecx$ from the input set being the unit ball,
which is the input set,
the output~$\vecz$ will also be from that input set.
Furthermore, it gives the same information.
The structure of this method reflects that for linear problems
in the deterministic setting interpolatory algorithms based on non-adaptive information
are optimal up to a factor~$2$,
see the book on \mbox{IBC by Traub et al.~\cite[pp.~51--53 and 57--67]{TWW88}}
for further details.
In this particular case linear algorithms are far worse than
interpolatory algorithms.

The crucial point is to find a good information mapping~$N$.
Several non-constructive ways are known,
e.g.\ taking an \mbox{$n \times M$}-Matrix
with independent standard Gaussian entries,
then with positive probability it will have the properties that ensure
the up to a constant optimal error bounds
\begin{equation} \label{eq:det:l1->l2}
	e(\phi \circ N, \ell_1^M \hookrightarrow \ell_2^M)
		\leq C \, \min \left\{1,\sqrt{\frac{1+\log \frac{M}{n}}{n}}\right\} \,,
\end{equation}
where~\mbox{$C > 0$} is a numerical constant,
see Foucart and Rauhut~\cite[Chap~10]{FR13CS} for a proof.
Almost surely, the matrix~$N$ will be such that the $\ell_1$-minimization
is solved by a unique~\mbox{$\vecz \in \R^M$}.
Computing~\mbox{$\phi(\vecy)$}, in fact, is a linear optimization problem,
see Foucart and Rauhut~\cite[Sec~3.1]{FR13CS}.
Even more generally,
for the parameter range~\mbox{$1 = p < q \leq 2$} and~\mbox{$n<M$},
the same algorithms give up to a constant optimal error rates
\begin{equation} \label{eq:det:l1->lq}
	e^{\deter}(n, \ell_1^M \hookrightarrow \ell_q^M)
		\asymp \min \left\{1,
											\left(\frac{1+\log \frac{M}{n}}{n}\right)^{1-1/q}
									\right\} \,.
\end{equation}

On the other hand,
the lower bounds known for the Monte Carlo setting actually state
\begin{equation*}
	e^{\ran}(n, \ell_1^M \hookrightarrow \ell_q^M)
		\succeq n^{-(1-1/q)}
	\quad \text{for\, $n \leq {\textstyle \frac{M}{2}}$,}
\end{equation*}
compare \secref{sec:Bernstein:lp->lq}.
As discussed there, it seems odd for the Monte Carlo error
to be independent from the size~$M$ of the problem,
so we conjecture that randomization may not help significantly in this setting.
The gap between the lower and the upper bound is logarithmic in~$\frac{M}{n}$.

In Foucart and Rauhut~\cite[p.~327]{FR13CS} one can also find a summary
on the worst case error
for 
\mbox{$1 < p < q \leq \infty$},
it is based on results from Kashin 1981~\cite{Ka81}.
The basic structure of algorithms in that case again is
that the information will be non-adaptive and the output interpolatory,
which can be achieved by~$\ell_p$-minimization
(instead of $\ell_1$-minimization).
For simplicity, we only cite the error for~\mbox{$1 < p < q = 2$},
\begin{equation} \label{eq:det:lp->l2}
	e^{\deter}(n,\ell_p^M \hookrightarrow \ell_2^M)
		\asymp \min\left\{1,
										\frac{M^{1-1/p}}{\sqrt{n}}
							\right\}
	\quad \text{for\, $n < M$,}
\end{equation}
where the hidden constants may depend on~$p$.
This stands in contrast to the best known lower bounds on the Monte Carlo error
from \secref{sec:Bernstein:lp->lq},
\begin{equation*}
	e^{\ran}(n, \ell_p^M \hookrightarrow \ell_2^M) \succeq n^{-(1/p-1/2)}
	\quad \text{for\, $n \leq {\textstyle \frac{M}{2}}$,}
\end{equation*}
which exhibit a polynomial gap
of a factor~\mbox{$(\frac{M}{n})^{1-1/p} \leq \sqrt{\frac{M}{n}}$}.
Still, we conjecture that randomization will not help a lot
in the parameter range~\mbox{$1 \leq p < q \leq 2$}.


\subsubsection{The Case $2 \leq p < q \leq \infty$
							-- Linear Monte Carlo Approximation}

In the case~\mbox{$p=2$} and \mbox{$q = \infty$},
since Smolyak 1965~\cite{Smo65}
it is well known that
\begin{equation} \label{eq:det:l1->linf}
	e^{\deter}(n, \ell_2^M \hookrightarrow \ell_{\infty}^M)
		= \sqrt{\frac{M-n}{M}} \,,
\end{equation}
see \exref{ex:DiagOps} for more details and references.
The optimal algorithm is an orthogonal rank-$n$ projection.
In particular for~\mbox{$n < \frac{M}{2}$},
the deterministic error cannot go below~\mbox{$\frac{\sqrt{2}}{2}$}.
By norm estimates\footnote{%
	With~$2 \leq p$, for the input set being the unit ball,
	we have \mbox{$B_2^M \subseteq B_p^M$}.
	On the other hand, for~\mbox{$q \leq \infty$}
	the error measuring norms are related
	by \mbox{$\|\cdot\|_{\infty} \leq \|\cdot\|_q$},
	making the problem even more difficult for~\mbox{$q < \infty$}.
	}
we obtain that this lower bound holds in general
for~\mbox{$2 \leq p < q \leq \infty$},
\begin{equation*} 
	e^{\deter}(n,  \ell_p^M \hookrightarrow \ell_q^M)
		\geq {\textstyle \frac{\sqrt{2}}{2}}
	\quad \text{for\, $n \leq {\textstyle \frac{M}{2}}$.}
\end{equation*}
Since this is no significant reduction of the initial error,
practically, for the deterministic setting we have the choice between
full information~\mbox{$n = M$}, and accepting the initial error.

In this parameter range, however, it is helpful to apply the
fundamental linear Monte Carlo method from Math\'e~\cite{Ma93},
see \propref{prop:Ma91_l2G},
as long as~$n$ is big enough for that the method's error does not exceed
the initial error~$1$.
The Monte Carlo algorithm~\mbox{$A_n = N^{\top} N$},
where~$N$ is an \mbox{$(n \times M)$-matrix} with independent
zero-mean Gaussian entries of variance~$\frac{1}{n}$,
has the error
\begin{equation*}
	e(A_n, \ell_p^M \hookrightarrow \ell_q^M)
		\leq M^{1/2-1/p} \, e(A_n, \ell_2^M \hookrightarrow \ell_q^M)
		\leq 2 \, M^{1/2-1/p} \, \frac{\expect \|\vecX\|_q}{\sqrt{n}},
\end{equation*}
where~$\vecX$ is a standard Gaussian vector in~$\R^M$.
Using the norm estimates for Gaussian vectors,
see \lemref{lem:gaussqnormvector},
we obtain
\begin{equation} \label{eq:ran:lp->lq,2<=p<q}
	e^{\ran}(n, \ell_p^M \hookrightarrow \ell_q^M)
		\preceq \begin{cases}
								\min \left\{1,
													M^{1/2 - 1/p + 1/q}
													/ \sqrt{n}
										\right\}
									\quad&\text{for $2 \leq p < q < \infty$,}\\
								\min \left\{1,
													M^{1/2 - 1/p}
													\, \sqrt{\frac{1 + \log M}{n}}
										\right\}
									\quad&\text{for $2 \leq p < q = \infty$.}\\
							\end{cases}
\end{equation}
Here, the hidden constant may depend on~$q$.
Comparing this to the best known lower bounds,
see \secref{sec:Bernstein:lp->lq},
where for~\mbox{$n \leq {\textstyle \frac{M}{2}}$} we have
\begin{equation} \label{eq:ran:lp->lq,1<=p<q,LB}
	e^{\ran}(n, \ell_p^M \hookrightarrow \ell_q^M)
		\succeq \begin{cases}
							n^{-(1/p-1/q)}
								\quad& \text{for $1 \leq p < q < \infty$,} \\
							n^{-1/p} \, \sqrt{1+\log n}
								\quad& \text{for $1 \leq p < q = \infty$,}
						\end{cases}
\end{equation}
we observe a gap which is at least logarithmic in~$M$
(for~\mbox{$p = 2$} and~\mbox{$q = \infty$}),
and can grow up to a factor of almost order~\mbox{$\sqrt{\frac{M}{n}}$}
(the limiting case is~\mbox{$p \rightarrow \infty = q$}).
Once more, this gap seems to be a deficiency of the lower bounds
and not of the algorithms proposed,
especially in the case~\mbox{$p = 2$} and~\mbox{$q = \infty$}.

\subsubsection{The Case~\mbox{$1 \leq p < 2 < q \leq \infty$}
							-- Non-Linear Monte Carlo Approximation}

It is easier to approximate with the error measured in
an $\ell_q$-norm for~\mbox{$2 < q \leq \infty$}
than with respect to the~$\ell_2$-norm, so
\begin{equation} \label{eq:det:lp->lq,p<2<q,UB}
	e^{\deter}(n,\ell_p^M \hookrightarrow \ell_q^M)
		\leq e^{\deter}(n,\ell_p^M \hookrightarrow \ell_2^M) \,.
\end{equation}
In view of the preceding paragraph it is not surprising that
for~\mbox{$n \leq \frac{M}{2}$} the order of the worst case error
cannot be improved.
For \mbox{$1 < p < 2 < q \leq \infty$} we have the estimate
\begin{equation} \label{eq:det:lp->lq,p<2<q}
	e^{\deter}(n,\ell_p^M \hookrightarrow \ell_q^M)
		\asymp \min\left\{1,
										\frac{M^{1-1/p}}{\sqrt{n}}
							\right\}
	\quad \text{for\, $n \leq {\textstyle \frac{M}{2}}$,}
\end{equation}
see Foucart and Rauhut~\cite[p.~327]{FR13CS}.
All in all, we can use the same algorithms that we used
for~\mbox{$\ell_p^M \hookrightarrow \ell_2^M$}.
Only for~\mbox{$p=1$} and \mbox{$2<q \leq \infty$}, best known lower bounds,
see Foucart and Rauhut~\cite[Thm~10.10]{FR13CS},
do not match the upper bounds we obtain by~\eqref{eq:det:lp->lq,p<2<q,UB}.
In this case, for~\mbox{$n<M$} we have
\begin{equation} \label{eq:det:l1->lq,2<q}
	\min\left\{1,\left(\frac{1+\log \frac{M}{n}}{n}\right)^{1-1/q}\right\}
		\preceq e^{\deter}(n,\ell_1^M \hookrightarrow \ell_q^M)
		\preceq \min\left\{1,\sqrt{\frac{1+\log \frac{M}{n}}{n}}\right\} \,.
\end{equation}

Randomization, however, enables us to exploit the advantage
of measuring the error in an $\ell_q$-norm.
Namely,
we combine the non-linear deterministic algorithms
that we have for~\mbox{$\ell_p \hookrightarrow \ell_2$}
with the linear Monte Carlo approximation
for~\mbox{$\ell_2 \hookrightarrow \ell_q$},
this idea is contained in Heinrich~\cite[Prop~3]{He92}.
In detail, split the cost~$n = n_1+n_2$, collecting information
\mbox{$N^{\omega} \vecx
				= (N_1 \vecx, N_2^{\omega} \vecx)
				= (\vecy_1,\vecy_2)
				= \vecy$}
for~\mbox{$\vecx \in \R^M$},
where~$N_1$ is an \mbox{$(n_1 \times m)$}-matrix as we would choose
it for~{$\ell_p \hookrightarrow \ell_2$},
and $N_2^{\omega}$~is a random \mbox{$(n_2 \times m)$}-matrix
with iid zero-mean Gaussian entries of variance~\mbox{$1/n_2$}.
The output is generated in two steps.
In the first step we compute a rough deterministic approximant
\begin{equation*}
	\vecz_1
		= \phi_1(\vecy_1)
		:= \argmin_{\substack{\vecz \in \R^m\\
													N_1 \vecz = \vecy_1}
								}
					\|\vecz\|_p \,,
\end{equation*}
in the second step we generate the refined output
\begin{equation*}
	\phi^{\omega}(\vecy_1,\vecy_2)
		:= \vecz_1 + [N_2^{\omega}]^{\top}(\vecy_2 - N_2^{\omega} \vecz_1) \,.
\end{equation*}
The error for~\mbox{$\|\vecx\|_p \leq 1$} can be estimated as
\begin{align*}
	e((\phi^{\omega} \circ N^{\omega})_{\omega \in \Omega},
		\ell_p^M \hookrightarrow \ell_q^M,
		\vecx)
		&= \expect
				\|\vecx - \vecz_1
					- (N_2^{\omega})^{\top}N_2^{\omega} (\vecx - \vecz_1)
					\|_q \\
		&\leq e(([N_2^{\omega}]^{\top} N_2^{\omega}))_{\omega \in \Omega},
						\ell_2^M \hookrightarrow \ell_q^M)
					\, \|\vecx - \vecz_1\|_2 \\
		&\leq e(([N_2^{\omega}]^{\top} N_2^{\omega})_{\omega \in \Omega},
						\ell_2^M \hookrightarrow \ell_q^M)
					\, e(\phi_1 \circ N_1,
							\ell_p^M \hookrightarrow \ell_2^M) \,.
\end{align*}
One could go with~\mbox{$n_1 = n_2$},
then
using~\eqref{eq:det:l1->l2} or \eqref{eq:det:lp->l2}, respectively,
together with~\eqref{eq:ran:lp->lq,2<=p<q}, we obtain\footnote{%
	Heinrich~\cite[Cor~3]{He92} contains only the case~$p=1$
	since the other cases were not needed for the application
	to Sobolev embeddings.}
\begin{equation} \label{eq:ran:lp->lq,p<2<q}
	e^{\ran}(n,\ell_p^M \hookrightarrow \ell_q^M)
		\preceq \frac{1}{n}
			\,\begin{cases}
					\sqrt{(1 +\log M)\left(1+\log \frac{M}{n}\right)}
						\quad&\text{for $1 = p$ and $q = \infty$,}\\
					M^{1-1/p} \, \sqrt{1 +\log M}
						\quad&\text{for $1 < p \leq 2$ and $q = \infty$,}\\
					M^{1/q} \, \sqrt{1 + \log \frac{M}{n}}
						\quad&\text{for $1 = p$ and $2 < q < \infty$,}\\
					M^{1-1/p+1/q}
						\quad&\text{for $1 < p \leq 2 < q < \infty$,}\\
				\end{cases}
\end{equation}
where the hidden constant may depend on~$q$.
However, if~$n$ is too small, it might be better to omit the second step
and choose~\mbox{$n_1 = n$}, \mbox{$n_2 = 0$},
thus simply taking~\mbox{$\phi_1 \circ N_1$} as a deterministic algorithm
that achieves the bound from~\eqref{eq:det:lp->lq,p<2<q}.
Also note that for~\mbox{$n \prec M^{2-2/p}$} or
\mbox{$n \prec 1+\log \frac{M}{n}$}, respectively,
this estimate is not optimal,
and one should rather use the relation between the case~$p<2$
and the case of the input space being~$\ell_2^M$,
\begin{equation} \label{eq:ran:lp->lq,p<2<q,alter}
	e^{\ran}(n,\ell_p^M \hookrightarrow \ell_q^M)
		\leq e^{\ran}(n,\ell_2^M \hookrightarrow \ell_q^M)
		\stackrel{\eqref{eq:ran:lp->lq,2<=p<q}}{\preceq}
			\frac{1}{\sqrt{n}}
				\,\begin{cases}
						M^{1/q}
							\quad&\text{for $q < \infty$,}\\
						\sqrt{1 + \log M}
							\quad&\text{for $q = \infty$.}
					\end{cases}
\end{equation}
Again, known lower bounds~\eqref{eq:ran:lp->lq,1<=p<q,LB} do not reflect the
size~$M$ of the problem.

\subsection{Speeding up the Convergence for Function Approximation}
\label{sec:ranApp:e(n)}

Several examples are known
where the order of convergence can be improved by randomization.
Heinrich~\cite{He92} considered Sobolev embeddings
\begin{equation*}
	\App: W_p^r([0,1]^d) \hookrightarrow L_q([0,1]^d) \,,
\end{equation*}
where~\mbox{$r,d \in \N$}, and~\mbox{$1 \leq p,q \leq \infty$},
with the compactness condition~\mbox{$\frac{r}{d} > \frac{1}{p} - \frac{1}{q}$}.
Here, \mbox{$W_p^r([0,1]^d)$}~denotes the Sobolev space of smoothness~$r$,
that is the space of all functions~\mbox{$f \in L_p([0,1]^d)$}
such that for all~\mbox{$\vecalpha \in \N_0^d$},
with~\mbox{$|\vecalpha|_1 \leq r$},
the partial derivatives~$D^{\vecalpha}f$ exist in a weak sense
and belong to~$L_p$.
We consider the norm\footnote{%
	Concerning the order of convergence,
	any equivalent norm will give the same results.}
\begin{equation*}
	\|f\|_{W_p^r}
		:=\begin{cases}
				\left(\sum_{\substack{\vecalpha \in \N_0^d \\
															|\vecalpha|_1 \leq r}}
								\|D^{\vecalpha} f\|_p^p
				\right)^{1/p}
					\quad& \text{for $1 \leq p < \infty$,} \\
				\max_{\substack{\vecalpha \in \N_0^d \\
												|\vecalpha|_1 \leq r}}
					\|D^{\vecalpha} f\|_{\infty}
					\quad& \text{for $ p = \infty$,}
			\end{cases}
\end{equation*}
see for example Evans~\cite[Sec~5.2]{Ev98},
or Triebel~\cite[Sec~2.3, esp.\ Thm~2.3.3 and Rem~2.3.3/5]{Tr95}.

For the deterministic setting we refer to Vyb{\'\i}ral~\cite[Thm~4.12]{Vyb08}.
For simplicity we only cite the result for smoothness~\mbox{$r > d$},
\begin{equation*}
	e^{\deter}(n,W_p^r([0,1]^d) \hookrightarrow L_q([0,1]^d))
		\asymp
			\begin{cases}
				n^{-r/d}
					\quad&\text{for $1 \leq q \leq p \leq \infty$,}\\
					\quad&\text{or $1 \leq p < q \leq 2$,}\\
				n^{-r/d + 1/2 - 1/q}
					\quad&\text{for $1 \leq p < 2 < q \leq \infty$,}\\
				n^{-r/d + 1/p - 1/q}
					\quad&\text{for $2 \leq p \leq q \leq \infty$.}
			\end{cases}
\end{equation*}
For the one-dimensional case see also Pinkus~\cite[Chap~VII]{Pin85}.
From Heinrich~\cite[Thm~2]{He92} we know the randomized setting
for smoothness~\mbox{$r > d$},
\begin{equation*}
	e^{\ran}(n,W_p^r([0,1]^d) \hookrightarrow L_q([0,1]^d))
		\preceq
			\begin{cases}
				n^{-r/d}
					\quad&\text{for $1 \leq p \leq \infty$}\\
					\quad&\text{and $1 \leq q < \infty$,}\\
					\quad&\text{or~$p = q = \infty$,}\\
				n^{-r/d} \sqrt{1 + \log n}
					\quad&\text{for $1 \leq p < q = \infty$.}
			\end{cases}
\end{equation*}
Heinrich also proved lower bounds for the adaptive Monte Carlo setting
that match the rate of the upper bounds --
except for the case~\mbox{$1 \leq q < p = \infty$},
where a logarithmic gap of a factor~\mbox{$1/\sqrt{1 + \log n}$}
occurrs.\footnote{%
	This gap can actually be closed in the non-adaptive Monte Carlo setting
	since the lower bounds of the Sobolev embeddings are based on estimates
	for the sequence space embedding~\mbox{$\ell_p^M \hookrightarrow \ell_q^M$}.
	For this, in the case~\mbox{$1 \leq q < p = \infty$},
	we have a better lower bound~\eqref{eq:MC(lp->lq)>=VolRat}
	when restricting to non-adaptive methods.}
Note that in all cases
the hidden constants may depend on~$r$, $d$, $p$, and~$q$.
In comparison of these two results,
Heinrich could show that,
for~\mbox{$p < q$} and \mbox{$2 < q \leq \infty$},
randomized algorithms can improve the rate of convergence by a factor
that can reach almost the order~\mbox{$1/\sqrt{n}$},
the most prominent case is~\mbox{$p = 2$} and \mbox{$q = \infty$}.
This phenomenon was already known to Math\'e~\cite{Ma91}
in the~$1$-dimensional case.

Similar gaps between the Monte Carlo and the worst case error have been found
in Fang and Duan~\cite{FD07} for multi-variate periodic Sobolev spaces
with bounded mixed derivative.\footnote{%
	In Fang and Duan~\cite{FD07},
	while lower bounds hold for adaptive Monte Carlo methods,
	upper bounds are obtained with non-adaptive but in some cases
	\emph{non-linear} methods.}
Again, gaps occur in parameter settings where we also know
that randomization can help
for the sequence space embedding~\mbox{$\ell_p^M \hookrightarrow \ell_q^M$}.
This is not surprising since the estimates for function space embeddings
heavily rely on results for sequence space embeddings.
In order to illustrate the connection to sequence spaces,
let us outline the methods of discretization.

For lower bounds one usually finds
$m$-dimensional subspaces~\mbox{$X_m \subseteq \widetilde{F}$}
with~\mbox{$m \geq 2n$}
such that the restriction~\mbox{$S|_{X_m}$}
resembles the sequence space embedding~\mbox{$\ell_p^m \hookrightarrow \ell_q^m$}. 

Upper bounds are based on Maiorov's discretization technique~\cite{Maiorov75},
where the solution operator
\begin{equation*}
	S: \widetilde{F} \rightarrow G
\end{equation*}
is split into finite rank operators, so-called \emph{blocks},
\begin{equation*}
	S = \sum_{i = 1}^{\infty} S_i \,, \quad \rank S_i = h_i \in \N \,,
\end{equation*}
that can be related to sequence space
embeddings~\mbox{$\ell_p^{h_i} \hookrightarrow \ell_q^{h_i}$}
by estimates
\begin{equation*}
	e^{\ran}(n,S_i)
		\leq\gamma_i \, e^{\ran}(n,\ell_p^{h_i} \hookrightarrow \ell_q^{h_i})
\end{equation*}
with~\mbox{$\gamma_i > 0$}.
Now, with~\mbox{$n = n_1 + \ldots + n_k$},
\mbox{$k \in \N$}, we have
\begin{align*}
	e^{\ran}(n,S)
		&\leq \|S_{-k}\|_{F \rightarrow G}
			+ \sum_{i=1}^{k} e^{\ran}(n_i,S_i) \\
		&\leq \|S_{-k}\|_{F \rightarrow G}
			+ \sum_{i=1}^{k}
								\gamma_i
								\, e^{\ran}(n_i,\ell_p^{h_i} \hookrightarrow \ell_q^{h_i}) \,,
\end{align*}
where~\mbox{$S_{-k} := S - (S_1 + \ldots + S_k)$}.
A common shape of the block operators could be that
for a Schauder basis \mbox{$(\psi_j)_{j \in \N}$}
of the input space~\mbox{$\widetilde{F}$}, we have disjoint index sets
\begin{equation*}
	\bigsqcup_{i = 1}^{\infty} J_i = \N
\end{equation*}
of cardinality~\mbox{$\# J_i = h_i$},
such that for~\mbox{$f = \sum_{j = 1}^{\infty} a_j \, \psi_j$} we may write
\begin{equation} \label{eq:MaiorovBlock}
	S_i(f) := \sum_{j \in J_i} a_j \, S(\psi_j) \,.
\end{equation}
This structure can be found in Fang and Duan~\cite{FD07},
however, in the case of Heinrich~\cite{He92} the discretization is based on
another decomposition that is described in the book of K\"onig~\cite{Koe86}.

As mentioned before,
the hidden constants for the results on the order of convergence
may depend on the problem parameters.
Indeed, the upper and the lower bounds may differ largely,
even exponentially in~$d$. In particular, for Heinrich's result,
the upper bounds are obtained with $n$ being exponential in~$d$.

We want to point out another drawback of splitting the operator,
especially in the randomized setting.
Let the input space be a Hilbert space
\mbox{$\widetilde{F} = \Hilbert$}
with orthonormal basis~\mbox{$(\psi_j)_{j \in \N}$},
and split the operator~$S$ into block operators~$S_i$
of the structure~\eqref{eq:MaiorovBlock}.
Now, performing Maiorov's technique,
we approximate the first~$k$ block operators
by known methods, using a part~$n_i$ of the total information, respectively,
where~\mbox{$n = n_1 + \ldots + n_k$}.
Assume that for each of the blocks, using~$n_i$ pieces of information,
the fundamental Monte Carlo approximation method
from \propref{prop:Ma91_l2G} is the best method we know.
Then we obtain
\begin{equation*}
	e^{\ran}(n,S_1 + \ldots + S_k)
		\leq \sum_{i=1}^k
						e^{\ran}(n_i,S_i)
		\leq \sum_{i=1}^k
					\frac{2 \,\expect \left\|
															\sum_{j \in J_i} X_j \, S (\psi_j)
														\right\|_G
								}{\sqrt{n_i}} \,,
\end{equation*}
where the~$X_j$ are independent standard Gaussian random variables.
However, we could apply the fundamental Monte Carlo approximation method
directly to the cluster~\mbox{$S_1 + \ldots + S_k$},
and obtain the far better estimate
\begin{equation*}
	e^{\ran}(n,S_1 + \ldots + S_k)
		\leq \frac{2 \, \expect \left\|
															\sum_{i=1}^k \sum_{j \in J_i} X_j \, S (\psi_j)
														\right\|_G
							}{\sqrt{n}} \,.
\end{equation*}
(Apply the triangle inequality for comparison to the Maiorov type upper bound.)
For this reason, for our analysis on breaking the curse,
we will take a more direct approach to the problem,
see \secref{sec:HilbertPlainMCUB}.

Let us add one final remark on the type of information.
In this chapter we aim for examples of~$d$-dependent problems
where randomized approximation using information
from arbitrary linear functionals~$\Lall$
can break the curse of dimensionality.
The examples of enhanced speed of convergence were also based on
general information~$\Lall$.
However, randomization can also help in some cases
where only function values~$\Lstd$ are available to the algorithms.
This was shown by Heinrich
in a series of papers~\cite{He08SobI,He09SobII,He09SobIII},
where he studied the randomized approximation of Sobolev embeddings,
and discovered cases of low smoothness where randomization can give
a speedup over deterministic methods.
It is an interesting task for future research
to find examples of function approximation problems based on standard
information~$\Lstd$ where randomization can break the curse of dimensionality,
or significantly improve the $d$-dependency of a problem.

\subsection{Breaking the Curse - a Sequence Space Modell}
\label{sec:l2->linf,curse}

Consider the following example
with a rather artificial\footnote{%
	One could regard~$\ell_2^{2^d}$ as an $L_2$-space
	on a Boolean domain~\mbox{$\{0,1\}^d$}
	equipped with the counting measure~$\#$.}
dimensional parameter~$d \in \N$,
\begin{equation} \label{eq:l2->linf,2^d}
	\App: \ell_2^{2^d} \hookrightarrow \ell_{\infty}^{2^d} \,.
\end{equation}
The initial error is~$1$, hence properly normalized.
By~\eqref{eq:det:l1->linf} we have
\begin{equation*}
	n^{\deter}(\eps, \ell_2^{2^d} \hookrightarrow \ell_{\infty}^{2^d})
		\geq 2^{d-1}
	\quad\text{for\, $0 < \eps \leq {\textstyle \frac{\sqrt{2}}{2}}$,}
\end{equation*}
which clearly is the curse of dimensionality.
Now, in the randomized setting, by~\eqref{eq:ran:lp->lq,2<=p<q} we have
\begin{equation}
	n^{\ran}(\eps, \ell_2^{2^d} \hookrightarrow \ell_{\infty}^{2^d})
		\leq C \, d \, \eps^{-2}
		\quad\text{for\, $\eps > 0$,}
\end{equation}
where~$C>0$ is a numerical constant.
This means that the problem is polynomially tractable for Monte Carlo methods.

We want to discuss briefly what happens for the problem
\begin{equation*}
	\App: \ell_p^{2^d} \hookrightarrow \ell_q^{2^d}
\end{equation*}
in other parameter settings where Monte Carlo methods are known to improve
the error significantly,
that is for \mbox{$1 \leq p < q$} and \mbox{$2 < q \leq \infty$},
see \secref{sec:An:lp->lq}.

In what cases do we have the curse of dimensionality
for the deterministic setting in the first place?\\
It turns out that, if~\mbox{$p = 1$},
then by~\eqref{eq:det:lp->lq,p<2<q,UB}
for~\mbox{$2 < q \leq \infty$} we have
\begin{equation*}
	n^{\deter}(\eps,\ell_1^{2^d} \hookrightarrow \ell_q^{2^d})
		\preceq d \, \eps^{-2} \,,
\end{equation*}
which implies polynomial tractability
in the deterministic setting already.\footnote{%
	By~\eqref{eq:det:l1->lq} we have polynomial tractability
	for~\mbox{$p = 1 < q \leq 2$} as well, yet with a worse
	$\eps$-dependency \mbox{$\eps^{-q/(q-1)}$}.}\\
For~\mbox{$2 \leq p < q \leq \infty$}, in turn,
the problem is more difficult than the problem~\eqref{eq:l2->linf,2^d}
and we obviously inherit the curse of dimensionality.\\
Furthermore, in the case~\mbox{$1 < p < 2 < q \leq \infty$},
by~\eqref{eq:det:lp->lq,p<2<q} we have the estimate
\begin{equation*}
	n^{\deter}(\eps,\ell_p^{2^d} \hookrightarrow \ell_q^{2^d})
		\succeq 2^{(2-2/p)\,d}
	\quad\text{for\, $0 < \eps < \eps_0$,}
\end{equation*}
so in this case
deterministic methods suffer from the curse of dimensionality, too.

Now, what do we know about randomized approximation for~\mbox{$p > 1$}?\\
If the target space is altered
compared to the original example~\eqref{eq:l2->linf,2^d},
i.e.~\mbox{$q < \infty$},
or if the input set\footnote{%
	Recall that the input set is the unit ball~$B_p^{2^d}$ of~$\ell_p^{2^d}$.}
is extended, that is the case for~\mbox{$p > 2$},
then by~\eqref{eq:ran:lp->lq,2<=p<q},
or by~\eqref{eq:ran:lp->lq,p<2<q} and~\eqref{eq:ran:lp->lq,p<2<q,alter},
respectively, we only know the upper bounds
for~\mbox{$0 < \eps < \eps_q$},\footnote{%
	The constant $\eps_q > 0$ is actually the hidden constant
	from the respective error estimates~\eqref{eq:ran:lp->lq,2<=p<q},
	\eqref{eq:ran:lp->lq,p<2<q} and \eqref{eq:ran:lp->lq,p<2<q,alter},
	the hidden constant for the complexity estimates
	is then~$\eps_q^2$, or $\eps_q$
	for~\mbox{$1 < p < 2 < q < \infty$}.}
\begin{equation*}
	n^{\ran}(\eps, \ell_p^{2^d} \hookrightarrow \ell_q^{2^d})
		\preceq 
			\begin{cases}
				2^{(1-2/p+2/q)\,d} \, \eps^{-2}
					\quad&\text{for $2 \leq p < q < \infty$,}\\
				d \, 2^{(1 - 2/p)\,d} \, \eps^{-2}
					\quad&\text{for $2 < p < q = \infty$,}\\
				\min \left\{2^{(1-1/p+1/q) \, d} \, \eps^{-1},
									2^{2d/q} \, \eps^{-2}
						\right\}
					\quad&\text{for $1 < p < 2 < q < \infty$,}
			\end{cases}
\end{equation*}
which is still exponential in~$d$.
We do not know whether the curse of dimensionality actually holds
in the randomized setting
because the lower bounds we obtain from~\eqref{eq:ran:lp->lq,1<=p<q,LB}
will be independent from~$d$.\\
The case left over is \mbox{$1<p<2$} and \mbox{$q = \infty$}.
Here, we can break the curse of dimensionality similarly to the
original example~\eqref{eq:l2->linf,2^d}.
Indeed, the $\ell_p$-ball is contained in the~$\ell_2$ ball,
hence the problem is easier.

To summarize, the~$\ell_{\infty}$ approximation of finite $\ell_p$-sequences
with~\mbox{$1 < p \leq 2$}
is the case where we know that randomization
can break the curse of dimensionality.
The most prominent case is~\mbox{$p = 2$}
where best known Monte Carlo methods are linear,
yet linear methods would suffice to break the curse for~\mbox{$1 < p < 2$}
as well.

This sequence space example motivates the restriction
to the~$L_{\infty}$-approximation of Hilbert space functions
in search of function approximation problems
where the curse of dimensionality holds in the worst case setting
but polynomial tractability can be found in the Monte Carlo setting.

\section{Tools for Function Approximation}
\label{sec:HilbertTools}

In \secref{sec:HilbertPlainMCUB}
we put the fundamental Monte Carlo method from \propref{prop:Ma91_l2G}
to an extreme
and obtain a function approximation analogue
to standard Monte Carlo integration~\eqref{eq:stdMCint}.
Here, we restrict the input set to functions from Hilbert spaces.
\lemref{lem:stdMCapp} stated below is still quite general,
its specification to $L_{\infty}$-approximation of functions
is the starting point for the study of Gaussian random fields
and their expected maximum.
In preparation for this, in \secref{sec:RKHS} we sketch major elements
of the theory of \emph{reproducing kernel Hilbert spaces} (RKHS).
\secref{sec:E|Psi|_sup} then outlines the well established theory
of Gaussian fields associated to a RKHS,
in particular the technique of majorizing measures due to Fernique,
and Dudley's entropy-based estimates.
The theory of RKHSs is also useful for the analysis
of the worst case setting,
for which we need lower bounds
in order to show the superiority of Monte Carlo approximation.
\secref{sec:HilbertWorLB} addresses a general approach
to deterministic \mbox{$L_{\infty}$-approximation} of Hilbert space functions,
that approach has been taken by several authors before~\cite{CKS16,KWW08,OP95}.

\subsection{A Plain Monte Carlo Upper Bound}
\label{sec:HilbertPlainMCUB}

\begin{lemma} \label{lem:stdMCapp}
	Consider the linear problem with a compact solution operator
	\begin{equation*}
		S: \Hilbert \rightarrow G
	\end{equation*}
	from a separable Hilbert space~$\Hilbert$ into a Banach space~$G$,
	the input set~\mbox{$F \subset \Hilbert$} being the unit ball.
	Assume that we have an orthonormal basis~\mbox{$(\psi_j)_{j \in \N}$}
	for~$\Hilbert$ such that
	the sum~\mbox{$\sum_{j=1}^{\infty} X_j \, S(\psi_j)$},
	with independent standard Gaussian random variables~\mbox{$X_j$},
	converges almost surely in~$G$.
	Then for~\mbox{$n \in \N$} we have
	\begin{equation*}
		e^{\ran}(n,S,\Lall)
			\leq \frac{2}{\sqrt{n}}
						 \, \expect \biggl\|\sum_{j=1}^{\infty}
																	X_j \, S(\psi_j)
												\biggr\|_G\,,
	\end{equation*}
	or equivalently, for~\mbox{$\eps > 0$},
	\begin{equation*}
		n^{\ran}(\eps,S,\Lall)
			\leq \left\lceil
							4\, \left(\frac{\expect \bigl\|\sum_{j=1}^{\infty}
																	X_j \, S(\psi_j)
												\bigr\|_G
										}{\eps}
								\right)^2
					\right\rceil \,.
	\end{equation*}
\end{lemma}
\begin{proof}
	For~\mbox{$m \in \N$} we define the linear Monte Carlo method
	\mbox{$A_{n,m} = (A_{n,m}^{\omega})_{\omega}$}
	which, for an input \mbox{$f \in \Hilbert$},
	returns the output
	\begin{equation*}
		g = A_{n,m}^{\omega}(f)
			:= \frac{1}{n} \, \sum_{i=1}^n
						L_{i,m}^{\omega}(f) \, g_{i,m}^{\omega} \,,
	\end{equation*}
	based on the information
	\begin{equation*}
		y_i = L_{i,m}^{\omega}(f)
			=  \sum_{j=1}^m X_{ij} \, \langle \psi_j, f \rangle_{\Hilbert} \,,
	\end{equation*}
	and with elements from the output space
	\begin{equation*}
		g_{i,m}^{\omega} := \sum_{j=1}^m X_{ij} \, S(\psi_j) \,.
	\end{equation*}
	Here, the~$X_{ij}$ are independent standard Gaussian random variables.
	This algorithm is actually the fundamental Monte Carlo method
	from \propref{prop:Ma91_l2G}
	when restricting~$S$
	to	the subspace~\mbox{$\Hilbert_m := \linspan\{\psi_1,\ldots,\psi_m\}$}
	that can be identified with~$\ell_2^m$.
	Let~$P_m$ denote the orthogonal projection onto~$\Hilbert_m$.
	Since~$S$ is compact, we have
	\begin{equation*}
		\|S \, (\id_{\Hilbert} - P_m) \|_{\Hilbert \rightarrow G}
			\xrightarrow[m \rightarrow \infty]{} 0 \,.
	\end{equation*}
	
	Then, for elements~$f$ from the input set,
	\mbox{$\|f\|_{\Hilbert} \leq 1$}, we have
	\begin{align*}
		e((A_{n,m}^{\omega})_{\omega},f)
			&\leq \|S \, (\id_{\Hilbert} - P_m) \, f \|_G
						+ e((A_{n,m}^{\omega})_{\omega}, P_m f) \\
			&\leq \|S(\id_{\Hilbert} - P_m) \|_{\Hilbert \rightarrow G}
				+ \frac{2}{\sqrt{n}}
						\, \expect\biggl\|\sum_{j=1}^m
																X_j \, S(\psi_j)
											\biggr\|_G\\
			&\xrightarrow[m \rightarrow \infty]{}
				\frac{2}{\sqrt{n}}
					\, \expect\biggl\|\sum_{j=1}^{\infty}
																X_j \, S(\psi_j)
										\biggr\|_G\,.
	\end{align*}
	Note that~\mbox{$\expect\left\|\sum_{j=1}^m X_j \, S(\psi_j) \right\|_G$}
	is monotonously increasing for~\mbox{$m \rightarrow \infty$}
	since the~$X_j$ are independent, see~\lemref{lem:E|Y|<E|Y+Z|}.
\end{proof}

\begin{remark}[Stochastically bounded information and algorithms]
	\label{rem:Lstoch}
	With the above lemma it seems natural to consider
	the idealized method~\mbox{$A_n = A_{n,\infty}$},
	\begin{equation*}
		A_n(f)
			:= \frac{1}{n} \, \sum_{i=1}^n L_i^{\omega}(f) \, g_i^{\omega}  \,,
	\end{equation*}
	with information
	\begin{equation*}
		y_i = L_i^{\omega}(f)
			:= \sum_{j=1}^{\infty} X_{ij} \, \langle \psi_j,f \rangle_{\Hilbert} \,,
	\end{equation*}
	and elements from the output space 
	\begin{equation*}
		g_i^{\omega} := \sum_{j=1}^{\infty} X_{ij} \, S(\psi_j) \,.
	\end{equation*}
	Observe the similarities
	with standard Monte Carlo integration~\eqref{eq:stdMCint}.
	Observe also the important difference that the present approximation method
	depends on the particular norm of the input space,
	whereas standard Monte Carlo integration is defined
	independently from the input set.
	
	Note that, almost surely, $L_i^{\omega}$~is an unbounded functional,
	so~\mbox{$L_i^{\omega} \notin \Lall$}.
	To see this,
	for fixed~$\omega$, consider the sequence~\mbox{$(f_{ik})_{k=1}^{\infty}$}
	of normalized Hilbert space elements
	\begin{equation*}
		f_{ik} := \frac{1}{\sqrt{\sum_{j=1}^k X_{ij}^2}}
						\, \sum_{j=1}^k X_{ij} \, S(\psi_j)
			\in \Hilbert \,,
	\end{equation*}
	where
	\begin{equation*}
		L_i^{\omega}(f_{ik}) = \sqrt{\sum_{j=1}^k X_{ij}^2}
			\xrightarrow[k \rightarrow \infty]{\text{a.s.}} \infty \,.
	\end{equation*}
	Indeed,
	from \lemref{lem:gaussqnormvector}
	we have~\mbox{$\expect L_i^{\omega}(f_{ik}) \geq \sqrt{2k/\pi}$}.
	The deviation result \propref{prop:dev2} can be applied,
	similarly to \corref{cor:deviationGauss},
	to bound the probability~\mbox{$\P(L_i^{\omega}(f_{ik}) < a)$}
	for~\mbox{$a > 0$}.
	The Borel-Cantelli lemma implies
	that the monotone sequence \mbox{$(L_i^{\omega}(f_{ik}))_{k=1}^{\infty}$}
	almost surely exceeds any~\mbox{$a > 0$} for sufficiently large~$k$.
	
	Specifically for embedding problems~%
	\mbox{$S = \App: \Hilbert \hookrightarrow G$},
	the functions~$g_i^{\omega}$, and therefore the output as well,
	are functions defined on the same domain as the input functions,
	but they are \emph{not} from the original Hilbert space~$\Hilbert$
	(for the same reasons
	that caused the functionals~\mbox{$L_i^{\omega}$} to be discontinuous).
	This underlines the non-interpolatory nature
	of the fundamental Monte Carlo approximation method,
	compare \remref{rem:FundMC}.
	Yet the functions~$g_i^{\omega}$
	correspond to the information functionals~$L_i^{\omega}$,
	similar to Hilbert space elements representing continuous linear functionals
	according to the Riesz representation theorem.
	
	Although the functionals~$L_i^{\omega}$ are almost surely discontinuous,
	for any fixed input~\mbox{$f \in \Hilbert$}
	the random information~\mbox{$L_i^{\omega}(f)$}
	is a standard Gaussian random variable with variance~$\|f\|_{\Hilbert}^2$,
	hence almost surely finite.
	Since, by assumption, the $g_i^{\omega}$~are almost surely defined,
	we have a method that is almost surely defined for any fixed~$f$.
	Even more, for any fixed~\mbox{$f \in \Hilbert$},
	the idealized algorithm~$A_n$
	can be approximated with almost sure convergence,
	\begin{equation*}
		A_{n,m}^{\omega}(f)
			\xrightarrow[m \rightarrow \infty]{\text{a.s.}}
				A_n^{\omega}(f) \,.
	\end{equation*}
	
	These considerations motivate to extend the class
	of admissible information functionals~$\Lall$
	to some class of ``stochastically bounded'' functionals~$\Lambda^{\rm stoch}$.
	Actually, this kind of stochastically defined functionals is quite common.
	For example, the problem of integrating $L_p$-functions by function values
	is only solvable in the randomized setting
	since in that case standard information~$\Lstd$ is discontinuous.
	Compare also the example from Heinrich and Milla~\cite{HeM11}
	which has been discussed in \secref{sec:measurable}.
\end{remark}

\subsection{Reproducing Kernel Hilbert Spaces}
\label{sec:RKHS}

We summarize several facts about \emph{reproducing kernel Hilbert spaces} (RKHS)
that are necessary for the numerical analysis of approximation problems
\begin{equation*}
	\App: \Hilbert \hookrightarrow L_{\infty}(D) \,,
\end{equation*}
with~\mbox{$\Hilbert$} being a separable Hilbert space
defined on a domain~\mbox{$D \subset \R^d$}.
For a general introduction to reproducing kernels,
refer to Aronszajn~\cite{Ar50}.
For an introduction with focus on the associated Gaussian field,
see Adler~\cite[Sec~III.2]{Adl90}. 
The theory of RKHSs is a powerful concept
for the analysis of many other numerical settings,
e.g.~for certain average case problems
(see for instance Ritter~\cite[Chap~III]{Rit00}),
or when standard information~$\Lstd$ is considered
(see Novak and Wo\'zniakowski~\cite{NW10,NW12} for a bunch of examples),
it also proves useful for statistical problems (see Wahba~\cite{Wah90}).

\subsubsection{Definition of Reproducing Kernels}

We assume function evaluations to be continuous on~\mbox{$\Hilbert$}.
Then by the Riesz representation theorem, for each~\mbox{$\vecx \in D$}
there exists a unique function~\mbox{$K_{\vecx}(\cdot) \in \Hilbert$}
such that for~\mbox{$f \in \Hilbert$} we have
\begin{equation*}
	f(\vecx) = \langle K_{\vecx},f \rangle_{\Hilbert} \,.
\end{equation*}
For~\mbox{$\vecx,\vecy \in D$} we define a \emph{symmetric} function
\begin{equation*}
	K(\vecx,\vecy)
		:= K_{\vecy}(\vecx)
		= \langle K_{\vecx},K_{\vecy} \rangle_{\Hilbert}
		= \langle K_{\vecy},K_{\vecx} \rangle_{\Hilbert}
		= K_{\vecx}(\vecy) = K(\vecy,\vecx) \,.
\end{equation*}
This function is called the \emph{reproducing kernel} of~\mbox{$\Hilbert$}.
The reproducing kernel~$K$ is \emph{positive-semidefinite}, that is,
for~\mbox{$\vecx_1,\ldots,\vecx_m \in D$}
and~\mbox{$a_1,\ldots,a_m \in \R$} we have
\begin{equation*}
	\sum_{i,j=1}^m a_i \, a_j \, K(\vecx_i,\vecx_j)
		= \left\langle \sum_{i=1}^m a_i \, K_{\vecx_i},
									\sum_{i=1}^m a_i\, K_{\vecx_i}
			\right\rangle_{\Hilbert}
		\geq 0 \,.
\end{equation*}
Reversely, any symmetric and positive-semidefinite
function~\mbox{$K: D \times D \rightarrow \R$}
defines an inner product on the linear space
\mbox{$\linspan\{K_{\vecx} := K(\vecx,\cdot):D \rightarrow \R
										\mid \vecx \in D\}$}
by
\begin{equation*}
	\left\langle \sum_{i=1}^m a_i \, K_{\vecx_i},
									\sum_{j=1}^n b_j \, K_{\vecz_j}
	\right\rangle_{K}
		:= \sum_{i=1}^m \sum_{j=1}^n a_i \, b_j \, K(\vecx_i,\vecz_j) \,,
\end{equation*}
for points~\mbox{$\vecx_i,\vecz_j \in D$} and \mbox{$a_i,b_j \in \R$}.
Its completion with respect to the corresponding norm
uniquely defines a Hilbert space~\mbox{$\Hilbert(K)$} 
which is then called the
\emph{reproducing kernel Hilbert space} with kernel~$K$.
This is the space~$\Hilbert$ we started with.

\subsubsection{Comparison to the $\sup$-Norm}

Knowing the kernel, it is easy to estimate the $\sup$-norm
of normalized functions~\mbox{$f \in \Hilbert$}.
Indeed, with \mbox{$\|f\|_{\Hilbert} = 1$} we get
\begin{multline} \label{eq:|f|sup,K}
	\|f\|_{\sup}
		= \sup_{\vecx \in D} |f(\vecx)|
		= \sup_{\vecx \in D} \langle K_{\vecx}, f \rangle_{\Hilbert}\\
		\leq \sup_{\vecx \in D} \|K_{\vecx}\|_{\Hilbert}
		= \sup_{\vecx \in D} \sqrt{\langle K_{\vecx}, K_{\vecx} \rangle_{\Hilbert}}
		= \sup_{\vecx \in D} \sqrt{K(\vecx,\vecx)} \,.
\end{multline}
This is the initial error for the $\sup$-norm approximation.
By the Cauchy-Schwarz inequality, the $\sup$-norm of $K$
is determined by values~\mbox{$K(\vecx,\vecx)$}
on the diagonal of~\mbox{$D \times D$},
\begin{equation*}
	\sup_{\vecx,\vecz \in D} |K(\vecx,\vecz)|
		\leq \sup_{\vecx \in D} K(\vecx,\vecx) \,.
\end{equation*}
Therefore from now on we assume~$K$ to be bounded.

\subsubsection{Decomposition of Reproducing Kernels
	and a Worst Case Upper Bound}

Let~\mbox{$(\psi_i)_{i \in \N}$} be an orthonormal basis
of~\mbox{$\Hilbert(K)$}, then we can write
\begin{equation} \label{eq:k=sumpsi}
	K(\vecx,\vecz)
		= \sum_{i = 1}^{\infty} \psi_i(\vecx) \, \psi_i(\vecz) \,.
\end{equation}
That way it is easy to determine the reproducing kernel
of a subspace~\mbox{$\Hilbert' \subset \Hilbert$}
spanned by~\mbox{$\{\psi_i\}_{i=n+1}^{\infty}$},
it is
\begin{equation} \label{eq:K'}
	K'(\vecx,\vecz)
		:= \sum_{i = n + 1}^{\infty} \psi_i(\vecx) \, \psi_i(\vecz)
		= K(\vecx,\vecz) - \sum_{i=1}^n \psi_i(\vecx) \, \psi_i(\vecz)\,.
\end{equation}
Therefore, via the linear algorithm
\begin{equation*}
	A_n(f) := \sum_{j=1}^n \langle \psi_j, f \rangle_{\Hilbert} \, \psi_j \,,
\end{equation*}
analogously to the initial error~\eqref{eq:|f|sup,K},
we can estimate the worst case error,
\begin{equation} \label{eq:RKHSworUB}
	e^{\deter,\lin}(n,\Hilbert(K) \hookrightarrow L_{\infty}(D),\Lall)
		\leq \sup_{\vecx \in D}
						\sqrt{K(\vecx,\vecx)
									- \sum_{i=1}^n \psi_i(\vecx) \, \psi_i(\vecz)} \,.
\end{equation}
Actually, the optimal error can be achieved that way,
see \secref{sec:HilbertWorLB}
for optimality and lower bounds.

\subsubsection{Tensor Product of Reproducing Kernel Hilbert Spaces}

Tensor products are a common way in IBC to define multivariate problems,
compare for instance Novak and Wo\'zniakowski~\cite[Sec~5.2]{NW08},
or Ritter~\cite[Sec~VI.2]{Rit00},
find many more examples in Novak and Wo\'zniakowski~\cite{NW10,NW12}.

Let~\mbox{$\Hilbert(K_1)$} and \mbox{$\Hilbert(K_2)$}
be reproducing kernel Hilbert spaces
defined on~$D_1$ and $D_2$, respectively.
Let~\mbox{$(\varphi_i)_{i \in \N}$} and~\mbox{$(\psi_j)_{j \in \N}$}
be corresponding orthonormal bases.
Then the tensor product space~%
\mbox{$\Hilbert(K_1) \otimes \Hilbert(K_2)$}
is the Hilbert space with orthonormal basis~%
\mbox{$(\varphi_i \otimes \psi_j)_{i,j \in \N}$}.
Here, \mbox{$f_1 \otimes f_2$} denotes the tensor product of functions~%
\mbox{$f_1 \in \Hilbert(K_1)$} and \mbox{$f_2 \in \Hilbert(K_2)$},
\begin{equation*}
	 [f_1 \otimes f_2](\vecx_1,\vecx_2)
			:= f_1(\vecx_1) \, f_2(\vecx_2)\,,
		\quad\text{defined for $(\vecx_1,\vecx_2) \in D_1 \times D_2$.}
\end{equation*}
With another tensor product function~\mbox{$g_1 \otimes g_2$} of this sort,
one easily obtains
\begin{equation*}
	\langle f_1 \otimes f_2,
					g_1 \otimes g_2
	\rangle_{\Hilbert1 \otimes \Hilbert_2}
		= \langle f_1,g_1 \rangle_{\Hilbert_1}
			\, \langle f_2,g_2 \rangle_{\Hilbert_2} \,.
\end{equation*}
Using the representation~\eqref{eq:k=sumpsi} for the reproducing kernel,
it is easy to see that the reproducing kernel~$K$ of the new space
is the tensor product of the kernels of the original spaces,
\begin{equation*}
	K((\vecx_1,\vecx_2),(\vecz_1,\vecz_2))
		:= K_1(\vecx_1,\vecz_1) \, K_2(\vecx_2,\vecz_2) \,,
\end{equation*}
where~\mbox{$(\vecx_1,\vecx_2),(\vecz_1,\vecz_2) \in D_1 \times D_2$}.

\subsubsection{Canonical Metric and Continuity}

We consider the \emph{canonical metric}\footnote{%
	If~$d_K(\vecx,\vecz) = 0$ for some distinct~{$\vecx \not= \vecz$},
	we only have a semimetric. Then we still obtain
	a metric for the set of equivalence classes
	of points that are at distance~$0$.}~%
\mbox{$d_K: D \times D \rightarrow [0,\infty)$},
\begin{equation*}
	d_K(\vecx,\vecz)	
		:= \|K_{\vecx} - K_{\vecz}\|_{\Hilbert}
		= \sqrt{K(\vecx,\vecx) - 2 \, K(\vecx,\vecz) + K(\vecz,\vecz)} \,.
\end{equation*}
Functions~\mbox{$f \in \Hilbert$} are Lipschitz continuous
with Lipschitz constant~\mbox{$\|f\|_{\Hilbert}$}
with respect to the canonical metric,
\begin{equation*}
	|f(\vecx) - f(\vecz)|
		= |\langle K_{\vecx} - K_{\vecz}, f \rangle_{\Hilbert}|
		\leq \|K_{\vecx} - K_{\vecz} \|_{\Hilbert} \, \|f\|_{\Hilbert}
		= \|f\|_{\Hilbert} \, d_K(\vecx,\vecz) \,.
\end{equation*}
Hence functions from~$\Hilbert$
are continuous with respect to any metric~$\delta$ on~$D$
that is topologically equivalent to the canonical metric~$d_K$.

Since we assume~$K$ to be bounded,
the domain~$D$ is bounded with respect to~$d_K$,
\begin{equation*}
	\diam(D) = \sup_{\vecx,\vecz \in D} d_K(\vecx,\vecz)
		\leq 2 \, \sup_{\vecx \in D} \sqrt{K(\vecx,\vecx)}\,.
\end{equation*}
Towards the end of the next section
on the boundedness (and continuity) of associated Gaussian fields,
we will need the stronger assumption
that~$D$ is \emph{totally bounded} with respect to~$d_K$.
That is, for any~\mbox{$r > 0$}
the set~$D$ can be covered by finitely many balls with radius~$r$.
In particular, if~$D$ is complete with respect to~$d_K$,
this implies compactness of~$D$.
Recall that compactness of a set in a metric space implies total boundedness.


\subsection{Expected Maximum of Zero-Mean Gaussian Fields}
\label{sec:E|Psi|_sup}

We discuss zero-mean Gaussian fields and their connection to
reproducing kernel Hilbert spaces.
All results presented here can be found in the notes by Adler~\cite{Adl90}.
For some results, Lifshits~\cite{Lif95} or Ledoux and Talagrand~\cite{LT91}
will also be good references.

\begin{definition} \label{def:Gauss-H(K)}
	The Gaussian field associated with a
	reproducing kernel Hilbert space~\mbox{$\Hilbert(K)$}
	is a random function~\mbox{$\Psi: D \rightarrow \R \cup \{\pm\infty\}$}
	such that, for any finite collection of
	points~\mbox{$\vecx_1,\ldots,\vecx_m \in D$},
	the vector~\mbox{$(\Psi_{\vecx_i})_{i = 1}^m$}
	is distributed according to a zero-mean Gaussian distribution in~$\R^m$,
	and
	\begin{equation*}
		\Cov(\Psi_{\vecx},\Psi_{\vecz}) = K(\vecx,\vecz) \,.
	\end{equation*}
\end{definition}

\subsubsection{Series Representation}

Let \mbox{$(\psi_i)_{i = 1}^M$} be an orthonormal basis
of~\mbox{$\Hilbert(K)$}, \mbox{$M \in \N \cup \{\infty\}$}.\footnote{%
	The interesting case of course is~$M = \infty$.
	For~$M < \infty$ almost sure boundedness is obvious,
	still, good upper bounds for the expected maximum are of interest.}
Then the pointwise definition
\begin{equation} \label{eq:Psi=sum(X*psi)}
	\Psi_{\vecx} := \sum_{i=1}^M X_i \, \psi_i(\vecx) \,,
\end{equation}
with $X_i$ being iid standard Gaussian random variables,
produces a version of the Gaussian field associated with~\mbox{$\Hilbert(K)$}.
Indeed, the covariance function of~$\Psi$ defined that way is the kernel~$K$,
\begin{equation*}
	\Cov(\Psi_{\vecx},\Psi_{\vecz})
		= \sum_{i,j=1}^{\infty}
				(\expect X_i \, X_j)
					\, \psi_i(\vecx) \, \psi_j(\vecz)
		= \sum_{i = 1}^{\infty}
				\psi_i(\vecx) \, \psi_i(\vecz)
		= K(\vecx,\vecz) \,.
\end{equation*}

\subsubsection{Continuity}

Note that for the \emph{canonical metric}~\mbox{$d_K$}
of the reproducing kernel Hilbert space~\mbox{$\Hilbert(K)$}
we have the alternative representation
\begin{equation*}
	d_K(\vecx,\vecz)
		= \sqrt{\expect (\Psi_{\vecx} - \Psi_{\vecz})^2} \,.
\end{equation*}
The question of the boundedness of~$\Psi$ is closely related
to continuity with respect to the canonical metric~\mbox{$d_K$}.
We say that the Gaussian field with covariance function~$K$ is \emph{continuous},
if there exists a version of~$\Psi$ with almost surely continuous sample paths.

The classical approach for continuity
starts with a countable dense subset~\mbox{$T \subseteq D$}
(we therefore assume~$D$ to be separable with respect to~$d_K$):\\
If~\mbox{$\Psi$} is continuous on~$T$ with respect to the canonical metric~$d_K$,
it can be uniquely extended to a continuous function on~$D$ by the limit
\begin{equation} \label{eq:Psi,C-Version}
	\Psi_{\vecx} := \lim_{T \ni \vecz \rightarrow \vecx} \Psi_{\vecz} \,,
\end{equation}
see Lifshits~\cite[Sec~15]{Lif95}.
Working with a countable subset~$T$ enables us to determine the probability
of~$\Psi$ being continuous on~$T$.
This probability is either~$0$ or~$1$, see e.g.\ Adler~\cite[Thm~3.12]{Adl90}.
If~$\Psi$ is continuous on~$T$ with probability~$1$,
indeed, \eqref{eq:Psi,C-Version}~defines an almost surely continuous version
of the Gaussian field associated with~\mbox{$\Hilbert(K)$} according
to \defref{def:Gauss-H(K)}.

If the Gaussian field associated with~\mbox{$\Hilbert(K)$}
has continuous sample paths and the domain~$D$ is totally bounded,
one can show that the series representation~\eqref{eq:Psi=sum(X*psi)}
converges uniformly\footnote{%
	That is, the series converges in the $L_{\infty}$-norm.}
on~$D$ with probability~$1$, see Adler~\cite[Thm~3.8]{Adl90}.
In the sequel, when talking about~$\Psi$, we always mean the
series representation, for which uniform convergence implies continuity
with respect to~$d_K$.

\subsubsection{Boundedness}

From now on we assume that the domain~$D$ is totally bounded.\footnote{%
	If $D$ is not totally bounded,
	the estimate in \propref{prop:Dudley} (Dudley) will be infinite.
	Actually, total boundedness of~$D$ is a necessary condition
	for boundedness of the process~$\Psi$.}
In this case, for Gaussian fields,
continuity with respect to the canonical metric~$d_K$
is equivalent to boundedness, see Adler~\cite[Thm~4.16]{Adl90}.

Let~\mbox{$B_K(\vecx,r)$}
denote the closed $d_K$-ball around~\mbox{$\vecx \in D$}
with radius~\mbox{$r > 0$}.

The following result can be found in Adler~\cite[Thm~4.1]{Adl90},
it is originally due to Fernique 1975~\cite{Fer75}.
\begin{proposition}[Fernique] \label{prop:Fernique}
	Let~$\mu$ be any probability measure on~$D$, then
	\begin{equation*}
		\expect \sup_{\vecx \in D} \Psi_{\vecx}
			\leq C_{\text{Fernique}} \, \sup_{\vecx \in D}
									\int_0^{\infty} \sqrt{\log(1/\mu(B_{K}(\vecx,r)))} \rd r \,,
	\end{equation*}
	where~\mbox{$C_{\text{Fernique}} > 0$} is a universal constant.
\end{proposition}
From Adler~\cite[Sec~IV.2]{Adl90} one can extract a value
\mbox{$C_{\text{Fernique}}
				\leq 4 \, \sqrt{3} \, (1/\sqrt{\pi \log 2} + 16)$}
\mbox{$= 115.5462...$}.
This constant is not optimal.
From the book of Ledoux and Talagrand~\cite[Prop~11.12]{LT91}
(via the Young function~\mbox{$\psi(x) = \exp(x^2)-1$})
we gain the much better estimate
\mbox{$C_{\text{Fernique}}
				\leq 8 \, (2 + 1/\sqrt{2})
				= 21.6568...$}.

A measure for which the right hand side of the above proposition is finite,
is called a \emph{majorizing measure}
for the metric space~\mbox{$(D,d_K)$}.
Majorizing measures must be -- in a certain way -- ``well spread''
because the integral will be infinite if~\mbox{$\mu(B_K(\vecx,r)) = 0$}
for some~\mbox{$\vecx \in D$} and \mbox{$r > 0$}.
Yet it may be discrete,
see the construction for the proof of \propref{prop:Dudley} below.
Furthermore, the integral vanishes for~$r$ exceeding the diameter of~$D$ with
respect to~$d_K$.

Sometimes it is inconvenient to work with majorizing measures.
An alternative way of estimating the maximum of a Gaussian field is based
on \emph{metric entropy}.
For~\mbox{$r > 0$}, let~\mbox{$N(r) = N(r,D,d_K)$}
denote the minimal number of $d_K$-balls with radius~$r$
needed to cover~$D$.
The function~\mbox{$H(r) := \log N(r)$}
is called the \emph{(metric) entropy} of~$D$.
The following inequality is based on this quantity,
it goes back to Dudley 1973~\cite[Thm~2.1]{Dud73}.
\begin{proposition}[Dudley] \label{prop:Dudley}
	There exists a universal constant~\mbox{$C_{\text{Dudley}} > 0$} such that
	\begin{equation*}
		\expect \sup_{\vecx \in D} \Psi_{\vecx}
			\leq C_{\text{Dudley}}
							\, \int_0^{\infty}
											\sqrt{\log N(r)}
										\rd r \,.
	\end{equation*}
\end{proposition}
For a direct proof with explicit numerical bound~%
\mbox{$C_{\text{Dudley}} \leq 4 \sqrt{2}$},
see Lifshits~\cite[Sec~14, Thm~1]{Lif95}.
In the book of Adler~\cite[Cor~4.15]{Adl90}
Dudley's inequality was derived from Fernique's estimate.
\begin{proof}[Idea of a derivation from Fernique's estimate.]
	By scaling, without loss of generality, the diameter of~$D$ is~$1$.
	For~\mbox{$k \in \N_0$},
	let~\mbox{$\{\vecx_{k,1},\ldots,\vecx_{k,N(2^{-k})}\} \subseteq D$}
	be a minimal collection of points
	such that $D$~is covered by balls of radius~\mbox{$r = 2^{-k}$},
	\begin{equation*}
		D = \bigcup_{j=1}^{N(2^{-k})} B_K(\vecx_{k,j},2^{-k}) \,.
	\end{equation*}
	Then, defining
	\begin{equation*}
		\mu(E) := \frac{1}{2} \sum_{k=0}^{\infty} 2^{-k} \,
								\left[\frac{1}{N(2^{-k})} \sum_{j=1}^{N(2^{-k})}
												\ind[\vecx_{k,j} \in E]
								\right]
	\end{equation*}
	for~$E \subseteq D$, we obtain a majorizing measure
	and may apply \propref{prop:Fernique} (Fernique),
	see Adler~\cite[Lem~4.14]{Adl90} for more details.\footnote{%
		In Adler~\cite[Lem~4.14]{Adl90}
		the construction of the measure is less explicit.
		Check the proof with the construction given here.}
\end{proof}

Since we are interested in the expected $\sup$-norm of~$\Psi$,
we also need the following easy lemma,
compare Adler~\cite[Lem~3.1]{Adl90}.
\begin{lemma} \label{lem:E|X|<init+EsupX}
	For the Gaussian field~$\Psi$ with covariance function~$K$,
	we have
	\begin{equation*}
		\expect \|\Psi\|_{\infty}
			\leq \sqrt{\frac{2}{\pi}}
							\, \inf_{\vecx \in D} \sqrt{K(\vecx,\vecx)}
						+ 2 \, \expect \sup_{\vecx \in D} \Psi_{\vecx} \,.
	\end{equation*}
\end{lemma}
\begin{proof}
	With~$\vecx_0 \in D$, by the triangle inequality we obtain
	\begin{equation*}
		\expect \|\Psi\|_{\infty}
			\leq \expect |\Psi_{\vecx_0}| + \expect\|\Psi - \Psi_{\vecx_0}\|_{\infty} \,.
	\end{equation*}
	Since $\Psi_{\vecx_0}$ is a zero-mean Gaussian random variable
	with variance~\mbox{$K(\vecx_0,\vecx_0)$},
	we get
	\begin{equation*}
		\expect |\Psi_{\vecx_0}|
									= \sqrt{\frac{2}{\pi}} \, \sqrt{K(\vecx_0,\vecx_0)} \,.
	\end{equation*}
	For the random field~\mbox{$\Phi_{\vecx} := \Psi_{\vecx} - \Psi_{\vecx_0}$},
	we have~\mbox{$\Phi_{\vecx_0} = 0$}, so by symmetry
	\begin{equation*}
		\expect\|\Phi_\vecx\|_{\infty}
			= \expect\max\{\sup_{\vecx \in D} \Phi_{\vecx},
							\sup_{\vecx \in D} (-\Phi_{\vecx})
						\}
			\leq \expect \sup_{\vecx \in D} \Phi_{\vecx}
						+ \expect \sup_{\vecx \in D} (-\Phi_{\vecx})
			= 2 \, \expect \sup_{\vecx \in D} \Phi_{\vecx} \,.
	\end{equation*}
	Finally, observe that
	\begin{equation*}
		\expect \sup_{\vecx \in D} \Phi_{\vecx}
			= \expect \sup_{\vecx \in D} \Psi_{\vecx} \,.
	\end{equation*}
	The lemma is obtained taking the infimum over~\mbox{$\vecx_0 \in D$}.
\end{proof}

\subsection{A Lower Bound in the Worst Case Setting}
\label{sec:HilbertWorLB}

As mentioned in \secref{sec:ranApp:e(n)},
commonly used discretization techniques
are not feasible for tractability analysis.
Osipenko and Parfenov 1995~\cite{OP95},
Kuo, Wasilkowski, and Wo\'zniakowski 2008~\cite{KWW08},
and Cobos, K\"uhn, and Sickel 2016~\cite{CKS16},
independently from each other found similar approaches
to relate the error of~$L_{\infty}$-approximation
to $L_2$-approximation.
The formulation of \propref{prop:H->L_inf/L_2} follows~\cite{CKS16,OP95},
giving a lower bound for the $L_{\infty}$-approximation
in terms of singular values of some~\mbox{$L_2(\rho)$}-approximation.
Parts of the proof in the original papers are based on the theory
of absolutely summing operators.
For this thesis, however,
a proof that is based on tools from IBC appears more natural.
Namely, in Kuo et al.~\cite{KWW08} the worst case~$L_{\infty}$ error
was compared to the average \mbox{$L_2(\rho)$}~error
with respect to the Gaussian field
associated with the reproducing kernel Hilbert
space~\mbox{$\Hilbert = \Hilbert(K)$}.

We start with a well known fact on the optimality of linear algorithms
for the approximation of Hilbert space functions.
The proofs are given for completeness.
\begin{lemma} \label{lem:H->,lin=opt,singular}
	Consider a linear problem~$S: \Hilbert \rightarrow G$
	with the input set~$F$ being the unit ball
	of a Hilbert space~$\Hilbert$.
	\begin{enumerate}[(a)]
		\item Linear algorithms are optimal, more precisely,
			optimal algorithms have the structure~\mbox{$A_n = S P$}
			where~$P$ is an orthogonal rank-$n$ projection on~$\Hilbert$.
			Hence we can write
			\begin{equation*}
				e^{\deter}(n,S,\Lall)
					= \inf_{\substack{\text{$P$ Proj.}\\
														\rank P = n}}
							\|S \, (\id_{\Hilbert} - P) \|_{\Hilbert \rightarrow G} \,.
			\end{equation*}
		\item \label{lemenum:H->H,singular}
			If~$G = \Hilbert_2$ is another Hilbert space and~$S$ is compact,
			we have a singular value decomposition.
			That is, there is an orthonormal system~%
			\mbox{$(\psi_k)_{k = 1}^M$} in~$\Hilbert$
			such that
			\mbox{$(S \psi_k)_{k =1}^M$} is orthogonal in~\mbox{$G = \Hilbert_2$},
			and $S$ can be written
			\begin{equation*}
				S f = \sum_{k=1}^{M}
								\langle \psi_k , f \rangle_{\Hilbert}
									\, S \psi_k \,,
			\end{equation*}
			where~\mbox{$M \in \N \cup \{\infty\}$}.
			Furthermore, the sequence~\mbox{$(\sigma_k)_{k = 1}^{\infty}$}
			of the singular values
			\begin{equation*}
				\sigma_k := \begin{cases}
											\|S \psi_k\|_G > 0
												\quad&\text{for $k < M+1$,}\\
											0 \quad&\text{for $k > M$,}
										\end{cases}
			\end{equation*}
			is ordered, \mbox{$\sigma_1 \geq \sigma_2 \geq \ldots \geq 0$}.
			Then%
			\footnote{%
				Following the axiomatic scheme of Pietsch~\cite{Pie74},
				singular values of linear operators between Hilbert spaces
				are a special case of $s$-numbers, \mbox{$s_n(S) := \sigma_n$}.
				All kinds of $s$-numbers
				-- which may differ for operators between arbitrary Banach spaces --
				coincide with the singular values in the Hilbert space setting.
				The correspondence to the error of deterministic methods
				exhibits the usual index shift we encounter
				when relating quantities from IBC to $s$-numbers,
				see also the discussion towards the end of \secref{sec:measurable},
				and the definition of Bernstein numbers in \secref{sec:BernsteinSetting}.
				}
			\begin{equation*}
				e^{\deter}(n,S,\Lall) = \sigma_{n+1} \,.
			\end{equation*}
	\end{enumerate}
\end{lemma}
\begin{proof}
	(a)
	Let~\mbox{$N : \Hilbert \rightarrow \R^n$} be any deterministic
	information mapping using adaptively chosen functionals~$L_{k,\vecy_{[k-1]}}$.
	We define the non-adaptive information
	\begin{equation*}
		N^{\non}(f)
			:= (L_1(f),L_{2,0}(f),L_{3,\zeros_{[2]}}(f),\ldots,L_{n,\zeros_{[n-1]}}(f)) \,,
	\end{equation*}
	which uses the functionals that~$N$ would choose in the case
	of~\mbox{$N(f) = \zeros$}.
	Let~$P$ be the orthogonal projection onto~\mbox{$\ker(N^{\non})^{\bot}$},
	i.e.\ \mbox{$\id_{\Hilbert} - P$}~is the orthogonal projection
	onto~\mbox{$\ker(N^{\non})$},
	and consider the linear rank-$n$ algorithm~\mbox{$A_n := S P$}
	(it is indeed based on the information~$N^{\non}$
	since \mbox{$\ker N^{\non} = \ker SP$}).
	The error of this algorithm is
	\begin{multline*}
		e(A_n,S)
			\,=\, \sup_{f \in F} \|S\, (\id_{\Hilbert} - P) \, f\|_G
			\,=\, \sup_{\substack{f \in \image(\id_{\Hilbert} - P) \\
														\|f\|_{\Hilbert} \leq 1}}
							\|S \, f \|_G
			\,=\, \sup_{\substack{f \in F \\
														N(f) = \zeros}}
							\|S \, f \|_G \\
			\,\leq\, \inf_{\phi} \sup_{\substack{f \in F \\
														N(f) = \zeros}}
							\|S \, f - [\phi \circ N](f)\|_G
			\,\leq\, \inf_{\phi} e(\phi \circ N, S)\,.
	\end{multline*}
	This shows that the error of~$A_n$ is maximal for zero information,
	a case which also occurs for the adaptive information mapping.
	
	(b)
	The singular value decomposition is a standard result from spectral theory.\\
	For~$n \geq M$, the statement is trivial
	since~$S$ itself, with rank~\mbox{$\rank S \leq n$},
	can be seen as a suitable algorithm.
	Now, for~\mbox{$n < M$}, take the rank-$n$ algorithm
	\begin{equation*}
		A_n(f) = \sum_{k=1}^n \langle \psi_k , f \rangle_{\Hilbert} \psi_k \,,
	\end{equation*}
	with the error
	\begin{align*}
		e(A_n,f)
			&\,=\, \Bigl\|\sum_{k=n+1}^{M}
											\langle \psi_k , f \rangle_{\Hilbert} \, \psi_k
						\Bigr\|_G
			\,=\, \sqrt{\sum_{k=n+1}^{M}
										\langle \psi_k , f \rangle_{\Hilbert}^2
											\, \sigma_k^2} \\
			&\,\leq\,
					\sigma_{n+1}
						\, \sqrt{\sum_{k=n+1}^{M}
										\langle \psi_k , f \rangle_{\Hilbert}^2}
			\,\leq\, \sigma_{n+1} \, \|f\|_{\Hilbert} \,,
	\end{align*}
	where equality is attained for~\mbox{$f = \psi_{n+1}$}.
	This gives us the upper bound. \\
	This upper bound is optimal.
	Indeed, for any rank-$n$ algorithm~\mbox{$A_n = S \, P$},
	there exists an element~\mbox{$f \in \linspan\{\psi_1,\ldots,\psi_{n+1}\}$}
	with~$\|f\|_{\Hilbert} = 1$ and~\mbox{$A_n(f) = 0$},
	wherefore
	\begin{equation*}
		e(A_n,f)
			= \|f\|_G
			= \sqrt{\sum_{k=1}^{n+1}
								\sigma_k^2
									\, \langle \psi_k , f \rangle_{\Hilbert}^2}
			\geq \sigma_{n+1} \|f\|_{\Hilbert} \,.
	\end{equation*}
	This implies the matching lower bound
	\begin{equation*}
		e^{\deter}(n,S,\Lall) \geq \sigma_{n+1} \,.
	\end{equation*}
\end{proof}

Let $\rho$ be a measure on~$D$
(defined for Borel sets in~$D$, with respect to the canonical metric~$d_K$).
Recall that~\mbox{$L_p(\rho)$} denotes the space of (equivalence classes of)
measurable functions defined on~$D$ and
bounded in the norm
\begin{equation*}
	\|f\|_{L_p(\rho)}
		:= \begin{cases}
					\left(\int_D f^p \rd \rho\right)^{1/p}
						\quad&\text{for $1 \leq p < \infty$,}\\
					\esssup_{D,\rho} |f| 
						\quad&\text{for $p = \infty$,}
				\end{cases}
\end{equation*}
where
\mbox{$\esssup_{D,\rho} |f|
				:= \sup \{\lambda \in \R \mid
									\rho\{\vecx \in D : |f(\vecx)| \geq \lambda\} > 0\}$}.
For continuous functions it makes sense to consider the supremum norm
\begin{equation*}
	\|f\|_{\sup} := \sup_{\vecx \in D} |f(\vecx)| \geq \|f\|_{L_{\infty}(\rho)} \,,
\end{equation*}
later, when the supremum norm and the $L_{\infty}$-norm coincide,
we will only write~\mbox{$\|\cdot\|_{\infty}$}.

The following version of a worst case lower bound
is close to the formulation of
Osipenko and Parfenov~\cite[Thm~3]{OP95},
also Cobos et al.~\cite[Lem~3.3]{CKS16},
however, it is essentially contained in Kuo et al.~\cite{KWW08}
as well.\footnote{%
	Kuo et al.\ in their research work with eigenvalues of an integral operator
	defined via the kernel function~$K$.
	These eigenvalues are the squared singular values,
	which in turn we prefer to use here.}
The first part of the proof follows Kuo et al.~\cite[Thm~1]{KWW08}.
\begin{proposition} \label{prop:H->L_inf/L_2}
	Let~$\rho$ be a probability measure on the domain~$D$,
	and let the separable reproducing kernel Hilbert space~%
	\mbox{$\Hilbert = \Hilbert(K)$}
	be compactly embedded into~\mbox{$L_{\infty}(\rho)$}.
	The embedding~\mbox{$\Hilbert \hookrightarrow L_2(\rho)$}
	is compact as well, and a singular value decomposition exists.
	This means, there is an orthonormal basis~\mbox{$(\psi_k)_{k=1}^M$}
	(with~\mbox{$M \in \N \cup \{\infty\}$}) of~$\Hilbert$
	which is orthogonal in~\mbox{$L_2(\rho)$} as well,
	and the corresponding
	singular values~\mbox{$\sigma_k := \|\psi_k\|_{L_2(\rho)}$},
	for~\mbox{$k < M+1$},
	are in decaying order~\mbox{$\sigma_1 \geq \sigma_2 \geq \ldots \geq 0$}.
	
	Then we have
	\begin{equation*}
		e^{\deter}(n,\Hilbert \hookrightarrow L_{\infty}(\rho),\Lall)
			\geq \sqrt{\sum_{k=n+1}^{\infty}
									\sigma_k^2
								} \,.
	\end{equation*}
\end{proposition}

\begin{proof}
	Without loss of generality \mbox{$n < M$}.
	
	By \lemref{lem:H->,lin=opt,singular} we know that optimal
	algorithms with cardinality~$n$ for the approximation
	problem~\mbox{$\App:\Hilbert \hookrightarrow L_{\infty}(\rho)$}
	can be built
	with orthonormal~\mbox{$\varphi_1,\ldots,\varphi_n \in \Hilbert$},
	\begin{equation*}
		A_n(f)
			= \sum_{j=1}^n
					\langle \varphi_j, f \rangle_{\Hilbert}
						\, \varphi_j \,.
	\end{equation*}
	The orthonormal system can be completed
	to an orthonormal basis~\mbox{$(\varphi_k)_{k = 1}^M$} of~$\Hilbert$.
	Then the worst case error is
	\begin{equation} \label{eq:e(A_n,H->Linf)}
		e(A_n,\Hilbert \hookrightarrow L_{\infty}(\rho))
			= \esssup_{\vecx \in D}
					\sqrt{\sum_{j=n+1}^{M} \varphi_j^2(\vecx)}
			\geq \sqrt{\int \sum_{j=n+1}^{M} \varphi_j^2(\vecx)
										\, \rho(\diff \vecx)}\,,
	\end{equation}
	compare~\eqref{eq:RKHSworUB}.\footnote{%
		The proof is simpler once knowing that the optimal algorithm,
		in fact, is built of an orthogonal projection within~$\Hilbert$
		in composition with the solution operator.
		In Kuo et al.~\cite[Thm~1]{KWW08} a little more work is needed
		because it was only used that optimal algorithms are linear.}
	
	Consider the Gaussian field~$\Psi$ associated to~$\Hilbert$,
	\begin{equation*}
		\Psi_{\vecx} = \sum_{j=1}^{M} X_j \, \varphi_j(\vecx) \,,
	\end{equation*}
	where the~$X_j$ are independent standard Gaussian random variables,
	and let~$\mu$ denote the distribution of~$\Psi$.
	The intention is to study the algorithm~$A_n$
	for the \mbox{$L_2(\rho)$}-approximation in the $\mu$-average setting.
	For~\mbox{$M = \infty$} however,
	\mbox{$\Psi \notin \Hilbert$} almost surely,
	but~$A_n$ uses functionals that are defined for functions from~$\Hilbert$.
	So instead, consider the random functions $\Psi^{(m)}$ for~\mbox{$m \in \N$},
	\begin{equation*}
			\Psi_{\vecx}^{(m)} := \sum_{j=1}^{m \wedge M} X_j \, \varphi_j(\vecx) \,,
	\end{equation*}
	and let $\mu^{(m)}$ denote the corresponding distribution in~$\Hilbert$.
	Clearly,
	\begin{equation*}
		\Psi_{\vecx}^{(m)} - [A_n \, \Psi^{(m)}](\vecx)
			= \sum_{j=n+1}^{m \wedge M} X_j \, \varphi_j(\vecx) \,,
	\end{equation*}
	and for the root mean square average $L_2(\rho)$-error we have
	\begin{align}
		e_2(A_n,L_2(\rho),\mu^{(m)})
			&= \sqrt{\expect \int
								\left(\sum_{j=n+1}^{m \wedge M} X_j \, \varphi_j(\vecx)
								\right)^2
							 \, \rho(\diff \vecx)} \nonumber \\
			[\text{Fubini, $X_j$ iid}]\quad&=
				\sqrt{\int \sum_{j=n+1}^{m \wedge M}
															\varphi_j^2(\vecx) \, \rho(\diff \vecx)} \,.
			\label{eq:e^avg(A_n,Psi)}
	\end{align}	
	In comparison with~\eqref{eq:e(A_n,H->Linf)}, this shows
	\begin{equation} \label{eq:ewor(Linf)>eavg(L2)}
		e(A_n,\Hilbert \hookrightarrow L_{\infty}(\rho))
			\geq \lim_{m \rightarrow \infty} e_2(A_n,L_2(\rho),\mu^{(m)}) \,,
	\end{equation}
	where the limit is approached monotonously from below.
	
	The RHS of~\eqref{eq:ewor(Linf)>eavg(L2)}
	can be expressed by means of singular values
	of the compact mapping~\mbox{$[\id - A_n] : \Hilbert \rightarrow L_2(\rho)$}.
	In detail,
	there exists an orthonormal system~\mbox{$(\chi_i)_{i = 1}^{M'}$}
	in~$\Hilbert$ such that \mbox{$([\id - A_n] \, \chi_i)_{i = 1}^{M'}$}~is
	orthogonal in~\mbox{$L_2(\rho)$} and
	\begin{equation*}
		[\id - A_n] f
			= \sum_{i=1}^{M'}
					\langle \chi_i , f \rangle_{\Hilbert}
						\, [\id - A_n] \chi_i \,,
	\end{equation*}
	with the singular values~%
	\mbox{$\tau_i := \|[\id - A_n] \, \chi_i\|_{L_2(\rho)}$},
	for~\mbox{$i < M'+1 := M - n + 1$},
	as always in decaying order
	\mbox{$\tau_1 \geq \tau_2 \geq \ldots > 0$}.
	(For~\mbox{$i > M'$} we have~\mbox{$\tau_i = 0$}.)
	With
	\begin{equation*}
		Z_i^{(m)} := \langle \chi_i , \Psi^{(m)} \rangle_{\Hilbert}
											= \sum_{j=1}^{m \wedge M}
										\langle \chi_i,\varphi_j \rangle_{\Hilbert} \, X_j \,,
	\end{equation*}
	we can write
	\begin{align*}
		e_2(A_n,L_2(\rho),\mu^{(m)})
			&= \sqrt{\expect\|[\id - A_n] \Psi^{(m)}\|_{L_2(\rho)}^2}
			= \sqrt{\expect\left\|
														\sum_{i=1}^{M'}
															Z_i^{(m)}
																\, [\id - A_n]\chi_i
											\right\|_{L_2(\rho)}^2} \\
			&= \sqrt{\sum_{i=1}^{M'} \tau_i^2 \, \expect (Z_i^{(m)})^2} \,.
	\end{align*}
	Note that
	\begin{equation*}
		\expect (Z_i^{(m)})^2
			\stackrel{\text{[$X_j$ iid]}}=
				\sum_{j = 1}^{m \wedge M}
					\langle \chi_i, \varphi_j \rangle_{\Hilbert}^2
			\xrightarrow[m \rightarrow \infty]{} \|\chi_i\|_{\Hilbert}^2
			= 1 \,,
	\end{equation*}
	where the sequence is monotonically increasing in~$m$,
	so we have
	\begin{equation*}
		\lim_{m \rightarrow \infty} e_2(A_n,L_2(\rho),\mu^{(m)})
			= \sqrt{\sum_{i = 1}^{\infty} \tau_i^2} \,.
	\end{equation*}
	
	It remains to show that \mbox{$\tau_i \geq \sigma_{n+i}$},
	and by~\eqref{eq:ewor(Linf)>eavg(L2)} we are done.
	Consider any algorithm~$A'_m$ for the approximation of
	\mbox{$[\id - A_n] : \Hilbert \rightarrow L_2(\rho)$}.
	Then~\mbox{$[A_n + A'_{i-1}]$} is
	an algorithm of cardinality~\mbox{$(n + i - 1)$}
	for the approximation of~\mbox{$\id : \Hilbert \hookrightarrow L_2(\rho)$}.
	By this observation
	and \lemref{lem:H->,lin=opt,singular}~\eqref{lemenum:H->H,singular},
	we obtain
	\begin{align*}
		\tau_i
			&= e^{\deter}(i-1,[\id - A_n] : \Hilbert \rightarrow L_2(\rho)) \\
			&= \inf_{A'_{i-1}}
						\sup_{\|f\|_{\Hilbert} \leq 1}
							\|[\id - A_n]f - A'_{i-1} f \|_{L_2(\rho)} \\
			&\geq \inf_{A''_{n+i-1}}
							\sup_{\|f\|_{\Hilbert} \leq 1}
								\|f - A''_{n+i-1}f\| \\
			&= e^{\deter}(n+i-1,\id : \Hilbert \hookrightarrow L_2(\rho)) \\
			&= \sigma_{n+i} \,.
	\end{align*}
\end{proof}

\begin{example}[Diagonal operators]
	\label{ex:DiagOps}
	The proposition above can be used to prove lower bounds
	for the approximation of diagonal operators on sequence spaces.
	This has already been pointed out
	by Osipenko and Parfenov~\cite[Sec~4]{OP95}.
	We repeat the example for its connection to the RKHS framework.
	
	Consider a compact matrix operator
	\begin{equation*}
		A: \ell_2 \rightarrow \ell_{\infty}, \quad
				(x_j)_{j \in \N}
					\mapsto \left({\textstyle \sum_{j}} a_{ij} \, x_j\right)_{i \in \N} \,.
	\end{equation*}
	This problem is equivalent to the embedding operator
	\begin{equation*} 
		\id : \Hilbert(K) \hookrightarrow \ell_{\infty}\,,
	\end{equation*}
	where we have the reproducing kernel
	\begin{equation*}
		K: \N \times \N \rightarrow \R, \quad
		K(i,j) = (A \, A^{\top})(i,j) \,.
	\end{equation*}
	Let~$\rho$ be a probability measure on~$\N$,
	with~\mbox{$\rho_i \geq 0$} denoting the probability of~\mbox{$i \in \N$},
	we have \mbox{$\sum_{i=1}^{\infty} \rho_i = 1$}.
	Then~$\ell_2(\rho)$ is the space of sequences~\mbox{$\vecx \in \R^\N$}
	that are bounded in the norm
	\begin{equation*}
		\|\vecx\|_{\ell_2(\rho)} := \sqrt{\sum_{i=1}^{\infty} \rho_i \, x_i^2} \,.
	\end{equation*}
	
	Now, let~$A$ be a diagonal operator
	\begin{equation*}
		[A \vecx](i) = \lambda_i \, x_i \, \quad\text{for\, $i \in \N$,}
	\end{equation*}
	where \mbox{$\lambda_1 \geq \lambda_2 \geq \ldots \geq 0$}
	and~\mbox{$\lambda_i \xrightarrow[i \rightarrow \infty]{} 0$}.
	The reproducing kernel for the Hilbert space of the
	corresponding embedding problem is
	\begin{equation*}
		K(i,j) = \delta_{ij} \, \lambda_i^2 \,.
	\end{equation*}
	Hence~\mbox{$(\lambda_i \vece_i)_{i = 1}^M$} is an orthonormal basis
	of~$\Hilbert(K)$,
	here \mbox{$M := \inf\{m \in \N_0 \mid \lambda_{m+1} = 0\}$}.
	It is also orthogonal in~$\ell_2(\rho)$ for any measure~$\rho$ on~$\N$.
	Taking
	\begin{equation*}
		\rho_i
			:= \frac{\lambda_i^{-2}}{\sum_{j=1}^m \lambda_j^{-2}}
						\, \ind[i \leq m] \,,
	\end{equation*}
	with~\mbox{$1 \leq m < M+1$},
	for~\mbox{$\Hilbert(K) \hookrightarrow \ell_2(\rho)$}
	we have the singular values
	\begin{equation*}
		\sigma_1 = \ldots = \sigma_m
			= \left(\sum_{j=1}^m \lambda_j^{-2}\right)^{-\frac{1}{2}}
			> 0 = \sigma_{m+1} = \sigma_{m+2} = \ldots
	\end{equation*}
	By \propref{prop:H->L_inf/L_2} we obtain
	\begin{equation*}
		e^{\deter}(n,A: \ell_2 \rightarrow \ell_{\infty},\Lall)
			\geq \sqrt{\frac{m-n}{\sum_{j=1}^m \lambda_j^{-2}}} \,,
	\end{equation*}
	taking the supremum over~$m$, \mbox{$n \leq m < M+1$},
	this gives sharp lower bounds.
	
	The proof for the lower bound does not reveal any information
	about the structure of optimal methods.
	A construction of methods can be found in the book of
	Osipenko~\cite[pp.~155-159]{Osi00}.
	The finite dimensional case goes back to Smolyak 1965~\cite{Smo65},
	the general case to Hutton, Morrell, and Retherford~\cite[Thm~2.12]{HMR76}.
\end{example}

\section{Breaking the Curse for Function Approximation}
\label{sec:HilbertExamples}

We study two examples of tensor product Hilbert spaces
where for $L_{\infty}$-Appro\-xi\-mation 
we can show the curse of dimensionality for the worst case,
but in the randomized setting we have polynomial tractability
with~\mbox{$n^{\ran}(\eps,d) \preceq d \, (1 + \log d) / \eps^2$}.
The first example on the approximation with the Brownian sheet,
see \secref{sec:BrownianSheet},
is more or less a toy example intended to demonstrate
the new techniques in a setting that is easy to visualize,
read~\remref{rem:BrownianToy} for further comments
on this rating of the example.
For the second example on unweighted periodic tensor product Hilbert spaces,
see \secref{sec:HilbertPeriodic},
we find some general conditions that are sufficient
for Monte Carlo to break the curse.
These conditions are specified for Korobov spaces, see \thmref{thm:Korobov}.
The latter constitutes the main application of the present chapter.
We will close the chapter with some final hints
that one should bear in mind when searching for further examples, see
\secref{sec:HilbertFinalRemarks}.

\subsection{Approximation with the Brownian Sheet}
\label{sec:BrownianSheet}

We consider the Hilbert space~\mbox{$\Hilbert(K_d)$}
with the \textit{Wiener sheet kernel}~$K_d$ on~\mbox{$[-1,+1]^d$},
the associated continuous Gaussian field~$W$ is also called
\emph{Brownian sheet},
see e.g.~Adler~\cite[p.~7, pp.~68/69]{Adl90}
for basic properties of this space
and the associated Gaussian random field.\footnote{%
	In the lecture notes of Adler, sometimes a \emph{set-indexed Brownian sheet}
	is dealt with.
	We do not need that concept here.}

For~\mbox{$d = 1$} we take the covariance kernel of
the ``two-armed'' Brownian motion,
\begin{align*}
	K_1(x,z)
		&:= \frac{|x+z| - |x-z|}{2} \\
		&= \ind[\sgn x = \sgn z] \, \min\{|x|,|z|\} \,.
\end{align*}
This space consists of once weakly differentiable functions,
\begin{equation*}
	\Hilbert(K_1)
		= \left\{
				f:[-1,+1] \rightarrow \R \mid
				f(0) = 0, \|f\|_{\Hilbert(K_1)} := \|f'\|_2
			\right\} \,,
\end{equation*}
the inner product is~\mbox{$\langle f,g\rangle_{\Hilbert(K_1)}
															:= \int_{-1}^{+1} f'(x) \, g'(x) \rd x$}.

For~\mbox{$d \in \N$} we take the tensor product kernel,
\begin{align*}
	K_d(\vecx,\vecz)
		&:= \prod_{j=1}^d K_1(x_j,z_j) \\
		&= \ind[\sgn \vecx = \sgn \vecz] \, \prod_{j=1}^d \min\{|x_j|,|z_j|\} \,,
\end{align*}
where~\mbox{$\sgn \vecx = (\sgn x_j)_{j=1}^d$}
is the vector-valued signum function.
Thus
\begin{align*}
	\Hilbert(K_d)
		&= \bigotimes_{j=1}^d \Hilbert(K_1) \\
		&= \{f:[-1,+1]^d \rightarrow \R \mid
				\text{$f(\vecx) = 0$ if $\exists_j \, x_j = 0$, and }
				\|f\|_{\Hilbert} := \|D^{\ones} f\|_{L_2} < \infty
			\} \,,
\end{align*}
where~\mbox{$D^{\ones}(f) := \partial_1 \cdots \partial_d f$}.
Note that the initial error is~$1$ (thus properly normalized)
since the kernel takes its maximum~$1$
in the corners~\mbox{$\{\pm 1\}^d$}. 
Furthermore, functions~\mbox{$f \in \Hilbert(K_d)$} can be identified with
functions~\mbox{$\tilde{f} \in \Hilbert(K_{d+1})$}, 
\begin{equation*}
	\tilde{f}(x_1,\ldots,x_{d+1})
		:= f(x_1,\ldots,x_d) \, [x_{d+1}]_+ \,,
\end{equation*}
then~\mbox{$\|f\|_{\Hilbert} = \|\tilde{f}\|_{\Hilbert}$},
and \mbox{$\|f\|_{\infty} = \|\tilde{f}\|_{\infty}$},
where for~$\tilde{f}$ the maximal absolute value
is attained with~\mbox{$x_{d+1} = 1$}.

\begin{theorem}
	Deterministic $L_{\infty}$-approximation
	of functions from the Wiener sheet space on~\mbox{$[-1,+1]^d$}
	suffers from the curse of dimensionality,
	in detail,
	\begin{equation*}
		n^{\deter}(\eps,\Hilbert(K_d) \hookrightarrow L_{\infty},\Lall)
			\geq 2^{d-1} \,,
		\quad \text{for\, $0 < \eps \leq \frac{\sqrt{2}}{2}$.}
	\end{equation*}
\end{theorem}
\begin{proof}
	We consider the $2^d$-dimensional subspace of~\mbox{$\Hilbert(K_d)$}
	spanned by~\mbox{$\{K_{\vecsigma}\}_{\vecsigma \in \{\pm 1\}^d}$},
	where
	\begin{equation*}
		K_{\vecsigma}(\vecx)
			\,:=\, K_d(\vecsigma,\vecx)
			\,=\, \ind[\sgn \vecx = \vecsigma] \, \prod_{j=1}^d |x_j| \,.
	\end{equation*}
	These functions are orthonormal in~$\Hilbert = \Hilbert(K_d)$ since
	\begin{equation*}
		\langle K_{\vecsigma}, K_{\vecsigma'} \rangle_{\Hilbert}
			\,=\, K_d(\vecsigma,\vecsigma')
			\,=\, \ind[\vecsigma = \vecsigma'] \,.
	\end{equation*}
	Besides, they have essentially disjoint supports
	(which are the subcubes of~\mbox{$[-1,+1]^d$}
	with constant sign in each coordinate),
	and take their supremum in~\mbox{$\vecx = \vecsigma$},
	which is~\mbox{$K_{\vecsigma}(\vecsigma) = 1$}.
	For~\mbox{$f = \sum_{\vecsigma \in \{\pm 1\}^d}
										a_{\vecsigma} \, K_{\vecsigma}$}
	we have
	\begin{equation*}
		\|f\|_{\Hilbert(K_d)}
			= \sqrt{\sum_{\vecsigma \in \{\pm 1\}^d} a_{\vecsigma}^2} \,,
		\quad\text{and}\quad
		\|f\|_{L_{\infty}}
			= \max_{\vecsigma \in \{\pm 1\}^d} |a_{\vecsigma}| \,.
	\end{equation*}
	Hence we can estimate the error from below
	by comparison to a sequence space embedding,
	\begin{equation*}
		e^{\deter}(n,\Hilbert(K_d) \hookrightarrow L_{\infty})
			\geq e^{\deter}(n, \ell_2^{2^d} \hookrightarrow \ell_{\infty}^{2^d})
			\geq \sqrt{1 - n \, 2^{-d}} \,,
	\end{equation*}
	see~\eqref{eq:det:l1->linf}.
	This implies the stated lower bound for the complexity,
	compare \secref{sec:l2->linf,curse}.
\end{proof}

\begin{theorem}
	Randomized $L_{\infty}$-approximation
	of functions from the Wiener sheet space,
	using linear functionals for information,
	is polynomially tractable.
	In detail,
	\begin{equation*}
		n^{\ran}(\eps,\Hilbert(K_d) \hookrightarrow L_{\infty},\Lall)
			\leq C \, \frac{d \, (1+ \log d)}{\eps^2} \,,
	\end{equation*}
	with a numerical constant~$C > 0$.
\end{theorem}
\begin{proof}
	We want to apply \lemref{lem:stdMCapp}, so we need to estimate
	\begin{equation*}
		\expect \|W\|_{\infty}
	\end{equation*}
	for the Brownian sheet~$W$ on~\mbox{$[-1,+1]^d$}
	defined by the covariance kernel~$K_d$.
	This will be done by an entropy estimate
	and \propref{prop:Dudley} (Dudley).\footnote{%
		The same approach for estimating the expected maximum of~$W$
		is taken in Adler~\cite[Prop~1.2]{Adl90}
		though for the domain~$[0,1]^d$ instead,
		which admittedly is a minor change.
		Steps that have been left to the reader there are explicated here.}
	
	The canonical metric of the Wiener sheet kernel is
	\begin{equation*}
		d_K(\vecx,\vecz)^2
			= \prod_{j=1}^d |x_j| + \prod_{j=1}^d |z_j|
				- 2\, \ind[\sgn \vecx = \sgn \vecz]
						\, \left(\prod_{j=1}^d \min\{|x_j|,|z_j|\}\right) \,.
	\end{equation*}
	If~\mbox{$\sgn \vecx \not= \sgn \vecz$}
	there exists an index~\mbox{$j \in \{1,\ldots,d\}$}
	such that~\mbox{$\sgn x_j \not= \sgn z_j$}, hence
	\begin{equation*}
		d_K(\vecx,\vecz)^2 = \prod_{j=1}^d |x_j| + \prod_{j=1}^d |z_j|
			\leq |x_j| + |z_j|
			= |x_j - z_j|
			\leq |\vecx - \vecz|_{\infty} \,.
	\end{equation*}
	If~\mbox{$\sgn \vecx = \sgn \vecz$},
	we obtain
	\begin{align*}
		d_K(\vecx,\vecz)^2
			&= \prod_{j=1}^d |x_j| + \prod_{j=1}^d |z_j|
				- 2\, \prod_{j=1}^d \min\{|x_j|,|z_j|\} \\
			&\leq 2 \, \left(\prod_{j=1}^d \max\{|x_j|,|z_j|\}
											- \prod_{j=1}^d \min\{|x_j|,|z_j|\}
									\right) \\
		\text{[telescoping sum]}\quad
			&\leq 2 \, \sum_{k=1}^d \Biggl[
										\underbrace{\left(\prod_{j=1}^{k-1} \min\{|x_j|,|z_j|\}
																\right)
																}_{\leq 1} \\
				&\quad\qquad\qquad\qquad
										\underbrace{\left(\max\{|x_k|,|z_k|\} - \min\{|x_k|,|z_k|\}
																\right)
																}_{=|x_k-z_k|} \\
				&\quad\qquad\qquad\qquad\qquad\qquad
										\underbrace{\left(\prod_{j=k+1}^d \max\{|x_j|,|z_j|\}
																\right)
																}_{\leq 1}
															\Biggr]\\
			&\leq 2 \, |\vecx - \vecz|_1 \,.
	\end{align*}
	This shows
	\begin{equation*}
		d_K(\vecx,\vecz)^2 \leq 2 \, d \, |\vecx - \vecz|_{\infty} \,,
	\end{equation*}
	and for~$\vecx \in [-1,1]^d$ and~$r>0$ we have the inclusion
	\begin{equation*}
		B_{\infty}\left(\vecx,\frac{r^2}{2 \, d}\right) \subseteq B_K(\vecx,r) \,,
	\end{equation*}
	where~$B_{\infty}$ denotes the ball in the $\ell_{\infty}^d$-metric.
	Since one can cover~\mbox{$B_{\infty}(\zeros,1) = [-1,1]^d$}
	by \mbox{$\lceil 2 \, d / r^2 \rceil^d$}~balls
	with~$\ell_{\infty}^d$-radius~$\frac{r^2}{2 \, d}$,
	for the metric entropy with respect to~$d_K$ we obtain
	\begin{equation*}
		H(r) = \log(N(r)) \leq d \log\left(1 + \frac{2 \, d}{r^2}\right)
			\leq C_1 \, d \, (1 + \log d) \, (1 - \log r) \,.
	\end{equation*}
	Note that~\mbox{$N(1) = 1$} with the ball around~\mbox{$\vecx = \zeros$}.
	
	Since the Brownian sheet~$W$ is zero on the coordinate hyperplanes,
	we do not need the additional term when applying \lemref{lem:E|X|<init+EsupX},
	\begin{align*}
		\expect \|W\|_{\infty}
			&\leq 2 \, \expect \sup_{\vecx \in [-1,1]^d} W_{\vecx} \\
		\text{[\propref{prop:Dudley}]} \quad
			&\leq 2 \, C_{\text{Dudley}} \, C_1
							\, \sqrt{d \, (1 + \log d)} \int_0^1 \sqrt{1 - \log r} \rd r \\
		\text{[subst.~${\textstyle \frac{s^2}{2}} = 1 - \log r$]}\quad
			&= C_2 \, \sqrt{d \, (1 + \log d)} \, \frac{\euler}{\sqrt{2}}
					\,\int_1^{\infty}
								s^2 \, \exp\left(-\frac{s^2}{2}\right)
							\rd s]\\
			&= C_3 \, \sqrt{d \, (1 + \log d)} \,.
	\end{align*}
	Hence by \lemref{lem:stdMCapp},
	\begin{equation*}
		n^{\ran}(\eps,\Hilbert(K_d) \hookrightarrow L_{\infty},\Lall)
			\leq 4 \, C_3^2 \, \frac{d \, (1 + \log d)}{\eps^2} \,.
	\end{equation*}
	This finishes the proof.
\end{proof}

\begin{remark} \label{rem:BrownianToy}
	This is a first example of a function approximation problem
	where Monte Carlo methods can break the curse of dimensionality.
	As mentioned before, the initial error is properly normalized,
	and the problem for lower dimensions
	is contained in the problem for higher dimensions.
	This example, however, has a downside:
	Actually, we treat the simultaneous approximation
	of $2^d$~entirely independent functions
	(that only need to be bounded in a common Euclidean norm).
	This view is justified by the fact that
	functions from the space~\mbox{$\Hilbert(K_d)$}
	are zero at the coordinate hyperplanes,
	that way the domain is split into subcubes of constant sign.
	Observe the similarities with the sequence space example
	in \secref{sec:l2->linf,curse}
	where we only lack the logarithmic term~\mbox{$(1 + \log d)$}
	in the Monte Carlo upper bound.
	Adding some proper ``function space nature'' by this Wiener sheet example,
	honestly, serves as a fig-leaf
	for the artificiality of the sequence space example.
	The next section treats much more natural problems.
	
	Including this example in this study, however, was motivated by the fact
	that the Brownian sheet is widely known,
	and that the loss of smoothness becomes palpable.
	In addition, this gives us a non-periodic example where it was convenient
	to use entropy methods for the estimate
	(in contrast to the next section where we will rely on the technique
	of majorizing measures).
	The Wiener sheet -- usually only defined on~\mbox{$[0,1]^d$} --
	is a common example for many topics in IBC,
	see for example Novak and Wo\'zniakowski~\cite{NW08,NW10,NW12}
	or Ritter~\cite{Rit00}.
\end{remark}

\subsection{Tensor Product Spaces of Periodic Functions}
\label{sec:HilbertPeriodic}

\subsubsection{The General Setting}

We consider the~$L_{\infty}$-approximation of Hilbert space functions
defined on the $d$-dimen\-sional torus~$\Torus^d$,
compare the notation in Cobos et al.~\cite{CKS16} (with slight modifications).
The Hilbert spaces we consider will be unweighted tensor product spaces.

A few words on the domain.
The one-dimensional torus~\mbox{$\Torus := \R \modulo \Z \equiv [0,1)$}
can be identified with the unit interval tying the endpoints together.
A natural way to define a metric on~$\Torus$ is
\begin{equation*}
	d_{\Torus}(x,z) := \min_{k \in\{-1,0,1\}} |x-z+k| \,,
	\quad \text{for\, $x,z \in [0,1)$.}
\end{equation*}
This is the length of the shortest connection between two points
along a closed curve of length~$1$.
For the~$d$-dimensional torus we take the summing metric
\begin{equation*}
	d_{\Torus^d}(\vecx,\vecz) := \sum_{j=1}^d d_{\Torus}(x_j,z_j) \,.
\end{equation*}
Smoothness and continuity are to be defined with respect to this metric.

We start with a basis representation of spaces under consideration.
First, for~$d=1$, the Fourier system
\begin{equation*}
	\{\varphi_0 := 1,\,
		\varphi_{-k} := \sqrt{2} \, \sin(2 \, \pi \, k \cdot),\,
		\varphi_k := \sqrt{2} \, \cos(2 \, pi \, k \cdot)
	\}_{k \in \N}
\end{equation*}
is an orthonormal basis for~\mbox{$L_2(\Torus) = L_2([0,1))$}.
We consider Hilbert spaces where these functions are still orthogonal.
Namely, let \mbox{$\Hilbert_{\veclambda}(\Torus)$}
denote the Hilbert space for which the system
\begin{equation*}
	\left\{\psi_0 := \lambda_0,\,
				\psi_{-k} := \lambda_k \, \sin(2 \, \pi \, k \cdot),\,
				\psi_k := \lambda_k \, \cos(2 \, \pi \, k \cdot)
	\right\}_{k \in \N} \setminus\{0\} \,,
\end{equation*}
is an orthonormal basis.\footnote{%
	If~$\lambda_k = 0$ for certain~$k$,
	of course, the corresponding zero-functions
	cannot be part of the orthonormal basis.
	In the proof of \thmref{thm:Korobov}
	we consider finite-dimensional subspaces
	where the corresponding orthonormal basis
	will be~\mbox{$\{\psi_k\}_{k=-m}^m$}.
	The same holds for the basis of the $d$-dimensional space.}
Here, \mbox{$\veclambda = (\lambda_k)_{k \in \N_0} \subset [0,\infty)$}
indicates the importance of the different frequencies.
Now, for general~\mbox{$d \in \N$},
we consider the \emph{unweighted}\footnote{%
	This means that every coordinate is equally important.
	For weighted tensor product spaces one would take different
	values for~$\veclambda$ for different dimensions~\mbox{$j=1,\ldots,d$},
	compare Cobos et al.~\cite{CKS16}, or Kuo et al.~\cite{KWW08}.}
tensor product space~\mbox{$\Hilbert_{\veclambda}(\Torus^d)$}
with the tensor product
orthonormal basis~\mbox{$\{\psi_{\veck}\}_{\veck \in \Z^d} \setminus \{0\}$},
\begin{equation} \label{eq:periodicpsi}
	\psi_{\veck}(\vecx) := \prod_{j=1}^d \psi_{k_j}(x_j) \,. 
\end{equation}
Analogously, we write~\mbox{$\{\varphi_{\veck}\}_{\veck \in \Z^d}$}
for the Fourier basis of~\mbox{$L_2(\Torus^d)$},
once more at the risk of some confusion
from using the same letter for the index
as in the one-dimensional case,
merely with a different font style.

For a suitable choice of the~$\lambda_k$,
we have the one-dimensional reproducing kernel
\begin{align}
	K_{\veclambda}(x,z)
		&:= \lambda_0^2
				+ \sum_{k=1}^{\infty}
						\lambda_k^2
							\, [\cos(2 \pi k \, x) \, \cos (2 \pi k \, z)
									+ \sin (2 \pi k \, x) \, \sin (2 \pi k \, z)] 
			\nonumber\\
		&= \sum_{k=0}^{\infty} \lambda_k^2 \, \cos (2 \pi k\,(x-z))
		\label{eq:K_la} \,,
\end{align}
for general dimensions~\mbox{$d \in \N$} we obtain the product kernel
\begin{equation*}
	K_{\veclambda}^d(\vecx,\vecz)
		:= \prod_{j=1}^d K_{\veclambda}(x_j,z_j) \,.
\end{equation*}
In particular, the initial error is
\begin{equation*}
	e(0,\Hilbert_{\veclambda}(\Torus^d) \hookrightarrow L_{\infty}(\Torus^d))
		= \sup_{\vecx \in \Torus^d} \sqrt{K_{\veclambda}^d(\vecx,\vecx)}
		= \left( \sum_{k=0}^{\infty} \lambda_k^2 \right)^{d/2} \,.
\end{equation*}
The condition~\mbox{$\sum_{k=0}^{\infty} \lambda_k^2 < \infty$}
is necessary and sufficient for the existence of a reproducing kernel
and for the embedding
\mbox{$\Hilbert_{\veclambda}(\Torus^d) \hookrightarrow L_{\infty}$}
to be compact,
see Cobos et al.~\cite[Thm~3.1]{CKS16}
with an extended list of equivalent properties.
We will assume \mbox{$\sum_{k=0}^{\infty} \lambda_k^2 = 1$}
for that the initial error be constant~$1$.

Note that under this last assumption,
functions~\mbox{$f \in \Hilbert_{\veclambda}(\Torus^d)$}
can be identified with
functions~\mbox{$\tilde{f} \in \Hilbert_{\veclambda}(\Torus^{d+1})$}
for~\mbox{$d < \tilde{d}$},
\begin{equation*}
	\tilde{f}(x_1,\ldots,x_{\tilde{d}})
		:= f(x_1,\ldots,x_d) \, K(0,x_{d+1}) \,,
\end{equation*}
the $\Hilbert_{\veclambda}$- and the $L_{\infty}$-norms coincide,
the maximum values of the function being attained for~\mbox{$x_{d+1} = 0$}.
So indeed,
the problems of lower dimensions are contained
in the problems of higher dimensions,
yet~$\tilde{f}$ is a bit lopsided in the redundant variable.

\begin{theorem} \label{thm:curseperiodic}
	Suppose that~\mbox{$0 \leq \lambda_0 < 1$}
	and~\mbox{$\sum_{k=0}^{\infty} \lambda_k^2 = 1$}
	for non-negative~$\lambda_k$.
	Then the approximation problem
	\begin{equation*}
		\App: \Hilbert_{\veclambda}(\Torus^d) \hookrightarrow L_{\infty}(\Torus^d)
	\end{equation*}
	suffers from the curse of dimensionality in the deterministic setting.
	
	In detail, while the initial error is constant~$1$, we have
	\begin{equation*}
		e^{\deter}(n,\Hilbert_{\veclambda}(\Torus^d)
									\hookrightarrow L_{\infty}(\Torus^d))
			\geq \sqrt{(1 - n \, \beta^d)_+} \,,
	\end{equation*}
	where~\mbox{$\beta := \sup\{\lambda_0^2,\,
															\lambda_k^2/2
														\}_{k \in \N} \in (0,1)$}.
	In other words,
	for~\mbox{$\eps \in (0,1)$} we have the complexity bound
	\begin{equation*}
		n^{\deter}(\eps,\Hilbert_{\veclambda}(\Torus^d)
									\hookrightarrow L_{\infty}(\Torus^d))
			\geq \beta^{-d}(1-\eps)^2 \,.
	\end{equation*}
\end{theorem}
\begin{proof}
	Following \propref{prop:H->L_inf/L_2}, we study the singular values
	of~\mbox{$\Hilbert_{\veclambda}(\Torus^d) \hookrightarrow L_2(\Torus^d)$}.
	Essentially, this can be traced back to the one dimensional case,
	\begin{equation*}
		\psi_k = \sigma_k \varphi_k
		\quad \text{for\, $k \in \Z$,}
	\end{equation*}
	where~\mbox{$\sigma_0 = \lambda_0$}
	and \mbox{$\sigma_k = \sigma_{-k} = \lambda_k / \sqrt{2}$} for~{$k \in \N$}
	denote the unordered singular values
	of~\mbox{$\Hilbert_{\veclambda}(\Torus) \hookrightarrow L_2(\Torus)$}.
	In the multi-dimensional case we have
	\begin{equation*}
		\psi_{\veck} = \sigma_{\veck} \varphi_{\veck}
		\quad \text{for\, $\veck \in \Z^d$,}
	\end{equation*}
	with the unordered singular values~%
	\mbox{$\sigma_{\veck} = \prod_{j=1}^d \sigma_{k_j}$},
	in particular
	\begin{equation*}
		\sigma_{\veck}^2
			\leq \left(\sup_{k' \in \Z} \sigma_{k'}^2\right)^d
			= \left(\sup\{\lambda_0^2, \, \lambda_{k'}^2 / 2\}_{k' \in \N}\right)^d
			= \beta^d \,.
	\end{equation*}
	On the other hand,
	\begin{equation*}
		\sum_{\veck \in \Z^d} \sigma_{\veck}^2
			= \left(\sum_{k' \in \Z} \sigma_{k'}^2 \right)^d
			= \left(\sum_{k' = 0}^{\infty} \lambda_{k'}^2 \right)^d
			= 1 \,.
	\end{equation*}
	So for any index set~\mbox{$I \subset \Z^d$} of size~\mbox{$\# I = n$},
	we have
	\begin{equation*}
		\sqrt{\sum_{\veck \in \Z^d \setminus I} \sigma_{\veck}^2}
			\geq \sqrt{(1 - n \beta^d)_+} \,.
	\end{equation*}
	By \propref{prop:H->L_inf/L_2}, this proves the lower bound.
\end{proof}

\begin{remark}
	Within the above proof we applied \propref{prop:H->L_inf/L_2}
	with~$\rho$ being the uniform distribution on~$\Torus^d$.
	If we consider complex-valued Hilbert spaces,
	this approach will always give sharp lower bounds,
	see Cobos et al.~\cite[Thm~3.4]{CKS16}.
	In the real-valued setting we obtain sharp error results at least
	for those~$n$ where the optimal index set~\mbox{$I \subset \Z^d$}
	contains all indices belonging to the same frequency, that is,
	\begin{equation*}
		\veck = (k_1,\ldots,k_d) \in I
			\quad\Leftrightarrow\quad
				\abs \veck := (|k_1|, \ldots, |k_d|) \in I \,.
	\end{equation*}
	Still, in most cases it is hard to estimate the number of
	singular values within a certain range.
\end{remark}

The following abstract result relies on estimates for the shape of the
kernel function~\mbox{$K_{\vecx}= K(\vecx,\cdot)$}.

\begin{theorem} \label{thm:MCUBperiodic}
	Consider the uniform approximation problem
	\begin{equation*}
		\App: \Hilbert(K_d) \hookrightarrow L_{\infty}(\Torus^d)
	\end{equation*}
	where~$\Hilbert(K_d)$ is a reproducing kernel Hilbert space on
	the $d$-dimensional torus~$\Torus^d$ with the following properties:
	\begin{enumerate}[(i)]
		\item \label{enum:MCUBperiodic,unweighted}
			$K_d$ is the unweighted product kernel built
			from the one-dimensional case,\\
			this means
			\mbox{$K_d(\vecx,\vecz) := \prod_{j=1}^d K_1(x_j,z_j)$}
			for~\mbox{$\vecx,\vecz \in \Torus^d$}.
		\item \label{enum:MCUBperiodic,init}
			\mbox{$K_1(x,x) = 1$} for all~$x \in \Torus$.\\
			(Consequently, \mbox{$K_d(\vecx,\vecx) = 1$}
			for all~\mbox{$\vecx \in \Torus^d$},
			in particular the initial error is constant~$1$.)
		\item \label{enum:MCUBperiodic,local}
			The kernel function can be locally estimated from below
			with an exponential decay,
			that is,
			there exist \mbox{$\alpha > 0$} and~\mbox{$0 < R_0 \leq \frac{1}{2}$}
			such that
			\begin{equation*}
				K_1(x,z) \geq \exp(-\alpha \, d_{\Torus}(x,z))
				\quad
				\text{for \mbox{$x,z \in \Torus$} with~\mbox{$d_{\Torus}(x,z) \leq R_0$}.}
			\end{equation*}
			(Hence \mbox{$K_d(\vecx,\vecz)
															\geq \exp(- \alpha \, d_{\Torus^d}(\vecx,\vecz))$}
			\quad for \mbox{$\max_j d_{\Torus}(x_j,z_j) < R_0$}.)\footnote{%
				If one is interested in a version of this theorem
				with better constants for particularly nice kernels,
				one could start with a stronger assumption
				\mbox{$K_1(x,z) \geq \exp(-\alpha \, d_{\Torus}(x,z)^2)$}
				which is a comparison to a bell-shaped curve.
				The asymptotics of the complexity result, however, will not change.
				See also \remref{rem:Korobov-small r} with a proposal
				for a modified version of this theorem.
				}
		\end{enumerate}
	Then the problem is polynomially tractable
	in the randomized setting with general linear information~$\Lall$,
	in detail,
	\begin{equation*}
		n^{\ran}(\eps,\Hilbert(K_d) \hookrightarrow L_{\infty}(\Torus^d),\Lall)
			\leq C \, (1 + \alpha^2 - \log 2 R_0) \, \frac{d \, (1 + \log d)}{\eps^2} \,,
	\end{equation*}
	with a universal constant~\mbox{$C > 0$}.
\end{theorem}
\begin{proof}
	We are going to apply the method of majorizing measures in order
	to estimate the expected maximum norm of the Gaussian field~$\Psi$
	associated with the reproducing kernel~$K_d$.
	The majorizing measure~$\mu$ we choose
	shall be the uniform distribution on~$\Torus^d$,
	this is the Lebesgue measure on~\mbox{$[0,1)^d$}.
	
	Supposed that \mbox{$\max_j d_{\Torus}(x_j,z_j) < R_0$},
	for the canonical metric in the $d$-dimensional case we have
	\begin{align*}
		d_K(\vecx,\vecz)^2
			&= K_d(\vecx,\vecx) + K_d(\vecz,\vecz) - 2 \, K_d(\vecx,\vecz) \\
			&\leq 2 \, (1 - \exp(-\alpha \, d_{\Torus^d}(\vecx,\vecz))) \\
			&\leq 2 \, \alpha \, d_{\Torus^d}(\vecx,\vecz) \,.
	\end{align*}
	By this, we have the inclusion
	\begin{equation*}
		B_{\Torus}\left(\sqrt{\frac{r}{2\,\alpha}},\,\vecx\right)
			\subseteq B_K(r,\vecx) \,,
	\end{equation*}
	where~\mbox{$B_{\Torus}(R,\vecx)$} denotes
	the $d_{\Torus^d}$-ball of radius~$R$ around~$\vecx \in \Torus^d$,
	and \mbox{$B_K(r,\vecx)$} is the ball of radius~$r$ in the canonical metric
	associated with~$K_d$.
	Hence
	\begin{equation*}
		\mu(B_K(r,\vecx))
			\geq \mu\left(B_{\Torus}\left(\sqrt{\frac{r}{2\,\alpha}},\,\vecx
															\right)
							\right) \,.
	\end{equation*}
	We distinguish three cases:
	\begin{itemize}
		\item For~\mbox{$0 \leq R \leq R_0 \leq \frac{1}{2}$},
			the $\mu$-volume \mbox{$\mu(B_{\Torus}(R,\vecx))$} of the torus metric ball
			is the volume~\mbox{$\Vol(R \, B_1^d)$}
			of an $\ell_1$-ball in~$\R^d$ with radius~$R$,
			so with Stirling's formula,
			\begin{align}
				\log(1/\mu(B_{\Torus}(R,\vecx)))
					&= \log(1/\Vol(R \, B_1^d)) \nonumber\\
					&= \log \Gamma(d + 1) - d \, \log 2 R \nonumber\\
					&\leq C_1 \, d \, (1 + \log d) \, (1 - \log 2 R)
					\label{eq:log(1/mu(B2))} \\
					&\leq C_1 \, d \, (1 + \log d) \, (1 - \log (2 R)^2) \nonumber \,.
			\end{align}
			Such an estimate can be used
			for~\mbox{$0 \leq r \leq 2\alpha \, R_0^2 \leq \alpha/2$}.
		\item For~\mbox{$R > R_0$},
			the $\mu$-volume of~\mbox{$B_{\Torus}(R,\vecx)$}
			can be estimated from below with the $\mu$-volume
			of an $\ell_1$-ball with radius~\mbox{$R_0 \leq \frac{1}{2}$}.
			We will use this in the case~\mbox{$2\alpha \, R_0^2 < r < 2$}.
		\item For~\mbox{$r \geq 2$},
			we know~\mbox{$B_K(r,\vecx) = \Torus^d$} with $\mu$-volume~$1$
			since~\mbox{$d_K(\vecx,\vecz) \leq 2$}.
			In this case the term~$\log(1/\mu(B_K(r,\vecx)))$
			vanishes.
	\end{itemize}
	Combining these cases, we can estimate
	\begin{align*}
		\int_0^{\infty} \sqrt{\log(1/\mu(B_K(\vecx,r)))} \rd r
			&\leq \int_0^{\alpha/2}
							\sqrt{\log \left(1 \middle/
																		\Vol\left(\sqrt{\frac{r}{2\,\alpha}}
																								\, B_1^d
																				\right)
													\right)
										} \,
								\rd r \\
			&\qquad + \left(2 - 2\alpha \, R_0^2\right)_+
								\, \sqrt{\log\left(1 \middle/
																			\Vol
																				\left(R_0 \, B_1^d
																				\right)
															\right)
													} \\
			&\stackrel{\text{\eqref{eq:log(1/mu(B2))}}}{\leq}
				\sqrt{C_1 \, d \, (1 + \log d)} \\
			&\qquad\qquad \left(\int_0^{\alpha/2} \sqrt{1 - \log \frac{2 r}{\alpha}} \rd r
													\,+\, 2 \, \sqrt{1 - \log 2 R_0}
										\right) \\
			&= C_2 \, (1 + \alpha + \sqrt{- \log 2 R_0}) \, \sqrt{d \, (1 + \log d)} \,.
	\end{align*}
	Here, the last integral can be transformed into a familiar integral
	by the substitution~\mbox{$s^2/2 = 1 - \log \frac{2 r}{\alpha}$},
	\begin{equation*}
		\int_0^{\alpha/2} \sqrt{1 - \log \frac{2 r}{\alpha}} \rd r
			= \frac{\euler \, \alpha}{2\, \sqrt{2}}
					\, \int_1^{\infty}
									s^2 \, \exp\left(-\frac{s^2}{2}\right)
								\rd s \,.
	\end{equation*}
	
	Now, consider the Gaussian field~$\Psi$
	associated with the reproducing kernel~$K_d$.
	Putting the above calculation into \propref{prop:Fernique} (Fernique),
	with \lemref{lem:E|X|<init+EsupX} we obtain
	\begin{align*}
		\expect \|\Psi\|_{\infty}
			&\leq \sqrt{\frac{2}{\pi}}
						+ 2 \, C_{\text{Fernique}} \, C_2 \, (1 + \alpha + \sqrt{- \log 2 R_0})
								\, \sqrt{d \, (1 + \log d)} \\
			&\leq C_3  \, (1 + \alpha + \sqrt{- \log 2 R_0}) \, \sqrt{d \, (1 + \log d)}\,.
	\end{align*}
	By \lemref{lem:stdMCapp}, this gives us a final upper bound on the complexity.
\end{proof}

We finish the general part of the periodic setting
with sufficient conditions for the parameters~$\veclambda$
of the kernels~$K_{\veclambda}$
for that we can apply \thmref{thm:MCUBperiodic}.
\begin{corollary} \label{cor:periodic-sufficient}
	Given a kernel
	\begin{equation*}
		K_1(x,z) := K_{\veclambda}(x,z)
			= \sum_{k=0}^{\infty} \lambda_k^2 \, \cos 2 \pi k\,(x-z)
	\end{equation*}
	with~\mbox{$\lambda_k \geq 0$},
	and tensor product kernels $K_{\veclambda}^d$ as before,
	it is sufficient
	for polynomial tractability in the randomized setting
	that the following holds:
	\begin{enumerate}[\quad(a)\quad]
		\item \label{enumcor:periodic,init}
			$\sum_{k=0}^{\infty} \lambda_k^2 = 1$,
		\item \label{enumcor:periodic,sum(k*la2)}
			$\sigma_{\veclambda}
				:= \sum_{k=1}^{\infty} k \, \lambda_k^2 < \infty$.
	\end{enumerate}
	In detail, there exists a universal constant~$C'> 0$ such that
	\begin{equation*}
		n^{\ran}(\eps,\Hilbert(K_{\veclambda}^d)
				\hookrightarrow L_{\infty}(\Torus^d),\Lall)
			\leq C' \, (1 + \sigma_{\veclambda}^2)
							\, \frac{d \, (1 + \log d)}{\eps^2} \,.
	\end{equation*}
\end{corollary}
\begin{proof}
	Condition~\eqref{enumcor:periodic,init}
	is for the normalization of the initial error,
	see~\eqref{enum:MCUBperiodic,init} in~\thmref{thm:MCUBperiodic}.
	
	We first check, when assumption~\eqref{enum:MCUBperiodic,local}
	of \thmref{thm:MCUBperiodic} holds with~\mbox{$R_0 = \frac{1}{2}$},
	that is, the inequality
	\mbox{$K_{\veclambda}(x,z) \geq \exp(-\alpha \, d_{\Torus}(x,z))$}
	is valid for all~\mbox{$x,z \in \Torus$}.
	It suffices to show that
	\begin{equation*}
		h(x) := \exp(\alpha \, x) \, K_{\veclambda}(x,0)
	\end{equation*}
	is monotonously increasing for~\mbox{$x > 0$},
	noting~\mbox{$h(0) = 1$} by~\eqref{enumcor:periodic,init}.
	Condition~\eqref{enumcor:periodic,sum(k*la2)} guarantees differentiability
	of~\mbox{$K_{\veclambda}(\cdot,0)$} with absolute convergence
	of the resulting series of sine functions.
	Moreover, for the derivative of~$h$ we obtain
	\begin{align*}
		h'(x)
			&= \exp(\alpha \, x)
					\, \left[\alpha \,
										\sum_{k=0}^{\infty}
											\lambda_k^2 \, \cos 2 \pi k \, x
									- 2 \pi \sum_{k=1}^{\infty}
														k \, \lambda_k^2 \, \sin 2 \pi k \, x
						\right] \\
			&\geq \exp(\alpha \, x)
					\, \left[\alpha \, \left(\lambda_0^2
																		- \sum_{k=1}^{\infty} \lambda_k^2
														\right)
										- 2 \pi \sum_{k=1}^{\infty} k \, \lambda_k^2
						\right] \,.
	\end{align*}
	Positivity of the left-hand term in~\mbox{$[\ldots]$} is ensured
	if~\mbox{$\lambda_0^2 > \frac{1}{2}$} holds in addition to
	\eqref{enumcor:periodic,init}.
	The right-hand term in~\mbox{$[\ldots]$} is finite thanks to
	condition~\eqref{enumcor:periodic,sum(k*la2)},
	hence we can choose
	\begin{equation*}
		\alpha
			:=\frac{2 \, \pi \sigma_{\veclambda}}{2 \, \lambda_0^2 - 1}
	\end{equation*}
	to guarantee non-negativity of~\mbox{$h'(x)$}.
	This gives us~\mbox{$K_{\veclambda}(x,z) \geq \exp(- \alpha \, |x-z|)$}
	as intended.
	Restricting to the case~\mbox{$\lambda_0^2 \geq \frac{2}{3}$}
	we have a better control over the constants,
	in that case getting
	\mbox{$\alpha \leq 6 \pi \, \sigma_{\veclambda}$}.
	Hence by~\thmref{thm:MCUBperiodic},
	\begin{align*}
		n^{\ran}(\eps,\Hilbert(K_{\veclambda}^d)
										\hookrightarrow L_{\infty}(\Torus^d),
						\Lall)
			&\leq C \, (1 + (6 \, \pi)^2 \, \sigma_{\veclambda}^2)
							\, d \, (1 + \log d) \, \eps^{-2} \\
			&\leq C_0 \, (1 + \sigma_{\veclambda}^2)
							\, d \, (1 + \log d) \, \eps^{-2} \,.
	\end{align*}
	
	If \mbox{$\lambda_0^2 < \frac{2}{3}$},\footnote{%
		In the case~\mbox{$\lambda_0^2 \leq \frac{1}{2}$}
		we need a local estimate
		because we can not a priori exlude negative or vanishing values
		for the kernel function~\mbox{$K(\vecx,\vecz)$}
		for far apart points~$\vecx$ and $\vecz$.
		In the case~\mbox{$\frac{1}{2} < \lambda_0^2 < \frac{2}{3}$}
		a localized view will make better constants possible,
		in the end we aim for a universal constant~$C'$.
		}
	we compare
	$K_{\veclambda}$ to a kernel~$K_{\veckappa}$
	with~\mbox{$\kappa_0^2 = \frac{2}{3}$}
	and~\mbox{$\kappa_k^2 = \lambda_k^2 / c$} for~\mbox{$k \in \N$},
	where~\mbox{$c := 3 \, (1-\lambda_0^2) > 1$} is a scaling factor
	such that~\eqref{enumcor:periodic,init} holds for~$\veckappa$ as well.
	Besides, \eqref{enumcor:periodic,sum(k*la2)} is inherited from~$\veclambda$.
	This shows the existence of a constant~%
	\mbox{$0 < \alpha \leq 6 \pi \, \sigma_{\veckappa}
										\leq 6 \pi \, \sigma_{\veclambda}$}
	such that
	\mbox{$K_{\veckappa}(x,0) \geq \exp(-\alpha \, x)$} for~\mbox{$x \geq 0$}.
	Note that
	\begin{align*}
		K_{\veclambda}(x,0)
			&= \lambda_0^2
					+ c \, \left(K_{\veckappa}(x,0) - {\textstyle \frac{2}{3}}\right) \\
			&\geq c \, \exp(-\alpha \, x)
				- \left({\textstyle \frac{2}{3}} c - \lambda_0^2\right) \\
			&= 3 \, (1-\lambda_0^2) \, \exp(-\alpha \, x) - (2 - 3 \lambda_0^2) \,. \\
	\intertext{%
		If we choose~\mbox{$\beta = 4 \alpha$},
		for~\mbox{$x \leq R_0
								:= \min\left\{\alpha^{-1} \, \log \frac{9}{8}, \,
															\frac{1}{2}
											\right\}$}
		we can finish with the estimate}
		K_{\veclambda}(x,0)
			&\geq \exp(- \beta \, x) \,.
	\end{align*}
	Here we used that with \mbox{$\beta > \alpha$} and \mbox{$\gamma > 0$},
	for \mbox{$0 \leq x \leq \alpha^{-1} \, \log\left((1-\alpha/\beta)\,(1+1/\gamma)
																				\right)$}
	the following inequality holds,
	\begin{equation*}
		\exp(- \beta x) \leq (1+\gamma) \, \exp(- \alpha x) - \gamma \,,
	\end{equation*}
	in our case~\mbox{$\beta = 4 \alpha$}
	and~\mbox{$\gamma = 2 - 3 \lambda_0^2 \leq 2$}.
	(A proof for this inequality can be done by observing
	that $\RHS/\LHS$ as a function in~$x$
	is monotonously growing for small~$x>0$.)
	By this, from~\thmref{thm:MCUBperiodic} we obtain the complexity bound
	\begin{align*}
		n^{\ran}(\eps,\Hilbert(K_{\veclambda}^d)
										\hookrightarrow L_{\infty}(\Torus^d),
							\Lall)
			&\leq C \, (1 + \beta^2 - \log 2 R_0)
										\, d \, (1 + \log d) \, \eps^{-2} \\
			&\leq C_1 \, (1 + \alpha^2 + \log \alpha)
										\, d \, (1 + \log d) \, \eps^{-2} \\
			&\leq C_2 \, (1 + \sigma_{\veclambda}^2)
										\, d \, (1 + \log d) \, \eps^{-2} \,.
	\end{align*}
	
	Finally, the constant in the corollary is~%
	\mbox{$C' := \max\{C_0,C_2\}$}.
\end{proof}

\subsubsection{Example: Korobov Spaces}

We apply the above results to unweighted Korobov spaces.
In the framework of this section, these are spaces~%
\mbox{$\Hilbert_r^{\Korobov}(\Torus^d)
				:= \Hilbert_{\veclambda}(\Torus^d)$}
with~\mbox{$\lambda_0 = \sqrt{\beta_0}$}
and~\mbox{$\lambda_k = \sqrt{\beta_1} \, k^{-r}$}
for~\mbox{$k \in \N$}, where~$\beta_0,\beta_1 > 0$.
For integers~\mbox{$r \in \N$},
the Korobov space norm can be given in a natural way
in terms of weak partial derivatives (instead of Fourier coefficients),
in the one-dimensional case we have
\begin{equation*}
	\|f\|_{\Hilbert_r^{\Korobov}(\Torus)}^2
		= \beta_0^{-1} \left| \int_{\Torus} f(x) \rd x \right|^2
			+ \beta_1^{-1} \, (2\pi)^{-2r} \, \|f^{(r)}\|_2^2 \,.
\end{equation*}
The $d$-dimensional case is a bit more complicated,
in a squeezed way, the norm is
\begin{equation*}
	\|f\|_{\Hilbert_r^{\Korobov}(\Torus^d)}^2
		= \sum_{J \subseteq [d]}
				\beta_0^{-(d - \#J)} \, (\beta_1^{-1} \, (2\pi)^{-2r})^{\#J}
					\, \left\| \int_{\Torus^{[d] \setminus J}}
												\Bigl(\prod_{j \in J} \partial_j^r\Bigr)
													f(\vecx) \, \rd \vecx_{[d] \setminus J}
						\right\|_{L_2(\Torus^J)}^2 \,,
\end{equation*}
see Novak and Wo\'zniakowski~\cite[Sec~A.1]{NW08} for details
on the derivation of this representation of the norm.
There one can also find some information
on the historical background concerning these spaces.
It should be pointed out that in the same book tractability
for $L_2$-approximation of Korobov functions based on~$\Lall$
has been studied~\cite[pp.~191--193]{NW08},
in that case randomization does not help a lot.

The condition~\mbox{$r > \frac{1}{2}$}
is necessary and sufficient for the existence of a reproducing kernel
(and the embedding~%
\mbox{$\Hilbert_r^{\Korobov}(\Torus^d)
					\hookrightarrow L_{\infty}(\Torus^d)$}
to be compact),
then
\begin{equation*}
	\sum_{k=1}^{\infty} \lambda_k^2
		= \beta_1 \, \sum_{k=1}^{\infty} k^{-2 r}
		= \beta_1 \, \zeta(2 r)
\end{equation*}
with the Riemann zeta function~$\zeta$.
Assuming
\begin{equation}\label{eq:beta12init}
	\beta_0 + \beta_1 \, \zeta(2 r) = 1 \,,
\end{equation}
the initial error will be constant~$1$ in all dimensions.
Furthermore, with \mbox{$\beta_1 > 0$} we have
the curse of dimensionality for the deterministic setting,
see \thmref{thm:curseperiodic}.

\begin{theorem} \label{thm:Korobov}
	Consider unweighted Korobov spaces~%
	\mbox{$\Hilbert_r^{\Korobov}(\Torus^d) = \Hilbert_{\veclambda}(\Torus^d)$}
	as described above.
	For smoothness~\mbox{$r > \frac{1}{2}$},
	fixing~\mbox{$\beta_0,\beta_1 \geq 0$}
	such that the initial error is constant~$1$ for all dimensions,
	we have polynomial tractability for the uniform approximation with Monte Carlo,
	in detail,
	\begin{multline*}
		n^{\ran}(\eps,\Hilbert_r^{\Korobov}(\Torus^d)
										\hookrightarrow L_{\infty}(\Torus^d),
						\Lall) \\
			\preceq \begin{cases}
								d \, (1 + \log d) \, \eps^{-2}
									\quad&\text{for $r > 1$,} \\
								d \, (1 + (\log d)^3) \, \eps^{-1} \, (1 + (\log \eps^{-1})^2)
									\quad&\text{for $r = 1$,} \\
								d^{1/(r-1/2) - 1} \, (1 + \log d) \, \eps^{-1/(r-1/2)}
									\quad&\text{for $\frac{1}{2} < r < 1$.}
							\end{cases}
	\end{multline*}
	The hidden constant may depend on~$r$.
\end{theorem}
\begin{proof}
	We start with the easiest case~\mbox{$r > 1$}.
	By~\eqref{eq:beta12init} we have~\mbox{$\beta_1 \leq 1/\zeta(2r - 1) < 1$},
	thus we satisfy~\eqref{enumcor:periodic,sum(k*la2)}
	in \corref{cor:periodic-sufficient}
	with
	\begin{equation*}
		\sigma_{\veclambda} = \sum_{k=1}^{\infty} k \, \lambda_k^2
			= \beta_1 \, \sum_{k=1}^{\infty} k^{-(2 r - 1)}
			\leq \zeta(2 r - 1) \,,
	\end{equation*}
	and obtain (with the constant~$C'$ from the corollary)
	\begin{equation*}
		n^{\ran}(\eps,\Hilbert_r^{\Korobov}(\Torus^d)
				\hookrightarrow L_{\infty}(\Torus^d),\Lall)
			\leq C' \, (1 + \zeta(2 r - 1)^2)
							\, \frac{d \, (1 + \log d)}{\eps^2} \,.
	\end{equation*}
	
	For~\mbox{$\frac{1}{2} < r \leq 1$},
	the quantity $\sigma_{\veclambda}$ is infinite.
	Therefore we apply the fundamental Monte Carlo method
	from \propref{prop:Ma91_l2G}
	to a finite dimensional subspace of finite Fourier sums
	up to frequencies~\mbox{$k_j \leq m$} in each dimension.
	With the orthonormal basis~\mbox{$\{\psi_{\veck}\}_{\veck \in \Z^d}$}
	of~$\Hilbert_{\veclambda}(\Torus^d)$, see \eqref{eq:periodicpsi},
	for~$f \in \Hilbert_{\veclambda}(\Torus^d)$ we define
	\begin{equation*}
		f_m
			:= \sum_{\substack{\veck \in \Z^d \\
												|\veck|_{\infty} \leq m}}
						\langle\psi_{\veck},f\rangle_{\Hilbert_{\veclambda}} \, \psi_{\veck} \,.
	\end{equation*}
	Taking a Monte Carlo method~\mbox{$(A_n^{\omega})_{\omega \in \Omega}$}
	with~\mbox{$A_n^{\omega}(f) = A_n^{\omega}(f_m)$},
	we can estimate the error for~\mbox{$f \in \Hilbert_{\veclambda}(\Torus^d)$} by
	\begin{equation*}
		e((A_n^{\omega})_{\omega}, f)
			\leq \|f - f_m\|_{\infty}
				+ e((A_n^{\omega})_{\omega}, f_m) \,.
	\end{equation*}
	For this term to be bounded from above by~$\eps$,
	we desire both summands to be bounded from above by~$\eps/2$.
	
	By the worst case error formula \eqref{eq:RKHSworUB},
	together with the kernel representation~\eqref{eq:K_la},
	for~\mbox{$\|f\|_{\Hilbert_{\veclambda}} \leq 1$} it easily follows
	\begin{equation*}
		\|f - f_m\|_{\infty}^2
			\,\leq\, 1 - \sum_{\substack{\veck \in \Z^d \\
																|\veck|_{\infty} \leq m}}
									\lambda_{\veck}^2
			\,=\, 1 - \left(\sum_{k=0}^m \lambda_k^2 \right)^d \,.
	\end{equation*}
	In our particular situation with~\mbox{$\lambda_k^2 = \beta_1 \, k^{-2r}$}
	for~\mbox{$k \in \N$},
	we can estimate
	\begin{equation*}
		\sum_{k = m+1}^{\infty} \lambda_k^2
			\leq \beta_1 \int_m^{\infty} t^{-2s} \rd t
			= \frac{\beta_1}{2r - 1} \, m^{-(2r - 1)} \,.
	\end{equation*}
	Hence, together with~\mbox{$\sum_{k=0}^{\infty} \lambda_k^2 = 1$}, we obtain
	\begin{align*}
		\|f - f_m\|_{\Hilbert_{\veclambda}}^2
			&\leq 1 - \left(1 - \frac{\beta_1}{2r - 1} \, m^{-(2r - 1)}\right)^d \\
		\text{[for ${\textstyle \frac{\beta_1}{2r - 1} \, m^{-(2r - 1)}
									< \frac{1}{2}}$]}\quad
			&\stackrel{(\ast)}{\leq}
				1 - \exp\left(- \log 2 \, \frac{\beta_1}{r - 1/2} \, d
															\, m^{-(2r - 1)}
										\right) \\
			&\leq \log 2 \, \frac{\beta_1}{r - 1/2} \, d \, m^{-(2r - 1)} \,.
	\end{align*}
	Here we used~\mbox{$1-t > \exp(-2 \, (\log 2) \, t)$}
	for~\mbox{$x \in (0,\frac{1}{2})$}.
	Choosing
	\begin{equation} \label{eq:Korobov,k_max}
		m := \left\lceil
						\left(4 \, (\log 2) \, \frac{\beta_1}{r - 1/2}
										\, d \, \eps^{-1}
						\right)^{1/(2r-1)}
				\right\rceil \,,
	\end{equation}
	step~$(\ast)$ is actually valid,
	and we bound~\mbox{$\|f-f_m\|_{\infty} \leq \eps/2$}.
	
	For the error analysis of~\mbox{$e((A_n^{\omega})_{\omega}, f_m)$},
	we need to understand the restricted approximation problem
	\begin{equation*}
		\App: \Hilbert_{\veclambda'}(\Torus^d)
						\hookrightarrow L_{\infty}(\Torus^d) \,,
		\quad\text{where~\mbox{$\lambda_k' := \lambda_k \, \ind[k \leq m]$}.}
	\end{equation*}
	The method $A_n$ shall be the fundamental Monte Carlo approximation method
	from \propref{prop:Ma91_l2G} applied to this problem.
	The initial error is smaller than~$1$,
	so we cannot apply \corref{cor:periodic-sufficient} directly to this problem.
	Therefore, consider another space~\mbox{$\Hilbert_{\veckappa}(\Torus^d)$}
	with~\mbox{$\kappa_k = \lambda_k'$} for~\mbox{$k \in \N$}
	and \mbox{$\kappa_0 := \sqrt{1 - \sum_{k = 1}^m \lambda_k^2} > \lambda_0$}.
	The initial error of the approximation problem
	\begin{equation*}
		\App: \Hilbert_{\veckappa}(\Torus^d)
						\hookrightarrow L_{\infty}(\Torus^d)
	\end{equation*}
	is then properly normalized by construction.
	Applying \propref{prop:Ma91_l2G}
	to~\mbox{$\Hilbert_{\veclambda'}(\Torus^d)$}
	and~\mbox{$\Hilbert_{\veckappa}(\Torus^d)$}
	means determining the expected $L_{\infty}$-norm
	of the corresponding Gaussian processes
	$\Psi^{(\veclambda',d)}$ and~$\Psi^{(\veckappa,d)}$, respectively.
	For better comparison it is useful to represent these
	via the Fourier basis~\mbox{$\{\varphi_{\veck}\}_{\veck \in \Z^d}$}
	of~\mbox{$L_2(\Torus^d)$},
	\begin{equation*}
		\Psi^{(\veclambda',d)}
			= \sum_{\substack{\veck \in \Z^d \\
												|\veck|_{\infty} \leq m}}
					2^{-|\veck|_0/2} \, \lambda_{\veck} \, X_{\veck} \, \varphi_{\veck} \,,
		\qquad\text{and}\qquad
		\Psi^{(\veckappa,d)}
			= \sum_{\substack{\veck \in \Z^d \\
												|\veck|_{\infty} \leq m}}
					2^{-|\veck|_0/2} \, \kappa_{\veck} \, X_{\veck} \, \varphi_{\veck} \,,
	\end{equation*}
	where the~$X_{\veck}$ are iid standard Gaussian random variables.
	Note that, by construction, \mbox{$\kappa_{\veck} \geq \lambda_{\veck}'$}
	for~\mbox{$k \in \Z^d$},
	so we have \mbox{$\expect \|\Psi^{(\veclambda',d)}\|_{\infty}
											\leq \expect \|\Psi^{(\veckappa,d)}\|_{\infty}$},
	see \lemref{lem:E|sum aXf|}.
	Consequently, complexity bounds from applying~\corref{cor:periodic-sufficient}
	to~\mbox{$\Hilbert_{\veckappa}(\Torus^d)$}
	also hold for~\mbox{$\Hilbert_{\veclambda'}(\Torus^d)$}.
	This gives
	\begin{align}
		n^{\ran}(\eps,\Hilbert_{\veclambda}(\Torus^d)
										\hookrightarrow L_{\infty}(\Torus^d)) \nonumber
			&\leq n^{\ran}(\eps/2,\Hilbert_{\veclambda'}(\Torus^d)
										\hookrightarrow L_{\infty}(\Torus^d)) \nonumber \\
			&\leq n^{\ran}(\eps/2,\Hilbert_{\veckappa}(\Torus^d)
														\hookrightarrow L_{\infty}(\Torus^d)) \nonumber\\
			&\leq 4 \, C' \, (1 + \sigma_{\veckappa}^2)
						\, \frac{d \, (1 + \log d)}{\eps^2} \,. \label{eq:Korroughlaka}
	\end{align}
	It remains to estimate~$\sigma_{\veckappa}$:
	\begin{align*}
		\sigma_{\veckappa}
			&= \sum_{k=1}^{\infty} k \, \lambda_k'^2 \\
			&= \beta_1 \sum_{k=1}^m k^{-(2r - 1)} \\
		[\text{neglect $\beta_1 < 1$}]\qquad
			&\leq 1 + \int_1^m t^{-(2r - 1)} \rd t \\
			&\leq \begin{cases}
						1 + \log m
							\quad& \text{for $r=1$,} \\
						\frac{1}{2(1-r)} \, m^{2(1-r)}
							\quad& \text{for $\frac{1}{2} < r < 1$.}
					\end{cases}
	\end{align*}
	By the choice of~$m$, see \eqref{eq:Korobov,k_max},
	putting this into \eqref{eq:Korroughlaka},
	we obtain the final upper bound with constants that depend on~$r$.
\end{proof}

\begin{remark}[Loss of smoothness]
	\label{rem:SmoothnessLost}
	In the case~$r>1$,
	with the simple approach from~\secref{sec:HilbertPlainMCUB}
	we loose smoothness~$\frac{1}{2}$.
	This stresses the non-interpolatory nature of the method.
	In detail, check that the Gaussian process~$\Psi$ associated
	to~$\Hilbert_r^{\Korobov}$ (for any equivalent norm)
	lies almost surely in~$\Hilbert_s^{\Korobov}$ for~\mbox{$r-s > \frac{1}{2}$},
	and it is almost surely not in~$\Hilbert_s^{\Korobov}$
	for~\mbox{$r-s \leq \frac{1}{2}$}.
	The argument is similar to that in \remref{rem:FundMC}.
\end{remark}

\begin{remark}[On the case of smaller smoothness] \label{rem:Korobov-small r}
	We had some difficulties with the case of smaller
	smoothness~\mbox{$\frac{1}{2} < r \leq 1$}.
	However, there is some indication
	that nevertheless the associated Gaussian process~$\Psi$ is bounded.
	Then we could apply the simple approach from~\secref{sec:HilbertPlainMCUB},
	and that way obtain better complexity bounds than in~\thmref{thm:Korobov}.
	
	Sufficient and necessary conditions on~$\veclambda$
	for boundedness of the univariate Gaussian Fourier series~$\Psi$
	associated to~\mbox{$\Hilbert_{\veclambda}(\Torus)$} are known,
	see Adler~\cite[Thm~1.5]{Adl90}.
	The book of Marcus and Pisier~\cite{MP81} contains a lot more information
	on random Fourier series and could serve as a starting point for further research.
	
	There is a second hint.
	Plots of the one-dimensional kernel~$K_1$
	for~\mbox{$\Hilbert_r^{\Korobov}(\Torus)$}
	nourish the conjecture that for~\mbox{$r > \frac{1}{2}$}
	we can find an estimate
	\begin{description}
		\item[\textnormal{\,\emph{(iii)'}\,}]
			\emph{$K_1(x,z) \geq \exp(-\alpha \, d_{\Torus}(x,z)^p)$
			for~\mbox{$x,z \in \Torus$} with~\mbox{$d_{\Torus}(x,z) \leq R_0$}},
	\end{description}
	with~\mbox{$0 < p < 2r - 1$} and \mbox{$\alpha > 0$}.
	An adapted version of \thmref{thm:MCUBperiodic}
	could give the desired upper bounds.
	
	Anyways, we were mainly interested
	in showing the superiority of Monte Carlo
	approximation over deterministic approximation in terms of tractability,
	and we have been successful
	for the whole range of continuous Korobov functions.
\end{remark}

\begin{remark}[Combined methods]
	The upper bounds in \thmref{thm:Korobov} do not give the optimal order
	of convergence for the approximation of Korobov functions.
	The rate of convergence we can guarantee
	is only~\mbox{$e^{\ran}(n) \preceq n^{-1/2}$} for~\mbox{$r > 1$},
	and it can be arbitrarily bad for low smoothness~$r$ close to~$\frac{1}{2}$.
	From results on similar settings, see Fang and Duan~\cite{FD07},
	one may conjecture a rate of
	something like~\mbox{$e^{\ran}(n) \preceq n^{-r} \, (\log n)^{r d}$}.\footnote{%
		This topic is part of ongoing cooperation
		with Glenn Byrenheid and Dr.~Van Kien Nguyen.}
	Proofs in that direction still rely on Maiorov's discretization technique.
	Maybe one can find bounds with better constants via more direct methods.
	
	In detail, propose an explicit Monte Carlo approximation method~$A_n$
	with the following properties:
	\begin{enumerate}[\quad(1)\;]
		\item \label{enum:ndet}
			The most relevant Fourier coefficients,
			belonging to the indices~\mbox{$I_{\deter} \subset \Z^d$},
			are approximated exactly at the cost of
			evaluating \mbox{$n_{\deter} = \# I_{\deter}$} functionals.
		\item \label{enum:nran}
			Fourier coefficients of medium importance,
			belonging to \mbox{$I_{\ran} \subset \Z^d$}
			are approximated altogether now using the
			fundamental Monte Carlo approximation method from \propref{prop:Ma91_l2G}
			with~\mbox{$n_{\ran} \ll \# I_{\ran}$} random functionals,
			the total information cost is \mbox{$n = n_{\deter} + n_{\ran}$}.
		\item \label{enum:nnot}
			The remaining Fourier coefficients,
			for indices~\mbox{$\Z^d \setminus (I_{\deter} \cup I_{\ran})$},
			are ignored.
	\end{enumerate}
	The effect of truncation~\eqref{enum:nnot} should be estimated with
	methods from \secref{sec:HilbertWorLB}, see also Cobos et al.~\cite{CKS16}
	or Kuo et al.~\cite{KWW08}.
	The Monte Carlo part~\eqref{enum:nran} should be treated
	with methods on Gaussian fields,
	see~\secref{sec:E|Psi|_sup} and Adler~\cite{Adl90},
	or maybe Marcus and Pisier~\cite{MP81}.
	Note that by the deterministic part~\eqref{enum:ndet}
	we will likely have no tensor product structure for the analysis
	of the Monte Carlo part~\eqref{enum:nran}.
\end{remark}

\subsection{Final Remarks on the Initial Error}
\label{sec:HilbertFinalRemarks}

For unweighted tensor product problems as in the above two subsections,
the assumption of a normalized initial error is crucial for the new
approach to work.
If not, that is, if
\begin{equation*}
	e(0,\Hilbert(K_1) \hookrightarrow L_{\infty}(D_1))
		= \sup_{x \in D_1} \sqrt{K_1(x,x)}
		=: 1 + \gamma > 1 \,,
\end{equation*}
we have
\begin{equation*}
	e(0,\Hilbert(K_d) \hookrightarrow L_{\infty}(D_d))
		= (1 + \gamma)^d \,,
\end{equation*}
where~$K_d$ is the product kernel and~\mbox{$D_d := \bigtimes_{j=1}^d D_1$}.
Then for the Gaussian field~$\Psi$ with covariance function~$K_d$ we have
\begin{equation*}
	\expect \|\Psi\|_{\infty}
		\geq \sup_{\vecx \in D_d} \expect |\Psi_{\vecx}|
		= \sqrt{\frac{2}{\pi}} \, (1 + \gamma)^d \,.
\end{equation*}
In this situation \lemref{lem:stdMCapp}
can only give an impractical upper complexity bound
that grows exponentially in~$d$ for fixed~\mbox{$\eps > 0$}.
However, if the constant function~\mbox{$f=1$}
is normalized in~\mbox{$\Hilbert(K_1)$},
but \mbox{$\Hilbert(K_1)$} is non-trivial
and contains more than just constant functions,
then~\mbox{$\sup_{x \in D_1} \sqrt{K_1(x,x)} > 1$},
contrary to our requirements.

Therefore we must accept that the constant function~\mbox{$f = 1$} cannot be
a normalized function in~$\Hilbert(K_d)$ if we want to
break the curse for the $L_{\infty}$-approximation
with the present tools.
This stands in contrast to many other problems:
\begin{itemize}
	\item
		Take the $L_2$~approximation of periodic Korobov spaces,
		see e.g.\ Novak and Wo\'zniakowski~\cite[pp.~191--193]{NW08}.
		The initial error is~$1$ iff the largest singular value is~$1$,
		so it is natural to let the constant function~\mbox{$f = 1$}
		be normalized within the input space~$\Hilbert$,
		hence it lies on the boundary of the input set.
	\item
		Consider multivariate integration over Korobov spaces,
		see for example Novak and Wo\'zniakowski~\cite[Chap~16]{NW10}.
		The integral is actually the Fourier coefficient
		belonging to the constant function.
		The initial error is properly normalized iff
		the constant function~\mbox{$f = 1$} has norm~$1$.
\end{itemize}
I wish to thank my colleagues David Krieg, and Van Kien Nguyen (meanwhile Dr.~rer.~nat.),
for making me aware of the difference to these particular two situations.

\chapter[Approximation of Monotone Functions]{
	 $L_1$-Approximation of Monotone Functions}
\label{chap:monotone}

Concerning the $L_1$~approximation
of $d$-variate monotone functions (and also monotone Boolean functions)
by function values,
in the deterministic setting the \emph{curse of dimensionality} holds,
see Hinrichs, Novak, and Wo\'zniakowski~\cite{HNW11}, and \secref{sec:monoCurse}.
For the randomized setting we still have \emph{intractability},
i.e.\ the problem is \emph{not weakly tractable},
see \secref{sec:monoMCLBs} where we improve known lower bounds,
the new bounds now exhibit a meaningful $\eps$-dependency.
Yet randomization may reduce the complexity significantly,
for any fixed tolerance~$\eps$
the complexity depends exponentially on~$\sqrt{d}$, roughly,
see \secref{sec:monoUBs}
for the analysis of a known algorithm for Boolean functions,
and a new extension to real-valued monotone functions
that is based on a Haar wavelet decomposition.

\section{The Setting and Background} \label{sec:intro}

Within this chapter we mainly consider
the $L_1$-approximation of $d$-variate \emph{monotone} 
functions using function values~$\Lstd$ as information,\footnote{%
	The input set contains different functions that belong to the same
	equivalence class in~\mbox{$L_1[0,1]^d$}.
	Actually, we do not care about these equivalence classes
	but only need the~$L_1$-norm as a seminorm.
	Furthermore, function evaluations are discontinuous functionals,
	but monotonicity provides a regularization to this type of information
	in some other useful ways such that deterministic approximation
	is actually possible.}
\begin{equation*}
	\App : F_{\mon}^d \hookrightarrow L_1([0,1]^d) \,,
\end{equation*}
where the input set
\begin{equation*}
	F_{\mon}^d := \{ f: [0,1]^d \rightarrow [0,1] \mid
						\vecx \leq \tilde{\vecx} \Rightarrow f(\vecx) \leq f(\tilde{\vecx})\}
\end{equation*}
consists of monotonously increasing functions
with respect to the partial order on the domain.
For~\mbox{$\vecx,\tilde{\vecx} \in \R^d$}, the partial order is defined by
\begin{equation}\label{eq:partord}
	\vecx \leq \tilde{\vecx}
		:\Leftrightarrow x_j \leq \tilde{x}_j \text{ for all } j=1,\ldots,d \,.
\end{equation}

This problem is closely related to the approximation of \emph{Boolean} monotone functions
\begin{equation*}
	F_{\Boole}^d
		:= \{ f: \{0,1\}^d \rightarrow \{0,1\} \mid
					\vecx \leq \tilde{\vecx} \Rightarrow f(\vecx) \leq f(\tilde{\vecx})\} \,.
\end{equation*}
One can identify~$F_{\Boole}^d$
with a subclass~\mbox{$\widetilde{F}_{\Boole}^d \subseteq F_{\mon}^d$}
if we split~\mbox{$[0,1]^d$} into $2^d$~subcubes indexed
by~\mbox{$\veci \in \{0,1\}^d$},
\begin{equation} \label{eq:cubes}
	C_{\veci} := \bigtimes_{j = 1}^{d} I_{i_j} \,,
	\quad\text{where $I_0 := [0, {\textstyle \frac{1}{2}})$
							and  $I_1 := [\textstyle{\frac{1}{2}},1]$}.
\end{equation}
Then, for any Boolean function \mbox{$f: \{0,1\}^d \rightarrow \{0,1\}$},
we obtain a subcubewise constant
function~\mbox{$\tilde{f}: [0,1]^d \rightarrow [0,1]$} by
setting \mbox{$\tilde{f}|_{C_{\veci}} := f(\veci)$}.
If~$f$~is monotone then so is~$\tilde{f}$.
The corresponding distance between two Boolean
functions~\mbox{$f_1,f_2 : \{0,1\}^d \rightarrow \{0,1\}$} is
\begin{equation} \label{eq:distBoole}
	\dist(f_1,f_2)
		:= \frac{1}{2^{d}}
					\, \#\{\veci \in \{0,1\}^d \mid
									f_1(\veci) \not= f_2(\veci) \} \,.
\end{equation}
The metric space of Boolean functions shall be named~$G_{\Boole}^d$,
we consider the approximation problem
\begin{equation*}
	\App: F_{\Boole}^d \hookrightarrow G_{\Boole}^d \,.
\end{equation*}
Note that the metric on~$G_{\Boole}^d$ corresponds to the $L_1$-distance of the
associated subcubewise constant functions defined on~\mbox{$[0,1]^d$}.
Another way to think of the metric on~$G_{\Boole}^d$ is
as the induced metric for~$G_{\Boole}^d$ as a subset of the
Banach space~$L_1(\Uniform\{0,1\}^d)$.\footnote{%
	This property is the reason why it is convenient to consider
	real-valued monotone functions with range~$[0,1]$.
	It is only in \secref{sec:monoUBs}
	that we switch to the range~\mbox{$[-1,+1]$}
	because there we use linear approximation methods.}

Approximation of monotone functions is not a linear problem
becasue the set~$F_{\mon}^d$
is not symmetric:
For non-constant functions~$f \in F_{\mon}^d$,
the negative~$-f$ is not contained in~$F_{\mon}^d$
as it will be monotonously decreasing.
The monotonicity assumption is very different
from common smoothness assumptions, yet it implies many other nice properties,
see for example Alberti and Ambrosio~\cite{AA99}.
Integration and Approximation of monotone functions has been studied
in several papers~\cite{HNW11,No92mon,Papa93}.
Monotonicity can also be an assumption for statistical problems~\cite{GW07,RM51}.
Similarly, a structural assumption could be convexity (more generally: $k$-monotonicity),
numerical problems with such properties have been studied for example
in~\cite{CDV09,HNW11,KNP96,Kop98,NP94}.

Within this research, Boolean monotone functions are considered
in order to obtain lower bounds for the Monte Carlo approximation
of real-valued monotone functions, see \secref{sec:monoMCLBs}.
We will show that the approximation of monotone (Boolean) functions
is \emph{not weakly tractable}, i.e.\ \emph{intractable} in the IBC sense.\footnote{%
	In learning theory there exist many similar sounding notions like
	\emph{weak learnability}, which, however, have different meanings.}
General Boolean functions~\mbox{$f: \{0,1\}^d \rightarrow \{0,1\}$}
are of interest for logical networks and cryptographic applications.
Monotone Boolean functions in particular
constitute a widely studied topic in computer science and discrete mathematics
with connections to graph theory amongst others.
Much research has been done on different
effective ways of exact representation of Boolean functions,
see the survey paper of Korshunov~\cite{Kor03}.
On the other hand, the approximation of monotone Boolean functions
(or subclasses thereof)
is a good example for learning theory~\cite{AM06,BBL98,BT96,KLV94,O'DS07}.
In cryptography one is interested in finding classes of easily representable
Boolean functions (not only monotone functions) that are hard to learn
from examples~\cite{KV94,PV88}, so the motivation is different from
the motivation for tractability studies in IBC.
Different perspectives on similar problems
explain the sometimes unfamiliar way of presenting results
in other scientific communities,
see for instance \secref{sec:monoMCLBs}.

Finally, we give some examples of naturally occurring monotone functions.

\begin{example}[Application of monotone functions]
	Think of a complex technical system with $d$~components.
	Some of these components could be damaged but the system as a whole
	would still work. However, if more critical components fail,
	the system will fail as well.
	It is a natural assumption that, once the system stopped working,
	it will not come back to life after more components break --
	this, in fact, is monotonicity.
	
	We are interested in predicting when the machine will cease to function.
	Our framework fits to a test environment where we can manually deactivate
	components and check whether the system still does its job.
	We are interested in a good (randomized) strategy to test the system.
	This approach could work for the testing of uncritical applications
	where a test environment can be set up.
	
	For rather critical systems such as aircrafts or running systems,
	we need to learn from bad experience (that we actually wish to avoid).
	Different components will have
	different probabilities of failing, and it is with these probabilities
	that we obtain samples from which we can learn.
	On the other hand, not every case is equally important,
	so we judge the approximation of the Boolean function by the
	probability of an event occurring where the prediction fails.
	This situation fits better to the framework that we have for example
	in Bshouty and Tamon~\cite{BT96}.
	(From that paper we know the method
	presented in \secref{sec:BooleanUBs}.)
	
	This picture can be extended to real-valued monotone functions~%
	\mbox{$f: [0,1]^d \rightarrow [0,1]$}.
	Now, each of the $d$~components of a system can work at a different level,
	we are interested in how much this affects the performance
	of the entire system.
	When thinking of a home computer, the question could be how much
	the PC slows down in different situations.
\end{example}

\section{First Simple Estimates}

In view of the order of convergence for the approximation
of real-valued monotone functions, see \secref{sec:MonoOrder},
randomization does not help.
Actually, for small errors~\mbox{$\eps > 0$}
the very simple deterministic algorithm to be found in this section
is the best method we know,
the randomized methods from \secref{sec:monoRealUBs}
will only help for larger~\mbox{$\eps$}.

In \secref{sec:monoCurse}
we cite a result of Hinrichs, Novak, and Wo\'{z}niakowski~\cite{HNW11},
which states that the approximation of real-valued monotone functions
suffers from the curse of dimensionality in the deterministic setting.
A similar statement holds for the approximation of Boolean monotone functions as well.

\subsection{The Classical Approach -- Order of Convergence}
\label{sec:MonoOrder}

The integration problem for monotone functions,
\begin{equation*}
	\Int: F_{\mon}^d \rightarrow \R, \quad
					f \mapsto \int_{[0,1]^d} f \rd\vecx \,,
\end{equation*}
based on standard information~$\Lstd$,
is an interesting numerical problem,
where in the randomized setting
adaption makes a difference for the order of convergence
(at least for~\mbox{$d = 1$}),
but non-adaptive randomization helps only for~\mbox{$d \geq 2$}
to improve the convergence compared to deterministic methods.
In the univariate case Novak~\cite{No92mon} showed
\begin{multline*}
	e^{\ran,\ada}(n,\Int,F_{\mon}^1) \asymp n^{-3/2} \\
		\prec e^{\ran,\nonada}(n,\Int,F_{\mon}^1)
			\asymp e^{\det}(n,\Int,F_{\mon}^1) \asymp n^{-1} \,.
\end{multline*}
Papageorgiou~\cite{Papa93} examined
the integration for $d$-variate monotone functions,
for dimensions~\mbox{$d \geq 2$} we have
\begin{multline*}
	e^{\ran,\ada}(n,\Int,F_{\mon}^d) \asymp n^{-\frac{1}{d} - \frac{1}{2}} \\
	\preceq
	e^{\ran,\nonada}(n,\Int,F_{\mon}^d) \preceq n^{- \frac{1}{2d} - \frac{1}{2}}\\
	\prec
	e^{\deter}(n,\Int,F_{\mon}^d) \asymp n^{-\frac{1}{d}} \,,
\end{multline*}
where the hidden constants depend on~$d$.
It is an open problem to find lower bounds for the non-adaptive Monte Carlo error
that actually show that adaption is better for~\mbox{$d \geq 2$} as well,
but from the one-dimensional case we conjecture it to be like that.

For the $L_1$-approximation,
the order of convergence does not reveal any differences
between the various algorithmic settings.
Applying Papageorgiou's proof technique to the problem of $L_1$-approximation,
we obtain the following theorem.
\begin{theorem}\label{thm:MonAppOrderConv}
	For the $L_1$-approximation of monotone functions, for fixed dimension~$d$
	and \mbox{$n\rightarrow \infty$}, we have the following asymptotic behaviour,
	\begin{equation*}
		e^{\ran}(n,\App,F_{\mon}^d)
			\asymp e^{\deter}(n,\App,F_{\mon}^d)
			\asymp n^{-\frac{1}{d}} \,.
	\end{equation*}
	This holds also for varying cardinality.
\end{theorem}
\begin{proof}
	We split~\mbox{$[0,1]^d$} into $m^d$~subcubes indexed
	by~\mbox{$\veci \in \{0,1,\ldots,m-1\}^d$}:
	\begin{equation*} 
		C_{\veci} := \bigtimes_{j = 1}^{d} I_{i_j}
	\end{equation*}
	where~\mbox{$I_i := [{\textstyle \frac{i}{m}}, {\textstyle \frac{i+1}{m}})$}
	for~\mbox{$i=0,1,\ldots,m-2$}
	and~\mbox{$I_{m-1} := [\textstyle{\frac{m-1}{m}},1]$}.
	
	For the lower bounds, we consider fooling functions
	that are constant on each of the subcubes, in detail,
	\begin{equation*}
		f|_{C_{\veci}} = \frac{|\veci|_1+\delta_{\veci}}{d(m-1)+1}
	\end{equation*}
	with $\delta_{\veci} \in \{0,1\}$
	and \mbox{$|\veci|_1 := i_1 + \ldots + i_d$}.
	Obviously, such functions are monotonously increasing.
	For a deterministic algorithm using $n < m^d$ function values,
	when applied to such a function,
	we do not know the function on at least~\mbox{$m^d-n$}~subcubes.
	There exist a maximal function~$f_{+}$ and a minimal function~$f_{-}$
	fitting to the computed function values, 
	and the diameter of information is at least
	\begin{equation*}
		\|f_{+} - f_{-}\|_1
			\geq \left(1 - \frac{n}{m^d}\right) \, \frac{1}{d(m-1)+1} \,.
	\end{equation*}
	Hence we have the error bound
	\begin{equation*}
		e^{\deter}(n,\App,F_{\mon}^d)
			\geq \frac{1}{2} \, \left(1 - \frac{n}{m^d}\right) \, \frac{1}{d(m-1)+1} \,.
	\end{equation*}
	
	For Monte Carlo lower bounds we switch to the average case setting,
	the measure $\mu$ may be described by such functions
	where the~$\delta_{\veci}$ are
	independent Bernoulli variables
	with~\mbox{$\mu\{\delta_{\veci}=0\}
								= \mu\{\delta_{\veci}=1\}
								= \frac{1}{2}$}.
	For any information~$\vecy$,
	let \mbox{$I^{\vecy} \subset \{0,\ldots,m-1\}^d$}
	be the set of indices~$\veci$
	where we do not know anything about the function
	on the corresponding subcube~$C_{\veci}$.
	Again,~\mbox{$\# I^{\vecy} \geq m^d-n$},
	and for any output~\mbox{$g:[0,1]^d \rightarrow [0,1]$}
	we have the following estimate on the local average error
	with respect to the conditional distribution~$\mu_{\vecy}$:
	\begin{multline*}
		\int \|f - g\|_1 \, \mu_{\vecy}(\diff f)\\
			\geq \sum_{\veci \in I^{\vecy}} \frac{1}{2}
				\int_{C_{\veci}}
					\underbrace{\left(\Bigl|\frac{|\veci|_1}{d(m-1)+1}
																	- g(\vecx)
														\Bigr|
													 +\Bigl|\frac{|\veci|_1 + 1}{d(m-1)+1}
																	- g(\vecx)
														\Bigr|
											\right)
										}_{\textstyle \geq \frac{1}{d\,(m-1)+1}}
						\rd \vecx\\
			\geq \frac{1}{2} \, \left(1 - \frac{n}{m^d}\right) \, \frac{1}{d(m-1)+1} \,.
	\end{multline*}
	Averaging over the information~$\vecy$, we obtain the same
	lower bound in this average setting as in the worst case setting,
	by virtue of \propref{prop:Bakh} (Bakhvalov's technique),
	this is a lower bound for the Monte Carlo error.
	Note that this lower bound is an estimate for the conditional error
	by a convex function~%
	\mbox{$\hat{\eps}(\bar{n})
						:= c_{m,d} \, (1 - \bar{n}/m^d)_+$},
	notably it holds for methods with varying cardinality~\mbox{$n(\vecy)$},
	putting~\mbox{$\bar{n} = n(\vecy)$}.
	Hence by \lemref{lem:n(om,f)avgspecial} we reason the alike error bounds
	for Monte Carlo methods with varying cardinality.\\
	Choosing~\mbox{$m := \lceil \sqrt[d]{2n} \rceil$},
	we obtain the general lower bound
	\begin{equation*}
		e^{\deter}(n,\App,F_{\mon}^d)
			\geq e^{\ran}(n,\App,F_{\mon}^d)
			\geq \frac{1}{4 \, (d \, \sqrt[d]{2n} + 1)}
			\geq \frac{1}{12 \, d} \, n^{-\frac{1}{d}} \,.
	\end{equation*}
	
	For the upper bounds, we give a deterministic, non-adaptive algorithm
	with cardinality~\mbox{$m^d$},
	i.e.~when allowed to use $n$~function values,
	we choose~\mbox{$m := \lfloor \sqrt[d]{n} \rfloor$}.
	We split the domain into $(m+1)^d$ subcubes as above,
	and define the output~\mbox{$g := A_m^d(f)$} by
	\begin{equation*}
		g|_{C_{\veci}}
			:= \frac{1}{2} \, \left[f\left(\frac{\veci}{m+1}\right)
														+ f\left(\frac{\veci+\ones}{m+1}\right)
												\right] \,,
	\end{equation*}
	here $\ones := (1,\ldots,1) \in \R^d$.
	Without loss of generality,
	we assume that on the boundary of the domain we have
	\begin{equation*}
		f|_{[0,1)^d \setminus (0,1)^d} = 0
			\quad\text{and}\quad
		f|_{[0,1]^d \setminus [0,1)^d} = 1 \,.
	\end{equation*}
	This means that we only need to compute~\mbox{$m^d$} function values
	on a grid in the interior~\mbox{$(0,1)^d$} of the domain.
	For each subcube we take the medium possible value based on our knowledge
	on the function~$f$ in the lower and upper corners of that particular subcube.
	When analysing this algorithm,
	we group the subcubes into diagonals collected by index sets
	\begin{equation*}
		D_{\vecj}
			:= \{ \vecj + k \, \ones \in \{0,\ldots,m\}^d \mid k \in \Z \} \,.
	\end{equation*}
	There are~\mbox{$(m+1)^d - m^d \leq d \, (m+1)^{d-1}$} such diagonals,
	each of them can be tagged by exactly
	one index~\mbox{$\vecj \in \{0,\ldots,m\}^d \setminus \{1,\ldots,m\}^d$}.
	By monotonicity we have the following estimate for the error:
	\begin{multline*}
		e(A_m^d,f) = \|f - g\|_1 \\
			\leq \frac{1}{(m+1)^d}
							\, \sum_{\vecj \in \{0,\ldots,m\}^d \setminus \{1,\ldots,m\}^d}
								\,	\underbrace{\sum_{\veci \in D_{\vecj}}
																\frac{1}{2} \,
																	\left[f\left(\frac{\veci+\ones}{m+1}\right)
																				- f\left(\frac{\veci}{m+1}\right)
																	\right]
															}_{= \frac{1}{2}} \\
			\leq \frac{d}{2(m+1)}
			\,=\, \frac{d}{2(\lfloor \sqrt[d]{n} \rfloor + 1)}
			\,\leq\, \frac{d}{2} \, n^{-\frac{1}{d}} \,.
	\end{multline*}
\end{proof}

\begin{remark}[On the impracticality of these results]
	\label{rem:MonAppOrderConv}
	The above proof yields the explicit estimate
	\begin{equation*}
		\frac{1}{12 \, d} \, n^{-\frac{1}{d}}
			\leq e^{\ran}(n,\App,F_{\mon}^d)
			\leq e^{\deter}(n,\App,F_{\mon}^d)
			\leq \frac{d}{2} \, n^{-\frac{1}{d}} \,.
	\end{equation*}
	At first glance, this estimate looks quite nice, with constants
	differing only polynomially in~$d$.
	This optimistic view, however, collapses dramatically
	when switching to the notion of $\eps$-complexity
	for~\mbox{$0<\eps<\frac{1}{2}$}:
	\begin{equation*}
		\left(\frac{1}{12 \, d}\right)^d \, \eps^{-d}
			\leq n^{\ran}(\eps,\App,F_{\mon}^d)
			\leq n^{\deter}(\eps,\App,F_{\mon}^d)
			\leq \left(\frac{d}{2}\right)^d \, \eps^{-d} \,.
	\end{equation*}
	Here, the constants differ superexponentially in~$d$.
	Of course, lower bounds for low dimensions also hold for higher
	dimensions,
	so given the dimension~$d_0$, one can optimize over~\mbox{$d=1,\ldots,d_0$}.
	Still, the upper bound is impractical for high dimensions since it
	is based on algorithms that use exponentially~(in~$d$) many function values.
	
	In fact, for the deterministic setting we cannot avoid a bad~$d$-dependency,
	as the improved lower bounds of \secref{sec:monoCurse} below show.
	For the randomized setting, however,
	we can significantly reduce the $d$-dependency (which is still high),
	at least as long as~$\eps$ is fixed,\footnote{%
		For small \mbox{$\eps \preceq 1/\sqrt{d}$} however,
		the best known method is the deterministic method
		from \thmref{thm:MonAppOrderConv} above,
		see \remref{rem:monoMCUBeps}.}
	see \secref{sec:monoRealUBs}.
	To summarize, if we only consider the order of convergence,
	we might think that randomization does not help,
	but for high dimensions randomization actually \emph{does} help.
\end{remark}

\subsection{Curse of Dimensionality in the Deterministic Setting}
\label{sec:monoCurse}

Hinrichs, Novak, and Wo\'{z}niakowski~\cite{HNW11} have shown
that the integration
(and hence also the $L_p$-approximation, \mbox{$1 \leq p \leq \infty$})
of monotone functions suffers from the curse of dimensionality
in the deterministic setting.
We want to recap their result for our particular situation.

\begin{theorem}[Hinrichs, Novak, Wo\'{z}niakowski 2011] \label{thm:HNW11}
	The $L_1$-approximation of monotone functions
	suffers from the curse of dimensionality in the worst case setting.
	In detail,
	\begin{equation*}
		e^{\deter}(n,\App,F_{\mon}^d,\Lstd)
			\geq {\textstyle \frac{1}{2}} \, \left(1 - n \, 2^{-d}\right) \,,
	\end{equation*}
	so for~\mbox{$0<\eps<\frac{1}{4}$}
	we have
	\begin{equation*}
		n^{\deter}(\eps,\App,F_{\mon}^d,\Lstd)
			\geq 2^{d-1} \,.
	\end{equation*}
\end{theorem}
\begin{proof}
	Let~$N$ be any adaptive information mapping.
	We consider functions~$f$ for which we obtain the same information~%
	\mbox{$\vecy = N(f) = (f(\vecx_1),\ldots,f(\vecx_n))$}
	as for the diagonal split function~%
	\mbox{$
					\ind\left[|\vecx|_1 \geq \frac{d}{2}\right]$}.
	Consequently, such functions will be evaluated
	at the same points~\mbox{$\vecx_1,\ldots,\vecx_n \in [0,1]^d$},
	and
	\begin{equation*}
		f(\vecx_i)
			= \begin{cases}
					0 &\quad\text{if $|\vecx_i|_1 < \frac{d}{2}$,}\\
					1 &\quad\text{if $|\vecx_i|_1 \geq \frac{d}{2}$,}
				\end{cases}
	\end{equation*}
	for~\mbox{$i=1,\ldots,n$}.
	Having this information, there are two areas,
	\begin{align*}
		D_0 &:= \{\vecx \in [0,1]^d \mid
						\exists \, \vecx_i \geq \vecx, f(\vecx_i) = 0 \}
					= \bigcup_{i:f(\vecx_i) = 0} \llbracket\zeros,\vecx_i\rrbracket\,,
				\quad\text{and} \\
		D_1 &:= \{\vecx \in [0,1]^d \mid
						\exists \, \vecx_i \leq \vecx, f(\vecx_i) = 1 \}
					= \bigcup_{i:f(\vecx_i) = 1} \llbracket\vecx_i,\ones\rrbracket\,,
	\end{align*}
	where we know the function for sure:
	\mbox{$f|_{D_0} = 0$} and \mbox{$f|_{D_1} = 1$}.
	The minimal monotone function fitting to that information is
	\mbox{$f_{-}(\vecx) := \ind[\vecx \notin D_0]$},
	the maximal function is
	\mbox{$f_{+}(\vecx) := \ind[\vecx \in D_1]$}.
	Their $L_1$-distance is
	\begin{equation*}
		\|f_{+} - f_{-}\|_1
			= \lambda^d\left([0,1]^d \setminus (D_0 \cup D_1)\right) \,,
	\end{equation*}
	so we have the error bound
	\begin{equation} \label{eq:CurseMonD0D1estimate}
		e(\phi \circ N,F_{\mon}^d)
			\geq {\textstyle \frac{1}{2}}
							\, \lambda^d\left([0,1]^d \setminus (D_0 \cup D_1)\right)\,,
	\end{equation}
	no matter which output~$\phi$ we choose.
	For~\mbox{$|\vecx_i|_1 < \frac{d}{2}$} we have 
	\begin{equation*}
		\lambda^d(\llbracket\zeros,\vecx_i\rrbracket)
			= \prod_{j=1}^d \vecx_i(j)
			\stackrel{\text{AM-GM ineq.}}{\leq}
				\left(\frac{|\vecx_i|_1}{d}\right)^d
			< 2^{-d} \,.
	\end{equation*}
	For~\mbox{$|\vecx_i|_1 \geq \frac{d}{2}$} we have a similar estimate,
	\begin{equation*}
		\lambda^d(\llbracket\vecx_i,\ones\rrbracket)
			= \prod_{j=1}^d (1-\vecx_i(j))
			\stackrel{\text{AM-GM ineq.}}{\leq}
				\left(\frac{|\ones - \vecx_i|_1}{d}\right)^d
			\leq 2^{-d} \,.
	\end{equation*}
	Consequently,
	\begin{equation*}
		\lambda^d\left([0,1]^d \setminus (D_0 \cup D_1)\right)
			\geq 1 - n \, 2^{-d} \,,
	\end{equation*}
	which together with~\eqref{eq:CurseMonD0D1estimate} finishes the proof.
\end{proof}

For the integration the proof follows exactly the same lines.
While for integration the standard Monte Carlo method easily achieves
strong polynomial tractability,
for the approximation we still have intractability in the randomized setting,
see \secref{sec:monoMCLBs}, yet the curse of dimensionality is broken,
see \secref{sec:monoRealUBs}.

With slight modifications an analogue lower bound
on the deterministic approximation of Boolean monotone functions is found.
\begin{theorem}
	For the approximation of monotone Boolean functions we have
	\begin{equation*}
		e^{\deter}(n,\App,F_{\Boole}^d,\Lstd)
			\geq \frac{1}{2} \, \left(1 - n \, 2^{-\lfloor d/2 \rfloor}\right) \,.
	\end{equation*}
	Hence, for~\mbox{$0<\eps<\frac{1}{4}$},
	the $\eps$-complexity is bounded from below by
	\begin{equation*}
		n^{\deter}(\eps,\App,F_{\Boole}^d,\Lstd)
			\geq 2^{\lfloor d/2 \rfloor - 1} \,.
	\end{equation*}
	In particular, the approximation of monotone Boolean functions
	suffers from the curse of dimensionality in the worst case setting.
\end{theorem}
\begin{proof}
	Similarly to the proof of \thmref{thm:HNW11},
	we consider fooling functions with
	\begin{equation*}
		f(\vecx_i)
			= \begin{cases}
					0 &\quad\text{if $|\vecx_i|_1 < \frac{d}{2}$,}\\
					1 &\quad\text{if $|\vecx_i|_1 \geq \frac{d}{2}$.}
				\end{cases}
	\end{equation*}
	For~\mbox{$|\vecx_i|_1 < \frac{d}{2}$} we have the estimate
	\mbox{$\#\{\vecx \in \{0,1\}^d \mid \vecx \leq \vecx_i\}
					\leq 2^{\lceil d/2 \rceil - 1}$},
	for~\mbox{$|\vecx_i|_1 \geq \frac{d}{2}$} it holds
	\mbox{$\#\{\vecx \in \{0,1\}^d \mid \vecx \geq \vecx_i\}
					\leq 2^{\lfloor d/2 \rfloor}$}.
	The remaining steps are analogous.
\end{proof}

\begin{remark}[On the initial error] \label{rem:monoInit}
	For the class~$F_{\mon}^d$ of real-valued monotone functions,
	it is easy to see that the initial error is~$\frac{1}{2}$
	with the initial guess being the constant~$\frac{1}{2}$-function.
	
	For Boolean functions, since there is no $\frac{1}{2}$-function,
	we need to exploit monotonicity properties in order to show that the
	initial error is~$\frac{1}{2}$ nevertheless.
	Let the initial guess be~\mbox{$g(\vecx) := x_1$},
	indeed~\mbox{$g \in F_{\Boole}^d$}.
	For any~\mbox{$f \in F_{\Boole}^d$} we have
	\begin{align*}
		\#\{\vecx \in \{0,1\}^d \mid f(\vecx) \not= x_1\}
			&= \#\{\vecx \mid f(\vecx) = 1, x_1 = 0\}
					+ \#\{\vecx \mid f(\vecx) = 0, x_1 = 1\} \\
		\text{[monotonicity]}\quad
			&\leq \#\{\vecx \mid f(\vecx) = 1, x_1 = 1\}
					+ \#\{\vecx \mid f(\vecx) = 0, x_1 = 1\} \\
			&= 2^{d-1} \,,
	\end{align*}
	from which we conclude~\mbox{$\dist(f,g) \leq \frac{1}{2}$}.
	
	In fact, there are several functions that are equally suitable
	for an initial guess, a more canonical function
	without bias to one single coordinate could be
	the diagonal split~\mbox{$g(\vecx) := \ind[|\vecx|_1 \geq \frac{d}{2}]$},
	for Boolean monotone functions, however, this only works properly for odd~$d$.
	
	The monotonicity assumption is crucial for us to prove the
	initial error to be independent from the algorithmic setting.
	In contrast to that, the approximation
	of general (non-monotone) Boolean functions with Boolean approximants
	is an interesting problem where it makes a difference
	whether we consider the randomized or the deterministic initial error.
	In the randomized setting, the algorithm that uses no information
	can return the constant~$0$ and the constant~$1$ functions
	with probability~$\frac{1}{2}$ each,
	that way obtaining the initial error~$\frac{1}{2}$.
	In the deterministic setting, however, for any initial guess~$g$,
	the opposite function~\mbox{$f := \neg g$} has a distance~$1$,
	the initial error is~$1$.
	
	Anyways, with the initial error being constant
	independently from the dimension~$d$,
	the problem of the approximation of monotone (Boolean) functions is
	properly normalized.
	In addition, problems of lower dimensions
	are canonically contained in the problem of higher dimensions,
	$F_{\mon}^d$~can be seen as a subset of~$F_{\mon}^{d+1}$
	consisting of all functions that do not depend on the coordinate~$x_{d+1}$.
	These two properties make the class of multivariate monotone functions
	particularly interesting for tractability studies.
\end{remark}

\begin{remark}[A combined lower bound] \label{rem:combiLBdeter}
	While the lower bound of \thmref{thm:MonAppOrderConv} contains
	bad $d$-dependent constants, the bound from \thmref{thm:HNW11}
	does not work for small~$\eps$.
	One could combine the ideas of both proofs in the following way:
	
	Given~$\vecm \in \N^d$,
	split the domain~\mbox{$[0,1]^d$}
	into~\mbox{$\prod \vecm := m_1 \ldots m_d$} sub-cuboids~$C_{\veci}$,
	each of side length~$m_j^{-1}$ in the $j$-th dimension,
	and indexed by~\mbox{$\veci \in \N_0^d$}, \mbox{$i_j \in \{0,\ldots,m_j-1\}$}.
	We consider those monotone functions~$f \in F_{\mon}^d$
	with different ranges on different sub-cuboids,
	\begin{equation*}
		f(\vecx) \in \left[\frac{|\veci|_1}{|\vecm|_1 - d + 1},
											\frac{|\veci|_1 + 1}{|\vecm|_1 - d + 1}
								\right]
			\quad \text{for\, $\vecx \in C_{\veci}$.}
	\end{equation*}
	Note that, by construction, on each of the sub-cuboids,
	\mbox{$f|_{C_{\veci}}$}~%
	can be chosen independently from the values of the function on other sub-cuboids.
	If an algorithm computes~$n_{\veci}$ function values of~$f$ on~$C_{\veci}$,
	we can apply \thmref{thm:HNW11}
	to the approximation of~\mbox{$f|_{C_{\veci}}$}
	by scaling. Altogether, with~\mbox{$n = \sum_{\veci} n_{\veci}$},
	we obtain the lower bound
	\begin{align*}
		e^{\deter}(n,\App,F_{\mon}^d,\Lstd)
			&\geq \frac{1}{|\vecm|_1 - d + 1} \,
				\left[\frac{1}{\prod \vecm} \sum_{\veci} {\textstyle \frac{1}{2}}
					\, \left(1 - n_{\veci} \, 2^{-d}\right) \right] \\
			&= \frac{1}{2(|\vecm|_1 - d + 1)}
					\, \left(1 - \frac{n}{2^d \, \prod \vecm}\right) \,.
	\end{align*}
	
	For~\mbox{$\frac{1}{4(d+1)} \leq \eps \leq \frac{1}{4}$},
	choose~\mbox{$\vecm = (2,\ldots,2,1,\ldots,1)$},
	splitting the domain only for the first~$k$ coordinates,
	\mbox{$k := \lfloor \frac{1}{4 \, \eps} \rfloor - 1
					\in \{0,\ldots,d\}$}.\\
	For~\mbox{$0 < \eps \leq \frac{1}{4(d+1)}$},
	split the domain into $m^d$~subcubes,
	that is, \mbox{$\vecm = (m,\ldots,m)$}
	with~\mbox{$m:= \lfloor\frac{1}{4 \,\eps \, d}
												- \frac{1}{d}
												+ 1
									\rfloor
								\geq 2$}.\\
	By this we obtain the complexity bound
	\begin{equation*}
		n^{\deter}(\eps,\App,F_{\mon}^d,\Lstd)
			\geq \begin{cases}
							2^{d+\lfloor 1/(4 \, \eps) \rfloor-2}
								\quad&\text{for $\eps \in [\frac{1}{4(d+1)},
																									\frac{1}{4}]$,} \\
							2^{d \, (1 + \log_2 \lfloor 1 / (4 \,\eps \, d)\rfloor) - 1}
								\quad&\text{for $\eps \in (0,\frac{1}{4(d+1)}]$.}
						\end{cases}
	\end{equation*}
	Still, the upper bounds in \thmref{thm:MonAppOrderConv}
	for the complexity are superexponential in~$d$, from \remref{rem:MonAppOrderConv}
	we have
	\begin{equation*}
		n^{\deter}(\eps,\App,F_{\mon}^d,\Lstd)
			\leq \exp\left(d \, \log \frac{d}{2\,\eps}
									\right) \,,
	\end{equation*}
	there is a logarithmic gap in the exponent.
	
	This idea of a combined lower bound can also be done for the randomized setting,
	see \remref{rem:combiLBran}.
\end{remark}

\section{Intractability for Randomized Approximation}
\label{sec:monoMCLBs}

\subsection{The Result -- A Monte Carlo Lower Bound}
\label{sec:monoMCLBs-Result}

As we will show in \secref{sec:monoUBs},
for the $L_1$-approximation of monotone functions,
the curse of dimensionality does not hold anymore in the randomized setting.
Within this section, however, we show
that, for fixed~\mbox{$\eps \in (0,\frac{1}{2})$},
the $\eps$-complexity depends at least exponentially on~$\sqrt{d}$
in the randomized setting. Yet worse, the problem is not weakly tractable.

For the proof we switch to an average case setting for Boolean functions,
an idea that has already been used by
Blum, Burch, and Langford~\cite[Sec~4]{BBL98}.\footnote{%
	I wish to thank Dr.~Mario~Ullrich for pointing me to this paper,
	after learning about my first version of a lower bound proof.}
They stated that for~\mbox{$n \geq d$},
and \emph{sufficiently large}~$d$,
we have
\begin{equation} \label{eq:BBL98claim}
	e^{\ran}(n,\App,F_{\Boole}^d,\Lstd)
		\geq \left(1-\exp\left(-{\textstyle \frac{n}{4}}\right)\right)
						\, \left(\frac{1}{2}
										- C \, \frac{\log(6 \, d \, n)}{\sqrt{d}}
							\right) \,,
\end{equation}
with a numerical constant~\mbox{$C > 0$}.
(From their proof the constant~\mbox{$C = 2.9895...$} can be extracted.)
Assuming this result to hold for all~$n$ and~$d$,
one could conclude that
for~\mbox{$n = \lfloor \frac{1}{6 \, d} \, \exp(\sigma \sqrt{d}) \rfloor$}
we have a lower error bound of roughly\footnote{%
	That is, ignoring the prefactor
	$(1-\exp(-\frac{n}{4}))$.}
\mbox{$\eps = \frac{1}{2} - C \sigma$},
where we may choose~\mbox{$0 < \sigma < \frac{1}{2 \, C}$}
for meaningful bounds.
Consequently, such an estimate gives us results
only for~\mbox{$n < \frac{1}{6 \, d} \, \exp(\sqrt{d}/(2 \, C))$}.
With the given value of~$C$, this means that~\mbox{$d = 8799$}
is the smallest dimension for which a positive error bound with~$n=1$
is possible in the first place.\footnote{%
	One more example: $d=39\,168$ is the first dimension for which
	a positive lower bound with~$n\geq d$ is possible.}
Actually, interpreting ``sufficiently large'' as~\mbox{$d \geq 8799$},
one can check their proof and will find out
that all proof steps work in those cases where we have non-trivial error bounds,
see~\remref{rem:1/2-...BBL98} for more information on the original proof.

Note that
\begin{equation*}
	\frac{1}{d} \, \exp(\sigma \, \sqrt{d}) \succ \exp(\sigma' \, \sqrt{d})
		\quad\text{for\, $d \rightarrow \infty$, if $0 < \sigma' < \sigma$.}
\end{equation*}
Therefore, by~\eqref{eq:BBL98claim},
Blum et al.\ were the first to show
that for fixed~\mbox{$\eps \in (0,\frac{1}{2})$}
the Monte Carlo complexity for the approximation
of monotone Boolean functions depends exponentially on~$\sqrt{d}$ at least.

From the IBC point of view the structure of~\eqref{eq:BBL98claim} appears unfamiliar.
Blum et al., however, wanted to show that
if we allow to use algorithms with cardinality~$n$ growing only polynomially
in the dimension~$d$, the error approaches the initial error
with a rate of almost~\mbox{$1/\sqrt{d}$}.
The interest for such a result is \emph{not} motivated
from practically \emph{solving} a problem,
but from proving that a problem is \emph{hard to solve},
and could therefore be used in cryptographic procedures.\footnote{%
	Think of two parties \textbf{A} and \textbf{B} communicating.
	Both have a key, say, they know a high-dimensional Boolean function~$f$.
	In order to check the identity of the other party they ask
	for some values of~$f$.
	That is, \textbf{A}~sends a list of points~\mbox{$\vecx_1,\ldots,\vecx_k$}
	to~\textbf{B}, and \textbf{B}~answers with a list of correct
	values~\mbox{$f(\vecx_1),\ldots,f(\vecx_k)$} in order to approve a message.
	A third party~\textbf{C} with bad intentions
	could intercept this communication and try to learn the function~$f$
	from sample pairs~$(\vecx_i,f(\vecx_i))$.
	This should be as hard as possible for that there is little chance
	that one day \textbf{C} could answer correctly
	to~\mbox{$\vecx_1,\ldots,\vecx_k$}.\\
	Monotone Boolean functions are just one model of such a problem.
	Besides \emph{hardness of learning},
	another criterion is \emph{simplicity of representation} (keeping the key),
	and probably other problems are better for that purpose.\\
	The big issue for complexity theory:
	With technical progress hackers have more capacities for cracking a code,
	with more intense communication they can collect more information.
	But technical progress also gives opportunities to make cryptography more
	complex in order to rule out cyber attacks.}

One objective in this study was to extract the best constants possible
whenever a constant error tolerance~$\eps$ is given.
Moreover, using some different inequalities than Blum et al.,\footnote{%
	Differences within the proof will be indicated in place by footnotes.
	}
it is possible to find a lower complexity bound
that includes the error tolerance~$\eps$ in a way
that we can prove \emph{intractability}
for the Monte Carlo approximation of monotone functions,
see \remref{rem:monMCLBintractable}.
A detailed comment on modifications of the proof needed for the case of
$\eps$~getting close to the initial error,
that is the case Blum et al.~\cite{BBL98} studied,
is made in \remref{rem:1/2-...BBL98}.

\begin{theorem} \label{thm:monotonLB}
	Consider the randomized approximation of monotone Boolean functions.
	There exist constants~\mbox{$\sigma_0,\nu,\eps_0 > 0$}
	and \mbox{$d_0 \in \N$}
	such that for~\mbox{$d \geq d_0$} we have
	\begin{equation*}
		n^{\ran}(\eps_0,\App,F_{\Boole}^d,\Lstd)
			> \nu \, \exp(\sigma_0 \sqrt{d}) \,,
	\end{equation*}
	and moreover, for~\mbox{$0 < \eps \leq \eps_0$}
	and \mbox{$d \geq d_0 \, \left(\eps_0 / \eps\right)^2$}
	we have
	\begin{equation*}
		n^{\ran}(\eps,\App,F_{\Boole}^d,\Lstd)
			> \nu \, \exp\left(\frac{c \, \sqrt{d}}{\eps}\right) \,,
	\end{equation*}
	with $c = \sigma_0 \, \eps_0$ \,.
	
	In particular, for~\mbox{$d \geq d_0 = 100$}
	and \mbox{$\eps_0 = \frac{1}{30}$} we have
	\begin{equation*}
		n^{\ran}({\textstyle \frac{1}{30}},\App,F_{\Boole}^d,\Lstd)
			> 108 \cdot \exp(\sqrt{d}-\sqrt{100}) \,,
	\end{equation*}
	for~\mbox{$d=100$} this means \mbox{$n^{\ran} > 108$}.
	For~\mbox{$0 < \eps \leq \frac{1}{30}$}
	and~\mbox{$d \geq 1/(9 \, \eps^2)$} we have
	\begin{equation*}
		n^{\ran}(\eps,\App,F_{\Boole}^d,\Lstd)
			> 108 \cdot \exp\left(\frac{\sqrt{d}}{30 \, \eps} -\sqrt{100}\right) \,.
	\end{equation*}
	
	All these lower bounds hold for varying cardinality as well.
\end{theorem}

Before we give the proof in \secref{sec:monoMCLBs-Proof}, 
we discuss some consequences of the theorem.

\begin{remark}[Intractability] \label{rem:monMCLBintractable}
	The above theorem shows that the approximation of monotone Boolean functions
	is \emph{not weakly tractable}.
	Indeed,	consider the sequence~\mbox{$(\eps_d)_{d=d_0}^{\infty}$}
	of error tolerances~\mbox{$\eps_d := \eps_0 \, \sqrt{d_0 / d}$}.
	Then
	\begin{equation*}
		\lim_{\eps_d^{-1} + d \rightarrow \infty}
				\frac{\log n(\eps_d,d)}{\eps_d^{-1} + d}
			\geq \lim_{d \rightarrow \infty}
					\frac{\sigma_0 \, d /\sqrt{d_0} + \log \nu
							}{\eps_0^{-1} \sqrt{d/d_0} + d}
			= \frac{\sigma_0}{\sqrt{d_0}} > 0 \,.
	\end{equation*}
	This contradicts the definition of weak tractability,
	see \secref{sec:tractability}.
	
	Actually, this behaviour has already been known since the paper
	of Bshouty and Tamon 1996~\cite[Thm~5.3.1]{BT96},
	however, research on weak tractability has not yet been started
	at that time.\footnote{%
		For the historical background of tractability notions,
		see Novak and Wo\'zniakowski 2008~\cite[p.~9]{NW08}.}
	Their lower bound can be summarized as follows:
	For moderately decaying error
	tolerances~\mbox{$\eps_d \preceq 1/(\sqrt{d} (1+ \log d))$}
	and sufficiently large~$d$, we have
	\begin{equation*}
		n^{\ran}(\eps_d,\App,F_{\Boole}^d,\Lstd)
			\geq c \, 2^d / \sqrt{d} \,,
	\end{equation*}
	with some numerical constant~$c > 0$.
	Interestingly, the proof is based on purely combinatorial arguments,
	without applying Bakhvalov's technique (\propref{prop:Bakh})
	and average case settings.
	Since a function value for Boolean functions may only be~$0$ or~$1$,
	after $n$ oracle calls we have at most~$2^n$ different possible information
	outcomes. Any Boolean function can approximate at most
	\begin{equation*}
		k(\eps,d)
			:= \sum_{l = 0}^{\lfloor \eps \, 2^d \rfloor} \binom{2^d}{l}
			\stackrel{\text{\hyperref[lem:BinomSum]{Lem~\ref*{lem:BinomSum}}}}{\leq}
				2^{ 2^d \, \eps \, \log_2 (\euler / \eps)}
	\end{equation*}
	Boolean functions up to a distance~\mbox{$\eps \in (0,1]$}.
	On the other hand,
	the total number of monotone Boolean functions is known to be
	\begin{equation*}
		\# F_{\Boole}^d \geq 2^{\binom{d}{\lfloor d/2 \rfloor}}
			\geq 2^{2^{d-1} / \sqrt{d}} \,,
	\end{equation*}
	see Korshunov~\cite[Sec~1.1]{Kor03} and the references therein
	for the first inequality,
	and \lemref{lem:(d d/2)} for the second inequality.
	Then for any realization of a Monte Carlo method,
	that is, fixing~$\omega$,
	a portion of at least
	\mbox{$(1 - 2^n \, k(\eps,d) / \#F_{\Boole}^d)$}~Boolean monotone
	functions is at distance more than~$\eps$
	to all of the output functions.
	This can be used to show the existence of poorly approximated functions,
	\begin{align*}
		\sup_{f \in F_{\Boole}^d} \P\{\dist(A_n^{\omega},f) > \eps\}
			&\geq \frac{1}{\# F_{\Boole}^d} \sum_{f \in F_{\Boole}^d}
							\P\{\dist(A_n^{\omega},f) > \eps\} \\
		\text{[Fubini]}\quad
			&= \expect \frac{\#\{f \in F_{\Boole}^d \mid \dist(A_n^{\omega},f) > \eps\}
										}{\# F_{\Boole}^d} \\
			&\geq \left(1 - 2^n \, \frac{k(\eps,d)}{\#F_{\Boole}^d}\right) \,.
	\end{align*}
	Hence we get the error bound
	\begin{equation*}
		e^{\ran}(n,\App,F_{\Boole}^d,\Lstd)
			\geq \eps \,
				\left(1 - 2^n \, \frac{k(\eps,d)}{\#F_{\Boole}^d}\right) \,,
	\end{equation*}
	which implies the complexity bound
	\begin{align*}
		n^{\ran}({\textstyle \frac{\eps}{2}},\App,F_{\Boole}^d,\Lstd)
			&\geq \log_2 \left(\frac{\# F_{\Boole}^d}{2 \, k(\eps,d)}\right) \\
			&\geq \frac{2^{d-1}}{\sqrt{d}}
				- 2^d \, \eps \, \log_2 \left(\frac{\euler}{\eps}\right)
				- 1 \,.
	\end{align*}
	This lower bound only makes sense
	for~\mbox{$\eps \log \eps^{-1} \preceq 1/\sqrt{d}$}
	as~\mbox{$d \rightarrow \infty$}, then it will give a
	complexity bound that is exponential in~$d$.
\end{remark}

\begin{remark}[Real-valued monotone functions]
	\label{rem:monoMCLBs-Realvalued}
	The lower bounds of \thmref{thm:monotonLB} 
	also hold for the problem~\mbox{$\App: F_{\mon}^d \hookrightarrow L_1[0,1]^d$},
	which therefore is intractable as well.
	
	Indeed, the proof of the theorem is done
	by switching to a $\mu$-average case setting
	on the set of monotone Boolean functions~$F_{\Boole}^d$
	(Bakhvalov's technique, see \propref{prop:Bakh}).
	Any measure~$\mu$ on~$F_{\Boole}^d$ can be associated with
	a measure~$\tilde{\mu}$ on the set of subcubewise constant
	functions~\mbox{$\widetilde{F}_{\Boole}^d \subset F_{\mon}^d$}
	that only take the values~$0$ and~$1$.
	Since for the Boolean setting the output function must be Boolean as well,
	in order to prove the equivalence of both problems,
	it remains to show that
	optimal outputs with respect to~$\tilde{\mu}$ are constant~$0$ or~$1$
	on each of the~$2^d$ subcubes of the domain~\mbox{$[0,1]^d$}.
	
	For the $\tilde{\mu}$-average setting on~$F_{\mon}^d$,
	take any deterministic information mapping
	\mbox{$\widetilde{N}:
					f \mapsto \vecy = (f(\vecx_1),\ldots,f(\vecx_n)) \in \{0,1\}^n$}
	using $n$~sample points (possibly adaptively chosen).
	By~\mbox{$\tilde{\mu}_{\vecy}(\cdot)
							= \tilde{\mu}(\cdot \mid \widetilde{N}(f) = \vecy)$}
	we denote the conditional measure,
	and \mbox{$F_{\vecy} := \widetilde{N}^{-1}(\vecy) \subseteq F_{\mon}^d$}
	is the preimage of the information~$\vecy$.
	We only need to consider the cases of information~$\vecy$
	with non-vanishing probability~\mbox{$\tilde{\mu}(F_{\vecy})$}.
	Taking any output mapping~$\tilde{\phi}$,
	the definition of the average error and the law of total probability
	lead to the following representation of the error:
	\begin{equation} \label{eq:r(N,mu)L1}
		\begin{split}
			e(\tilde{\phi} \circ \widetilde{N}, \tilde{\mu})
				&\stackrel{\phantom{\text{Fubini}}}{=}
					\sum_{\vecy} \tilde{\mu}(F_{\vecy})
						\, \int_{F_{\vecy}}
									\bigl\|[\tilde{\phi}(\vecy)](\vecx) - f(\vecx)\bigr\|_{L_1}
								\tilde{\mu}_{\vecy}(\diff f) \\
				&\stackrel{\text{Fubini}}{=}
					\sum_{\vecy} \tilde{\mu}(F_{\vecy})
						\, \int_{[0,1]^d}
							\left[
								\int_{F_{\vecy}}
										\bigl|[\tilde{\phi}(\vecy)](\vecx) - f(\vecx)\bigr|
									\tilde{\mu}_{\vecy}(\diff f)
							\right] \diff \vecx
		\end{split}
	\end{equation}
	Since~\mbox{$f(\vecx) \in \{0,1\}$}, for the integrand we have
	\begin{equation*}
		[\ldots] \, = \, \tilde{\mu}_{\vecy}\{f(\vecx) = 0\} \,
											\bigl| [\tilde{\phi}(\vecy)](\vecx) \bigr|
								+ \tilde{\mu}_{\vecy}\{f(\vecx) = 1\} \,
										\bigl|1 - [\tilde{\phi}(\vecy)](\vecx) \bigr| \,,
	\end{equation*}
	which is minimized for
	\mbox{$[\tilde{\phi}(\vecy)](\vecx)
					:= \ind\bigl[\tilde{\mu}_{\vecy}\{f(\vecx) = 1\}
												\geq \frac{1}{2}
									\bigr]$}.
	This, of course, is a function that is constant~$0$ or~$1$
	on each of the $2^d$~subcubes.
	
	Note that, since the set of Boolean functions is finite,
	by minimax principles
	there exists a measure~$\mu^{\ast}$ on~$F_{\Boole}^d$ such that
	the $\mu^{\ast}$-average error coincides with the Monte Carlo error,
	see Math\'e~\cite{Ma93}.
	By this we have that the problem of approximating Boolean functions
	is strictly easier than the problem of~$L_1$-approximation
	of real-valued functions,
	\begin{equation*}
		e^{\ran}(n,\App,F_{\Boole}^d,\Lstd) = e^{\avg}(n,\App,\mu^{\ast},\Lstd)
			\leq e^{\ran}(n,\App,F_{\mon}^d,\Lstd) \,.
	\end{equation*}
\end{remark}

\subsection{The Proof of the Monte Carlo Lower Bound}
\label{sec:monoMCLBs-Proof}

We start the proof with two preparatory lemmas.

The calculation~\eqref{eq:r(N,mu)L1} in \remref{rem:monoMCLBs-Realvalued}
actually brings us to a direct representation of the best error
possible with the information~$\widetilde{N}$,
\begin{equation*}
	\inf_{\tilde{\phi}} e(\tilde{\phi} \circ \widetilde{N}, \tilde{\mu})
		= \sum_{\vecy} \tilde{\mu}(F_{\vecy})
					\, \int_{[0,1]^d}
									\min\bigl\{\tilde{\mu}_{\vecy}\{f(\vecx) = 0\},
														 \tilde{\mu}_{\vecy}\{f(\vecx) = 1\}\bigr\}
								\diff \vecx \,.
\end{equation*}
We summarize a similar identity for Boolean functions in the following Lemma.
\begin{lemma}[Error for the optimal output] \label{lem:01outputavg}
	Let~$\mu$ be a probability measure on the set~$F_{\Boole}^d$,
	and let~\mbox{$N : F_{\Boole}^d \rightarrow \{0,1\}^n$}
	be any deterministic information mapping.
	Then
	\begin{equation} 
		\inf_{\phi} e(\phi \circ N, \mu)
			= \sum_{\vecy} \mu(F_{\vecy}) \, 2^{-d}
						\sum_{\vecx \in \{0,1\}^d}
									\min\bigl\{\mu_{\vecy}\{f(\vecx) = 0\},
														 \mu_{\vecy}\{f(\vecx) = 1\}\bigr\} \,,
	\end{equation}
	where \mbox{$\mu_{\vecy}(\cdot) = \mu(\,\cdot \mid N(f) = \vecy)$}
	is the conditional measure,
	and \mbox{$F_{\vecy} := N^{-1}(\vecy) \subseteq F_{\Boole}^d$}
	is the preimage of the information~$\vecy$.
\end{lemma}

By \lemref{lem:01outputavg}, for a given measure $\mu$ on~$F_{\Boole}^d$,
the average case analysis reduces to understanding the conditional
measure~$\mu_{\vecy}$. The concept of \emph{augmented information}
allows us to simplify the conditional measure as long as we are concerned
about lower bounds.

\begin{lemma}[Augmented information] \label{lem:AugmentedInfo}
	Consider the general problem~\mbox{$S:F \rightarrow G$}.
	Let~\mbox{$N: F \rightarrow \mathcal{Y}$}
	be an arbitrary measurable information mapping and
	\mbox{$\widetilde{N}: F \rightarrow \widetilde{\mathcal{Y}}$}
	be an \emph{augmented information mapping}, that is,
	for all possible information
	representers~\mbox{$\tilde{y} \in \widetilde{\mathcal{Y}}$}
	there exists an information representer~\mbox{$\vecy \in \mathcal{Y}$}
	such that \mbox{$\widetilde{N}^{-1}(\tilde{y}) \subseteq N^{-1}(\vecy)$}.
	
	With this property, the augmented information allows for smaller errors,
	that means
	\begin{equation*}
		\inf_{\tilde{\phi}} e(\tilde{\phi} \circ \widetilde{N},\mu)
			\leq \inf_{\phi} e(\phi \circ N,\mu) \,.
	\end{equation*}
\end{lemma}
\begin{proof}
	There exists
	a mapping~\mbox{$\psi:\widetilde{\mathcal{Y}} \rightarrow \mathcal{Y}$}
	such that~\mbox{$N = \psi \circ \widetilde{N}$}.
	For any mapping \mbox{$\phi:\mathcal{Y} \rightarrow G$} we can define
	\mbox{$\tilde{\phi} := \phi \circ \psi$} such that
	\mbox{$\phi \circ N = \phi \circ \psi \circ \widetilde{N}
											= \tilde{\phi} \circ \widetilde{N}$}, and so
	\begin{equation*}
		r^{\avg}(\widetilde{N},\mu)
			:= \, \inf_{\tilde{\phi}} e(\tilde{\phi} \circ \widetilde{N},\mu)
			\leq \inf_{\phi} e(\phi \circ N, \mu)
			=: r^{\avg}(N,\mu) \,.
	\end{equation*}
	(The quantity~\mbox{$r^{\avg}(N,\mu)$} is called
	\emph{$\mu$-average radius} of the information~$N$.)
\end{proof}

We are now ready to proof the theorem.
\begin{proof}[Proof of \thmref{thm:monotonLB}.] 
	We will use Bakhvalov's technique (\propref{prop:Bakh}) for lower bounds
	and switch to an average case setting on the set of monotone Boolean functions.
	The proof is organized in seven steps.
	\begin{description}
		\item[{\proofstepref{proof:monoLB1}}:]
			The general structure of the measure~$\mu$ on~$F_{\Boole}^d$.
		\item[{\proofstepref{proof:monoLB2}}:]
			Introduce the augmented information.
		\item[{\proofstepref{proof:monoLB3}}:]
			Estimate the number of points~$\vecx \in \{0,1\}^d$
			for which $f(\vecx)$ is still -- to some extend -- undetermined,
			even after knowing the augmented information.
		\item[{\proofstepref{proof:monoLB4}}:]
			Further specify the measure~$\mu$,
			and give estimates on the conditional probability
			for the event~\mbox{$f(\vecx) = 0$} for the set of still fairly
			uncertain~$\vecx$ from the step before.
		\item[{\proofstepref{proof:monoLB5}}:]
			A general formula for the lower bound.
		\item[{\proofstepref{proof:monoLB6}}:]
			Connect estimates for~$\eps_0$ and~$d_0$ with estimates
			for smaller~$\eps$ and larger~$d$.
		\item[{\proofstepref{proof:monoLB7}}:]
			Explicit numerical values.
	\end{description}
	
	\proofstep{proof:monoLB1}{%
		General structure of the measure~$\mu$.}
	We define a measure $\mu$ on the set of functions
	that can be represented by a randomly drawn
	set~\mbox{$U \subseteq W := \{\vecx \in \{0,1\}^d \mid |\vecx|_1 = t\}$},
	with \mbox{$t \in \N$} being a suitable parameter, and a boundary
	value~\mbox{$b \in \N$}, \mbox{$t \leq b \leq d$}.
	We define~\mbox{$f_U : \{0,1\}^d \rightarrow \{0,1\}$} by
	\begin{equation} \label{eq:f_U}
		f_{U}(\vecx) := \ind\left[(|\vecx|_1 > b)
																		\text{\, or \,}
																	(\exists \vecu \in U:
																	\vecu \leq \vecx)
														\right] \,.
	\end{equation}
	The boundary value~\mbox{$b \in \N$}
	will facilitate considerations in connection with the augmented
	information.\footnote{\label{fn:f_U with b}%
		The proof of Blum et al.~\cite{BBL98} worked without
		a boundary value~$b$ when defining functions~$f_U$ for the measure~$\mu$.
		Then in \proofstepref{proof:monoLB2} one needs to replace
		the first inequality of~\eqref{eq:|V01|bound}.
		They used the Chernoff bound in order to control the size of the set~$V_0$
		of the augmented information~$\tilde{y}$ with high probability,
		\begin{equation*}
			\mu\{\# V_0 \leq {\textstyle \frac{2n}{p}}\}
				\geq 1 - \exp(-{\textstyle \frac{n}{4}}) \,.
		\end{equation*}
		The error bound, see \proofstepref{proof:monoLB5}, then needs
		to be multiplied with this factor, which for large~$n$ is neglectable, though.
		In \proofstepref{proof:monoLB4} we specify~\mbox{$p = \varrho / \binom{a}{t}$},
		so we could use the estimate~%
		\mbox{$\# V_0 \leq n \, \frac{2}{\varrho} \, \binom{a}{t}$}.
		The estimate~\eqref{eq:|V01|bound} we take instead can be compared
		to this using~\eqref{eq:sigma_abt(d)}, we have
		\mbox{$\# V_0 \leq n \, \binom{b}{t}
									\leq n \, \sigma_{\alpha\beta\tau}(d) \, \binom{a}{t}$}.
		For high dimensions we get the rough estimate
		\mbox{$\sigma_{\alpha\beta\tau}(d) \approx \exp((\beta-\alpha)\tau)$}.
		With the numerical values listed in \proofstepref{proof:monoLB7} we have
		\mbox{$\frac{2}{\varrho} = 7.7041...$} versus
		\mbox{$\sigma_{\alpha\beta\tau}(d) = 6.5622...$}.
		In this case our version is slightly better and avoids
		the usage of Chernoff bounds.
		
		For the original result, however, Chernoff bounds prove to
		be useful, see \remref{rem:1/2-...BBL98}.
		In turn, the case of varying cardinality becomes more difficult
		because the situation of \lemref{lem:n(om,f)avgspecial}
		does not apply anymore, compare \proofstepref{proof:monoLB5}.
		}
	We draw~$U$ such that the~\mbox{$f(\vecw)$} are independent
	Bernoulli random variables
	with~\mbox{$p = \mu\{f(\vecw) = 1\} = 1 - \mu\{f(\vecw) = 0\}$}.
	The parameter~\mbox{$p \in (0,1)$}
	will be specified in \proofstepref{proof:monoLB4}.
	
	\proofstep{proof:monoLB2}{%
		Augmented information.}
	Now, for any (possibly adaptively obtained)
	info~\mbox{$\vecy = N(f) = (f(\vecx_1),\ldots,f(\vecx_n))$}
	with~\mbox{$\vecx_i \in \{0,1\}^d$}, we define the augmented information
	\begin{equation} \label{eq:augmented y}
		\tilde{y} := (V_0,V_1),
	\end{equation}
	where~\mbox{$V_0 \subseteq W \setminus U$} and~\mbox{$V_1 \subseteq U$}
	represent knowledge about the instance~$f$ that implies the information~$\vecy$.
	We know \mbox{$f(\vecu) = 0$ for $\vecu \in V_0$},
	and \mbox{$f(\vecu) = 1$ for $\vecu \in V_1$}.
	In detail,
	let $\leq_{\text{L}}$ be the lexicographic order\footnote{%
		Any other total order will be applicable as well.}
	of the elements of~$W$,
	then \mbox{$\min_{\text{L}} V$} denotes the first element
	of a set~$V \subseteq W$ with respect to this order.
	For a single sample~$f(\vecx)$ the augmented oracle reveals the sets
	\begin{equation*}
		\begin{split}
			V_0^{\vecx}
				&:= \begin{cases}
							\emptyset
									&\quad\text{if $|\vecx|_1 > b$,}\\
							\{\vecv \in W \mid \vecv \leq \vecx\}
									&\quad\text{if $f(\vecx) = 0$,}\\
							\{\vecv \in W
								\, \mid \,
									\vecv \leq \vecx
										\text{ and }
											\vecv <_{\text{L}}
													{\textstyle \min_{\text{L}}}
														\{\vecu \in U \, \mid \,
															\vecu \leq \vecx\}
							\}
									&\quad\text{if $f(\vecx) = 1$} \\
									&\quad\text{and $|\vecx|_1 \leq b$,}\\
						\end{cases}\\
			V_1^{\vecx}
				&:= \begin{cases}
							\emptyset
									&\quad\text{if $|\vecx|_1 > b$
															or $f(\vecx) = 0$,}\\
							\{{\textstyle\min_{\text{L}}}
									\{\vecu \in U \, \mid \,
										\vecu \leq \vecx\}
							\}
									&\quad\text{if $f(\vecx) = 1$
															and $|\vecx|_1 \leq b$,}
						\end{cases}
		\end{split}
	\end{equation*}
	and altogether the augmented information is
	\begin{equation*}
		\tilde{y} = (V_0,V_1)
			:=\left(\bigcup_{i=1}^n V_0^{\vecx_i}
							\, , \,
							\bigcup_{i=1}^n V_1^{\vecx_i}
				\right) \,.
	\end{equation*}
	Note that computing~$f(\vecx)$ for~\mbox{$|\vecx|_1 > b$}
	is a waste of information,
	so no algorithm designer would decide to compute such samples.
	Since~\mbox{$\# V_0^{\vecx} \leq \binom{|\vecx|_1}{t}
																	 \leq \binom{b}{t}$}
	for~\mbox{$|\vecx|_1 \leq b$},
	and~\mbox{$\# V_1^{\vecx} \leq 1$},
	we have the estimates\footnotemark[\getrefnumber{fn:f_U with b}]
	\begin{equation}\label{eq:|V01|bound}
		\# V_0 \leq n \, \binom{b}{t} \,,
		\quad\text{and}\quad
		\# V_1 \leq n \,.
	\end{equation}
	
	\proofstep{proof:monoLB3}{%
		Number of points~\mbox{$\vecx \in \{0,1\}^d$}
		where
		\mbox{$f(\vecx)$} is still fairly uncertain.}
	For any point~\mbox{$\vecx \in \{0,1\}^d$} we define the set
	\begin{equation*}
		W_{\vecx} := \{\vecw \in W \mid
												\vecw \leq \vecx\} \,
	\end{equation*}
	of points that are ``relevant'' to~\mbox{$f(\vecx)$}.
	Given an augmented information~\mbox{$\tilde{y} = (V_0,V_1)$},
	we are interested in points where it is not yet clear
	whether~\mbox{$f(\vecx) = 1$} or~\mbox{$f(\vecx) = 0$}.
	In detail, these are points~$\vecx$ where~\mbox{$W_{\vecx} \cap V_1 = \emptyset$},
	for that \mbox{$f(\vecx) = 0$} be still possible.
	Furthermore, \mbox{$W_{\vecx} \setminus V_0$} shall be big enough,
	say \mbox{$\#(W_{\vecx} \setminus V_0) \geq M$} with~\mbox{$M \in \N$},
	so that the conditional probability~%
	\mbox{$p_{\vecx} := \mu_{\tilde{y}}\{f(\vecx) = 1\}$}
	is not too small.
	For our estimates it will be necessary to restrict
	to points~\mbox{$|\vecx|_1 \geq a \in \N$},
	we suppose~\mbox{$t \leq a \leq b$}.
	The set of all these points will be denoted by
	\begin{equation*}
		\begin{split}
			B &:= \{\vecx \in D_{ab} \mid
							W_{\vecx} \cap V_1 = \emptyset ,\,
							\#(W_{\vecx} \setminus V_0) \geq M
						\} \,,\\
			\text{where} \quad
			D_{ab} &:= \{\vecx \in \{0,1\}^d \mid
										a \leq |\vecx|_1 \leq b\}\,.
		\end{split}
	\end{equation*}
	We aim to find a lower bound for the cardinality of $B$.
	
	\proofsubstep{proof:monoLB3.1}{%
		Bounding~$\# D_{a,b}$.}
	Let~\mbox{$a := \lceil \frac{d}{2} + \alpha \frac{\sqrt{d}}{2} \rceil$}
	and~\mbox{$b := \lfloor \frac{d}{2} + \beta \frac{\sqrt{d}}{2} \rfloor$}
	with~\mbox{$\alpha < \beta$},
	then by the Berry-Esseen inequality (on the speed of convergence of the
	Central Limit Theorem),
	see \propref{prop:BerryEsseen}
	and in particular \corref{cor:BerryEsseenBinom},
	we have
	\begin{equation} \label{eq:card(D_ab)}
		\frac{\# D_{ab}}{\# \{0,1\}^d}
			\,=\, \frac{1}{2^d} \sum_{k=a}^{b} \binom{d}{k}
			\,\geq\, \underbrace{\Phi(\beta) - \Phi(\alpha)
													}_{=: C_{\alpha\beta}}
							 - \frac{2 \, C_0}{\sqrt{d}}
			\, =: r_0(\alpha,\beta,d) \,.
	\end{equation}
	Here, $\Phi$ is the cumulative distribution function of standard Gaussian
	variables.\footnote{%
		The original proof of Blum et al.~\cite{BBL98} uses Hoeffding bounds
		\begin{equation*}
			\frac{\#D_{ab}}{\#\{0,1\}^d} \geq 1 - \exp(-\alpha^2) - \exp(-\beta^2) \,.
		\end{equation*}
		The Berry-Esseen inequality enables us to obtain better constants.
		Moreover, for the considerations on small~$\eps$ in \proofstepref{proof:monoLB6},
		we need to take~$\alpha,\beta \rightarrow 0$,
		but this cannot be done with the Hoeffding bound,
		where for~\mbox{$|\alpha|,|\beta| < \sqrt{\log 2}$}
		we would have trivial negative bounds.}
	
	\proofsubstep{proof:monoLB3.2}{%
		The influence of $\vecw \in W$ (in particular~$\vecw \in V_1$).\footnote{%
			This step helps to get better constants, but it becomes essential
			for small~$\eps$ in \proofstepref{proof:monoLB6}.
			For the focus of Blum et al.~\cite{BBL98}
			with $\eps$ being close to the initial error,
			the estimate
			\mbox{$\# (Q_{\vecw} \cap D_{ab}) / \# Q_{\vecw} \leq 1$}
			will be sufficient.}}
	Now, let~\mbox{$t := \lceil \tau \sqrt{d} \rceil$}
	with~\mbox{$\tau > 0$}, and for~\mbox{$\vecw \in W$} define
	\begin{equation*}
		Q_{\vecw} := \{\vecx \in \{0,1\}^d \mid
												\vecw \leq \vecx\} \,,
	\end{equation*}
	this is the set of all points inside the area of influence of~$\vecw$.
	Applying \corref{cor:BerryEsseenBinom} once more, we obtain
	\begin{equation} \label{eq:Q_w-influenced}
		\begin{split}
			\frac{\# (Q_{\vecw} \cap D_{ab})}{\# Q_{\vecw}}
				&= \, \frac{\# \{\vecx \in \{0,1\}^{d-t} \mid
														a-t \leq |\vecx|_1 \leq b-t\}
											}{2^{d-t}}\\
				&= \, \frac{1}{2^{d-t}} \sum_{k=a-t}^{b-t} \binom{d-t}{k}\\
			[\text{\hyperref[cor:BerryEsseenBinom]{Cor~\ref*{cor:BerryEsseenBinom}}}]
			\quad &\leq \,
							\Phi\left(\frac{2b - t}{\sqrt{d-t}}\right)
										- \Phi\left(\frac{2a - t}{\sqrt{d-t}}\right)
										+\frac{2 \, C_0}{\sqrt{d-t}}\\
				[\text{\eqref{eq:Gauss-stretch}, \eqref{eq:Gauss-shift}}]\quad&\leq \,
							\biggl[\underbrace{\Phi\left(\beta-\tau\right)
																- \Phi\left(\alpha - \tau\right)
																}_{=: C_{\alpha\beta\tau}}
											+\underbrace{\left(\frac{1}{\sqrt{2 \pi}} + 2 \, C_0\right)
																	}_{=: C_1}
													\frac{1}{\sqrt{d}}
								\biggr]
								\, \frac{1}{\sqrt{1 - t/d}}\,,
		\end{split}
	\end{equation}
	where~\mbox{$1 \leq 1/\sqrt{1-t/d}
									\leq 1/\sqrt{1 - \tau / \sqrt{d} - 1/d}
									=: \kappa_{\tau}(d)$}
	for~\mbox{$\tau < \sqrt{d} - 1/\sqrt{d}$},
	and this factor converges,
	\mbox{$\kappa_{\tau}(d) \xrightarrow[d\rightarrow\infty]{} 1$}.
	Within the above calculation,
	we exploited that the density of the Gaussian distribution
	is decreasing with growing distance to~$0$,
	in detail, for $t_0<t_1$ and $\kappa \geq 1$ we have
	\begin{multline} \label{eq:Gauss-stretch}
		\underline{\underline{\Phi(\kappa \, t_1) - \Phi(\kappa \, t_0)}}
			\,=\, \frac{1}{\sqrt{2\pi}} \int_{\kappa \, t_0}^{\kappa \, t_1}
				\exp\left(-\frac{t^2}{2}\right) \rd t
			\,=\, \frac{\kappa}{\sqrt{2\pi}} \int_{t_0}^{t_1}
				\exp\left(-\frac{\kappa^2 \, s^2}{2}\right) \rd s \\
			\,\leq\, \frac{\kappa}{\sqrt{2\pi}} \int_{t_0}^{t_1}
				\exp\left(-\frac{s^2}{2}\right) \rd s
			\,=\, \underline{\underline{\kappa \left[\Phi(t_1) - \Phi(t_0)\right]}} \,.
	\end{multline}
	Namely, we took~\mbox{$\kappa = 1/\sqrt{1 - t/d}$}
	which comes from replacing~\mbox{$1/\sqrt{d-t}$} by~\mbox{$1/\sqrt{d}$}.
	Furthermore, we shifted the \mbox{$\Phi$-function},
	knowing its derivative being bounded between~$0$ and~\mbox{$1/\sqrt{2\pi}$},
	so for~\mbox{$t_0<t_1$} and \mbox{$\delta \in \R$} we have
	\begin{equation} \label{eq:Gauss-shift}
		\Bigl|\bigl[\Phi(t_1 + \delta) - \Phi(t_0 + \delta)\bigr]
								-\bigl[\Phi(t_1) - \Phi(t_0)\bigr]
					\Bigr|
			\leq \frac{|\delta|}{\sqrt{2\pi}} \,,
	\end{equation}
	in our case \mbox{$\delta = t / \sqrt{d} - \tau
												\leq 1 / \sqrt{d}$}.
	
	\proofsubstep{proof:monoLB3.3}{%
		The influence of~$V_0$.}
	Markov's inequality gives us
	\begin{equation} \label{eq:V_0Markov}
		\sum_{\vecw \in V_0} \# (Q_{\vecw} \cap D_{ab})
			= \sum_{\vecx \in D_{ab}} \# (W_{\vecx} \cap V_0)
			\,\geq\, N \, \#\{\vecx \in D_{ab} \mid
										\#(W_{\vecx} \cap V_0) \geq N\} \,,
	\end{equation}
	with~\mbox{$N \in \N$}.
	Using this, we can carry out the estimate
	\begin{equation} \label{eq:Markov W_x-V_0>M}
		\begin{split}
			\# \{\vecx \in D_{ab} \mid
						\#(W_{\vecx} \setminus V_0) \geq M \}
				&\,=\, \# \{\vecx \in D_{ab} \mid
										\#(W_{\vecx} \cap V_0) \leq \# W_{\vecx} - M \} \\
				&\,\geq\,
					\# \{\vecx \in D_{ab} \mid
							\#(W_{\vecx} \cap V_0) \leq {\textstyle\binom{a}{t}} - M \} \\
				&\,=\,
					\# D_{ab}
					- \# \{\vecx \in D_{ab} \mid
								\#(W_{\vecx} \cap V_0) > {\textstyle\binom{a}{t}} - M \} \\
			\text{[\eqref{eq:V_0Markov}]}\quad
				&\,\geq\,
					\# D_{ab}
					- \frac{1}{{\textstyle \binom{a}{t}} - M + 1}
							\, \sum_{\vecw \in V_0} \# (Q_{\vecw} \cap D_{ab}) \,.
		\end{split}
	\end{equation}
	
	\proofsubstep{proof:monoLB3.4}{%
		Final estimates on~$\# B$.}
	Putting all this together, we estimate the cardinality of~$B$:
	\begin{equation*}
		\begin{split}
			\frac{\# B}{\#\{0,1\}^d}
				&= \frac{\# \left(\{\vecx \in D_{ab} \mid
														\#(W_{\vecx} \setminus V_0) \geq M \}
													\setminus \bigcup_{\vecw \in V_1} Q_{\vecw}
										\right)
								}{\#\{0,1\}^d} \\
				\text{[\eqref{eq:Markov W_x-V_0>M}, any $\vecw \in W$]}\quad &\geq
					\frac{\# D_{ab}}{\#\{0,1\}^d} \\
				&\qquad - \frac{\# Q_{\vecw}}{\# \{0,1\}^d}
								\, \left(\frac{\# V_0}{\binom{a}{t} - M + 1} + \#V_1\right)
								\, \frac{\# (Q_{\vecw} \cap D_{ab})}{\# Q_{\vecw}}\\
			\text{[\eqref{eq:|V01|bound},
						 \eqref{eq:card(D_ab)},
						 \eqref{eq:Q_w-influenced}]}
			\quad &\geq
				C_{\alpha\beta} - \frac{2 \, C_0}{\sqrt{d}}\\
				&\qquad
				- n \, 2^{-t} \, \left(\frac{\binom{b}{t}}{\binom{a}{t} - M + 1} + 1\right)
					\, \left[C_{\alpha\beta\tau}
										+\frac{C_1}{\sqrt{d}}
						\right]
						\, \kappa_{\tau}(d) \,.
		\end{split}
	\end{equation*}
	We set~\mbox{$M := \lceil \lambda \binom{a}{t} \rceil$}
	with~\mbox{$0<\lambda<1$},
	and provided~\mbox{$t \leq a \leq b$},
	which can be guaranteed for~\mbox{$\alpha - 2\tau > - \sqrt{d} + 2/\sqrt{d}$}
	and \mbox{$\beta-\alpha > 2/\sqrt{d}$},
	we estimate the ratio
	\begin{equation} \label{eq:sigma_abt(d)}
		\begin{split}
			\binom{b}{t}\bigg/\binom{a}{t}
				&\leq \left(\frac{a+1}{a-t+1}\right)^{b-a}
				\leq \left(\frac{\frac{d}{2} + \alpha \frac{\sqrt{d}}{2} + 1
												}{\frac{d}{2} + (\alpha - 2\tau) \frac{\sqrt{d}}{2}}
							\right)^{(\beta-\alpha)\,\sqrt{d}/2}\\
				&\leq \exp\Biggl((\beta - \alpha) \, \tau
												\underbrace{\left(1 + \frac{\alpha - 2 \tau}{\sqrt{d}}
																		\right)^{-1}
																	}_{=: \kappa_{\alpha\tau}(d)}
												+\underbrace{\frac{\beta-\alpha
																					}{\sqrt{d} + \alpha - 2\tau}
																		}_{=: K_{\alpha\beta\tau}(d)}
									\Biggr) =: \sigma_{\alpha\beta\tau}(d) \,,
		\end{split}
	\end{equation}
	where we have~\mbox{$1 \leq \kappa_{\alpha\tau}(d)
														\xrightarrow[d\rightarrow\infty]{} 1$}
	and \mbox{$0 \leq K_{\alpha\beta\tau}(d) \xrightarrow[d\rightarrow\infty]{} 0$}.
	(Note that the above estimate is asymptotically optimal,
	\mbox{$1 \leq \binom{b}{t}/\binom{a}{t}
										\xrightarrow[d\rightarrow\infty]{}
											\exp\left((\beta-\alpha)\,\tau\right)$}.)
	We finally choose the information cardinality~\mbox{$n = \lfloor\nu 2^t \rfloor$},
	and obtain the estimate
	\begin{align}
		\frac{\# B}{\#\{0,1\}^d}
			&\geq \left[C_{\alpha\beta} - \frac{2 \, C_0}{\sqrt{d}} \right]
					- \nu \, \left(\frac{\sigma_{\alpha\beta\tau}(d)
															}{1-\lambda} + 1
										\right)
						\, \left[C_{\alpha\beta\tau}
											+\frac{C_1}{\sqrt{d}}
							\right]
							\, \kappa_{\tau}(d)
				\nonumber\\
			&=: r_0(\alpha,\beta,d) - \nu \, r_1(\alpha,\beta,\tau,\lambda,d)
				\label{eq:|B|>=r0-nu*r1}\\
			&=: r_B(\alpha,\beta,\tau,\lambda,\nu,d) 
				\nonumber\,.
	\end{align}
	With all the other conditions on the parameters imposed before,
	for sufficiently large~$d$ we will have~\mbox{$r_0(\ldots) > 0$}.
	Furthermore, we always have~\mbox{$r_1(\ldots) > 0$},
	so choosing~\mbox{$0 < \nu < r_0(\ldots) / r_1(\ldots)$}
	will guarantee $r_B(\ldots)$ to be positive.
		
	\proofstep{proof:monoLB4}{%
		Specification of~$\mu$ and bounding of conditional probabilities.}
	We specify the measure~$\mu$
	on the set of functions~\mbox{$\{f_U\} \subset F_{\Boole}^d$}
	defined as in~\eqref{eq:f_U} with~\mbox{$U \subseteq W$}.
	Remember that the~\mbox{$f(\vecw)$} (for~\mbox{$\vecw \in W$})
	shall be independent Bernoulli random variables with
	probability~\mbox{$p = \mu\{f(\vecw) = 1\}$}.
	Having the augmented information~\mbox{$\tilde{y} = (V_0,V_1)$},
	the values~\mbox{$f(\vecw)$} are still independent random variables
	with probabilities
	\begin{equation*}
		\mu_{\tilde{y}}\{f(\vecw) = 1\} =
			\begin{cases}
				0 &\text{if $\vecw \in V_0$,} \\
				1 &\text{if $\vecw \in V_1$,} \\
				p &\text{if $\vecw \in W \setminus (V_0 \cup V_1)$.}
			\end{cases}
	\end{equation*}
	Then for~\mbox{$\vecx \in B$} we have the estimate
	\begin{equation*}
		\mu_{\tilde{y}}\{f(\vecx) = 0\} \leq (1 - p)^M
			\leq \exp\left( - p \lambda \binom{a}{t}\right)
			= \exp( - \lambda \varrho) \,,
	\end{equation*}
	where we write~\mbox{$p := \varrho/\binom{a}{t}$}
	with~\mbox{$0 < \varrho < \binom{a}{t}$}.
	The other estimate is
	\begin{multline} \label{eq:q_0}
		\mu_{\tilde{y}}\{f(\vecx) = 0\} \geq (1 - p)^{\binom{b}{t}}
			= \exp\left(\log(1 - p) \, \binom{b}{t}\right) \\
			\geq
				\exp\Biggl(
							- \varrho \, \sigma_{\alpha\beta\tau}(d) \,
								\biggl(
									\underbrace{\frac{1}{2}
															+ \frac{1}{2\,(1- \varrho/\gamma_{\alpha\tau}(d))}
														}_{=: \kappa_{\varrho\gamma}(d)}
								\biggr)
						\Biggr)
			=: q_0(\alpha,\beta,\tau,\varrho,d) \\
			\xrightarrow[d \rightarrow \infty]{}
				\exp\Bigl(-\varrho \,
										\exp\left((\beta-\alpha)\, \tau\right)
						\Bigr)\,,
	\end{multline}
	Here we used that, for $0 \leq p < 1$,
	\begin{equation*}
		0 \geq \log(1-p)
			= - \left(p+\sum_{k=2}^{\infty} \frac{p^k}{k}\right)
			\geq - \left(p+\sum_{k=2}^{\infty} \frac{p^k}{2}\right)
			= -p\,\left(\frac{1}{2} + \frac{1}{2\,(1-p)}\right) \,,
	\end{equation*}
	together with the estimates
	\begin{equation*}
		p \, \binom{b}{t} \leq \varrho \, \sigma_{\alpha\beta\tau}(d) \,,
	\end{equation*}
	and, provided~\mbox{$t\leq a$},
	\begin{equation*}
		\binom{a}{t} \geq \left(\frac{a}{t}\right)^t
			\geq
				\left(\frac{\sqrt{d}+\alpha}{2\left(\tau + 1 /\sqrt{d}\right)}
				\right)^{\tau \sqrt{d}}
			=: \gamma_{\alpha\tau}(d) \,.
	\end{equation*}
	Note that \mbox{$\gamma_{\alpha\tau}(d)
								\xrightarrow[d \rightarrow \infty]{} \infty$}
	implies
	\mbox{$\kappa_{\varrho\gamma}(d)
								\xrightarrow[d \rightarrow \infty]{} 1$}.
	It follows that for~\mbox{$\vecx \in B$},
	\begin{multline} \label{eq:def_q}
		\min\Bigl\{\mu_{\tilde{y}}\{f(\vecx) = 1\},\,
							\mu_{\tilde{y}}\{f(\vecx) = 0\}
				\Bigr\}\\
			\geq \min\left\{1- \exp\left( - \varrho \, \lambda\right), \,
										q_0(\alpha,\beta,\tau,\varrho,d)
								\right\}
			=: q(\alpha,\beta,\tau,\lambda,\varrho,d) \,.
	\end{multline}
	
	\proofstep{proof:monoLB5}{%
		The final error bound.}
	By \lemref{lem:01outputavg}
	and Bakhvalov's technique (\propref{prop:Bakh})
	we obtain the final estimate
	for~\mbox{$n \leq \nu \, 2^{\tau\sqrt{d}} = \nu \, \exp(\sigma\sqrt{d})$},
	where~\mbox{$\sigma = \tau \, \log 2$},
	\begin{align}
		e^{\ran}(n,\App,F_{\Boole}^d,\Lstd)
			&\geq e^{\avg}(n,\App,\mu,\Lstd)
				\nonumber\\
		[\text{any valid $\tilde{y}$}]\quad
			&\geq \frac{\# B}{\#\{0,1\}^d}
				\, \min\bigl\{\mu_{\tilde{y}}\{f(\vecx) = 0\},\,
											\mu_{\tilde{y}}\{f(\vecx) = 1\}
									\mid \vecx \in B
								\bigr\} 
				\nonumber\\
		[\text{\eqref{eq:|B|>=r0-nu*r1} and \eqref{eq:def_q}}]\quad
			&\geq [r_0(\alpha,\beta,\tau) - \nu \, r_1(\alpha,\beta,\tau,\lambda,d)]
							\cdot q(\alpha,\beta,\tau,\lambda,\varrho,d)
				\nonumber\\
			&=: \hat{\eps}(\alpha,\beta,\tau,\lambda,\nu,\varrho,d)
				\label{eq:err>=(r0-nu*r1)*q}\,.
	\end{align}
	Fixing~\mbox{$d = d_0$}, and with appropriate values
	for the other parameters, we can provide~\mbox{$r_B = r_0 - \nu \, r_1 > 0$}.
	The value of $\varrho$~should be adapted for that~$q(\ldots)$ is big
	(and positive in the first place).
	The function \mbox{$\hat{\eps}(\ldots,d)$} is monotonously
	increasing in~$d$, so an error bound for~\mbox{$d=d_0$}
	implies error bounds for~\mbox{$d \geq d_0$} while keeping in particular
	$\nu$ and~$\tau$.
	Clearly, for any~\mbox{$0 < \eps_0 < \hat{\eps}(\ldots)$},
	this gives lower bounds for the $\eps$-complexity,
	\begin{equation*}
		n^{\ran}(\eps_0,\App,F_{\Boole}^d,\Lstd) > \nu \, \exp(\sigma \sqrt{d}) \,.
	\end{equation*}
	
	Note that the definition of the measure does not depend on~$n$.
	Moreover, by the above calculations,
	we have a general lower bound~\mbox{$\hat{\eps}(n(\vecy))$} which holds
	for the conditional error
	in the case of varying cardinality as well.
	This estimate can be seen as a convex function
	in~\mbox{$\bar{n} \geq 0$}, indeed, fixing all parameters
	but~\mbox{$\nu := \bar{n} \, 2^{-\tau\sqrt{d}}$}, we have
	\begin{equation*}
		\hat{\eps}(\bar{n})
			= [r_0 - \bar{n} \, 2^{-\tau\sqrt{d}} \, r_1]_+ \, q \,.
	\end{equation*}
	By \lemref{lem:n(om,f)avgspecial} the lower bounds
	extend to methods with varying cardinality.
	
	\proofstep{proof:monoLB6}{%
		Smaller~$\eps$ and bigger exponent~$\tau$ for higher dimensions.\footnote{%
			These considerations give results of a new quality
			compared to Blum et al.~\cite{BBL98}.}}
	More sophisticated,
	if we have a lower bound~\mbox{%
		$\hat{\eps}(\alpha_0,\beta_0,\tau_0,\lambda,\nu,\varrho,d_0)
			> \eps_0$},
	then for~\mbox{$\tau \geq \tau_0$}
	and \mbox{$d \geq d_0 \, \bigl(\frac{\tau}{\tau_0}\bigr)^2$}
	we obtain the lower bound
	\begin{equation*}
		\hat{\eps}(\alpha(\tau),\beta(\tau),\tau,\lambda,\nu,\varrho,d)
			> \frac{\tau_0}{\tau} \, \eps_0 =: \eps
	\end{equation*}
	with~\mbox{$\alpha(\tau) = \alpha_0 \, \frac{\tau_0}{\tau}$}
	and \mbox{$\beta(\tau) = \beta_0 \, \frac{\tau_0}{\tau}$},
	supposed that in addition we fulfil the conditions~\mbox{$\tau_0 \geq \beta_0$}
	and \mbox{$-\tau_0 \leq \alpha_0 \leq 0$}.
	This gives us a valid estimate
	\begin{equation*}
		n^{\ran}(\eps,\App,F_{\Boole}^d,\Lstd)
			\geq \nu \, 2^{\tau \sqrt{d}}
			= \nu \, 2^{\tau_0 \, \eps_0 \, \sqrt{d} / \eps}
	\end{equation*}
	under the restriction~\mbox{$d \geq d_0 \, \bigl(\frac{\eps_0}{\eps}\bigr)^2$}.
	
	In detail,
	the constraint \mbox{$\tau \leq \tau_0 \, \sqrt{d/d_0}$}
	is needed to contain several correcting terms that occur
	because $a$, $b$, and~$t$ can only take integer values.
	For example, from~\eqref{eq:Q_w-influenced} we have the correcting
	factor~\mbox{$\kappa_{\tau}(d)$}, for which holds
	\begin{equation*}
		1 \leq \kappa_{\tau}(d) = 1/\sqrt{1-\tau/\sqrt{d}-1/d}
			\leq 1/\sqrt{1-\tau_0/\sqrt{d_0}-1/d_0} = \kappa_{\tau_0}(d_0) \,.
	\end{equation*}
	Furthermore, with the choice of~\mbox{$\alpha(\tau)$} and~\mbox{$\beta(\tau)$},
	the product~\mbox{$(\beta-\alpha)\tau = (\beta_0-\alpha_0) \tau_0$}
	is kept constant. This is the key element for the estimate
	\begin{equation*}
		\sigma_{\alpha\beta\tau}(d)
			\leq \sigma_{\alpha_0,\beta_0,\tau_0}(d_0) \,,
	\end{equation*}
	see its definition~\eqref{eq:sigma_abt(d)}.
	For the $d$-dependent correcting terms that occur therein,
	we have
	\mbox{$1 \leq \kappa_{\alpha\tau}(d) \leq \kappa_{\alpha_0 \tau_0}(d_0)$},
	and \mbox{$0 \leq K_{\alpha\beta\tau}(d) \leq K_{\alpha_0 \beta_0 \tau_0}(d_0)$},
	where the assumption~\mbox{$\alpha_0 \leq 0$} comes into play.
	For bounding~\mbox{$\kappa_{\alpha\tau}(d)$},
	we also need \mbox{$\tau \leq \tau_0 \, \sqrt{d/d_0}$}.
	
	Having~\mbox{$\sigma_{\alpha\beta\tau}(d)$} under control,
	one can easily show\footnote{%
		This effectively means examining
		\mbox{$q_0(\alpha(\tau),\beta(\tau),\tau,\varrho,d)$},
		see~\eqref{eq:q_0},
		where in particular we need to show
		\mbox{$\gamma_{\alpha\tau}(d) \geq \gamma_{\alpha_0,\tau_0}(d_0)$},
		which relies on~\mbox{$\alpha_0 \leq 0$}
		and \mbox{$\tau_0 \leq \tau \leq \tau_0 \sqrt{d/d_0}$}.
		}
	\begin{equation*}
		q(\alpha(\tau),\beta(\tau),\tau,\lambda,\varrho,d)
			\geq q(\alpha_0,\beta_0,\tau_0,\lambda,\varrho,d_0) \,,
	\end{equation*}
	and, more complicated,
	\begin{equation*}
		r_B(\alpha(\tau),\beta(\tau),\tau,\lambda,\nu,d)
			\geq \frac{\tau_0}{\tau} \, r_B(\alpha_0,\beta_0,\tau_0,\lambda,\nu,d_0).
	\end{equation*}
	For the latter we need in particular the inequalities
	\begin{equation*} 
		C_{\alpha\beta} \geq \frac{\tau_0}{\tau} \, C_{\alpha_0, \beta_0} \,,
		\quad \text{and} \quad
		C_{\alpha\beta\tau} \leq \frac{\tau_0}{\tau} \, C_{\alpha_0, \beta_0, \tau_0} \,.
	\end{equation*}
	We start with the first inequality, 
	\begin{align*}
		C_{\alpha\beta}
			&= \frac{1}{\sqrt{2\pi}}
						\int_{\alpha}^{\beta}
								\exp\left(-\frac{x^2}{2}\right)
							\rd x\\
		[x = {\textstyle \frac{\tau_0}{\tau}} \, u]\quad
			&= \frac{\tau_0}{\tau \, \sqrt{2\pi}}
						\int_{\alpha_0}^{\beta_0}
								\underbrace{\exp\left(- \left(\frac{\tau_0}{\tau}\right)^2
																				\, \frac{u^2}{2}
																\right)
													}_{[\tau \geq \tau_0]\quad
															\geq \exp\left( - \frac{u^2}{2}\right)}
							\rd u\\
			&\geq \frac{\tau_0}{\tau} \, C_{\alpha_0, \beta_0} \,.
	\end{align*}
	The second inequality is a bit trickier, 
	\begin{align*}
		C_{\alpha\beta\tau}
			&= \frac{1}{\sqrt{2\pi}}
						\int_{\alpha-\tau}^{\beta-\tau}
								\exp\left(-\frac{x^2}{2}\right)
							\rd x\\
		[\text{subst.\ }x + \tau = {\textstyle \frac{\tau_0}{\tau}} \, (u + \tau_0)]\quad
			&= \frac{\tau_0}{\tau \, \sqrt{2\pi}}
						\int_{\alpha_0-\tau_0}^{\beta_0-\tau_0}
								\underbrace{\exp\left(- \frac{1}{2} \,
																					\left(\frac{\tau_0}{\tau}
																									\, (u + \tau_0)
																								- \tau
																					\right)^2
																\right)
													}_{\geq \exp\left( - \frac{u^2}{2}\right)}
							\rd u\\
		[\textstyle{\frac{\tau_0}{\tau} \, (u + \tau_0) - \tau}
				\leq u \leq 0]\quad
			&\leq \frac{\tau_0}{\tau} \, C_{\alpha_0, \beta_0, \tau_0} \,.
	\end{align*}
	Here, \mbox{$u \leq 0$} followed
	from the the upper integral boundary~\mbox{$u \leq \beta_0-\tau_0$}
	and the assumption~\mbox{$\beta_0 \leq \tau_0$}.
	The other constraint,
	\mbox{$\psi(\tau) := \frac{\tau_0}{\tau} \, (u + \tau_0) - \tau
						\leq u$},
	followed from the monotonous decay of~\mbox{$\psi(\tau)$}
	for~\mbox{$\tau \geq \tau_0$},
	taking~\mbox{$\alpha_0 - \tau_0 \leq u$}
	from the lower integral boundary into account,
	and recalling the assumption~\mbox{$\alpha_0 \geq -\tau_0$}.
	
	\proofstep{proof:monoLB7}{%
		Example for numerical values.}
	The stated numerical values result from the setting
	\mbox{$\alpha = -0.33794$}, \mbox{$\beta = 0.46332$},
	\mbox{$\tau = 1.47566 > \frac{1}{\log 2}$} and \mbox{$\lambda = 0.77399$}.
	We adapt \mbox{$\varrho = 0.25960$},
	and for starting dimension~\mbox{$d_0 = 100$} and \mbox{$n_0 = 108$}
	(choosing~$\nu$ appropriately)
	we obtain the lower error bound \mbox{$\hat{\eps}(\ldots)
																					= 0.03333335...
																					> \frac{1}{30} =: \eps_0$}.
\end{proof}

\subsection{Remarks on the Proof}
\label{sec:monoMCLBs-Remarks}

\begin{remark}[On finding parameters for good lower bounds]
	Fixing $d_0$ and~$n_0$, one may vary $\alpha$, $\beta$, $\tau$, and~$\lambda$
	so that the error bound~\mbox{$\hat{\eps}(\ldots)$} is maximized
	(meanwhile adjusting \mbox{$\nu = n_0 \, 2^{-\tau \sqrt{d_0}}$} and~$\varrho$).
	Numerical calculations indicate that
	\mbox{$d_0 = 24$}~is likely to be the first dimension where with~$n_0 = 1$
	we can obtain a positive error bound
	(which is at around~\mbox{$\eps_0 \approx 10^{-6}$}).
	Choosing~\mbox{$d_0 = n_0$} we first obtain
	a positive error bound~\mbox{$\eps_0 \approx 3 \times 10^{-5}$}
	for~\mbox{$d_0 = 40$}.
	
	That way it is also possible to find the maximal value of~$n_0$
	such that for a given dimension~$d_0$
	the error bound exceeds a given value $\eps_0$.
	In doing so, we find a result with a particular \mbox{$\tau > 0$}
	and an estimate for the $\eps_0$-complexity for~\mbox{$d\geq d_0$}:
	\begin{equation*}
		n^{\ran}(\eps_0,d)
		\geq n_0 \, 2^{\tau(\sqrt{d}-\sqrt{d_0})}
		= n_0 \, \exp(\sigma(\sqrt{d}-\sqrt{d_0}))
		\,.
	\end{equation*}
	The following tabular lists the maximal~$n_0$ for given~$d_0$
	such that we still find a lower bound
	that exceeds~\mbox{$\eps_0 = \frac{1}{30}$}.
	In addition, we give the maximal possible value for~$\tau$
	\mbox{(and~$\sigma = \tau \, \log 2$)} such that we still obtain
	the error bound~$\eps_0$ with the same~$n_0$.
	\begin{equation*}
		\begin{array}{c|c|c|c}
			d_0 & n_0 & \tau & \sigma = \tau \, \log 2 \\
			\hline
			\phantom{0}51 & \phantom{000\,00}1 & 1.0696 & 0.7414 \\
								100 & \phantom{000\,}108 & 1.4795 & 1.0255 \\
								200	& 					498\,098 & 1.9796 & 1.3721 \\
		\end{array}
	\end{equation*}	
	As observable in the examples,
	the value for~$\tau$ is increasing for growing dimension,
	so if we aim to find a good lower bound for the $\eps_0$-complexity
	for a particular dimension~$d$, it is preferable to use an estimate based
	on a big value~\mbox{$d_0 \leq d$}.
	For example, for~\mbox{$d=200$} we obtain
	\begin{itemize}
		\item $n^{\ran}(\eps_0,d) > \phantom{000\,}179$,
			\quad based on $d_0 = 51$ and $\tau = 1.0696$,
		\item $n^{\ran}(\eps_0,d) > \phantom{00}7\,554$,
			\quad based on $d_0 = 100$ and $\tau = 1.4795$,
		\item $n^{\ran}(\eps_0,d) > 498\,098$,
			\quad computed directly for $d_0 = 200$.
	\end{itemize}
	The question on how big the exponent can get
	is answered within the next remark.	
\end{remark}

\begin{remark}[Maximal value for the exponential constant
		\mbox{$c = \sigma_0 \, \eps_0$}] \label{rem:taueps}
	We have results of the type
	\begin{equation*}
		n^{\ran}(\eps,d) \geq \nu \, \exp\left(c  \, \frac{\sqrt{d}}{\eps}\right)
	\end{equation*}
	that hold for ``\emph{large}~$d$'' and \mbox{$\eps \succeq 1/\sqrt{d}$}.
	In the asymptotics of~\mbox{$d \rightarrow \infty$},
	any estimate with a larger exponent will outstrip an estimate with a smaller exponent,
	so it is preferable to have a big constant~$c$ in the exponent,
	but $\nu$ can be arbitrarily small.
	In order to find the maximal value for~$c$,
	we consider the limiting case for the detailed
	error bounds of \thmref{thm:monotonLB}, see~\proofstepref{proof:monoLB5} of the proof.
	First, we ask the question what error bounds are possible for a given~$\tau$,
	\begin{equation*}
		\lim_{\lambda \rightarrow 1}
		\lim_{\substack{d \rightarrow \infty\\
										\nu \rightarrow 0}}
			\hat{\eps}(\ldots) =
		\left(\Phi(\beta)-\Phi(\alpha)\right)
		\, \max_{\varrho>0}\min\left\{1- \exp(-\varrho),\,
									\exp\Bigl(-\varrho\,\exp((\beta-\alpha)\,\tau)\Bigr)
					\right\} \,.
	\end{equation*}
	The maximal value for optimal~$\alpha$, $\beta$, and~$\varrho$ gives
	us a limiting value~\mbox{$\bar{\eps}(\tau)$}.
	Amongst all settings with constant difference~\mbox{$(\beta-\alpha)$},
	the factor~\mbox{$(\Phi(\beta)-\Phi(\alpha))$} is maximized
	(and hence also the asymptotic lower bound)
	for the symmetrical choice~\mbox{$\alpha = - \beta$}.
	Writing~\mbox{$\beta \, \tau =: \vartheta$}, we obtain
	\begin{equation*}
		\bar{\eps}(\tau)
			:= \max_{\vartheta>0}
					\underbrace{
						\left(2 \Phi\left(\frac{\vartheta}{\tau}\right) - 1\right) \,
						\max_{\varrho>0} \min
							\left\{1- \exp(-\varrho),\,
									\exp\Bigl(-\varrho\, \exp(2\vartheta)\Bigr)
							\right\}
						}_{=:\tilde{\eps}(\tau,\vartheta)}\,,
	\end{equation*}
	so the second factor is formally independent from~$\tau$ now.\footnote{%
		Compare with the choice of~$\alpha(\tau)$ and $\beta(\tau)$
		in \proofstepref{proof:monoLB5} within the proof of \thmref{thm:monotonLB}.}
	The product~\mbox{$\tau \, \tilde{\eps}(\tau,\vartheta)$} is growing with~$\tau$
	when~$\vartheta$ is fixed. Therefore, the product~\mbox{$\tau \, \bar{\eps}(\tau)$}
	is maximized in the limit~\mbox{$\tau \rightarrow \infty$},
	\begin{equation*}
		\lim_{\tau \rightarrow \infty} \tau \, \bar{\eps}(\tau)
			= \max_{\vartheta>0} \sqrt{\frac{2}{\pi}} \, \vartheta
					\max_{\varrho>0} \min
						\left\{1- \exp(-\varrho),\,
									\exp\Bigl(-\varrho\, \exp(2\vartheta)\Bigr)
							\right\}
			\approx 0.1586\,.
	\end{equation*}
	In other words, switching the basis of the exponential expression,
	now considering~\mbox{$\sigma = \tau \, \log 2$},
	the constant~\mbox{$c = \sigma_0 \, \eps_0$}
	in \thmref{thm:monotonLB}
	cannot exceed~\mbox{$0.1100$} when relying on the given proof technique.
	For comparison,
	in the numerical example of the theorem with~\mbox{$d_0 = 100$},
	\mbox{$n_0 = 108$}, and \mbox{$\eps_0 = \frac{1}{30}$},
	we have \mbox{$c = \frac{1}{30} \approx 0.3333$}.
	
	Compared to the upper bounds for Boolean functions,
	see \thmref{thm:BooleanUBs} in \secref{sec:monoUBs},
	there is a significant gap in the exponent that can reach arbitrarily high factors
	if~\mbox{$1/\sqrt{d} \prec \eps < 1/2$}
	(the growth, however, is only logarithmic).
\end{remark}

\begin{remark}[Close to the initial error] \label{rem:1/2-...BBL98}
	We discuss necessary modifications to the proof	of \thmref{thm:monotonLB}
	in order to reproduce the result
	of Blum, Burch, and Langford~\cite[Sec~4]{BBL98},
	which states that for sufficiently large dimensions~$d$,
	and~\mbox{$n \geq d$}, we have
	\begin{equation} \label{eq:BBL98claimb}
		e^{\ran}(n,\App,F_{\Boole}^d,\Lstd)
			\geq \frac{1}{2} - C \, \frac{\log(d \, n)}{\sqrt{d}} \,,
	\end{equation}
	where \mbox{$C > 0$} is a numerical constant.
	
	We start with direct modifications for a weaker version of~\eqref{eq:BBL98claimb}.
	Several parameters are chosen with regard to 
	the estimates of \proofstepref{proof:monoLB3}.
	First, taking \mbox{$-\alpha = \beta = \sqrt{(\log d)/2}$},
	by Hoeffding bounds we have
	\begin{equation*}
		\frac{\#D_{ab}}{\#\{0,1\}^d} \geq 1 - \frac{2}{\sqrt{d}} \,,
	\end{equation*}
	compare \proofstepref{proof:monoLB3.1}.
	The calculations of \proofstepref{proof:monoLB3.2} can be replaced
	by the trivial estimate
	\begin{equation*}
		\frac{\# (Q_{\vecw} \cap D_{ab})}{\# Q_{\vecw}} \leq 1 \,.
	\end{equation*}
	We choose~{$\nu = 1/d$} and \mbox{$\lambda = 1-1/\sqrt{d}$},
	thus the final estimate~\eqref{eq:|B|>=r0-nu*r1}
	of \proofstepref{proof:monoLB3.4} reduces to
	\begin{equation} \label{eq:|B|>...BBL}
		\frac{\# B}{\#\{0,1\}^d}
			\geq 1 - \frac{1}{d} - \frac{2 + \sigma_{\alpha\beta\tau}(d)}{\sqrt{d}} \,.
	\end{equation}
	The choice of~$\nu$ implies that, for given~$n$, we need to take
	\begin{equation*}
		\tau := \frac{\log_2(d \, n)}{\sqrt{d}} \,,
	\end{equation*}	
	thus~\mbox{$t := \lceil \log_2(d \, n) \rceil$}.
	Note that \mbox{$\sigma_{\alpha\beta\tau}(d)$} can be bounded by a constant
	as long as~\mbox{$\beta \tau \preceq 1$},
	which is equivalent to~\mbox{$n \leq \exp(c \, \sqrt{d/\log d})$} for some~$c>0$.
	This log-term in the exponent is unpleasant when trying to reproduce
	the result of Blum et al.
	The term \mbox{$\sigma_{\alpha\beta\tau}(d)$}
	occurs from estimating~\mbox{$\# V_0$},
	see~\eqref{eq:|V01|bound} in \proofstepref{proof:monoLB2}.
	In the original paper, for this purpose, Chernoff bounds are used,
	and we do not need to include the boundary value~$b$
	in the definition of the measure, see \proofstepref{proof:monoLB1}.
	Then the cardinality of~$\# B$ can be estimated by terms
	that only depend on~$\tau$ and~$d$,
	the term \mbox{$\sigma_{\alpha\beta\tau}(d)$} can be replaced by a constant.
		
	More effort than before has to be put in estimating
	the conditional distribution of~\mbox{$f(\vecx)$} for~\mbox{$\vecx \in B$},
	compare \proofstepref{proof:monoLB4}.
	For detailed calculations
	refer to the original proof of Blum et al.~\cite{BBL98}.
	The parameter~$p$ of the distribution is determined by the equation
	\begin{equation*}
		(1-p)^{\binom{d/2}{t}} \stackrel{\text{!}}{=} \frac{1}{2} \,,
	\end{equation*}
	thus, for~\mbox{$|\vecx|_1 = \frac{d}{2}$},
	the function values~$f(\vecx)$ are~$0$ and~$1$
	with equal probability under~$\mu$.
	Then we obtain
	\begin{equation} \label{eq:muy><1/2+-...BBL}
		\begin{split}
			\mu_{\tilde{y}}\{f(\vecx) = 0\}
				&\leq 2^{\lambda \binom{a}{t}\big/\binom{d/2}{t}}
				\leq \frac{1}{2} + \frac{(\beta+1) \, t}{\sqrt{d}}\,, \quad \text{and} \\
			\mu_{\tilde{y}}\{f(\vecx) = 0\}
				&= 2^{- \binom{b}{t}\big/\binom{d/2}{t}}
				\geq \frac{1}{2} - \frac{(\beta + 1) \, t}{\sqrt{d}} \,.
		\end{split}
	\end{equation}
	Note that these two inequalities become trivial
	if~\mbox{$(\beta + 1) \, t / \sqrt{d}$} exceeds~$\frac{1}{2}$.
	In particular, \mbox{$(\beta + 1) \, t / \sqrt{d} \leq \frac{1}{2}$}
	is an assumption needed for the proof of the second inequality
	when following the steps in Blum et al.
	
	Combining~\eqref{eq:|B|>...BBL}	and~\eqref{eq:muy><1/2+-...BBL},
	and inserting the values for~$\beta$ and~$t$,
	we obtain the estimate
	\begin{align*}
		e(n,d) \geq \frac{1}{2}
							- C \, \frac{(1 + \sqrt{\log d}) \, \log d \, n}{\sqrt{d}} \,,
	\end{align*}
	with~$C > 0$ being a numerical constant.
	This is a weaker version of the result~\eqref{eq:BBL98claimb}
	by a logarithmic factor in~$d$.
	
	Blum et al. proved a stronger version without this logarithmic factor
	by integrating over~$\beta$ from~$1$ to~$\sqrt{d}$.
	The weaker version with constant~\mbox{$\beta = \sqrt{(\log d) / 2}$}
	has also been mentioned in the original paper already.
	The integration over~$\beta$ is only possible
	if we estimate~$\# V_1$ by Chernoff bounds,
	the boundary value~$b$ may not be part of the definition of the functions
	for the measure~$\mu$.
	Interestingly, this integral runs also over such~$\beta$ where
	the estimates on the conditional measure~\eqref{eq:muy><1/2+-...BBL}
	give negative values, thereby weakening the lower bounds.
	Still, this refined proof technique gives an improvement by a logarithmic term.
\end{remark}

\begin{remark}[A combined lower bound for real-valued monotone functions]
\label{rem:combiLBran}
	Similarly to \remref{rem:combiLBdeter}, which was about the deterministic setting,
	we can find lower bounds for the Monte Carlo approximation
	of real-valued monotone functions that include arbitrarily small~$\eps$
	for small dimensions already, but still reflect the $d$-dependency
	known from \thmref{thm:monotonLB}.
	
	With the notation from \remref{rem:combiLBdeter},
	given~\mbox{$\vecm \in \N^d$},
	we split the domain into \mbox{$\prod \vecm$}~sub-cuboids~$C_{\veci}$,
	and consider monotone functions~$f \in F_{\mon}^d$ that on each cuboid
	only have function values
	within an interval of length~\mbox{$1/(|\vecm|_1 - d + 1)$}
	such that monotonicity is guaranteed whenever the function is monotone
	on each of the sub-cuboids.
	Having lower bounds from \thmref{thm:monotonLB},
	\begin{equation*}
		e^{\ran}(n = \nu \, 2^{\tau \sqrt{d}},
						\App, F_{\mon}^d , \Lstd)
			\geq (r_0 - \nu \, r_1) \, q =: \eps_1 \,,
	\end{equation*}
	see also~\eqref{eq:err>=(r0-nu*r1)*q}
	for the inner structure of the lower bound,
	we can estimate the error we make on each of the sub-cuboids
	using \mbox{$n_{\veci} = \nu_{\veci} \, 2^{\tau \sqrt{d}}$}~function
	values only,
	\begin{align*}
		e^{\ran}(n = \nu \, 2^{\tau \sqrt{d}},
						\App, F_{\mon}^d , \Lstd)
			&\geq \frac{1}{|\vecm|_1 - d + 1} \,
					\left[\frac{1}{\prod \vecm}
									\sum_{\veci} (r_0 - \nu_{\veci} \, r_1) \, q
					\right] \\
			&= \frac{1}{|\vecm|_1 - d + 1}\,
					\left( r_0 - \frac{\nu}{\prod \vecm} \right) \, q \,,
	\end{align*}
	where~\mbox{$n = \sum_{\veci} n_{\veci}$}.
	For~\mbox{$0 < \eps < \eps_1$},
	choose an appropriate splitting parameter~$\vecm$,
	and obtain the complexity bound
	\begin{equation} \label{eq:monMCLBcombi}
		n^{\deter}(\eps,\App,F_{\mon}^d,\Lstd)
			\geq \begin{cases}
							\nu \, 2^{\tau \sqrt{d}+\lfloor \eps_1 / \eps \rfloor-2}
								\quad&\text{for $\eps \in [\frac{\eps_1}{d+1},
																									\eps_1]$,} \\
							\nu \, 2^{\tau \sqrt{d} + d \, \log_2 \lfloor \eps_1 / (\eps \, d)\rfloor) - 1}
								\quad&\text{for $\eps \in (0,\frac{\eps_1}{d+1}]$.}
						\end{cases}
	\end{equation}
	Knowing
	\begin{equation*}
		n^{\ran}(\eps,\App,F_{\Boole}^d,\Lstd)
			> \nu \, 2^{c \, \sqrt{d}/\eps} \,,
	\end{equation*}
	for~\mbox{$d > d_0$} and \mbox{$\eps \geq \eps_0 \sqrt{d_0/d}$},
	we can take~\eqref{eq:monMCLBcombi}
	with the $d$-dependent values \mbox{$\eps_1 = \eps_0 \, \sqrt{d_0/d}$}
	and~\mbox{$\tau = c/\eps_1 = c/\eps_0 \, \sqrt{d/d_0}$}
	in order to get enhanced error bounds for \mbox{$0 < \eps < \eps_0 \, \sqrt{d_0/d}$}.
\end{remark}

\section{Breaking the Curse with Monte Carlo}
\label{sec:monoUBs}

A new algorithm for the approximation of real-valued monotone functions
on the unit cube is presented and analysed in \secref{sec:monoRealUBs}.
It is the first algorithm to show that for this problem
the curse of dimensionality does not hold in the randomized setting.
The idea is directly inspired by a method
for Boolean monotone functions due to Bshouty and Tamon~\cite{BT96}.
We start with the presentation of the less complicated Boolean case
in~\secref{sec:BooleanUBs}.
The structure of the proofs in each of the two sections
is analogous to the greatest possible extend
so that one can always find the counterpart within the other setting
(if there is one).

\subsection{A Known Method for Boolean Functions}
\label{sec:BooleanUBs}

We present a method known from Bshouty and Tamon~\cite{BT96}
for the randomized approximation of Boolean functions
that comes close to the lower bounds from \thmref{thm:monotonLB}.
Actually, they considered a slightly more general setting,
allowing product weights for the importance of different entries
of a Boolean function~$f$,
we only study the special case that fits to our setting.
The analysis of the original paper was done for the
\emph{margin of error} setting,
but it can be easily converted into results on the Monte Carlo error
as we prefer to define it by means of expectation.\footnote{%
	In the margin of error setting we want to
	determine~\mbox{$\eps,\delta > 0$} such that
	for any input function~$f$ the actual error of the randomized algorithm
	exceeds~$\eps$ only with probability~$\delta$.
	Since for Boolean functions the error cannot exceed~$1$,
	the corresponding expected error is bounded from above
	by~\mbox{$\eps + \delta$}.\\
	Conversely, for any~$\eps$ that exceeds
	the Monte Carlo error~$e^{\ran}$, we obtain
	\mbox{$\delta \leq e^{\ran}/\eps$} for the uncertainty level
	by Chebyshev's inequality.
	Practically, if we aim for a small~$\delta$,
	we lose a lot in this direction,
	and it is advisable to analyse the margin of error setting directly
	whenever it is of interest.}

For the formulation and the analysis of the algorithm
it is convenient to redefine the notion of Boolean functions
and to consider the class of functions
\begin{equation*}
	G_{\pm}^d := \{f: \{-1,+1\}^d \rightarrow \{-1,+1\} \} \,,
\end{equation*}
the input set of Boolean functions is renamed
\mbox{$F_{\pm}^d \subset G_{\pm}^d$}.
We keep the distance~$\dist$ that we had for the old version
of~$G_{\Boole}^d$, for~\mbox{$f_1,f_2 \in G_{\pm}^d$} we have
\begin{equation*}
	\dist(f_1,f_2)
		:= \frac{1}{2^{d}}
					\, \#\{\veci \in \{-1,1\}^d \mid
									f_1(\veci) \not= f_2(\veci) \} \,,
\end{equation*}
compare~\eqref{eq:distBoole},
thus the diameter of~$G_{\pm}^d$ is still~$1$.
This metric differs by a factor~$2$ from the induced metric
that we obtain when regarding~$G_{\pm}^d$ as a subset
of the Euclidean space~\mbox{$L_2(\Uniform\{-1,+1\}^d)$}.
For this space we choose the orthonormal
basis~\mbox{$\{\psi_{\vecalpha}\}_{\vecalpha \in \{0,1\}^d}$},
\begin{equation*}
	\psi_{\vecalpha}(\vecx) := \vecx^{\vecalpha} = \prod_{j=1}^d x_j^{\alpha_j}
	\quad \text{for\, $\vecx \in \{-1,+1\}^d$.}
\end{equation*}
Every Boolean function can be written as the Fourier decomposition
\begin{equation*}
	f = \sum_{\vecalpha \in \{0,1\}^d} \hat{f}(\vecalpha) \, \psi_{\vecalpha}
\end{equation*}
with the Fourier coefficients
\begin{equation*}
	\hat{f}(\vecalpha)
		:= \langle \psi_{\vecalpha}, f \rangle
		= \expect \vecX^{\vecalpha} \, f(\vecX) \,,
\end{equation*}
where~$\vecX$ is uniformly distributed on~\mbox{$\{-1,+1\}^d$}.
The idea of the algorithm is to use random samples~\mbox{$f(\vecX_1),\ldots,f(\vecX_n)$},
with \mbox{$\vecX_i \stackrel{\text{iid}}{\sim} \Uniform\{-1,+1\}^d$},
in order to approximate the low-degree Fourier coefficients
for~\mbox{$|\vecalpha|_1 \leq k$}, $k \in \N$,
\begin{equation*}
	\hat{f}(\vecalpha)
		\approx \hat{h}(\vecalpha)
		:= \frac{1}{n} \sum_{i=1}^n \vecX_i^{\vecalpha} \, f(\vecX_i) \,.
\end{equation*}
Based on the Fourier approximation
\begin{equation*}
	h := \sum_{\substack{\vecalpha \in \{0,1\}^d\\
											|\vecalpha|_1 \leq k}}
					\hat{h}(\vecalpha) \, \psi_{\vecalpha} \,,
\end{equation*}
we return the output~$g := A_{n,k}^{\omega}(f)$ with
\begin{equation*}
	g(\vecx)
		:= \sgn h(\vecx)
		= \begin{cases}
				+1 \quad&\text{if $h(\vecx) \geq 0$,} \\
				-1 \quad&\text{if $h(\vecx) < 0$.}
			\end{cases}
\end{equation*}

We will give a complete analysis of the above algorithm
which is based on~$L_2$-approximation. In fact, from
\begin{equation*}
	f(\vecx) \not= g(\vecx)
		\quad\Leftrightarrow\quad
			f(\vecx) \not= \sgn h(\vecx)
		\quad\Rightarrow\quad
			(f(\vecx) - h(\vecx))^2 \geq 1
\end{equation*}
we obtain
\begin{equation} \label{eq:dist<L2}
	\dist(f,g) \leq \|f - h\|_{L_2}^2 \,.
\end{equation}

Note that every Boolean function~\mbox{$f: \{-1,+1\}^d \rightarrow \{-1,+1\}$}
has norm~$1$ in the~$L_2$-norm, so by Parseval's equation,
\begin{equation*}
	 \sum_{\vecalpha \in \{0,1\}^d} \hat{f}^2(\vecalpha) = \|f\|_{L_2} = 1 \,.
\end{equation*}

A key result for the analysis of the above algorithm is the following fact
about the Fourier coefficients that are dropped.

\begin{lemma}[{Bshouty and Tamon~\cite[Sec~4]{BT96}}]
	\label{lem:monBooleSmallFourier}
	For any monotone Boolean function~$f$ we have
	\begin{equation*}
		 \sum_{\substack{\vecalpha \in \{0,1\}^d\\
											|\vecalpha|_1 > k}}
				\hat{f}^2(\vecalpha) \leq \frac{\sqrt{d}}{k+1} \,.
	\end{equation*}
\end{lemma}
\begin{proof}
	Within the first step, we consider
	the special Fourier coefficients~\mbox{$\hat{f}(\vece_j)$},
	which measure the sensitivity of~$f$ with respect to a single variable~$x_j$.
	These are the only Fourier coefficients
	where monotonicity guarantees a non-negative value.
	For~\mbox{$j=1,\ldots,d$} we consider
	the restricted functions
	\begin{align*}
		f_{-j}(\vecx) = f(\vecz) \,,
			\quad&\text{with $z_{j'} = x_{j'}$ for $j' \not= j$, and $z_j = -1$,} \\
		f_{+j}(\vecx) = f(\vecz) \,,
			\quad&\text{with $z_{j'} = x_{j'}$ for $j' \not= j$, and $z_j = +1$.}
	\end{align*}
	Due to the monotonicity of~$f$, we have~\mbox{$f_{-j} \leq f_{+j}$},
	using this and Parseval's equation, we obtain
	\begin{align*}
		\hat{f}(\vece_j) = <\psi_{\vece_j},f>
			&= \frac{1}{2^d} \, \sum_{\vecx \in \{-1,+1\}^d} x_j \, f(\vecx) \\
			&= \frac{1}{2^d} \, \sum_{\vecx \in \{-1,+1\}^d}
														\underbrace{\frac{f_{+j}(\vecx)-f_{-j}(\vecx)}{2}
																				}_{\in \{0,+1\}} \\
			&= \left\| \frac{f_{+j}-f_{-j}}{2} \right\|_{L_2}^2 \\
			&= \sum_{\vecalpha \in \{0,1\}^d}
						\left\langle \psi_{\vecalpha},
													\frac{f_{+j}-f_{-j}}{2}
						\right\rangle^2 \,.
	\end{align*}
	Since the functions~$f_{-j}$ and~$f_{+j}$ are independent from~$x_j$,
	the summands with~\mbox{$\alpha_j = 1$} vanish.
	For all the other summands, with~\mbox{$\alpha_j = 0$}, we have
	\begin{align*}
		\left\langle \psi_{\vecalpha},
													\frac{f_{+j}-f_{-j}}{2}
		\right\rangle
			&= \frac{1}{2^d} \, \sum_{\vecx \in \{-1,+1\}^d}
					\vecx^{\vecalpha} \, \frac{f_{+j}(\vecx)-f_{-j}(\vecx)}{2}\\
			&= \frac{1}{2^d} \, \sum_{\vecx \in \{-1,+1\}^d}
					\underbrace{x_j \, \vecx^{\vecalpha}}_{= \vecx^{\vecalpha'}} \, f(\vecx)\\
			&= \langle \psi_{\vecalpha'}, f \rangle
			= \hat{f}(\vecalpha') \,,
	\end{align*}
	where~\mbox{$\alpha'_{j'} = \alpha_{j'}$} for \mbox{$j' \not= j$},
	and \mbox{$\alpha'_j = 1$}.
	This leads to the identity
	\begin{equation*}
		\hat{f}(\vece_j)
			= \sum_{\substack{\vecalpha \in \{0,1\}^d\\
															\alpha_j = 1}}
								\hat{f}^2(\vecalpha) \,.
	\end{equation*}
	
	Summing up over all dimensions, we obtain
	\begin{equation*}
		\sum_{j=1}^d \hat{f}(\vece_j)
			= \sum_{j=1}^d \sum_{\substack{\vecalpha \in \{0,1\}^d\\
																		\alpha_j = 1}}
					\hat{f}^2(\vecalpha)
			= \sum_{\vecalpha \in \{0,1\}^d} |\vecalpha|_1 \, \hat{f}^2(\vecalpha)
			\geq (k+1) \, \sum_{\substack{\vecalpha \in \{0,1\}^d \\
																	 |\vecalpha|_1 > k}}
											\hat{f}^2(\vecalpha) \,.
	\end{equation*}
	Finally,
	\begin{equation*}
		1 = \sum_{\vecalpha \in \{0,1\}^d} \hat{f}^2(\vecalpha)
			\geq \sum_{j=1}^d \hat{f}^2(\vece_j)
			\geq \frac{1}{d} \left(\sum_{j=1}^d \hat{f}(\vece_j)\right)^2 \,,
	\end{equation*}
	which, combined with the inequality above, proves the lemma.
\end{proof}

This helps us to obtain the following error and complexity bound,
which is a simplification of Bshouty and Tamon~{\cite[Thm~5.1]{BT96}},
where the setting was more general,
and the more demanding margin of error was considered.
\begin{theorem} \label{thm:BooleanUBs}
	For the algorithm~\mbox{$A_{n,k} = (A_{n,k}^{\omega})_{\omega \in \Omega}$},
	\mbox{$k \in \{1,\ldots,d\}$},
	we have the error bound
	\begin{equation*}
		e(A_{n,k},F_{\pm}^d \hookrightarrow G_{\pm}^d)
			\leq \frac{\sqrt{d}}{k+1}
				+ \frac{\exp\left(k \, \left(1 + \log \frac{d}{k}\right)\right)
								}{n} \,.
	\end{equation*}
	In particular, given~\mbox{$0 < \eps < \frac{1}{2}$},
	the $\eps$-complexity
	of the Monte Carlo approximation of monotone Boolean functions
	is bounded by
	\begin{equation*}
		n^{\ran}(\eps,F_{\pm}^d \hookrightarrow G_{\pm}^d,\Lstd)
			\leq \min\left\{
									\exp\left(C \, \frac{\sqrt{d}}{\eps}
															\, \left(1 + [\log \sqrt{d} \, \eps]_+
																\right)
													\right),\,
									2^d
							\right\} \,,
	\end{equation*}
	where \mbox{$C > 0$} is some constant.
	Hence the curse of dimensionality does \emph{not} hold.
\end{theorem}
\begin{proof}
	We first compute the accuracy at which we approximate each of the
	Fourier coefficients,
	\begin{equation} \label{eq:(f-h)^2}
		\begin{split}
			\expect[\hat{f}(\vecalpha) - \hat{h}(\vecalpha)]^2
				&= \frac{1}{n^2}
						\, \expect\left[\sum_{i=1}^n
															(\hat{f}(\vecalpha)
																- \psi_{\vecalpha}(\vecX_i) \, f(\vecX_i))
											\right]^2 \\
				\text{[independent mean~$0$ variables]}\quad
				&= \frac{1}{n^2}
						\, \sum_{i=1}^n
									\expect [\hat{f}(\vecalpha)
														- \underbrace{\psi_{\vecalpha}(\vecX_i)
																					\, f(\vecX_i)
																				}_{ \in \{-1,+1\}}
													]^2 \\
				&\leq \frac{1}{n} \,.
		\end{split}
	\end{equation}
	This estimate is needed
	only for the set~%
	\mbox{$A := \{\vecalpha \in \{0,1\}^d \mid |\vecalpha|_1 \leq k\}$},
	we obtain the general bound
	\begin{align*}
		\expect \dist(f,g)
			&\stackrel{\eqref{eq:dist<L2}}{\leq}
				\expect \|f - h\|_{L_2}^2 \\
			&\,\leq\,
				\sum_{\substack{\vecalpha \in \{0,1\}^d\\
												|\vecalpha|_1 > k}}
							\hat{f}^2(\vecalpha)
					+ \sum_{\substack{\vecalpha \in \{0,1\}^d\\
														|\vecalpha|_1 \leq k}}
							\expect[\hat{f}(\vecalpha) - \hat{h}(\vecalpha)]^2 \\
		\text{[\hyperref[lem:monBooleSmallFourier]{Lem~\ref*{lem:monBooleSmallFourier}}
						and \eqref{eq:(f-h)^2}]}\quad
			&\,\leq\,
				\frac{\sqrt{d}}{k+1} + \frac{\# A}{n} \,.
	\end{align*}
	By \lemref{lem:BinomSum}, for \mbox{$k \in \{1,\ldots,d\}$},
	we estimate
	\begin{equation*}
		\#A = \sum_{l=0}^k \binom{d}{l}
			\leq \left(\frac{\euler \, d}{k}\right)^k
			= \exp\left(k \, \left(1 + \log \frac{d}{k}\right) \right) \,.
	\end{equation*}
	This gives us the error bound for the Monte Carlo method~\mbox{$A_{n,k}$}.
	
	Choosing~\mbox{$k := \lfloor 2 \, \sqrt{d} / \eps \rfloor$}
	guarantees~\mbox{$\sqrt{d}/(k+1) \leq \eps/2$}.
	The second term can be bounded by~\mbox{$\eps/2$}
	if we choose
	\begin{equation*}
		n := \left\lceil \frac{2 \, \# A}{\eps} \right\rceil
			\leq\left\lceil \frac{2}{\eps} \,
						\exp\left(\frac{2 \, \sqrt{d}}{\eps}
												\, \left(1 + [\log \sqrt{d} \, \eps]_+
													\right)
								\right)
					\right\rceil \,.
	\end{equation*}
	
	For~\mbox{$\eps \leq 2/\sqrt{d}$}, however,
	according to the error estimate of~$A_{n,k}$,
	we would need to take~\mbox{$k=d$}.
	In this case \mbox{$\# A = 2^d$}, and~$n$ should be even larger.
	But then deterministically collected complete information~\mbox{$n=2^d$}
	is the best solution with already exact approximation.
\end{proof}

\begin{remark}[Limiting case]
	The present upper bounds fit
	the lower bounds from \thmref{thm:monotonLB}
	up to a factor in the exponent which is logarithmic in~$\eps$ and~$d$,
	but if we consider a sequence~\mbox{$\eps_d = \eps_0/\sqrt{d}$},
	there is no logarithmic gap at all.

	There is a natural transition to complete information
	as $n$ approaches~$2^d$.
	Indeed, take~\mbox{$n=2^k$} for some natural number~\mbox{$k \leq d$}
	and sample~$f$ from function values computed
	for independent~$\vecX_i$ chosen uniformly
	from~$2^k$ disjoint subsets in~\mbox{$\{-1,+1\}^d$} of equal size.
	For instance,
	let the first~$k$~entries within the random vector~$\vecX_i$ be given
	by the binary representation of~$i$,
	and let the remaining~\mbox{$d-k$} entries be independent Bernoulli random variables.
	The calculation~\eqref{eq:(f-h)^2} will still work essentially the same.
\end{remark}

\begin{remark}[Non-interpolatory]
	The algorithm~$A_{n,k}$
	is not always consistent with the knowledge we actually have on the function,
	and it does not even preserve monotonicity in general,
	so it is \emph{non-interpolatory}.
	
	Take, for example, \mbox{$d \geq 2$} and {$k = 1$},
	and the constant function~\mbox{$f=1$}.
	Assume that -- for bad luck -- all the sample points~\mbox{$\vecX_i$}
	happened to be~\mbox{$(-1,\ldots,-1)$}.
	Of course, the information \mbox{$f(-1,\ldots,-1) = 1$}
	already implies \mbox{$f = 1$}, thanks to monotonicity.
	But with the Fourier based algorithm,
	\begin{equation*}
		h(\vecx) = 1 - \sum_{i=1}^d x_i
	\end{equation*}
	and for the output we have
	\begin{equation*}
		g(1,\ldots,1) = -1 < g(-1,\ldots,-1) = +1 \,,
	\end{equation*}
	which violates monotonicity.
	
	We could modify the output,
	making the algorithm interpolatory,
	actually this is possible without affecting the error bounds.
	Obviously, it is an improvement to replace the original output~$g$ by
	\begin{equation*}
		g'(\vecx)
			:= \begin{cases}
						+1 \quad& \text{if $\exists i: \, \vecX_i \leq \vecx$
																						and $f(\vecX_i) = +1$,}\\
						-1 \quad& \text{if $\exists i: \, \vecX_i \geq \vecx$
																						and $f(\vecX_i) = -1$,}\\
						g(\vecx) \quad& \text{else.}
					\end{cases}
	\end{equation*}
	Restoring monotonicity for the output is a bit more complicated
	since one needs to survey the output~$g$ as a whole.\footnote{%
		Usually we would not store all values of an approximant~$g$ in a computer
		but only the coefficients that are necessary
		for a computation of~\mbox{$g(\vecx)$} on demand.
		This process should be significantly cheaper
		than asking the oracle for a value~$f(\vecx)$,
		compare \remref{rem:monoMCUBphicost}.}
	The general idea is to find pairs of points~$\vecz_1 \leq \vecz_2$
	with~\mbox{$g(\vecz_1) = +1 > g(\vecz_2) = -1$}.
	At least one of these values is a misprediction of the input~$f$.
	If we flip these values, that is, we create a new output
	\begin{equation*}
		g'(\vecx)
			:= \begin{cases}
						-1 \quad& \text{for~$\vecx = \vecz_1$,}\\
						+1 \quad& \text{for~$\vecx = \vecz_2$,}\\
						g(\vecx) \quad& \text{else,}
					\end{cases}
	\end{equation*}
	then at least one of these values predicts~$f$ correctly,
	so~$g'$ approximates~$f$ not worse than~$g$ does.
	We could proceed like this until we obtain an interpolatory output,
	if needed.
\end{remark}

\subsection{Real-Valued Monotone Functions}
\label{sec:monoRealUBs}

We present a generalization of the method from the above section
to the situation of real-valued monotone functions.
The method is based on Haar wavelets.
For convenience, we change the range and now consider monotone functions
\begin{equation*}
	f: [0,1]^d \rightarrow [-1,+1] \,,
\end{equation*}
the altered input set shall be named~$F_{\mon\pm}^d$.\footnote{%
	By the bijection~\mbox{$F_{\mon}^d = \frac{1}{2} \, (F_{\mon\pm}^d + 1)$},
	we can transfer results for~$F_{\mon\pm}^d$ to results for~\mbox{$F_{\mon}^d$},
	which comes along with a reduction of the error quantities
	by a factor~$\frac{1}{2}$.\\
	For proofs on lower bounds it was more convenient to have
	functions~\mbox{$f:[0,1]^d \rightarrow [0,1]$},
	because then the distance of Boolean functions
	coincides with the $L_1$-distance of the corresponding subcubewise
	constant functions.}

We define dyadic cuboids on~\mbox{$[0,1]^d$}
indexed by~\mbox{$\vecalpha \in \N^d$}, or equivalently
by an index vector pair~\mbox{$(\veclambda,\veckappa)$}
with~\mbox{$\veclambda \in \N_0^d$}
and~\mbox{$\veckappa \in \N_0^d$},
\mbox{$\kappa_j < 2^{\lambda_j}$},
such that~\mbox{$\alpha_j \equiv 2^{\lambda_j} + \kappa_j$}
for~\mbox{$j = 1,\ldots,d$}:
\begin{equation*}
	C_{\vecalpha} = C_{\veclambda,\veckappa}
		:= \bigtimes_{j=1}^d I_{\alpha_j} \,,
\end{equation*}
where
\begin{equation*}
	I_{\alpha_j} = I_{\lambda_j,\kappa_j}
		:=\begin{cases}
				[\kappa_j \, 2^{-\lambda_j} , (\kappa_j+1) \, 2^{-\lambda_j})
					\quad&\text{for~$\kappa_j = 0,\ldots,2^{\lambda_j}-2$,} \\
				[1 - 2^{-\lambda_j} , 1]
					\quad&\text{for~$\kappa_j = 2^{\lambda_j}-1$.}
			\end{cases}
\end{equation*}
Note that for fixed~$\lambda_j$ we have a decomposition
of the unit interval~\mbox{$[0,1]$}
into $2^{\lambda_j}$~disjoint intervals of length~$2^{-\lambda_j}$.
One-dimensional Haar wavelets~\mbox{$h_{\alpha_j} : [0,1] \rightarrow \R$}
are defined for~\mbox{$\alpha_j \in \N_0$}
(if~\mbox{$\alpha_j=0$}, we set~\mbox{$\lambda_j = -\infty$} and~\mbox{$\kappa_j=0$}),
\begin{equation*}
	h_{\alpha_j} 
		:=\begin{cases}
				\ind_{[0,1]}
					\quad&\text{if $\alpha_j = 0$
											(i.e.\ $\lambda_j = -\infty$ and $k=0$),}\\
				2^{\lambda_j/2}
					\, (\ind_{I_{\lambda_j+1, 2\kappa_j+1}}
							- \ind_{I_{\lambda_j+1, 2\kappa_j}})
					\quad&\text{if $\alpha_j \geq 1$ (i.e.\ $\lambda_j \geq 0$).}
			\end{cases}
\end{equation*}
In~\mbox{$L_2([0,1]^d)$} we have the orthonormal basis
\mbox{$\{\psi_{\vecalpha}\}_{\vecalpha \in \N_0^d}$} with
\begin{equation*}
	\psi_{\vecalpha}(\vecx) := \prod_{j=1}^d h_{\alpha_j}(x_j) \,.
\end{equation*}
The volume of the support of~$\psi_{\vecalpha}$
is~\mbox{$2^{-|\veclambda|_{+}}$}
with~\mbox{$|\veclambda|_{+} := \sum_{j=1}^d \max\{0,\lambda_j\}$}.
The basis function~$\psi_{\vecalpha}$ only takes discrete
values~\mbox{$\{0,\pm 2^{|\veclambda|_{+}/2}\}$},
hence it is normalized indeed.

We can write any monotone function~$f$ as the Haar wavelet decomposition
\begin{equation*}
	f = \sum_{\vecalpha \in \N_0^d} \tilde{f}(\vecalpha) \, \psi_{\vecalpha}
\end{equation*}
with the wavelet coefficients
\begin{equation*}
	\tilde{f}(\vecalpha)
		:= \langle \psi_{\vecalpha}, f \rangle
		= \expect \psi_{\vecalpha}(\vecX) \, f(\vecX) \,,
\end{equation*}
where~$\vecX$ is uniformly distributed on~\mbox{$[0,1]^d$}.
For the algorithm we will use random samples~\mbox{$f(X_1),\ldots,f(X_n)$},
with~\mbox{$\vecX_i \stackrel{\text{iid}}{\sim} \Uniform [0,1]^d$},
in order to approximate the most important wavelet coefficients
\begin{equation*}
	\tilde{f}(\vecalpha) \approx \tilde{g}(\vecalpha)
		:= \frac{1}{n} \sum_{i=1}^n \psi_{\vecalpha}(\vecX_i) \, f(\vecX_i) \,.
\end{equation*}
In particular, we choose a resolution~\mbox{$r \in \N$},
and a parameter~\mbox{$k \in \{1,\ldots,d\}$},
and only consider indices~\mbox{$\vecalpha \equiv (\veclambda,\veckappa)$}
with~\mbox{$\lambda_j < r$}
and~\mbox{$|\vecalpha|_0 := \#\{ j \mid \alpha_j > 0\} < k$}.
The Monte Carlo method~\mbox{$(A_{n,k,r}^{\omega})_{\omega\in\Omega}$}
will give the output
\begin{equation*}
	g := A_{n,k,r}^{\omega}(f)
		:= \sum_{\substack{\vecalpha \in \N_0^d\\
											|\vecalpha|_0 \leq k \\
											\veclambda < r}}
					\tilde{g}(\vecalpha) \, \psi_{\vecalpha} \,.
\end{equation*}

We start with an analogue of \lemref{lem:monBooleSmallFourier}.
\begin{lemma}\label{lem:monSmallWavelet}
	For any monotone function~$f \in F_{\mon\pm}^d$ we have
	\begin{equation*}
		\sum_{\substack{\vecalpha \in \N_0^d\\
											|\vecalpha|_0 > k \\
											\veclambda < r}}
				\tilde{f}(\vecalpha)^2
			\leq \frac{\sqrt{d \, r}}{k+1} \,.
	\end{equation*}
\end{lemma}
\begin{proof}
	Within the first step, we consider special wavelet
	coefficients~\mbox{$\tilde{f}(\alpha \, \vece_j)$}
	that measure the average growth of~$f$
	for the $j$-th coordinate within
	the interval~\mbox{$I_{\alpha}$}.
	We will frequently use the alternative
	indexing~\mbox{$I_{\lambda,\kappa}$}
	with~\mbox{$\alpha = 2^{\lambda} + \kappa \in \N$},
	where~\mbox{$\lambda \in \N_0$} and~\mbox{$\kappa = 0,\ldots,2^{\lambda}-1$}.
	We define the two functions
	\begin{align*}
		f_{-\alpha j}(\vecx)
			&:= \begin{cases}
						0
							\quad&\text{if $x_j \notin I_{\alpha}$,}\\
						2^{\lambda+1} \,
							\int_{I_{(\lambda+1,2\kappa)}}
								f(\vecz)\Big|_{\substack{z_{j'} = x_{j'}\\
																		\text{for $j \not= j'$}}}
									\rd z_j
							\quad&\text{if $x_j \in I_{\alpha}$,}
					\end{cases} \\
		f_{+\alpha j}(\vecx)
			&:= \begin{cases}
						0
							\quad&\text{if $x_j \notin I_{\alpha}$,}\\
						2^{\lambda+1} \,
							\int_{I_{(\lambda+1,2\kappa+1)}}
								f(\vecz)\Big|_{\substack{z_{j'} = x_{j'}\\
																		\text{for $j \not= j'$}}}
									\rd z_j
							\quad&\text{if $x_j \in I_{\alpha}$.}
					\end{cases}
	\end{align*}
	Due to monotonicity of~$f$, we have \mbox{$f_{-\alpha j} \leq f_{+\alpha j}$}.
	Using this and Parseval's equation, we obtain
	\begin{align*}
		\tilde{f}(\alpha \, \vece_j)
			= \langle \psi_{\alpha \, \vece_j} , f \rangle
			&= 2^{\lambda/2} \,
					\left[\langle \ind_{I_{\lambda+1,2\kappa+1}}, f \rangle
								- \langle \ind_{I_{\lambda+1,2\kappa}}, f \rangle
					\right] \\
			&= 2^{\lambda/2} \,
					\biggl\| \underbrace{\frac{f_{+\alpha j} - f_{-\alpha j}
																		}{2}
														}_{\in [0,1]} 
					\biggr\|_{L_1} \\
			&\geq 2^{\lambda/2} \,
					\biggl\| \frac{f_{+\alpha j} - f_{-\alpha j}
											}{2} 
					\biggr\|_{L_2}^2 \\
			&= 2^{\lambda/2} \, \sum_{\vecalpha' \in \N_0^d}
						\left\langle \psi_{\vecalpha'},
												\frac{f_{+\alpha j} - f_{-\alpha j}}{2}
						\right\rangle^2 \,.
	\end{align*}
	Since the functions~$f_{-\alpha j}$ and~$f_{+\alpha j}$
	are constant in~$x_j$ on~$I_{\alpha \vece_j}$ and vanish outside,
	we only need to consider summands
	with coarser resolution~\mbox{$\lambda_j' < \lambda$}
	in that coordinate,
	and where the support of~$\psi_{\vecalpha}$
	contains the support of~$f_{\pm\alpha j}$.
	That is the case for
	\mbox{$\kappa_j' = \lfloor 2^{\lambda_j' - \lambda} \kappa \rfloor$}
	with~\mbox{$\lambda_j' = -\infty,0,\ldots,\lambda-1$}.
	For these indices we have
	\begin{equation*}
		\left\langle \psi_{\vecalpha'},
												\frac{f_{+\alpha j} - f_{-\alpha j}}{2}
		\right\rangle^2
			= 2^{\max\{0,\lambda_j'\} - \lambda} \,
							\left\langle \psi_{\vecalpha''},f
							\right\rangle^2
			= 2^{\max\{0,\lambda_j'\} - \lambda} \, \tilde{f}^2(\vecalpha'') \,,
	\end{equation*}
	where~\mbox{$\alpha''_{j'} = \alpha'_{j'}$} for~\mbox{$j' \not= j$},
	and~\mbox{$\alpha''_j = \alpha$}. Hence we obtain
	\begin{equation} \label{eq:f(aej)>}
		\tilde{f}(\alpha \, \vece_j)
			\geq 2^{\lambda/2}
					\, \underbrace{\left(2^{-\lambda}
															+ \sum_{l = 0}^{\lambda-1} 2^{l-\lambda}
													\right)
												}_{= 1}
					\, \sum_{\substack{\vecalpha'' \in \N_0^d\\
														\alpha_j'' = \alpha}}
								\tilde{f}^2(\vecalpha'') \,.
	\end{equation}
	
	Based on this relation between the wavelet coefficients, we can estimate
	\begin{align*} \allowdisplaybreaks
		1 = \|f\|_{L_2}^2
			&\,=\,
				\sum_{\vecalpha \in \N_0^d} \tilde{f}^2(\vecalpha)\\
			&\,\geq\,
				\sum_{j=1}^d
					\sum_{\lambda = 0}^{r-1}
						\sum_{\kappa = 0}^{2^{\lambda}-1}
							\tilde{f}^2((2^{\lambda}+\kappa) \, \vece_j)\\
			&\,\geq\,
				\sum_{j=1}^d \sum_{\lambda = 0}^{r-1}
					\biggl(2^{-\lambda/2}
									\, \sum_{\kappa = 0}^{2^{\lambda}-1}
											\tilde{f}((2^{\lambda}+\kappa) \, \vece_j)
					\biggr)^2 \\
			&\stackrel{\eqref{eq:f(aej)>}}{\geq}
				\sum_{j=1}^d \sum_{\lambda = 0}^{r-1}
					\biggl(\sum_{\substack{\vecalpha \in \N_0^d\\
														\lambda_j = \lambda}}
									\tilde{f}^2(\vecalpha)
					\biggr)^2 \,.
	\end{align*}
	Taking the square root, and using the norm estimate
	\mbox{$\|\vecx\|_1 \leq \sqrt{m} \, \|\vecx\|_2$} for~\mbox{$\vecx \in \R^m$},
	here with~\mbox{$m = d \, r$}, we get
	\begin{align*}
		1 &\geq \frac{1}{\sqrt{d \, r}} \,
					\sum_{j=1}^d \sum_{\lambda = 0}^{r-1}
									\sum_{\substack{\vecalpha \in \N_0^d\\
														\lambda_j = \lambda}}
										\tilde{f}^2(\vecalpha) \\
			&\geq \frac{1}{\sqrt{d \, r}} \,
					\sum_{\substack{\vecalpha \in \N_0^d\\
														\veclambda < r}}
						|\vecalpha|_0 \, \tilde{f}^2(\vecalpha) \\
			&\geq \frac{k+1}{\sqrt{d \, r}} \,
					\sum_{\substack{\vecalpha \in \N_0^d\\
														|\vecalpha|_0 > k \\
														\veclambda < r \\}}
										\tilde{f}^2(\vecalpha) \,.
	\end{align*}
	This proves the lemma.
\end{proof}

\begin{theorem} \label{thm:monoUBsreal}
	For the algorithm~\mbox{$A_{n,k,r} = (A_{n,k,r}^{\omega})_{\omega \in \Omega}$}
	we have the error bound
	\begin{equation*}
		e(A_{n,k,r},F_{\mon\pm}^d \hookrightarrow L_1([0,1]^d))
			\leq \frac{d}{2^{r+1}}
				+ \sqrt{\frac{\sqrt{d \, r}}{k+1}
								+\frac{\exp[k(1+\log \frac{d}{k} + (\log 2) \, r)]}{n}
								} \,.
	\end{equation*}
	Given~\mbox{$0 < \eps < 1$},
	the $\eps$-complexity for the Monte Carlo approximation
	of monotone functions is bounded by
	\begin{align*}
		n^{\ran}(\eps,F_{\mon\pm}^d \hookrightarrow L_1([0,1]^d),\Lstd)
			\,\leq\, \min\Biggl\{&
								\exp\left[C \, 
																	\frac{\sqrt{d}}{\eps^2} 
														\, \left( 1 + \left(\log \frac{d}{\eps}\right)^{3/2}
															\right)
													\right],\\
								&\exp\left[d \, \log \frac{d}{2 \eps}\right]
							\Biggr\} \,,
	\end{align*}
	with some numerical constant~\mbox{$C > 0$}.
	In particular, the curse of dimensionality does \emph{not} hold
	for the randomized $L_1$-approximation of monotone functions.
\end{theorem}
\begin{proof}
	Since we only take certain wavelet coefficients
	until a resolution~$r$ into account,
	the output will be a function that is constant on each
	of~$2^{rd}$ subcubes~$C_{r \ones,\veckappa}$
	where~\mbox{$\veckappa \in \{0,\ldots,2^r-1\}^d$}.
	The algorithm actually approximates
	\begin{equation*}
		f_r := \sum_{\substack{\vecalpha \in \N_0^d\\
														\veclambda < r}}
							\tilde{f}(\vecalpha) \, \psi_{\vecalpha} \,.
	\end{equation*}
	Since on the one hand,
	the Haar wavelets are constant on each of the~$2^{rd}$~subcubes,
	and on the other hand,
	we have $2^{rd}$~wavelets up to this resolution,
	the function~$f_r$ averages the function~$f$
	on each of the subcubes. That is,
	for~\mbox{$\vecX,\vecX' \sim \Uniform C_{r \ones,\veckappa}$}
	we have
	\begin{equation} \label{eq:frsubcube}
		\expect|f(\vecX) - f_r(\vecX)|
			= \expect|f(\vecX) - \expect' f(\vecX')|
			\leq \frac{1}{2} \,
					\left[\sup_{\vecx \in C_{r \ones,\veckappa}} f(\vecx)
								-\inf_{\vecx \in C_{r \ones,\veckappa}} f(\vecx)
					\right] \,.
	\end{equation}
	Following the same arguments
	as in the proof of \thmref{thm:MonAppOrderConv},
	we group the subcubes into diagonals, each diagonal being uniquely represented
	by a~$\veckappa$ with at least one $0$-entry.
	By monotonicity,
	summing up~\eqref{eq:frsubcube} for all cubes of a diagonal,
	we obtain the upper bound~$\frac{1}{2}$.
	Now that there are
	\mbox{$2^{rd} - 2^{r (d-1)}
					\leq d \, 2^{r (d-1)}$}~%
	diagonals,
	and the volume of each subcube is~$2^{-rd}$,
	we obtain the estimate
	\begin{equation} \label{eq:|f-fr|}
		\|f - f_r\|_{L_1} \leq \frac{d}{2^{r + 1}} \,.
	\end{equation}
	
	Surprisingly, the fact that the wavelet basis functions have a small
	support, actually helps to keep the error for estimating the wavelet
	coefficients small. Exploiting independence
	and unbiasedness (compare \eqref{eq:(f-h)^2}),
	for~\mbox{$\vecalpha \in \N_0^d$} we have
	\begin{equation} \label{eq:waveletEst}
		\begin{split}
			\expect [\tilde{f}(\vecalpha) - \tilde{g}(\vecalpha)]^2
				&= \frac{1}{n^2} \,
							\sum_{i=1}^n \expect[\tilde{f}(\vecalpha)
																		- \psi_{\vecalpha}(\vecX_i) \, f(\vecX_i)
																	]^2 \\
				&= \frac{1}{n} \left(\expect[\psi_{\vecalpha}(\vecX_i) \, f(\vecX_i)]^2
														- \tilde{f}^2(\vecalpha)
											\right) \\
				&\leq \frac{1}{n}
							\, \underbrace{\P\{ \vecX_1 \in C_{\vecalpha} \}
														}_{= 2^{-\lambda}}
							\, \expect\bigl[\underbrace{(\psi_{\vecalpha}(\vecX_1)
																						\, f(\vecX_1))^2
																					}_{\in [0,2^{\lambda}]}
													\mid \vecX_1 \in C_{\vecalpha} 
												\bigr] \\
				&\leq \frac{1}{n} \,.
		\end{split}
	\end{equation}
	
	Then by
	\eqref{eq:|f-fr|}, \lemref{lem:monSmallWavelet}, and \eqref{eq:waveletEst},
	the expected distance between input~$f$ and output~$g$
	is
	\begin{align*}
		\expect \|f - g\|_{L_1}
			&\leq \|f - f_r\|_{L_1} + \expect\|f_r - g\|_{L_2} \\
			&\leq \|f - f_r\|_{L_1}
				+ \biggl(\sum_{\substack{\vecalpha \in \N_0^d\\
																|\vecalpha|_0 > k \\
																\veclambda < r}}
									\tilde{f}(\vecalpha)^2
								+ \sum_{\substack{\vecalpha \in \N_0^d\\
																	|\vecalpha|_0 \leq k \\
																	\veclambda < r}}
										\expect [\tilde{f}(\vecalpha)
															- \tilde{g}(\vecalpha)]^2
					\biggr)^{\frac{1}{2}} \\
			&\leq \frac{d}{2^{r + 1}}
				+ \sqrt{\frac{\sqrt{d \, r}}{k+1}
								+ \frac{\# A}{n}} \,,
	\end{align*}
	where
	\begin{equation*}
		A := \{\vecalpha \in \N_0^d
						\mid |\vecalpha|_0 \leq k
								\text{ and }
								\veclambda < r\} \,.
	\end{equation*}
	Using \lemref{lem:BinomSum}, we can estimate the size of~$A$
	for~\mbox{$k \in \{1,\ldots,d\}$},
	\begin{equation*}
		\# A = \sum_{l=0}^k \binom{d}{l} \, (2^r - 1)^l
			\leq 2^{r \, k} \, \left(\frac{\euler \, d}{k}\right)^k \,.
	\end{equation*}
	This gives us the error bound for the Monte Carlo method~$A_{n,k,r}$.
	
	Choosing the resolution~\mbox{$r := \lceil \log_2 \frac{d}{\eps}
																			\rceil$}
	will bound the first term~\mbox{$2^{-(r + 1)} \, d \leq \eps / 2$}.
	Taking~\mbox{$k:= \bigl\lfloor 8 \, \sqrt{d \, (1 + \log_2 \frac{d}{\eps})}
																	/ \eps^2
										\bigr\rfloor
								\stackrel{!}{\leq} d$}
	then guarantees \mbox{$\sqrt{d \, r} / (k+1) \leq \eps^2 / 8$}.
	Finally, \mbox{$\# A / n$} can be bounded from above
	by~\mbox{$\eps^2 / 8$} if we choose
	\begin{equation*}
		n := \left\lceil \frac{8}{\eps^2}
						\, \exp\left(
										\frac{8 \, \sqrt{d \, (1 + \log_2 \frac{d}{\eps})}
												}{\eps^2}
											\left( 1
												+ \log\frac{d}{k}
												+ \log \frac{2 \,d}{\eps}
											\right)
									\right)
				\right\rceil \,.
	\end{equation*}
	By this choice we obtain the error bound~$\eps$ we aimed for.
	
	Note that if~$\eps$ is too small, we can only choose~\mbox{$k = d$}
	for the algorithm~$A_{n,k,r}$.
	In this case, for the approximation of~$f_r$,
	we would take $2^{rd}$~wavelet coefficients into account,
	$n$~would become much bigger
	in order to achieve the accuracy we aim for.
	Instead, one can approximate~$f$ directly
	via the deterministic algorithm~$A_m^d$ from \thmref{thm:MonAppOrderConv},
	which is based on~\mbox{$m^d$} function values on a regular grid.
	The worst case error is bounded by~\mbox{$e(A_m^d) \leq d/(2(m+1))$}.
	Taking \mbox{$m := 2^r-1$},
	this gives the same bound
	that we have for the accuracy at which~$f_r$ approximates~$f$,
	see \eqref{eq:|f-fr|}.
	So for small~$\eps$ we take the deterministic upper bound
	\begin{equation*}
		n^{\deter}(\eps,\App,F_{\mon}^d,\Lstd)
			\leq \exp\left(d \, \log \frac{d}{2\,\eps}
									\right) \,,
	\end{equation*}
	compare \remref{rem:MonAppOrderConv}.
\end{proof}

\begin{remark}[Non-linear algorithm with improved $\eps$-dependency]
	\label{rem:monoMCUBeps}	
	It is rather unpleasant that the estimate	in~\thmref{thm:monoUBsreal}
	depends exponentially on~$\eps^{-2}$,
	at least for~\mbox{$\eps_d \succeq 1 / \sqrt[4]{d}$}.
	In other words, for $d$-dependent error tolerances~\mbox{$\eps_d \asymp 1/\sqrt[4]{d}$},
	the cardinality~\mbox{$n = n(\eps_d,d)$}
	of~$A_{n,k,r}$ (with appropriately chosen parameters)
	depends exponentially on~$d$.
	However, there is a way to improve the $\eps$-dependency
	at the price of losing the linearity of the algorithm.
	
	For the subclass of sign-valued monotone functions
	\begin{equation*}
		F_{\mon\{\pm\}}^d
			:= \{f:[0,1]^d \rightarrow \{-1,+1\} \mid f \in F_{\mon\pm}^d\} \,,
	\end{equation*}
	we can modify the algorithm in a way similar to the Boolean setting,
	the new version~$\tilde{A}_{n,k,r}$ now returning
	an output~\mbox{$\tilde{g} := \sgn g$}.
	For~\mbox{$f \in F_{\mon\{\pm\}}^d$} we can estimate
	\begin{align*}
		\expect \|f - \tilde{g} \|_{L1}
			&\leq \|f - f_r \|_{L_1} + 2 \, \|f_r - g\|_{L_2}^2 \\
			&\leq \frac{d}{2^{r+1}} + 2 \, \frac{\sqrt{d r}}{k+1}
				+ 2 \, \frac{\# A}{n} \,,
	\end{align*}
	and for the restricted
	input set~\mbox{$F_{\mon\{\pm\}}^d \subset F_{\mon\pm}^d$}
	we obtain the complexity bound
	\begin{equation} \label{eq:monoUBnonlin}
		n^{\ran}(\eps,F_{\mon\{\pm\}}^d\hookrightarrow L_1[0,1]^d,\Lstd)
			\leq \exp\left[C' \,
										\frac{\sqrt{d}
												}{\eps} \,
											\left( 1 + \left(\log \frac{d}{\eps}\right)^{3/2}
											\right)
								\right] \,,
	\end{equation}
	with a numerical constant~$C' > 0$.
	
	This complexity bound holds actually for the whole class~$F_{\mon\pm}^d$.
	Indeed, any bounded monotone function~\mbox{$f \in F_{\mon\pm}^d$}
	can be written as an integral composition of sign-valued functions~%
	\mbox{$f_t := \sgn (f - t) \in F_{\mon\{\pm\}}^d$},
	\begin{equation*}
		f = \frac{1}{2} \int_{-1}^1 f_t \rd t \,.
	\end{equation*}
	We define a new algorithm~\mbox{$\bar{A}_{n,k,r}$} by
	\begin{equation*}
		\bar{A}_{n,k,r}^{\omega}(f)
			:= \frac{1}{2} \int_{-1}^1 \tilde{A}_{n,k,r}^{\omega}(f_t) \rd t \,.
	\end{equation*}
	Note that the information needed for the computation
	of~\mbox{$\tilde{A}_{n,k,r}^{\omega}(f_t)$} can be derived
	from the same information mapping applied to~$f$ directly
	since the algorithm is non-adaptive.
	We can write~\mbox{$\bar{A}_{n,k,r}^{\omega}
											= \bar{\phi}_{n,k,r}^{\omega} \circ N^{\omega}$}
	in contrast to~\mbox{$A_{n,k,r}^{\omega} = \phi_{n,k,r}^{\omega} \circ N^{\omega}$}.
	By the triangle inequality we get
	\begin{align*}
		e(\bar{A}_{n,k,r},f)
			&= \expect \|f - \bar{A}_{n,k,r}^{\omega}(f)\| \\
			&\leq \frac{1}{2}
							\int_{-1}^1
									\expect \|f_t - \tilde{A}_{n,k,r}^{\omega}(f_t)\|
								\rd t \\
			&\leq e(\tilde{A}_{n,k,r},F_{\mon\{\pm\}}^d) \,.
	\end{align*}
\end{remark}

\begin{remark}[Realization of the non-linear method $\bar{A}_{n,k,r}$]
	\label{rem:monoMCUBphicost}
	The resulting algorithm~$\bar{A}_{n,k,r}$ is not linear anymore.
	It is an interesting question
	how much the \emph{combinatory cost} for~$\bar{\phi}_{n,k,r}$
	differs from the cost for~$\phi_{n,k,r}$.
	For the model of computation we refer to
	the book on IBC of Traub et al.~\cite[p.~30]{TWW88},
	and to Novak and Wo\'zniakowski~\cite[Sec~4.1.2]{NW08}.
	
	The cost for computing the \textbf{linear representation $\phi_{n,k,r}$}
	is dominated by the following operations:
	\begin{itemize}
		\item Each sample~$f(\vecX_i)$ contributes to
			\mbox{$\sum_{l=0}^k \binom{d}{l} (r+1)^l
								\leq \left(\frac{\euler \, (r+1) \, d}{k}\right)^k$}
			wavelet coefficients.
			The relevant indices~$\vecalpha \in \N_0^d$
			can be determined
			effectively based on the binary representation of~$\vecX_i$.
		\item Compute the linear combination of
			\mbox{$\#A = \sum_{l=0}^k \binom{d}{l} (2^r-1)^l
								\leq \left(\frac{\euler \, 2^r \, d}{k}\right)^k$}
			wavelets.
	\end{itemize}
	The first part is the most costly part with more than~$n$ operations needed.
	If the parameters are chosen according to~\thmref{thm:monoUBsreal},
	the second part only needs about~\mbox{$n \, \eps^2/4$} operations
	in~\mbox{$L_2([0,1]^d)$}.
	We can roughly summarize the cost as
	\mbox{$n \preceq \cost \phi_{n,k,r} \preceq n^2$}.
	
	For the \textbf{non-linear representation~$\bar{\phi}_{n,k,r}$}, proceed as follows:
	\begin{itemize}
		\item Sort the information~\mbox{$(\vecX_1,y_1),\ldots,(\vecX_n,y_n)$}
			such that~\mbox{$y_1 \leq y_2 \leq \ldots \leq y_n$}.
		\item Define~\mbox{$y_0 := -1$} and~\mbox{$y_{n+1} := +1$},
			and use the representation
			\begin{equation*}
				\bar{\phi}_{n,k,r}^{\omega}(\vecy)
					= \frac{1}{2}
							\sum_{i=0}^{n}
								(y_{i+1}-y_i)
									\, \sgn \underbrace{
															\phi_{n,k,r}^{\omega}(\overbrace{-1,\ldots,-1
																											}^{\text{$i$ times}},
																						\overbrace{1,\ldots,1
																											}^{\text{$(n-i)$ times}})
														}_{=: g_i} \,.
			\end{equation*}
	\end{itemize}
	Here, the cost for computing~$g_0$
	is the cost for computing~\mbox{$\phi_{n,k,r}$}.
	For~\mbox{$i=1,\ldots,n$}, we obtain~$g_i$ from modifying~$g_{i-1}$,
	indeed, by linearity of~$\phi_{n,k,r}$ we have
	\begin{equation*}
		g_i := g_{i-1} - 2 \, \phi_{n,k,r}^{\omega}(\vece_i) \,.
	\end{equation*}
	Since for~\mbox{$\phi_{n,k,r}^{\omega}(\vece_i)$} we only need to take~%
	\mbox{$\sum_{l=0}^k \binom{d}{l} (r+1)^l
						\leq \left(\frac{\euler \, (r+1) \, d}{k}\right)^k$}
	wavelet coefficients into account,
	doing this~$n$ times, the cost for computing~\mbox{$g_0,\ldots,g_n$} is
	only twice the cost of~$\phi_{n,k,r}$.
	(Here, however, we need more operations in~\mbox{$L_2([0,1]^d)$},
	and less operations in~$\R$, but we assumed them to have the same cost,
	no matter how realistic that is.)
	The signum operator and the final sum contribute to the cost
	only linearly in~$n$.
	Ordering the information has an expected cost of~\mbox{$n \log n$},
	which is likely to be dominated by the cost of~\mbox{$\phi_{n,k,r}$}.
	Assuming this, we obtain
	\begin{equation*}
		\cost \phi_{n,k,r} \asymp \cost \bar{\phi}_{n,k,r} \,.
	\end{equation*}
	
	Heinrich and Milla~\cite[Sec~6.2]{HeM11} pointed out
	that for problems with functions as output,
	the interesting question is not about a complete picture of the
	output~\mbox{$\phi(\vecy)$},
	but about effective computation of approximate
	function values~\mbox{$[\phi(\vecy)](\vecx)$} on demand.
	In our situation it makes sense to distinguish between pre-processing operations
	and operations on demand.
	
	For the \textbf{linear representation $\phi_{n,k,r}$} we have
	\begin{description}
		\item[\quad \textnormal{\emph{Pre-processing:}}]
			Compute and store the wavelet coefficients that are needed for the output.
		\item[\quad \textnormal{\emph{On demand:}}]
			Compute~\mbox{$[\phi_{n,k,r}^{\omega}(\vecy)](\vecx)$}.\\
			(Only \mbox{$\sum_{l=0}^k \binom{d}{l} (r+1)^l$}
			wavelet coefficients are relevant to \mbox{$[\phi_{n,k,r}^{\omega}(\vecy)](\vecx)$}.)
	\end{description}
	The pre-processing is approximately as expensive
	as the cost with the above computational model,
	computation on demand is significantly cheaper.
	
	For the \textbf{non-linear representation~$\bar{\phi}_{n,k,r}$} we have
	\begin{description}
		\item[\quad \textnormal{\emph{Pre-processing:}}]
			Rearrange the information.\\
			Store the wavelet coefficients needed for~$g_0$.
		\item[\quad \textnormal{\emph{On demand:}}]
			Compute~\mbox{$g_0(\vecx)$}.\\
			(For this, \mbox{$\sum_{l=0}^k \binom{d}{l} (r+1)^l$}~wavelet coefficients
			are relevant.)\\
			In order to compute \mbox{$[\bar{\phi}_{n,k,r}^{\omega}(\vecy)](\vecx)$},
			we need in particular the values~\mbox{$[\phi_{n,k,r}^{\omega}(\vece_i)](\vecx)$}
			for~\mbox{$i=1,\ldots,n$}.
			These can be determined in a very effective way.\footnote{%
				Observe that
				\begin{equation*}
					n \, [\phi_{n,k,r}^{\omega}(\vece_i)](\vecx)
						= \sum_{\substack{\vecalpha \in \N_0^d\\
															|\vecalpha|_0 \leq k\\
															|\vecalpha|_{\infty} < 2^r}}
								\psi_{\vecalpha}(\vecX_i) \, \psi_{\vecalpha}(\vecx)
						= \sum_{\substack{\vecalpha \in \N_0^d\\
															|\vecalpha|_0 \leq k\\
															|\vecalpha|_{\infty} < 2^r}}
								\prod_{j=1}^d h_{\alpha_j}(\vecX_i(j)) \, h_{\alpha_j}(x_j)
						= \sum_{\substack{\vecbeta \in \{0,1\}^d\\
															|\beta|_1 \leq k}}
								\vecZ^{\vecbeta} \,,
				\end{equation*}
				where~\mbox{$Z_j := \sum_{\alpha_j=1}^{2^r-1}
															h_{\alpha_j}(\vecX_i(j))
																\, h_{\vecalpha_j}(x_j)$}.
				It is readily checked that
				\begin{equation*}
					Z_j = \begin{cases}
									2^r - 1
										\quad&\text{if $\lfloor 2^r \, \vecX_i(j) \rfloor
																			= \lfloor 2^r \, x_j \rfloor$,}\\
									-1
										\quad&\text{else,}
								\end{cases}
				\end{equation*}
				so a comparison of the first~$r$ digits of the binary
				representation of~$\vecX_i(j)$ and $x_j$
				is actually enough for determining~$Z_j$.
				In the end,
				we only need the number~$b$ of coordinates
				where \mbox{$\lfloor 2^r \, \vecX_i(j) \rfloor
												= \lfloor 2^r \, x_j \rfloor$},
				and obtain
				\begin{equation*}
					n \, [\phi_{n,k,r}^{\omega}(\vece_i)](\vecx)
						= \sum_{l = 0}^{b \wedge k} \binom{b}{l} \, (2^r - 1)^l \,
								\sum_{m = 0}^{(d-b)\wedge(k-l)} \binom{d-b}{m} \, (-1)^m
						=: \chi(b) \in \Z\,.
				\end{equation*}
				These values~\mbox{$\chi(b)$}
				are needed for \mbox{$b \in \{0,\ldots,d\}$}.
				Since they only depend on parameters of the algorithm,
				they can be prepared before any information was collected.
				}
	\end{description}
	For the pre-processing there is no big difference from the setting before.
	The part ``on demand'' it is a little cheaper than the pre-processing part,
	however, we need more than~$n$ operations.
	Hence a linear algorithm with the same information cardinality on the one hand
	is less costly, on the other hand the error is larger.
	
	We conclude that it depends on~$\eps$ and the ratio of information cost
	versus combinatory cost
	whether the linear or the non-linear algorithm should be preferred.
\end{remark}

\begin{remark}[Deterministic methods for~$\Lall$]
	If we are allowed asking the oracle for wavelet coefficients directly,
	the same algorithmic idea is implementable as a deterministic method
	and less information is needed.
	(This reduction of the complexity is
	by a factor~\mbox{$\eps^2/8$} or \mbox{$\eps/4$}.) 
	In particular, the curse of dimensionality does \emph{not} hold in
	the deterministic setting with~$\Lall$.
	
	It is an open problem whether similar lower bounds
	to those from \secref{sec:monoMCLBs}
	can be found for~$\Lall$.
	The proof technique of \thmref{thm:monotonLB},
	however, will not work for~$\Lall$,
	because one could choose a functional
	that injectively maps all possible Boolean functions onto the real line,
	and thus identify any given function by just one measurement.
	
	Indeed, one measurement is sufficient. Let~\mbox{$\vecx \in \{0,1\}^d$}
	be the binary representation of an integer,
	\begin{equation*}
		b(\vecx) := \sum_{j=1}^d x_j \, 2^{j-1} \in \N_0.
	\end{equation*}
	Define the functional~$L_1$
	for Boolean functions~\mbox{$f: \{0,1\}^d \rightarrow \{0,1\}$}
	by
	\begin{equation*}
		L_1(f) := \sum_{\vecx \in \{0,1\}^d} f(\vecx) 2^{b(\vecx)} \,.
	\end{equation*}
	This functional maps~$G_{\Boole}^d$ bijectively
	onto the set~\mbox{$\{0,\ldots,2^{2^d}-1\}$} of $2^d$-bit representable integers.
	This shows that~$\Lall$ is an inappropriate model
	for the approximation of Boolean functions.
\end{remark}

\appendix
\automark[chapter]{chapter}

\chapter{On Gaussian Measures}

Let $\Phi(x) = \frac{1}{\sqrt{2\pi}}
									\, \int_{-\infty}^x \exp\left(- \frac{t^2}{2}\right) \rd t$
denote the cumulative distribution function of a standard Gaussian variable.\\
Let~\mbox{$\vecX = (X_1,\ldots,X_m)$} be a standard Gaussian vector
in~\mbox{$\R^m = \ell_2^m$},
i.e.\ the~$X_i$ are iid standard Gaussian variables.
For any linear mapping~\mbox{$J: \ell_2^m \rightarrow \widetilde{F}$}
with $\widetilde{F}$ being a Banach space,
then~\mbox{$J\vecX$} is a zero-mean Gaussian vector in~$\widetilde{F}$.
For the norm of a vector~$\vecx \in \ell_2^m$ we write~\mbox{$\|\vecx\|_2$}.
The operator norm of~$J$ shall be denoted~\mbox{$\|J\|_{2 \rightarrow F}$}.

\section{Comparison of Gaussian Measures}

\begin{lemma}\label{lem:E|Y|<E|Y+Z|}
	Let~$\vecY$ and $\vecZ$ be independent zero-mean Gaussian vectors
	in a normed space~$\widetilde{F}$.
	Then
	\begin{equation*}
		\expect\|\vecY\|_F \leq \expect \|\vecY + \vecZ\|_F \,.
	\end{equation*}
\end{lemma}
\begin{proof}
	Due to symmetry of zero-mean Gaussian measures,
	the distributions of \mbox{$\vecY + \vecZ$} and \mbox{$\vecY - \vecZ$}
	are identical.
	Hence, by the triangle inequality,
	\begin{align*}
		\expect \|\vecY\|_F
			&= \expect \left\|{\textstyle \frac{1}{2} \sum_{\sigma = \pm 1}}
													(\vecY + \sigma \vecZ )
								\right\|_F \\
			&\leq \textstyle{\frac{1}{2} \sum_{\sigma = \pm 1}}
							\expect \|\vecY + \sigma \vecZ\|_F \\
			&=	\expect \|\vecY + \vecZ\|_F \,.
	\end{align*}
\end{proof}

\begin{lemma} \label{lem:E|sum aXf|}
	Let~\mbox{$X_1,\ldots,X_m$} be independent standard Gaussian random variables,
	\mbox{$f_1,\ldots,f_m$} be elements in a normed vector space~$\widetilde{F}$,
	and~\mbox{$a_k' \geq a_k \geq 0$} be real numbers.
	Then
	\begin{equation*}
		\expect \left\| \sum_{k=1}^m a_k \, X_k \, f_k
						\right\|_F
			\leq \expect \left\| \sum_{k=1}^m a_k' \, X_k \, f_k
									\right\|_F \,.
	\end{equation*}
\end{lemma}
\begin{proof}
	Take another~$m$ independent standard Gaussian random variables~%
	\mbox{$X_1',\ldots,X_m'$}.
	Then the \mbox{$a_k \, X_k + \sqrt{a_k'^2 - a_k^2} \, X_k'$}
	are identically distributed to~\mbox{$a_k' \, X_k$}.
	Thus, applying \lemref{lem:E|Y|<E|Y+Z|}, we obtain
	\begin{align*}
		\expect \left\| \sum_{k=1}^m a_k \, X_k \, f_k
						\right\|_F
			&\leq \expect \left\| \sum_{k=1}^m a_k \, X_k \, f_k
													+ \sum_{k=1}^m \sqrt{a_k'^2 - a_k^2} \, X_k' f_k
									\right\|_F \\
			&= \expect \left\| \sum_{k=1}^m a_k' \, X_k \, f_k
								\right\|_F \,.
	\end{align*}
\end{proof}

There are several other known comparison principles for Gaussian vectors,
mainly based on covariance comparisons, and concerning the expected maximum component
instead of arbitrary norms,
see for example Lifshits~\cite[pp.~186--192]{Lif95}.
For Gaussian fields one important result is Slepian's inequality,
see e.g.\ Adler~\cite[Cor~2.4]{Adl90}.

\section{Restrictions of Gaussian Measures}

With~$P$ being an orthogonal projection in~$\ell_2^m$,
the random vector~\mbox{$JP \vecX$} can be
interpreted as a restriction of the Gaussian measure
to the subspace~\mbox{$\image JP \subseteq \widetilde{F}$}.
We start with a rather simple comparison of the expected norms.
\begin{lemma} \label{lem:E|JPX|<=E|JX|}
	For orthogonal projections~$P$ on~$\ell_2^m$ we have
	\begin{equation*}
		\expect \|J P \vecX\|_F \leq \expect \|J \vecX\|_F \,.
	\end{equation*}
\end{lemma}
\begin{proof}
	Due to orthogonality, \mbox{$\vecY := J P \vecX$} and \mbox{$\vecZ := J(\id - P) \vecX$}
	are stochastically independent.
	The claim follows from \lemref{lem:E|Y|<E|Y+Z|}.
\end{proof}

As an application we have a bound for the operator norm of~$J$.
\begin{lemma} \label{lem:E|JX|>c|J|}
	For any linear operator~$J : \ell_2^m \rightarrow \widetilde{F}$ we have
	\begin{equation*}
		\expect\|J \vecX\|_F
			\geq \sqrt{\textstyle \frac{2}{\pi}} \, \|J\|_{2 \rightarrow F} \,.
	\end{equation*}
\end{lemma}
\begin{proof}
	Applying \lemref{lem:E|JPX|<=E|JX|} to rank-$1$ orthogonal projections,
	we obtain
	\begin{equation*} \label{eq:expect-opnorm}
		\expect\|J \vecX\|_F
			\geq
				\sup_{\substack{\text{$P$ orth. Proj.}\\
												\rank{P}=1}}
					\expect\|JP \vecX\|_F
			= \|J\|_{2 \rightarrow F} \, \expect |X_1|
			= \sqrt{\textstyle \frac{2}{\pi}} \, \|J\|_{2 \rightarrow F} \,.
	\end{equation*}
\end{proof}

\section{Deviation Estimates for Gaussian Measures}

\corref{cor:deviationGauss} is a deviation result
for norms of Gaussian vectors
known from Pisier~\cite[Thm~2.1,~2.2, Rem~p.~180/181]{Pis86}.
It has been used in this form by Heinrich~\cite[Prop~1]{He92} for his proof
on lower bounds for the Monte Carlo error via Gaussian measures,
which is reproduced in Chapter~\ref{chap:Bernstein}.
The deviation result for the norms of Gaussian vectors is
a consequence of the following proposition.

\begin{proposition} \label{prop:dev2}
	Let $f : \ell_2^m \rightarrow \R$ be a Lipschitz function with Lipschitz
	constant~$L$. Then for $t>0$ we have
	\begin{equation*}
		\P\{f(\vecX) - \expect f(\vecX) > t\}
			\leq \exp\left(- \frac{t^2}{
														2 L^2}
								\right) \,.
	\end{equation*}
\end{proposition}

For this result several proofs are known.
One way uses stochastic integration,
a rather short description of this approach
can be found in Pisier~\cite[Rem~p.~180/181]{Pis86}.
A more direct proof is contained in Adler and Taylor~\cite[Lem~2.1.6]{AT07}.
Both methods require higher smoothness for~$f$ first, an assumption
that can be removed in the end with Fatou's inequality.
For a simpler proof with a slightly worse constant in the exponent,
see Pisier~\cite[p.~176--178]{Pis86}.

\begin{corollary} \label{cor:deviationGauss}
	For \mbox{$\lambda > 1$} and
	with \mbox{$\rho := \expect \|J \vecX\|_F / \|J\|_{2 \rightarrow F}$}
	we have
	\begin{equation*}
		\P\{\|J \vecX\|_F > \lambda \, \expect \|J \vecX\|_F\}
			\leq \exp\left(- \frac{(\lambda-1)^2 \, \rho^2}{2} \right)
			\leq \exp\left(- \frac{(\lambda-1)^2}{\pi} \right) \,.
	\end{equation*}
\end{corollary}
\begin{proof}
	We define~\mbox{$f:\ell_2^m \rightarrow \R$} by
	\begin{equation*}
		f(\vecx) := \|J \vecx\|_F \,.
	\end{equation*}
	This function is Lipschitz continuous with Lipschitz constant
	\begin{equation*}
		L = \sup_{\substack{\vecx,\vecz \in \ell_2^m \\
												\vecx \not= \vecz}}
					\frac{|f(\vecx)-f(\vecz)|}{\|\vecx-\vecz\|_2}
			\stackrel{\text{$\Delta$-ineq.}}{\leq}
				\sup_{\substack{\vecx,\vecz \in \ell_2^m \\
												\vecx \not= \vecz}}
					\frac{\|J(\vecx-\vecy)\|_F}{\|\vecx-\vecz\|_2}
			= \|J\|_{2 \rightarrow F} \,,
	\end{equation*}
	so that we can apply \propref{prop:dev2}
	with~\mbox{$t = (\lambda-1) \, \expect\|J \vecX\|_F$}.
	Actually, we have equality~\mbox{$L = \|J\|_{2 \rightarrow F}$},
	to see this, consider the LHS
	with the supremum for~\mbox{$\vecx \not= \zeros = \vecz$}.
	
	For this kind of estimate
	it is sufficient to bound~\mbox{$\|J\|_{2 \rightarrow F}^2$} from above,
	or the ratio~\mbox{$\rho$} from below.
	In the most general way, \mbox{$\rho \geq \sqrt{2/\pi}$},
	see \lemref{lem:E|JX|>c|J|}.
\end{proof}

\section{Gaussian Vectors in Sequence Spaces}
\label{sec:E|X|_p}

The next two lemmas are needed for the proof of \lemref{lem:gaussqnormvector}
which follows the lines of Pisier~\cite[Lem~4.14]{Pis89}.
In there the inequalities of \lemref{lem:gausstail} and~\ref{lem:gaussqnorm}
have been mentioned
without giving a proof or detailed information about the constants.

\begin{lemma}[Tail estimate] \label{lem:gausstail}
	There exists a constant~\mbox{$\alpha>0$}
	such that
	\begin{equation}\label{eq:gausstail}
		\P\left\{|X|>\alpha\sqrt{\log x}\right\}
			\geq \frac{1}{x} \quad \text{for\, $x\geq\euler$,}
	\end{equation}
	where~$X$ is a standard Gaussian random variable.\\
	The value
	\mbox{$\alpha = \Phi^{-1}\left(1-\frac{1}{2\euler}\right)
				= - \Phi^{-1}\left(\frac{1}{2\euler}\right)
				= 0.90045...$} 
	is optimal,
	providing equality for~\mbox{$x=\euler$}.
	Here, $\Phi$ denotes the cumulative distribution function
	of the standard normal distribution.
\end{lemma}
\begin{proof}
	Having chosen~$\alpha$ as above,
	we consider the left-hand side of the inequality to be shown,
	\begin{align*}
		\LHS = \P\left\{|X|>\alpha\sqrt{\log x}\right\}
			&= \sqrt{\frac{2}{\pi}}
					\int_{\alpha\sqrt{\log x}}^{\infty}
						\exp\left(-\frac{t^2}{2}\right) \rd t \\
		[t = \alpha \sqrt{\log s}]\quad
			&= \frac{\alpha}{\sqrt{2\pi}}
					\int_x^{\infty}
						\frac{s^{-(\alpha^2/2+1)}}{\sqrt{\log s}} \rd s
	\end{align*}
	Its derivative is
	\begin{equation*}
		\frac{\diff}{\diff x} \LHS
			= - \frac{\alpha}{2\pi}
					\, \frac{x^{-(\alpha^2/2 + 1)}}{\sqrt{\log x}} \,,
	\end{equation*}
	whereas for the right-hand side we have
	\begin{equation*}
		\frac{\diff}{\diff x} \RHS = - \frac{1}{x^2} \,.
	\end{equation*}
	The derivatives of both sides agree iff
	\begin{equation} \label{eq:gausstaildiff}
		\frac{\alpha^2}{2\pi} x^{2-\alpha^2} = \log x \,.
	\end{equation}
	Since~\mbox{$\alpha < 1$}, we have that $x^{2-\alpha^2}$ is convex,
	whereas \mbox{$\log x$} is concave for~\mbox{$x>0$},
	so there are no more than two points~\mbox{$0<x_1<x_2$}
	that solve~\eqref{eq:gausstaildiff}.
	Comparing both sides of~\eqref{eq:gausstaildiff}, we obtain:
	\begin{align*}
		\frac{\alpha}{\sqrt{2\pi}}
				> 0
			\quad&\text{for\, $x = 1$,}\\
		\frac{\alpha^2}{2\pi} \euler^{2-\alpha^2}
				\stackrel{\alpha<1}{<}
					\frac{\euler}{2\pi} < 1
			\quad&\text{for\, $x = \euler$, and}\\
		\frac{\alpha^2}{2\pi} x^{2-\alpha^2}
				> \log x
			\quad&\text{for sufficiently large~$x$, say $x > x_2$.}
	\end{align*}
	Therefore $1<x_1<\euler<x_2$.
	
	Now for~\eqref{eq:gausstail}, by choice of $\alpha$, we have
	\begin{align*}
		\LHS &= \RHS \qquad \text{for\, $x=\euler$,}\\
		\lim_{x\rightarrow \infty} \LHS
			&= \lim_{x \rightarrow \infty} \RHS = 0 \,.
	\end{align*}
	With
	\begin{multline*}
		\frac{\diff}{\diff x} \LHS \mid_{x=\euler}
				= - \frac{\alpha}{\sqrt{2\pi}} \, \euler^{-(\alpha^2/2+1)}
				\approx -0.088107\\
			> \frac{\diff}{\diff x} \RHS \mid_{x=\euler}
				= - \euler^{-2}
				\approx -0.135335 \,,
	\end{multline*}
	and only having one point~\mbox{$x=x_2>\euler$}
	where the difference~\mbox{$\LHS - \RHS$} is locally extreme,
	we obtain that \mbox{$\LHS \geq \RHS$} for $\euler \leq x < \infty$.
\end{proof}

\begin{lemma} \label{lem:gaussqnorm}
	There exists a constant~\mbox{$K>0$} such that
	for all~\mbox{$1 \leq q < \infty$} we have
	\begin{equation*}
		\left(\expect |X|^q\right)^{1/q} \leq K \sqrt{q} \,,
	\end{equation*}
	where~$X$ is a standard Gaussian variable.
\end{lemma}
\begin{proof}
	We can write
	\begin{equation*}
		\gamma(q)
			:= \left(\expect |X|^q\right)^{1/q}
			= \left(\sqrt{\frac{2}{\pi}} \,
								\int_0^{\infty} x^q \, \exp\left(-\frac{x^2}{2}\right) \rd x
				\right)^{1/q} \,.
	\end{equation*}
	Given~\mbox{$\alpha \in (0,1)$},
	one can find~\mbox{$c>0$} such that
	\begin{equation*}
		x^q \, \exp\left(-\frac{x^2}{2}\right)
			\leq c \, x \, \exp\left(-\alpha \, \frac{x^2}{2}\right)\,,
	\end{equation*}
	in detail,
	\mbox{$\displaystyle c
					= \left(\frac{q-1}{\euler \,(1-\alpha)}\right)^{(q-1)/2}$}
	is optimal with equality holding
	in~\mbox{$\displaystyle x = \sqrt{\frac{q-1}{1-\alpha}}$}.
	We obtain
	\begin{equation*}
		\int_0^{\infty} x^q \, \exp\left(-\frac{x^2}{2}\right) \rd x
			\leq c \, \int_0^{\infty} x \, \exp\left(-\alpha \, \frac{x^2}{2}\right) \rd x
			= \frac{c}{\alpha}
			= \frac{1}{\alpha}
					\, \left(\frac{q-1}{\euler \, (1-\alpha)}\right)^{(q-1)/2}
			\,.
	\end{equation*}
	This estimate can be minimized
	by choosing~\mbox{$\displaystyle \alpha = \frac{2}{q+1}$},
	so
	\begin{equation*}
		\gamma(q) \leq
			\underbrace{\frac{1}{\sqrt{\euler}} \,
									\left(\frac{\euler \, (q+1)}{2\pi}\right)^{1/(2q)}
									\, \sqrt{1 + \frac{1}{q}}
									}_{
					=: K(q)}
				\, \sqrt{q} \,.
	\end{equation*}
	We can carry out a final estimate (finding the extreme point by
	\mbox{$0 \stackrel{\text{!}}{=} \frac{\diff}{\diff q} \log K(q)$}),
	\begin{equation*}
		K := \sup_{q\geq 1} K(q)
			= K\left(\frac{2\pi}{\euler} - 1\right)
			= \sqrt{\frac{2\pi}{\euler \, (2\pi - \euler)}}
			\approx 0.805228 \,.
	\end{equation*}
	This value for~$K$ is not optimal because for~\mbox{$q>1$} our estimate is rough.
	But it is close to optimal
	since~\mbox{$\gamma(1) = \sqrt{2/\pi} \approx 0.797885$},
	i.e.\ if one could show that the lemma is true with~\mbox{$K = \sqrt{2/\pi}$},
	one would have a sharp estimate with equality holding for~\mbox{$q=1$}.
	Numerical calculations strongly encourage this conjecture.
\end{proof}

Now we can state and prove the norm estimates for Gaussian vectors,
the proof (without explicit constants) is found in Pisier~\cite[Lem~4.14]{Pis89}.

\begin{lemma} \label{lem:gaussqnormvector}
	Let~\mbox{$\vecX= (X_1,\ldots,X_m)$} be a standard Gaussian vector on~$\R^m$.
	\begin{enumerate}[(a)]
		\item For~\mbox{$1 \leq q < \infty$} we have
			\begin{equation*}
				\sqrt{\frac{2}{\pi}} \, m^{1/q}
					\leq \expect \|\vecX\|_q
					\leq K \sqrt{q} \, m^{1/q} \,.
			\end{equation*}
		\item There exist constants $c,C>0$ such that
			\begin{align*}
				c \, \sqrt{1+\log m}
					\leq \expect \|\vecX\|_{\infty} \leq C \, \sqrt{1+\log m} \,.
			\end{align*}
	\end{enumerate}
\end{lemma}
\begin{proof}
	(a) The upper bound follows from
	\begin{equation*}
		\expect \|\vecX\|_q
			= \expect \left(\sum_{j=1}^m |X_j|^q \right)^{1/q}
			\stackrel{\text{Jensen}}{\leq}
				\left(\sum_{j=1}^m \expect |X_j|^q\right)^{1/q} \,,
	\end{equation*}
	and \lemref{lem:gaussqnorm}.
	
	For the lower bound we use the triangle inequality
	with the vector of absolute values~\mbox{$\abs \vecX := (|X_1|,\ldots,|X_m|)$},
	\begin{equation*}
		\expect \|\vecX\|_q
			\geq \| \expect \abs \vecX  \|_q
			= \left(\sum_{j=1}^m (\expect|X_j|)^q \right)^{1/q}
			= \sqrt{\frac{2}{\pi}} \, m^{1/q} \,.
	\end{equation*}
	
	(b) The upper bound is obtained by a comparison to some $q$-norm
	for~\mbox{$q \geq 1$}, from~(a) we have
	\begin{equation*}
		\expect \|\vecX\|_{\infty}
			\leq
				\expect \|\vecX\|_q 
			\stackrel{\text{Jensen}}{\leq}
				\left(\expect \|\vecX\|_q^q\right)^{1/q}
			\leq
				K \, \sqrt{q} \, m^{1/q}\,.
	\end{equation*}
	For the choice~\mbox{$q = 1+\log m$} we have
	\mbox{$m^{1/q}
					= \exp\left(\frac{\log m}{1+\log m}\right)
					< \euler$},
	so that we obtain the desired upper bound with
	\mbox{$C = K \euler \approx 2.18884$}.
	
	Finally the lower bound.
	Using \lemref{lem:gausstail} with~\mbox{$x = \euler \, m$},
	we get
	\begin{equation*}
		\P\left\{|X_j|>\alpha\sqrt{1 + \log m}\right\}
			\geq \frac{1}{\euler \, m} \,.
	\end{equation*}
	By this,
	\begin{align*}
		\P\Bigl\{\max_{j=1,\ldots,m} |X_j|
							\leq \alpha \sqrt{1+\log m}\Bigr\}
			\leq \left(1 - \frac{1}{\euler \, m}\right)^m
			\leq \exp\left({-\euler^{-1}}\right).
	\end{align*}
	We conclude
	\begin{align*}
		\expect \|\vecX\|_{\infty} = \expect \max_{j=1,\ldots,m} |X_j|
			&\geq \alpha \sqrt{1+\log m} \,
					\left(1 - \P\Bigl\{\max_{j=1,\ldots,m} |X_j|
															\leq \alpha \sqrt{1+\log m}\Bigr\}
					\right) \\
			&\geq \alpha \left(1 - \exp\left({-\euler^{-1}}\right)\right)
							\, \sqrt{1 + \log m} \,.
	\end{align*}
	This gives us
	\mbox{$c = \alpha \, \left(1 - \exp\left(-\euler^{-1}\right)\right)
				\approx 0.277159$}.
\end{proof}

\chapter{Useful Inequalities}

\section{Combinatorics}

The following two lemmas are well known.

\begin{lemma}[Sum of binomial coefficients] \label{lem:BinomSum}
	For~$k=1,\ldots,d$ we have
	\begin{equation*}
		\sum_{l=0}^k \binom{d}{l} < \left(\frac{\euler \, d}{k}\right)^k \,.
	\end{equation*}
\end{lemma}
\begin{proof}
	For $k=1$ we have
	\begin{equation*}
		\LHS = 1 + d
			< \euler \, d
			= \RHS \,.
	\end{equation*}
	For~$k=d$ we have
	\begin{equation*}
		\LHS = 2^d
			< \euler^d
			= \RHS \,.
	\end{equation*}
	This completes the cases~$d=1,2$.
	
	Assume that the lemma holds for~$d$ and~\mbox{$k=1,\ldots,d$}.
	We show that it holds for~$d+1$ and~\mbox{$k=2,\ldots,d$} as well.
	\begin{align*}
		\sum_{l=0}^k \binom{d+1}{l}
			& = \sum_{l=0}^k \binom{d}{l} + \sum_{l=0}^{k-1} \binom{d}{l} \\
			&< \left(\frac{\euler \, d}{k-1}\right)^{k-1}
							+ \left(\frac{\euler \, d}{k}\right)^k \\
			&= \left(\frac{\euler}{k}\right)^k
							\,\left(\underbrace{\frac{\left(1 + \frac{1}{k-1}\right)^{k-1}
																			}{\euler}
																}_{< 1}
												\, k \, d^{k-1}
											+ d^k
								\right) \\
			&< \left(\frac{\euler \, (d+1)}{k}\right)^k \,.
	\end{align*}
	By induction, this completes the proof.
\end{proof}

\begin{lemma}[Central binomial coefficient]\label{lem:(d d/2)}
For $d > 0$ we have
\begin{equation*}
	\frac{1}{2} \, \frac{2^d}{\sqrt{d}} \leq \binom{d}{\lfloor d/2 \rfloor}
		\leq \sqrt{\frac{2}{\pi}} \, \frac{2^d}{\sqrt{d}} \,.
\end{equation*}
\end{lemma}
\begin{proof}
	For odd~\mbox{$d = 2k + 1$}, \mbox{$k \in \N_0$},
	by induction we see that~\mbox{$\binom{2k+1}{k} \, 2^{-(2k+1)} \, \sqrt{2k+1}$}
	is monotonously increasing in~$k$.
	Similarly, for even~\mbox{$d = 2k$}, \mbox{$k \in \N$},
	we observe that \mbox{$\binom{2k}{k} \, 2^{-2k} \, \sqrt{2k}$}
	is increasing.
	Hence the lower bound is done by considering~\mbox{$d=1,2$}.
	For the upper bound we need the limit as~\mbox{$k \rightarrow \infty$},
	this can be obtained by Stirling's formula.
\end{proof}

\section{Quantitative Central Limit Theorem}

\begin{proposition}[Berry-Esseen inequality] \label{prop:BerryEsseen}
	Let \mbox{$X_1,X_2,\ldots$} be iid random variables with
	zero mean, unit variance and finite third absolute moment~$\beta_3$.
	Then there exists a universal constant~$C_0$
	such that
	\begin{equation*}
		\left|\P\biggl\{\frac{1}{\sqrt{d}} \sum_{j=1}^d X_j \leq x \biggr\}
					- \Phi(x)\right|
			\leq \frac{C_0 \, \beta_3}{\sqrt{d}} \,,
	\end{equation*}
	where $\Phi(\cdot)$ is the cumulative distribution function of the
	univariate standard normal distribution.\\
	The best known estimates on~$C_0$ are
	\begin{equation*}
		C_E := \frac{\sqrt{10}+3}{6\sqrt{2\pi}} = 0.409732\ldots
			\leq C_0
			< 0.4748
	\end{equation*}
	see Shevtsova~\cite{Shev11}.
\end{proposition}

\begin{corollary} \label{cor:BerryEsseenBinom}
	Let~\mbox{$a := \lceil \frac{d}{2} + \alpha \frac{\sqrt{d}}{2} \rceil$}
	and~\mbox{$b := \lfloor \frac{d}{2} + \beta \frac{\sqrt{d}}{2} \rfloor$}
	with real numbers~\mbox{$\alpha < \beta$}.
	Then we have
	\begin{equation*}
		\frac{1}{2^d} \sum_{k=a}^{b} \binom{d}{k}
			\,\geq\, \Phi(\beta) - \Phi(\alpha) - \frac{2 \, C_0}{\sqrt{d}} \,.
	\end{equation*}
\end{corollary}
\begin{proof}
	Let~\mbox{$X_1,\ldots,X_d \stackrel{\text{iid}}{\sim} \Uniform \{0,1\}$}
	be Bernoulli random variables
	and~\mbox{$Z_j := 2 \, X_j - 1$} the corresponding Rademacher random variables.
	Note that the~$Z_i$ have zero mean, unit variance,
	and third absolute moment~\mbox{$\beta_3 = 1$}.
	Applying \propref{prop:BerryEsseen} twice to the~$Z_j$, we obtain
	\begin{align*}
		\frac{1}{2^d} \sum_{k=a}^{b} \binom{d}{k}
			&= \P\biggl\{\frac{d}{2} + \alpha \frac{\sqrt{d}}{2}
										\leq \sum_{j=1}^d X_j
										\leq \frac{d}{2} + \beta \frac{\sqrt{d}}{2}\biggr\} \\
			&= \P\biggl\{\alpha \leq \frac{1}{\sqrt{d}}
																\, \sum_{j=1}^d Z_j \leq \beta \biggr\} \\
			&\geq \Phi(\beta) - \Phi(\alpha) - \frac{2 \, C_0}{\sqrt{d}} \,.
	\end{align*}
\end{proof}

\chapter*{Abbreviations and Symbols}
\addcontentsline{toc}{chapter}{Abbreviations and Symbols}
\pagestyle{scrplain}

\begin{tabularx}{\textwidth}{cX}
\multicolumn{2}{X}{
\textbf{Abbreviations}
\bigskip}\\
	a.s.
		& almost surely \\
	IBC
		& information-based complexity \\
	iid
		& independent and identically distributed (random variables) \\
	$\LHS$/$\RHS$
		& \emph{left-hand side}/\emph{right-hand side}
			(of an equation or inequality referred to)\\
	RKHS
		& reproducing kernel Hilbert space \\
\phantom{void}& \\
\multicolumn{2}{X}{\textbf{On Functions and Real Numbers}
\bigskip}\\
	$a_{+} := \max\{a,0\}$
		& positive part of a real number~$a \in \R$ \\
	$\lfloor a \rfloor$ 
		& $a \in \R$ rounded down to an integer\\
	$\lceil a \rceil$ 
		& $a \in \R$ rounded up to an integer\\
	$\ind[Statement]$
		& characteristic function, yielding~$1$ if $Statement$ is true,
			and taking the value~$0$ if $Statement$ is false \\
	$\delta_{ij} = \ind[i=j]$
		& Kronecker symbol \\
	$\log$
		& natural logarithm \\
	$f \preceq g$
		& $f$ and~$g$ non-negative functions on a common domain
			and \mbox{$f \leq C \, g$} for a constant~\mbox{$C > 0$}\\
	$f \asymp g$
		& $f \preceq g$ and $g \preceq f$, that is,
			\mbox{$c\, g \leq f \leq C \, g$} with $c,C>0$\\
	$a_n \prec b_n$ for $n \rightarrow \infty$
		& \mbox{$(a_n)_{n \in \N},(b_n)_{n \in \N} \subset [0,\infty)$}
			with~\mbox{$\frac{a_n}{b_n} \xrightarrow[n \rightarrow \infty]{} 0 \,$}\\
	$a_n \preceq b_n$ for $n \rightarrow \infty$
		& there exist \mbox{$n_0 \in \N$} and \mbox{$C > 0$}
			such that for~\mbox{$n \geq n_0$} we have~\mbox{$a_n \leq C \, b_n$} \\
\phantom{void}& \\
\multicolumn{2}{X}{\textbf{Vectors and Normed Spaces}
\bigskip}\\
	$\vecx = (x_1,\ldots,x_m)$
		& vector in~$\R^m$ (or~$\ell_p$) with entries $\vecx(i) = x_i$\\
	$\vecx_A = (x_i)_{i \in A} \in \R^A$
		& sub-vector for $A \subseteq \{1,\ldots,m\}$\\
	$[k] := \{1,\ldots,k\}$
		& first $k$ natural numbers\\
	$\zeros := (0,\ldots,0)$
		& vector with all the entries set to~$0$ \\
	$\ones := (1,\ldots,1)$
		& vector with all the entries set to~$1$ \\
	$\llbracket\vecx,\vecz\rrbracket := \bigtimes_{j=1}^m [x_j,z_j]$
		& closed cuboid with vector-valued interval boundaries
			\mbox{$\vecx \leq \vecz$} \\
	$\langle \vecx, \vecz \rangle$
		& scalar product for vectors~$\vecx,\vecz \in \R^m = \ell_2^m$\\
	$\ell_p^m$
		& $\R^m$ with the norm
			\mbox{$\|\vecx\|_p := \left(\sum_{i=1}^m |x_i|^p\right)^{\frac{1}{p}}$}
			for~\mbox{$1 \leq p < \infty$},
			or \mbox{$\|\vecx\|_{\infty} := \max |x_i|$} for~\mbox{$p = \infty$},
			analogously for~\mbox{$\ell_p = \R^{\N}$}\\
	$|\vecx|_p$
		& in some contextes instead of~$\|\vecx\|_p$ for the $\ell_p^d$-norm \\
	$|\vecx|_0$
		& number of non-zero entries of~$\vecx \in \R^d$ \\
	$B_p^m$
		& unit ball within~$\ell_p^m$ \\
	$\vece_i$
		& standard basis vector in $\R^m = \ell_p^m$, or~$\R^{\N}$ \\
	$L_p(\rho)$
		& for~\mbox{$(Q,\Sigma,\rho)$} being a measure space, \mbox{$L_p(\rho)$}
			is the space of measureable functions~\mbox{$f:Q \rightarrow \R$}
			that are bounded in the norm
			\mbox{$\|f\|_p := \left(\int |f|^p \, \diff \rho \right)^{1/p}$}
			for~\mbox{$1 \leq p < \infty$},
			or~\mbox{$\|f\|_{\infty} := \esssup_{\vecx \in Q} |f(\vecx)|$}
			for~\mbox{$p = \infty$}; more precicely it is the space of equivalence
			classes of functions that are indistinguishable with respect to
			the metric induced by the norm \\
	$L_p(Q)$
		& $L_p$ on a domain~$Q \subseteq \R^d$
			with the $d$-dimensional Lebesgue measure~\mbox{$\lambda^d = \Vol$} \\
\phantom{void}& \\
\multicolumn{2}{X}{\textbf{Operators}
\bigskip}\\
	$\Linop(\R^m)$
		& set of linear operators $\R^m \rightarrow \R^m$ \\
	$\|J\|_{2 \rightarrow F}$
		& operator norm of a linear operator~\mbox{%
					$J: \ell_2^m \rightarrow \widetilde{F}$}
		between normed spaces \\
	$\rank T$
		& rank of a linear operator~$T$\\
	$\trace P$
		& trace of a linear operator~$P \in \Linop(\R^m)$\\
	$\image J$
		& image $J(\R^m)$ of an operator~$J : \R^m \rightarrow \widetilde{F}$\\
	$\widetilde{F} \hookrightarrow G$
		& identity mapping, $\widetilde{F}$~is identified with a subset of~$G$\\
\phantom{void}& \\
\multicolumn{2}{X}{\textbf{Analysis and Topology}
\bigskip}\\
	$\nabla_{\vecv} f$
		& directional derivative~\mbox{$[\nabla_{\vecv} f](\vecx)
																			:= \lim_{h \rightarrow 0}
																					\frac{f(\vecx + h \vecv) - f(\vecx)
																							}{h}$}
			of a $d$-variate function~$f$ \\
	$\partial_i f$
		& partial derivative of a $d$-variate function~$f$,
			that is~\mbox{$\nabla_{\vece_i} f$},
			into the direction~$\vece_i$ of the $i$-th coordinate \\
\phantom{void}& \\
\multicolumn{2}{X}{\textbf{Stochastics and Measure Theory}
\bigskip}\\
	$(\Omega,\Sigma,\P)$
		& suitable probability space\\
	$\expect$
		& expectation, i.e.\ the integration over~\mbox{$\omega \in \Omega$}
			with respect to~$\P$\\
	$\Phi(x)$
		& cummulative distribution function of a standard normal variable\\
	$\vecX = (X_1,\ldots,X_m)$
		& random vector \\
	$\Uniform(A)$
		& uniform distribution on a finite set~$A$ \\
	$\# A$
		& number of elements of a set~$A$
\end{tabularx}

\newpage
\renewcommand{\refname}{Bibliography}

\newpage
\cleardoublepage
\chapter*{Ehrenw\"ortliche Erkl\"arung}
\thispagestyle{empty}

Hiermit erkl\"are ich,
\begin{itemize}
	\item dass mir die Promotionsordnung der Fakult\"at bekannt ist,
	\item dass ich die Dissertation selbst angefertigt habe,
		keine Textabschnitte oder Ergebnisse eines Dritten
		oder eigenen Pr\"ufungsarbeiten ohne Kennzeichnung \"ubernommen
		und alle von mir benutzten Hilfsmittel, pers\"onliche Mitteilungen
		und Quellen in meiner Arbeit angegeben habe,
	\item dass ich die Hilfe eines Promotionsberaters nicht in Anspruch genommen habe
		und dass Dritte weder unmittelbar noch mittelbar geldwerte Leistungen von mir
		f\"ur Arbeiten erhalten haben, die im Zusammenhang mit dem Inhalt der vorgelegten
		Dissertation stehen,
	\item dass ich die Dissertation noch nicht als Pr\"ufungsarbeit f\"ur eine staatliche
		oder andere wissenschaftliche Pr\"ufung eingereicht habe.
\end{itemize}

\noindent
Bei der Auswahl und Auswertung des Materials sowie bei der Herstellung des Manuskripts
wurde ich durch Prof.~Dr.~Erich Novak unterst\"utzt.

\bigskip

\noindent
Ich habe weder die gleiche, noch eine in wesentlichen Teilen \"ahnliche
oder andere Abhandlung bei einer anderen Hochschule als Dissertation eingereicht.

\vspace*{\fill}

\noindent
Jena, 31.\ M\"arz 2017
\hfill Robert Kunsch

\end{document}